\documentclass{amsart}

\usepackage[hiresbb]{graphicx}
\usepackage{amsmath,amssymb}
\usepackage{amsthm}
\usepackage{amsfonts}
\usepackage{amscd}
\usepackage{comment}

\newtheorem{df}{Definition}[section]
\newtheorem{thm}[df]{Theorem}
\newtheorem{prop}[df]{Proposition}
\newtheorem{lemm}[df]{Lemma}
\newtheorem{cor}[df]{Corollary}

\newtheorem{rem}[df]{Remark}
\newtheorem{fact}[df]{Fact}

\newcommand{\id}{\mathrm{id}}
\newcommand{\Q}{\mathbb{Q}}
\newcommand{\R}{\mathbb{R}}
\newcommand{\Z}{\mathbb{Z}}

\newcommand{\shuugou}[1]{\{ #1 \}}
\newcommand{\zettaiti}[1]{\lvert #1 \rvert}

\newcommand{\gyaku}[1]{ #1^{-1}}

\newcommand{\skein}[1]{\mathcal{S}( #1 )}
\newcommand{\tskein}[1]{\mathcal{A}( #1 )}
\newcommand{\tzeroskein}[1]{\mathcal{A}_0 ( #1 )}

\newcommand{\kukakko}[1]{\langle #1 \rangle}
\newcommand{\defeq}{\stackrel{\mathrm{def.}}{=}}
\newcommand{\Aut}{\mathrm{Aut}}
\newcommand{\bch}{\mathrm{bch}}

\newcommand{\filtn}[1]{\{ #1 \}_{n \geq 0}}

\newcommand{\comp}[1]{\underleftarrow{\lim}_{#1 \rightarrow \infty}}

\newcommand{\gauss}[1]{\lfloor #1 \rfloor}
\newcommand{\homfly}{\mathcal{H}_\mathrm{HOMFLY-PT}}
\newcommand{\arcsinh}{\mathrm{arcsinh}}

\newcommand{\tknot}[1]{\mathcal{T} (#1)}
\newcommand{\ad}{\mathrm{ad}}

\begin{document}

\title[Completed HOMFLY-PT skein algebras]
{A formula for the action of Dehn twists on
HOMFLY-PT skein modules and its applications}
\author{Shunsuke Tsuji}
\date{}
\maketitle

\begin{abstract}
We introduce a formula for the action of Dehn twists
on the HOMFLY-PT type skein module of a surface.
As an application of the formula to mapping class group, 
we give an embedding from the Torelli group of a surface
$\Sigma_{g,1}$ of genus $g$ with non-empty connected boundary
into the completed HOMFLY-PT type skein algebra.
As an application of the formula to  integral
homology $3$-spheres, we construct an invariant 
$z(M) \in \mathbb{Q} [\rho ] [[h]]$ for an integral homology $3$-sphere $M$.
The invariant $z(M) \mod (h^{n+1})$
is a finite type invariant of order $n$.

\end{abstract}

\section{Introduction}

A skein algebra plays an important role 
in both theories of mapping class groups of surfaces
and  finite type invariants for integral homology $3$-spheres.
Actually, in our preceding papers
\cite{TsujiCSAI} \cite{Tsujipurebraid} \cite{TsujiTorelli} \cite{Tsujihom3},
Kauffman bracket skein algebras
lead to some new results concerning 
mapping class groups of surfaces and 
finite type invariants for integral homology $3$-spheres.
The Kauffman bracket skein algebra of a surface $\Sigma$
is the quotient of the $\Q [A^{\pm 1}]$-free module 
with basis the set of framed unoriented  links in $\Sigma \times I$
modulo the relations defining the Kauffman bracket.
Throughout this paper, we denote the unit interval $[0,1]$ by $I$.
In \cite{TsujiTorelli},
we introduce
an embedding of the Torelli group of a compact connected oriented surface
with non-empty connected boundary into
the completed Kauffman bracket skein algebra of the surface,
which gives a new construction of the core of the Casson invariant
defined by Morita \cite{Morita_Casson_core}.
Furthermore, in \cite{Tsujihom3},
using the embedding, we construct an invariant $z_\mathcal{S} (M) \in \Q [[A+1]]$
for an integral homology $3$-sphere $M$.
We expect the invariant $(z_\mathcal{S} (M))_{|A^4 =q}$ to equal  the Ohtsuki series 
\cite{Ohtsuki1995}.

The aim of this paper is to establish some analogues
of the results stated above
for HOMFLY-PT type skein algebras.
In this paper, we present 3 main theorems
(Theorem \ref{thm_main_intro_1}, Theorem \ref{thm_main_intro_2} and Theorem \ref{thm_main_intro_3}).
In theory of mapping class groups of surfaces,
the HOMFLY-PT type skein algebra has  information
including the Kauffman bracket skein algebra and
the Goldman Lie algebra.
We remark that
the Kauffman bracket skein algebra does not include
 information coming from the Goldman Lie algebra.
In theory of integral homology $3$-spheres,
we construct an invariant for integral homology
$3$-spheres via the HOMFLY-PT type skein algebra.
We hope the invariant
to recover all the quantum invariant via the quantum group
of $sl_n$ for any $n$. 
In this paper, a HOMFLY-PT type skein algebra
is the quotient of the $\Q [\rho] [[h]]$-free module 
with basis the set of framed oriented links in $\Sigma \times I$
modulo the relations in Definition \ref{df_skein_algebra}.
Using the relations, we  can construct some polynomial invariant
for oriented framed links
having the same information as the HOMFLY-PT polynomial.
For an oriented unframed link $L$ in $S^3$,
the HOMFLY-PT polynomial $\homfly (L)(X,h) \in 
\Q [X^{\pm 1},  h^{\pm 1}]$ is defined using the relations in the Figure
\ref{figure_HOMFLYpolynomial}.

\begin{figure}
\begin{picture}(300,140)
\put(20,70){\includegraphics[width=40pt]{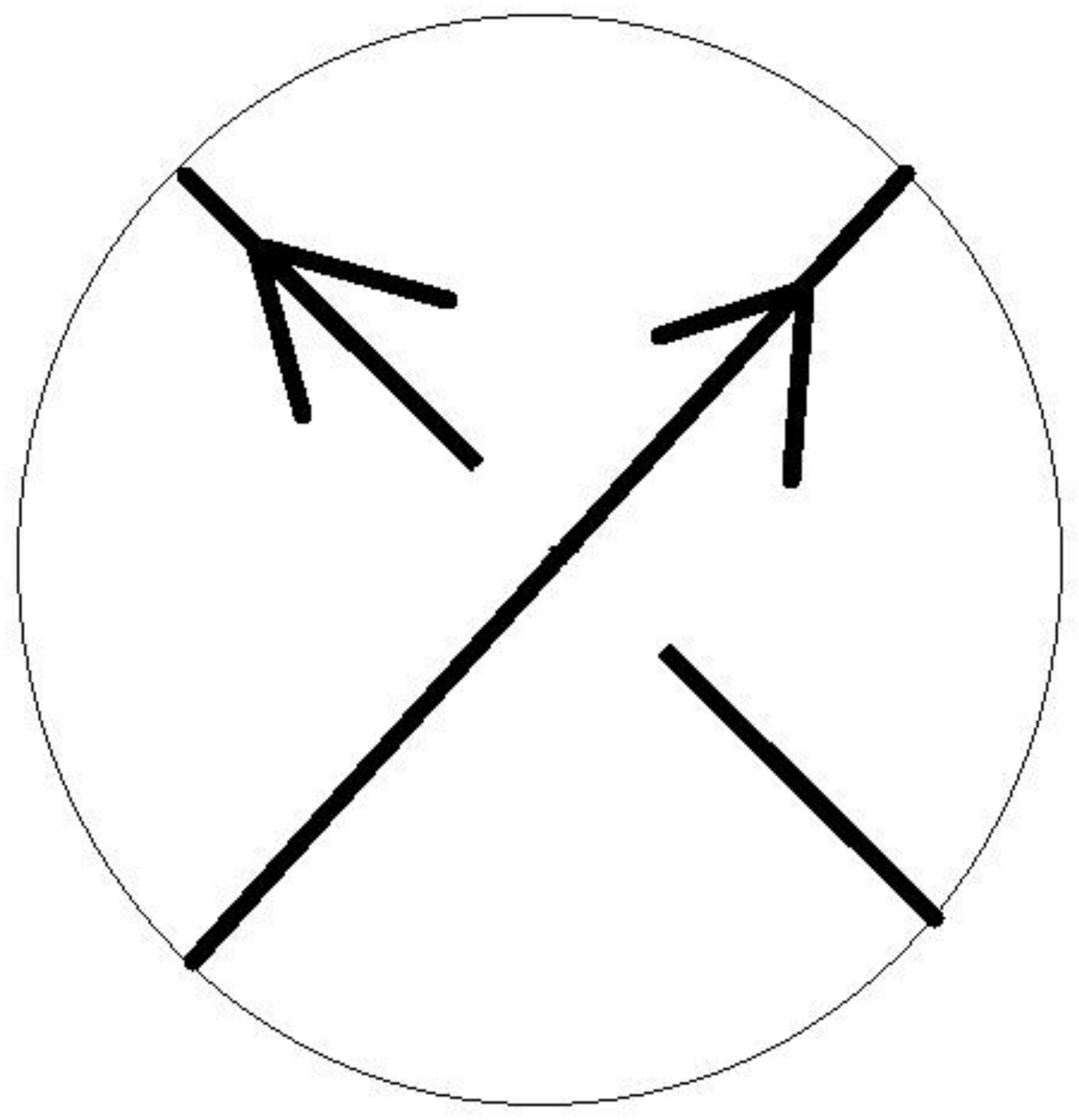}}
\put(100,70){\includegraphics[width=40pt]{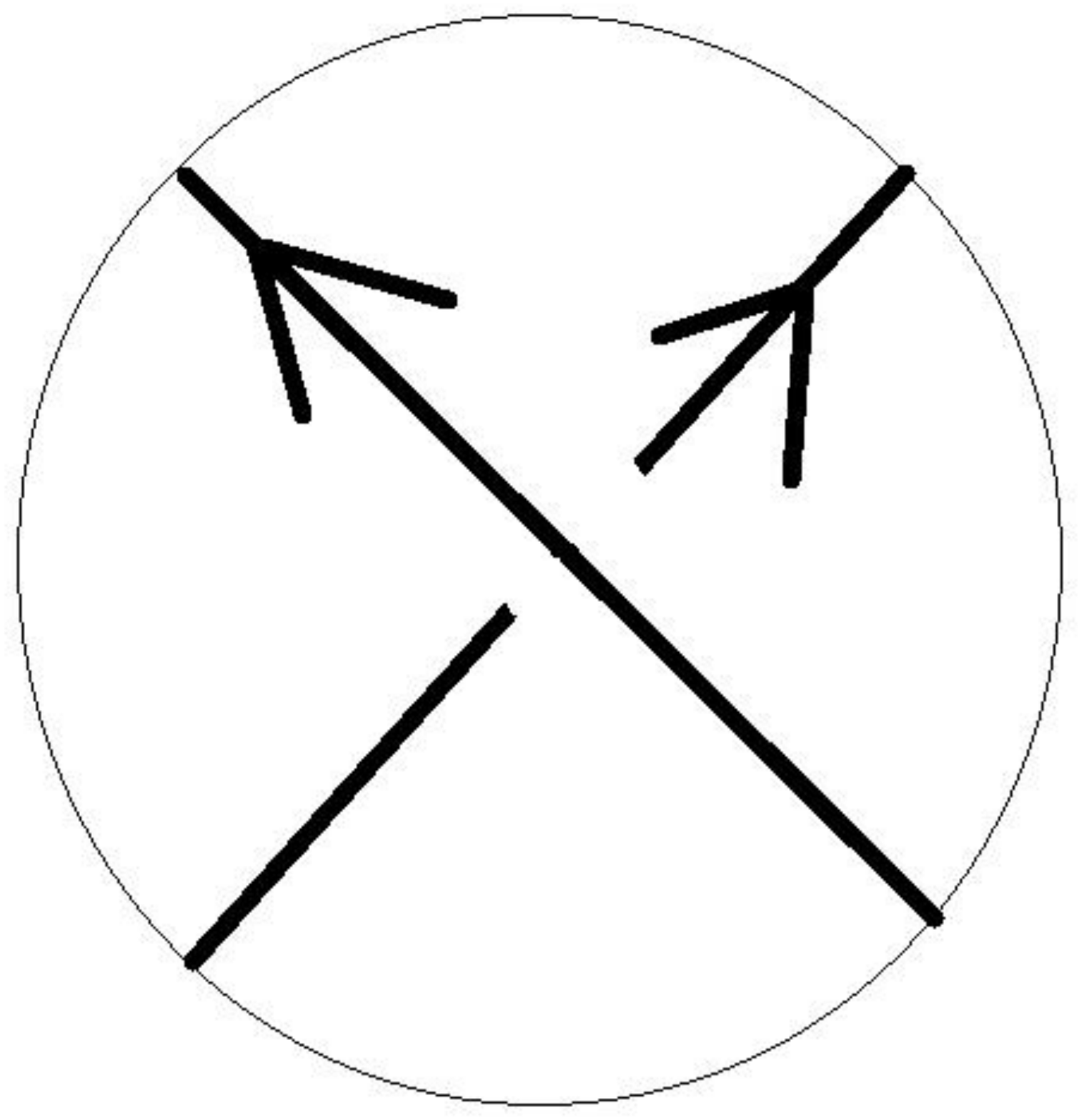}}
\put(180,70){\includegraphics[width=40pt]{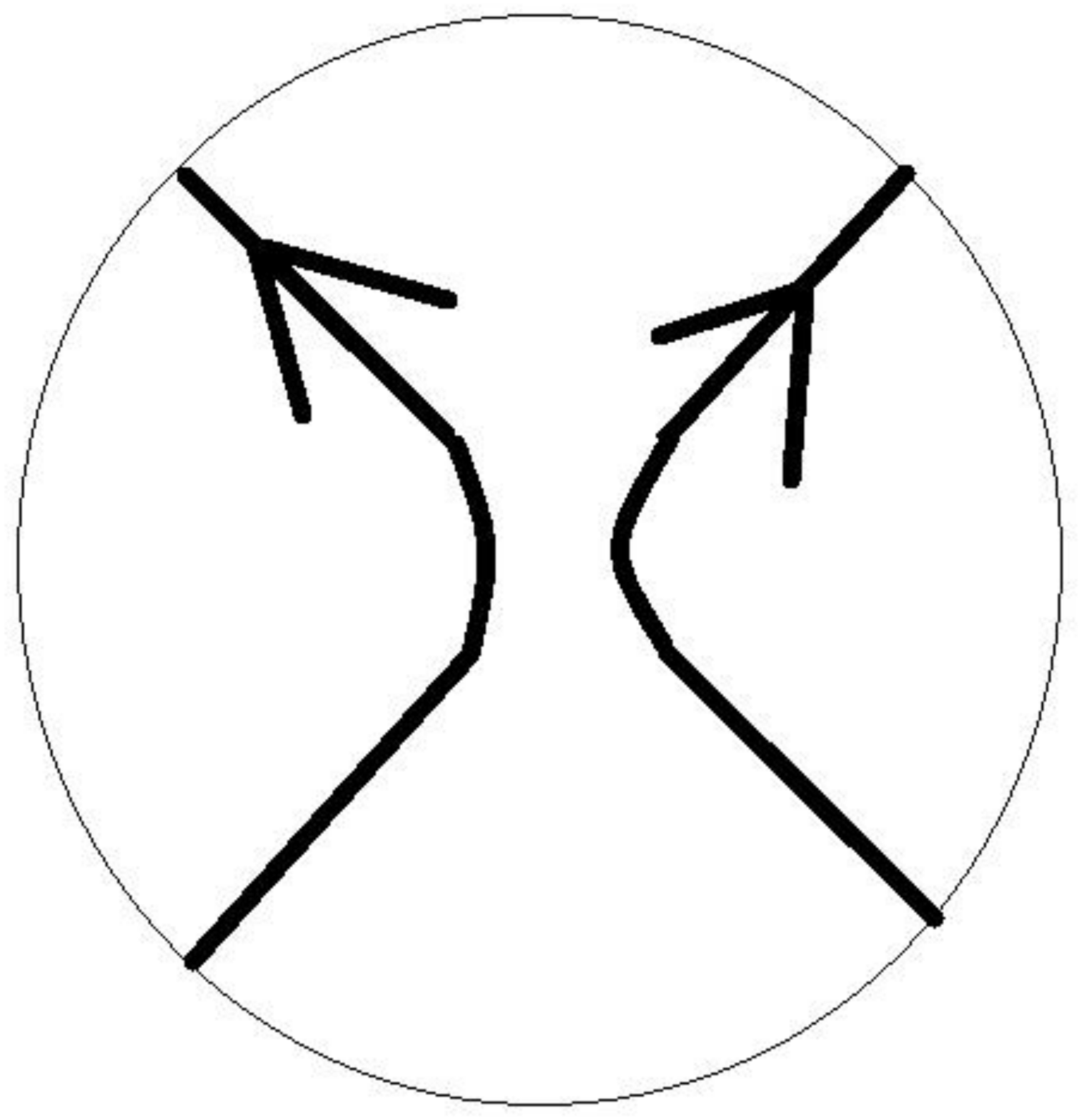}}
\put(0,90){$X$ \ \ \ \ \ \ \ \ \ \ \ \ \ \ \ \ \ \ $-\gyaku{X}$ \ \ \ \ \ \ \ \ \ \ \ \ 
 \ \ \ $=h$}
\put(0,10){\includegraphics[width=40pt]{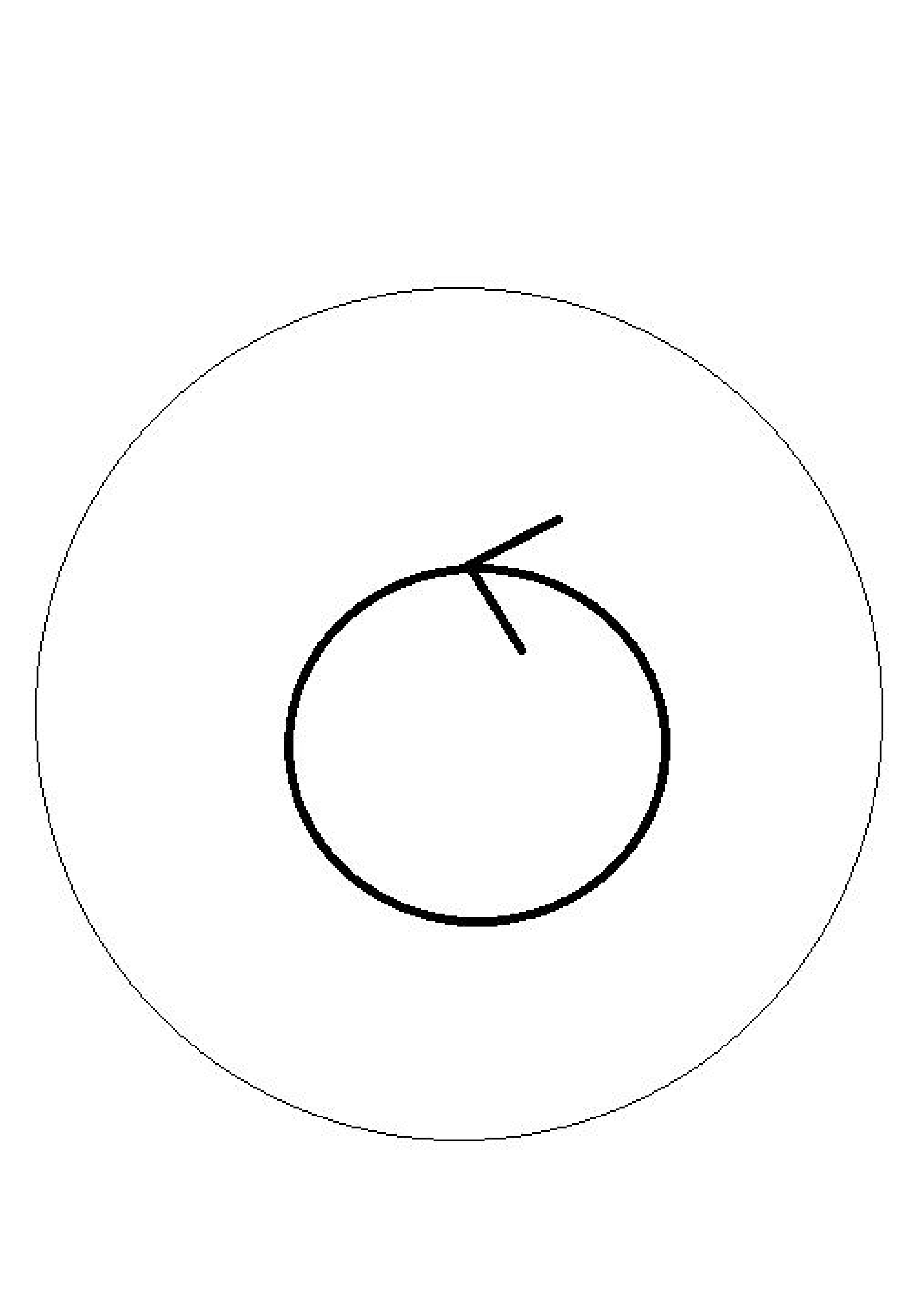}}
\put(100,10){\includegraphics[width=40pt]{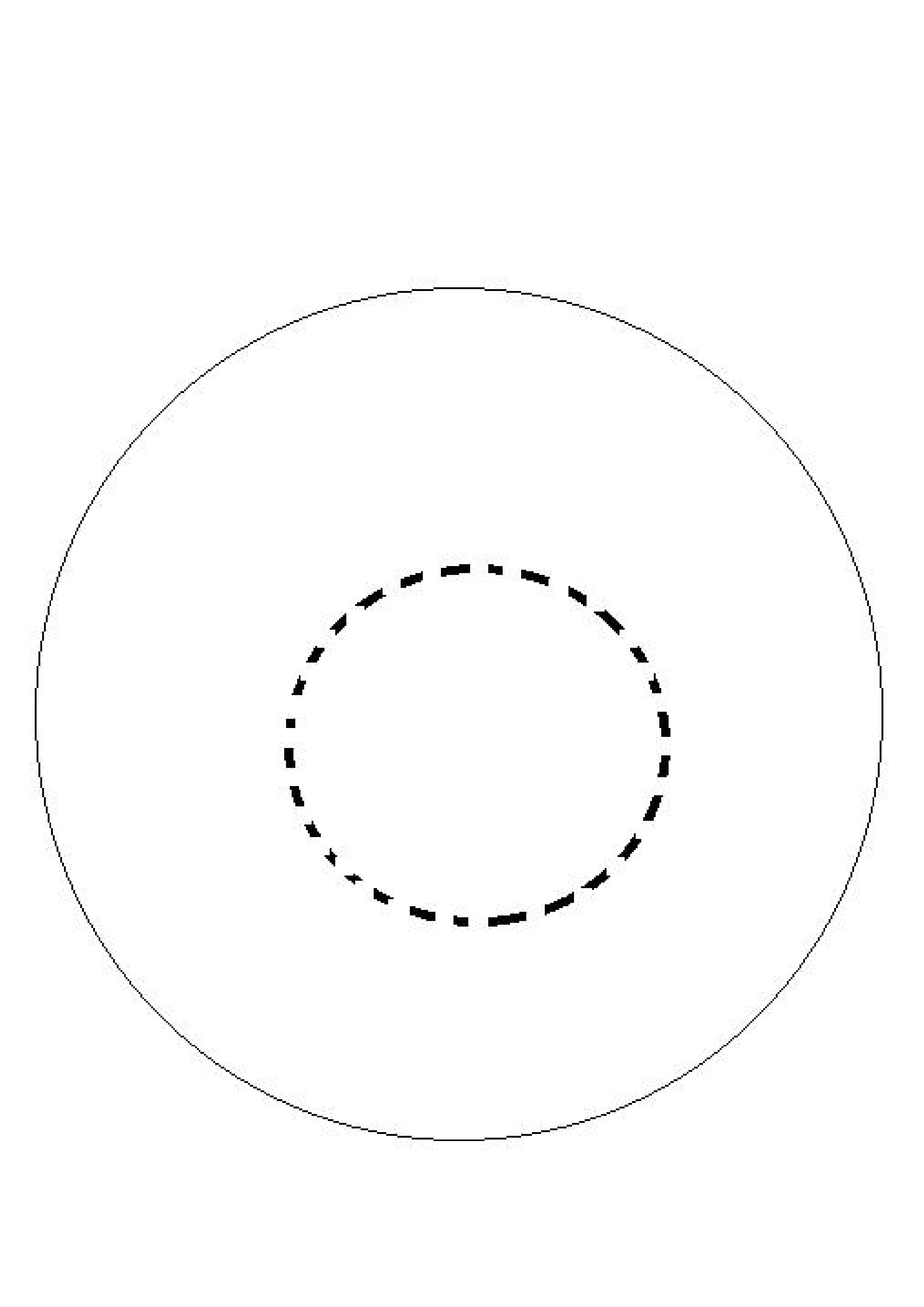}}
\put(0,30){\ \ \ \ \ \ \ \ \ \ \ \ \ \ \ $=\frac{X-\gyaku{X}}{h}$}
\end{picture}
\caption{HOMFLY polynomial}
\label{figure_HOMFLYpolynomial}
\end{figure}

Let $\Sigma$ be a compact connected oriented surface
with non-empty boundary and
$\mathcal{M} (\Sigma)$ the mapping class group
of $\Sigma$
which is
$\mathrm{Diff} (\Sigma, \partial \Sigma)/
\mathrm{Diff}_0 (\Sigma, \partial \Sigma)$
by definition.
An action $\sigma$ of the Goldman Lie algebra on
the group ring of the fundamental group of $\Sigma$ was introduced by
Kawazumi-Kuno \cite{KK}
\cite{Kawazumi}.
Through this action, Kawazumi-Kuno \cite{KK} \cite{Kawazumi} and
Massuyeau-Turaev \cite{MT} obtained a formula for the action of the 
right handed
Dehn twist $t_c \in \mathcal{M} (\Sigma)$ along a simple closed curve $c$
\begin{equation}
\label{equation_KawazumiKuno}
t_c ( \cdot )=\exp (\sigma(L_{\mathrm{Gol}}(c)))( \cdot ) :\widehat{\Q \pi_1(\Sigma, *)} \to
\widehat{\Q \pi_1 (\Sigma, *)},
\end{equation}
where $\widehat{\Q \pi_1 (\Sigma, *)}$ is the completed 
group ring of the fundamental 
group of $\Sigma$ with base point $* \in \partial \Sigma$.
Here, we denote $L_{\mathrm{Gol}} (c) \defeq
\frac{1}{2} \zettaiti{(\log(c))^2} \in \widehat{\Q \widehat{\pi_1} (\Sigma)}$ where
$\widehat{\Q \widehat{\pi_1} (\Sigma)}$ is the completed 
Goldman Lie algebra and
$\zettaiti{ \cdot } : \widehat{\Q \pi_1 (\Sigma, *)}
\to \widehat{\Q \widehat{\pi_1} (\Sigma)}$
is the quotient map.
The formula (\ref{equation_KawazumiKuno})
relates the Goldman Lie algebra
to the mapping class group.
We also consider the completion 
$\widehat{\skein{\Sigma}}$ of  the Kauffman bracket skein algebra
$\mathcal{S} (\Sigma)$,
 the completion $\widehat{\skein{\Sigma,J}}$ of the Kauffman bracket skein module
$\skein{\Sigma,J}$
with base point set $J \subset \partial \Sigma$
and an Lie action $\sigma$ of $\skein{\Sigma}$
on $\skein{\Sigma,J}$.
In \cite{TsujiCSAI}, we also obtained a formula for the action of the 
Dehn twist $t_c \in \mathcal{M} (\Sigma)$ along a simple closed curve $c$
\begin{equation}
\label{equation_Dehn_Kauff}
t_c (\cdot) = \exp ( \sigma (L_\mathcal{S} (c))) (\cdot ):
\widehat{\skein{\Sigma,J}} \to \widehat{\skein{\Sigma,J}}
\end{equation}
where $L_\mathcal{S} (c) \defeq \dfrac{-A+\gyaku{A}}{4 \log (-A)}
( \mathrm{arccosh} (-\frac{c}{2}))^2$.
The formula (\ref{equation_Dehn_Kauff})  relates
the Kauffman bracket skein algebra to the mapping class group.
As the first one of our $3$ main theorems, we establish
a HOMFLY-PT type skein algebra
version of these formulas.
\begin{thm}[Theorem \ref{thm_main_Dehn}]
\label{thm_main_intro_1}
There exists an element $\Lambda (c)$
of the completed HOMFLY-PT type skein algebra 
$\widehat{\tskein{\Sigma}}$ of $\Sigma$
for any simple closed curve $c$ satisfying
\begin{equation*}
t_c (\cdot) = \exp ( \sigma (\Lambda (c))) (\cdot ):
\widehat{\tskein{\Sigma,J^-,J^+}} \to \widehat{\tskein{\Sigma,J^-,J^+}}
\end{equation*}
where $\widehat{\tskein{\Sigma,J^-,J^+}}$
is the completion of the HOMFLY-PT type skein algebra
of $\Sigma$ with start point set and end point set $J^-$
and $J^+$.
\end{thm}

One motivation to study a HOMFLY-PT type skein algebra in this paper 
comes from mapping class group theory.
Let $\Sigma_{g,1}$ be a compact connected oriented surface of genus $g$
with nonempty connected boundary,
$(\widehat{\Q \hat{\pi} (\Sigma_{g,1})}, \filtn{F^n \widehat{\Q \hat{\pi} (\Sigma_{g,1})}})$
the completed filtered Goldman Lie algebra of $\Sigma_{g,1}$
and $(\widehat{\skein{\Sigma_{g,1}}}, \filtn{F^n \widehat{\skein{\Sigma_{g,1}}}})$
the completed filtered Kauffman bracket skein algebra of $\Sigma_{g,1}$.
In this paper, we introduce a completed filtered
HOMFLY-PT type skein algebra 
$(\widehat{\tskein{\Sigma_{g,1}}}, \filtn{F^n \widehat{\tskein{\Sigma_{g,1}}}})$
of the surface.
We denote by $\mathcal{I} (\Sigma_{g,1})$ the Torelli group,
that is
the kernel of the action of $\mathcal{M} (\Sigma_{g,1})$
on $H_1 (\Sigma_{g,1},\Z)$.
Since $[F^n \mathfrak{g},F^m \mathfrak{g}] \subset
F^{n+m-2} \mathfrak{g}$ for $\mathfrak{g} = \widehat{\Q \hat{\pi} (\Sigma)},
\widehat{\skein{\Sigma}}, \widehat{\tskein{\Sigma}}$,
we can consider $F^3 \mathfrak{g}$ as a group whose
group law is the Baker-Campbell-Hausdorff series.
There exist embeddings
\begin{align*}
&\zeta_\mathrm{Gol} : \mathcal{I} (\Sigma_{g,1}) 
\to F^3 \widehat{\Q \hat{\pi} (\Sigma_{g,1})}, \\
&\zeta_\mathcal{S} :\mathcal{I} (\Sigma_{g,1})
\to F^3 \widehat{ \skein{\Sigma_{g,1}}}.
\end{align*}
We define some filtrations $\filtn{\mathcal{I}_g(n)}$,
$\filtn{\mathcal{M}_g (n)}$ and $\filtn{\mathcal{M}^\mathcal{S}_g (n)}$
by $\mathcal{I}_g (0) \defeq \mathcal{I} (\Sigma_{g,1})$,
$\mathcal{I}_g (n)=[\mathcal{I}_g(n-1), \mathcal{I} (\Sigma_{g,1})]$,
$\mathcal{M}_g(n)=\zeta_\mathrm{Gol}^{-1} (F^{n+3} \widehat{\Q \hat{\pi} (\Sigma_{g,1})})$
and $\mathcal{M}^\mathcal{S}_g (n)
=\zeta_\mathcal{S}^{-1} (F^{n+3} \widehat{\skein{\Sigma_{g,1}}}$.
We remark that $\filtn{\mathcal{M}_g (n)}$ is the Johnson filtration.
Then we have $\mathcal{I}_g (n) \subset \mathcal{M}_g(n)$ and
$\mathcal{I}_g(n) \subset \mathcal{M}^\mathcal{S}_g (n)$.
We obtain $\mathcal{M}_g (1) = \mathcal{M}^\mathcal{S}_g(1)$
but $\mathcal{M}_g^\mathcal{S} (2) 
\not\supset \mathcal{M}_g(2)$ and
$\mathcal{M}_g^\mathcal{S} (2) 
\not\subset \mathcal{M}_g(2)$.
So we expect that there exists  a filtration
$\filtn{\mathcal{M}'_g (n)}$
of $\mathcal{I} (\Sigma_{g,1})$ satisfying
$\mathcal{I}_g (n) \subset \mathcal{M}'_g (n)
\subset \mathcal{M}_g^\mathcal{S} (n) 
\cap \mathcal{M}_g(n)$ for any $n$.
In Section \ref{section_Torelli},
we construct the following embedding
as the second main theorem.
\begin{thm}[Theorem \ref{thm_main_embedding}, Corollary \ref{cor_main_embedding}]
\label{thm_main_intro_2}
There is an injective group homomorphism 
$\zeta_\mathcal{A} : \mathcal{I} (\Sigma)
\to I \tskein{\Sigma}$ defined by 
$\zeta (t_{c_1 c_2}) = \Lambda(c_1)-\Lambda(c_2)$
for a pair $(c_1, c_2)$ of simple closed curves 
bounding a surface.
\end{thm}
If $\mathcal{M}'_g (n) =
\mathcal{M}^\mathcal{A}_g (n) \defeq \zeta_\mathcal{A}^{-1}
(F^{n+3} \widehat{\tskein{\Sigma_{g,1}}})$, this filtration satisfies the above condition.
For details, see Subsection \ref{subsection_some_filtrations}.

Another motivation to establish a HOMFLY-PT type skein algebra version
comes from theory of quantum invariants for links and integral homology
$3$-spheres.
A number-theoretical expansion of the quantum invariants 
defined via the quantum group of $sl_2$
into power series was given by Ohtsuki \cite{Ohtsuki1995}.
The power series is called the Ohtsuki series.
For any simple Lie algebra,
the power series invariant was given by Le \cite{Le2000}.
Since the HOMFLY-PT polynomial is a universal invariant of 
quantum invariants defined via the quantum group via $sl_n$ for links in $S^3$,
we expect that there exists a power series invariant
for integral homology $3$-spheres
which is a universal invariant of power series invariant
defined via the quantum group of $sl_n$.
Using the HOMLY-PT type skein algebra, we also
construct a power series invariant
$z_\mathcal{A} (M) \in \Q [\rho ] [[h]]$
for an integral homology $3$-spheres $M$.
\begin{thm}[Theorem \ref{thm_main}]
\label{thm_main_intro_3}
Fix a Heegaard splitting $H_g^+ \cup_\iota H_g^-$ of $S^3$
where $H_g^+$ and $H_g^-$ are handle bodies and
$\iota:\partial H_g^+ \to \partial H_g^-$ is a diffeomorphism.
Let $e$ be a standard embedding $\Sigma \times I \to S^3$,
which induces $e_*:\widehat{\tskein{\Sigma_{g,1}}} \to \Q[\rho][[h]]$.
The map $Z: \mathcal{I}(\Sigma_{g,1}) \to \Q[\rho][[h]]$
defined by 
\begin{equation*}
Z(\xi) \defeq \sum_{i=0}^\infty \frac{1}{(h)^i i!}e_*
((\zeta (\xi))^i)
\end{equation*}
induces an invariant for integral homology $3$-spheres
\begin{equation*}
z:\mathcal{H} (3) \to \Q[\rho][[h]], H_g^+ \cup \iota \circ \xi H_g^- \to Z(\xi).
\end{equation*}
Here we denote by $\mathcal{H} (3)$ the set of integral homology $3$-spheres.
\end{thm}
We hope that $z_\mathcal{A}$ is a universal invariant
of power series invariant 
for integral homology $3$-spheres
defined via the quantum group of $sl_n$.

\tableofcontents

\section*{Acknowledgements}

The author is very grateful to his adviser, Nariya Kawazumi, for helpful discussion
and encouragement.
He thanks to Kazuo Habiro, Yusuke Kuno, Gw\'{e}na\"{e}l Massuyeau,
Jun Murakami, Tomotada Ohtsuki and Masatoshi Sato for helpful comments.
This work was supported by JSPS KAKENHI Grant Number 15J05288 
and the Leading Graduate Course for Frontiers of Mathematical Sciences and Phsyics.

\section{Definition of oriented framed tangles in $\Sigma \times I$}

In this section, let $\Sigma $ be a compact connected oriented surface.

We define $\tilde{\mathcal{E}} (\Sigma)$ to be the set consisting
of all the embedding map $E =E_{S^1} \sqcup E_{I}:
(\coprod_\mathrm{finite} S^1 ) \times I \sqcup  (\coprod_\mathrm{finite} I) \times I
 \to  \Sigma \times I$
satisfying the following.

\begin{enumerate}
\item The embedding satisfies
$
E_I ( (\coprod_\mathrm{finite} \shuugou{0,1}) \times I)
\subset p_1 \circ E_I ( (\coprod_\mathrm{finite} \shuugou{0,1}) \times 
\shuugou{0}) \times I \subset \partial \Sigma \times I
$
where $p_1 :\Sigma \times I \to \Sigma$ is the first projection.
\item The embedding $E_{I|( \coprod_\mathrm{finite} \shuugou{0,1})
\times I} :( \coprod_\mathrm{finite} \shuugou{0,1})
\times I \to p_1 \circ E_I ( (\coprod_\mathrm{finite} \shuugou{0,1}) \times 
\shuugou{0}) \times I$ preserves the orientation.
\item The map $p_1 \circ E_{I|( \coprod_\mathrm{finite} \shuugou{0,1})
\times \shuugou{0}} $ is injective.
\end{enumerate}

Two elements $E_0$ and $E_1$ of $\tilde{\mathcal{E}} (\Sigma)$
are isotopic if there exists a continuous map
$H : ((\coprod_\mathrm{finite} S^1 ) \times I \sqcup  (\coprod_\mathrm{finite} I) \times I)
 \times I \to \Sigma \times I $ satisfying the following.
\begin{enumerate}
\item For any $t \in I$, $H( \cdot , t) =H_{t ,S^1} \sqcup H_{t, I} \in 
\mathcal{E} (\Sigma)$.
\item For any $t \in I$, $p_1 \circ H_{0, I} (( \coprod_\mathrm{finite} \shuugou{0,1}
\times \shuugou{0}) = p_1 \circ H_{t ,I} (( \coprod_\mathrm{finite} \shuugou{0,1})
\times \shuugou{0})$.
\item We have $H(\cdot , 0)=E_0$ and $H (\cdot ,1) =E_1$.
\end{enumerate}
We call an isotopic class of $\tilde{\mathcal{E}} (\Sigma)$
a framed oriented tangle in $\Sigma \times I$.
Let $T$ be a framed oriented tangle in $\Sigma \times I$
represented by $E = E_{S^1 } \sqcup E_I$.
We call the set $E_I ((\coprod_\mathrm{finite} \shuugou{0}) 
\times \shuugou{0})$ and
$E_I ((\coprod_\mathrm{finite} \shuugou{1})
\times \shuugou{0})$ the start point set of $T$
and the end point set of $T$, respectively.
Let $J^-$ and $J^+$ be disjoint subsets of $\partial \Sigma$
with $\sharp (J^-) = \sharp (J^+)$.
We denote by $\mathcal{E} (\Sigma ,J^-,J^+)$
the subset of $\tilde{\mathcal{E}}$
consisting of all elements 
 representing oriented framed tangles
whose start point sets and end point sets are
$J^-$ and $J^+$, respectively.
We denote by $\mathcal{T}(\Sigma, J^-,J^+)$ 
the set of 
isotopic
classes of elements of $\mathcal{E} (\Sigma, J^-,J^+)$,
in other words,
the set of framed oriented tangles with start point set 
and end point set $J^-$ and $J^+$, respectively.
We denote by $\kukakko{\cdot}$ the quotient map $\mathcal{E} (\Sigma, J^-,J^+) \to 
\mathcal{T}(\Sigma,J^-,J^+)$.
If $J^- =J^+ =\emptyset$, we simply denote $\mathcal{T}(\Sigma,J^-,J^+)$ and 
$\mathcal{E}(\Sigma, J^-,J^+)$ by 
$\mathcal{T}(\Sigma)$ and $\mathcal{E}(\Sigma)$.

The definition of `tangles' is similar to the definition of `link'
of marked surfaces in \cite{Mu2012}.
But a tangle in this definition has one arc on each point of $J^- \cap J^+$.

\begin{df}

Let $J^+$ and $J^-$ be  disjoint finite subsets of $\partial \Sigma $
with $\sharp (J^+) =\sharp (J^-)$.
An element $E$ of $\mathcal{E}(\Sigma, J^-,J^+)$ is
generic if $E :(\coprod_\mathrm{finite} S^1 ) \times I \sqcup  
(\coprod_\mathrm{finite} I) \times I \to \Sigma \times I$
 satisfies the following.
 
 \begin{enumerate}
 \item For $x \in \coprod_\mathrm{finite} S^1$
or $x \in \coprod_\mathrm{finite} I$,
 $I \to I, t \mapsto p_2 \circ E(x,t)$ is an orientation preserving
 embedding map where $p_2:\Sigma \times I \to I$ is the second
projection.
 \item $(\coprod_\mathrm{finite} S^1)
 \sqcup (\coprod_\mathrm{finite} I) \to \Sigma, x \mapsto p_1 \circ
E (x,0)$ is an immersion such that the intersections of the image consist of 
transverse double points
where $p_1: \Sigma \times I \to \Sigma$ is the first projection.
\end{enumerate}

 \end{df}

It is convenient to present tangles in $\Sigma \times I$
by tangle diagrams on $\Sigma$ in the same fashion
in which links in $\R^3$ may be presented by planar link diagrams.

\begin{df}
Let $J^+$ and $J^-$ be  disjoint finite subsets of $\partial \Sigma $
with $\sharp (J^+) =\sharp (J^-)$,
$T$ an element of $\mathcal{T}(\Sigma,J^-,J^+)$ and 
$E$ an element of $\mathcal{E}(\Sigma, J^-,J^+)$
which is generic and representing $T$.
The tangle diagram of $T$ is $p_1 \circ  E ((\coprod_\mathrm{finite} S^1 ) \times I
 \sqcup  
(\coprod_\mathrm{finite} I) \times I \to \Sigma \times I)$
together with orientation and height-information, i.e., the choice of the upper branch
of the curve at each crossing.
The chosen branch is called an over-crossing; the other branch is 
called an under crossing.
We denote by $T(d)$ the element of 
$\mathcal{T}(\Sigma,J^-,J^+)$ presented by a tangle diagram $d$

\end{df}

\begin{prop}[see, for example, \cite{BuZi85}]

Let $J^+$ and $J^-$ be  disjoint finite subsets of $\partial \Sigma $
with $\sharp (J^+) =\sharp (J^-)$.
Then, $T(d)$ equals 
$T(d')$ if and only if $d$ can be transformed into $d'$ by a sequence of isotopies of 
$\Sigma$ and the Reidemeister moves RI, RII, RIII shown in Figure \ref{fig:RI},
\ref{fig:RIIi}, \ref{fig:RIIii} and \ref{fig:RIII}.
We remark that the Reidemeister move I (RI) implies that
there are two ways to describe a positive (right handed)
twist as Figure \ref{fig:RI}.

\end{prop}

\begin{figure}
\begin{picture}(300,80)
\put(0,0){\includegraphics[width=60pt]{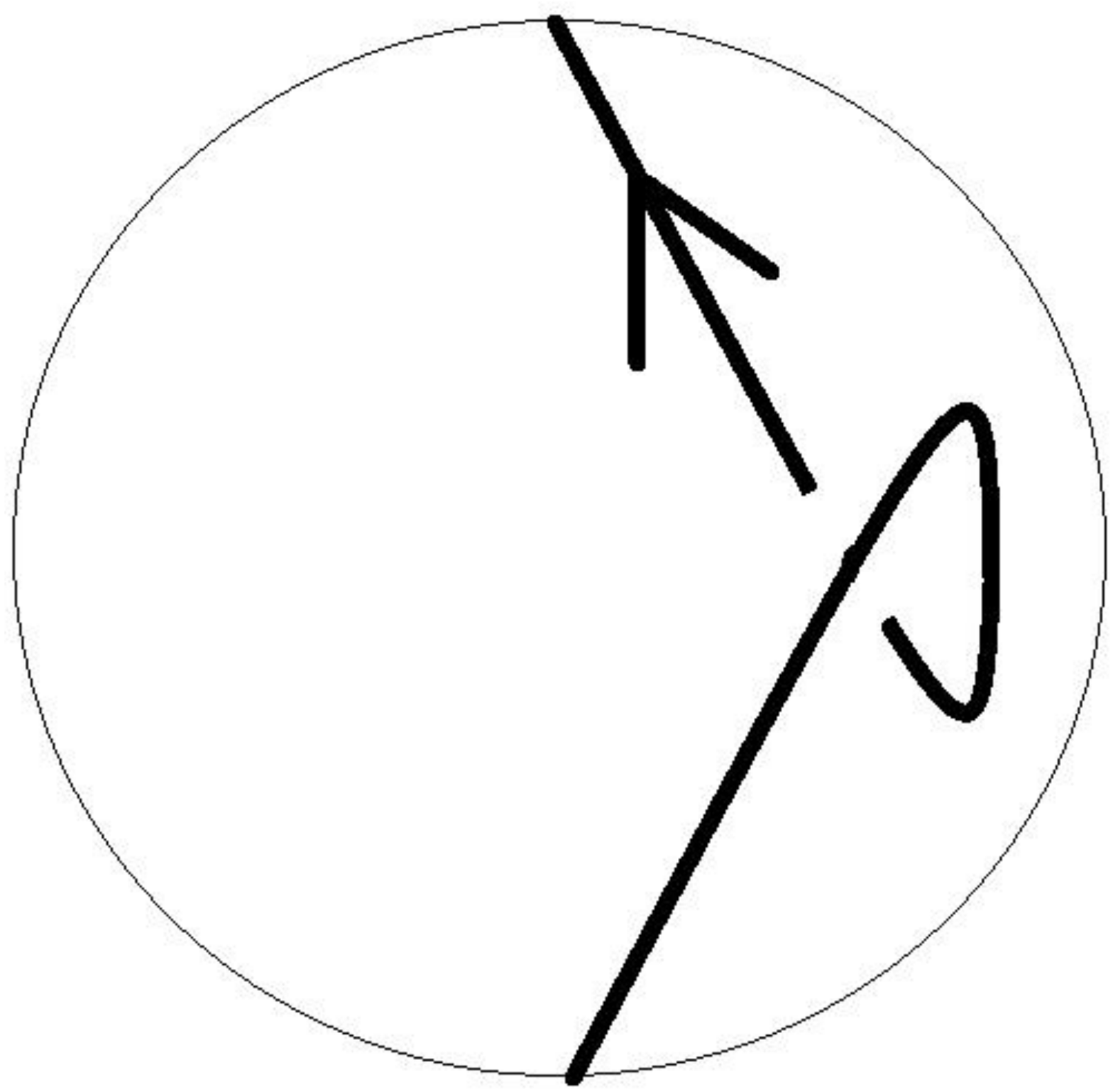}}
\put(70,0){\includegraphics[width=60pt]{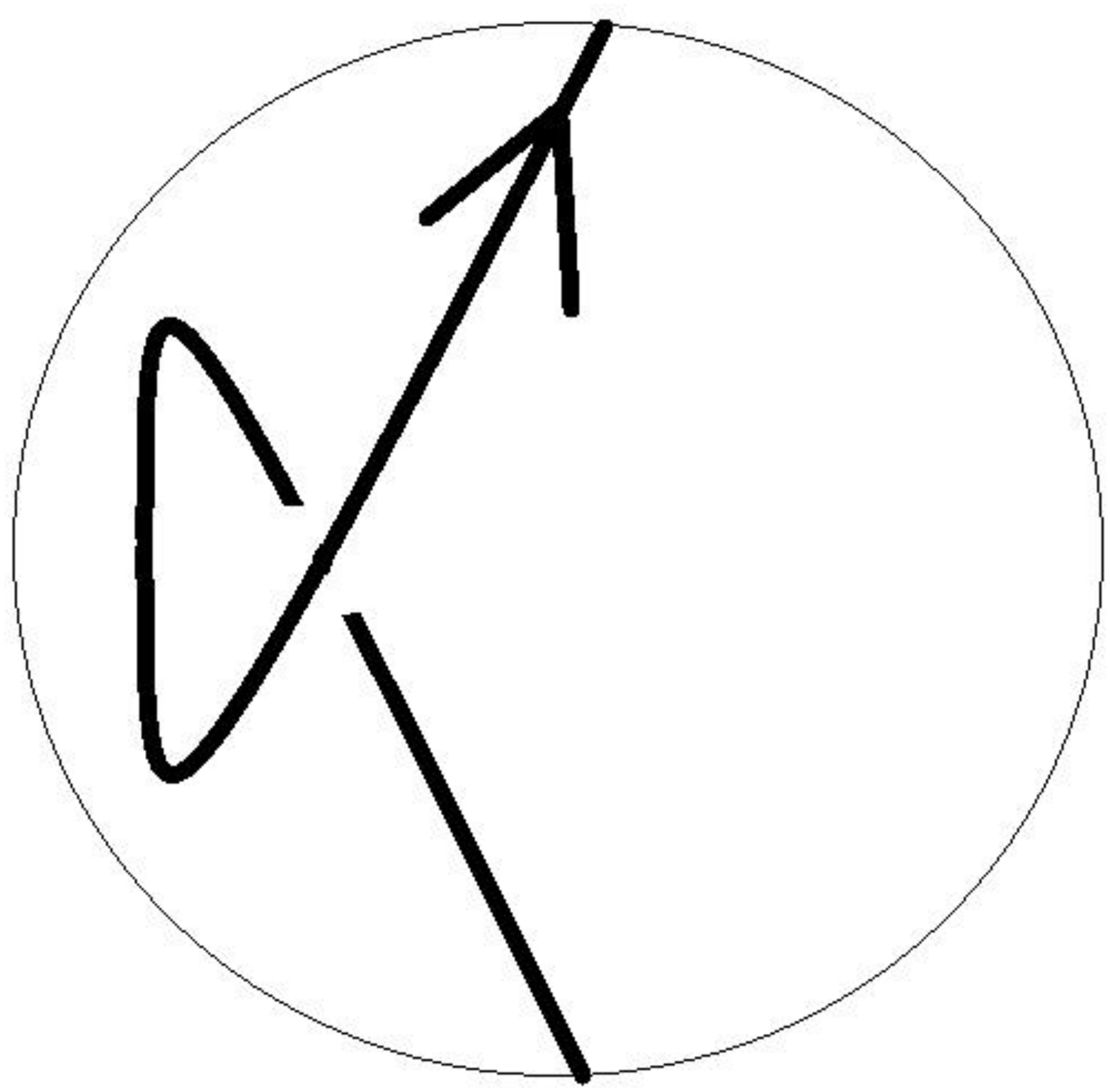}}
\put(60,30){\LARGE{$\leftrightarrow$}}
\end{picture}
\caption{RI: Reidemester move I}
\label{fig:RI}
\end{figure}

\begin{figure}
\begin{picture}(300,80)
\put(0,0){\includegraphics[width=60pt]{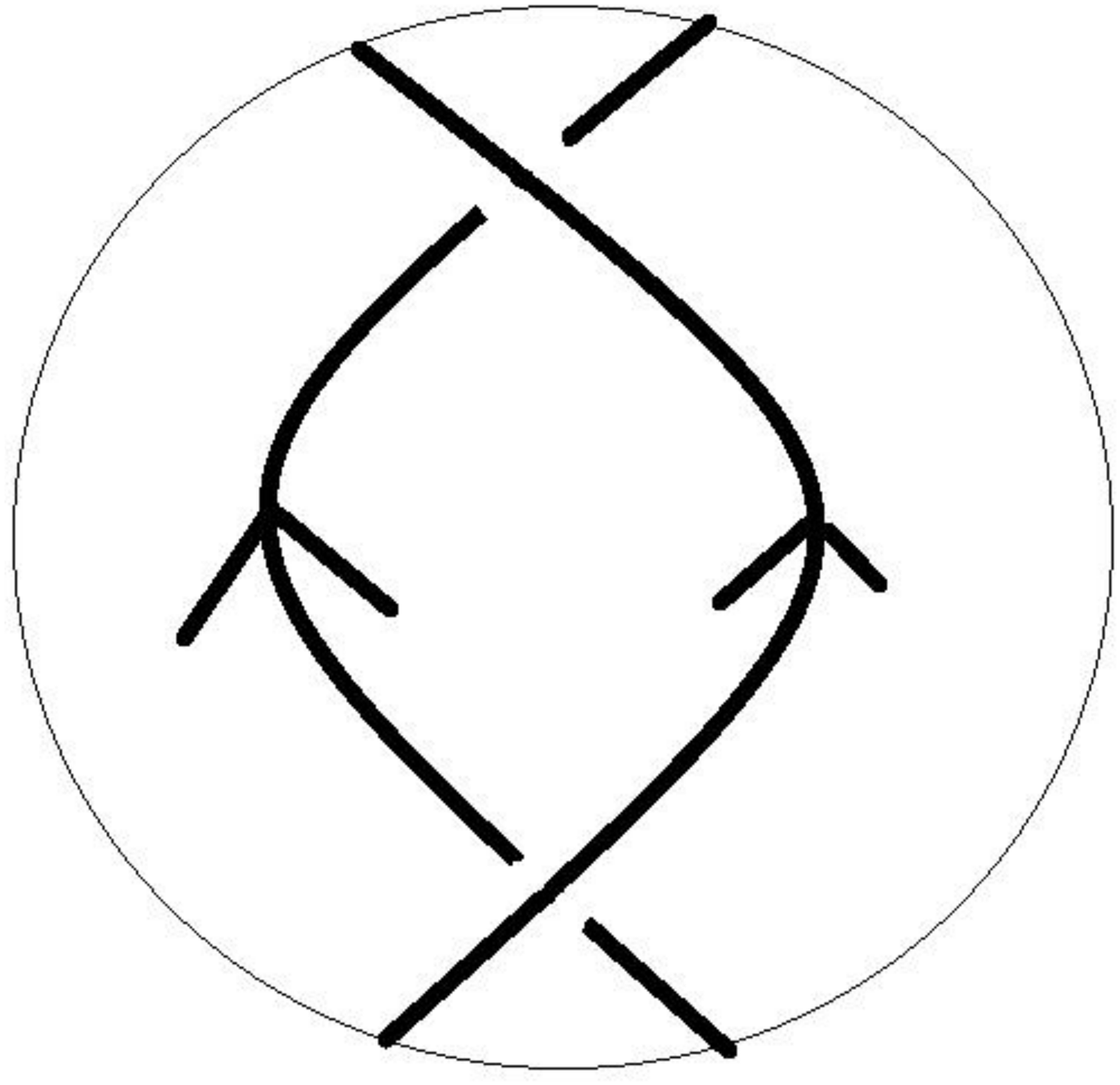}}
\put(70,0){\includegraphics[width=60pt]{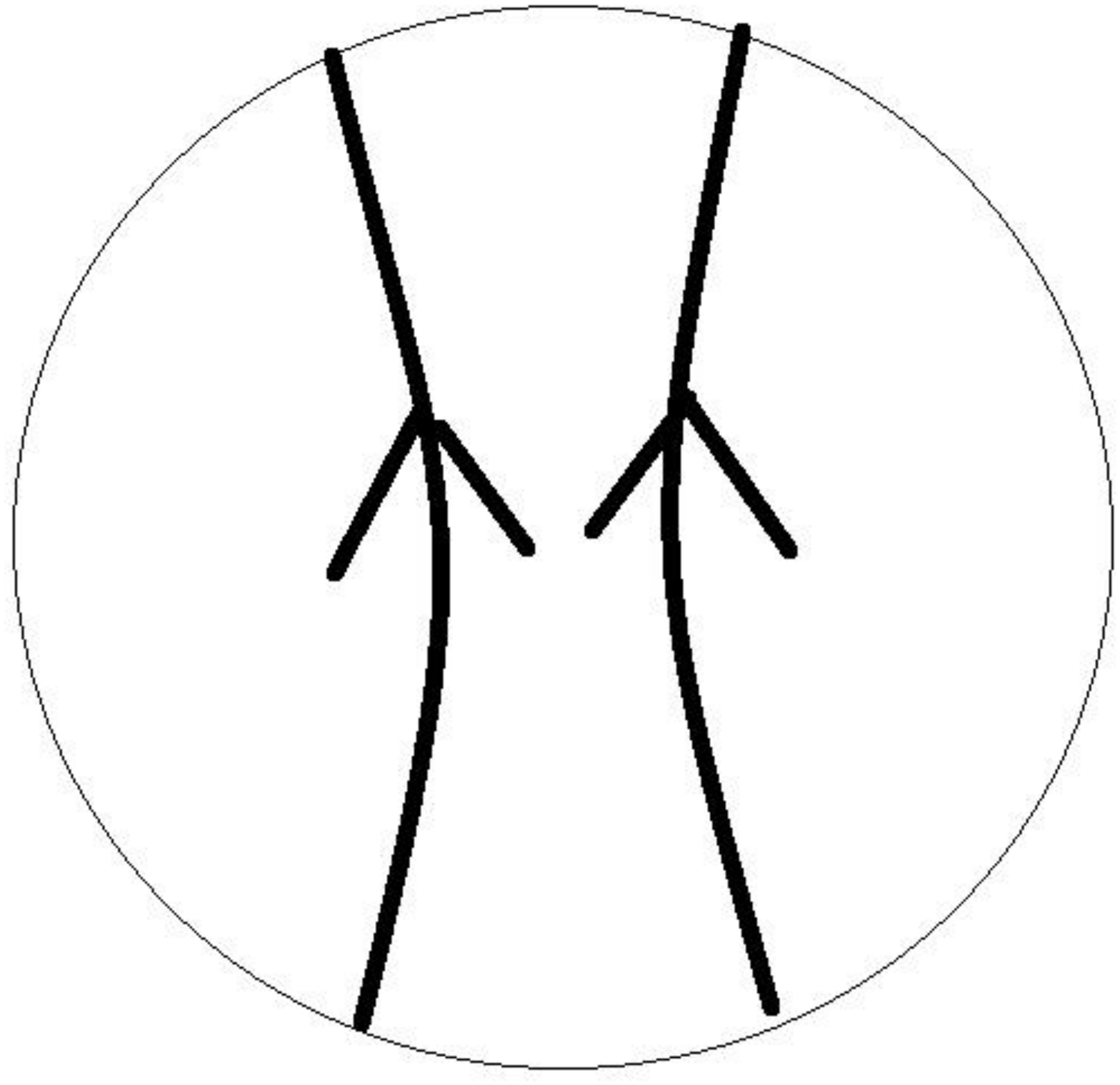}}
\put(140,0){\includegraphics[width=60pt]{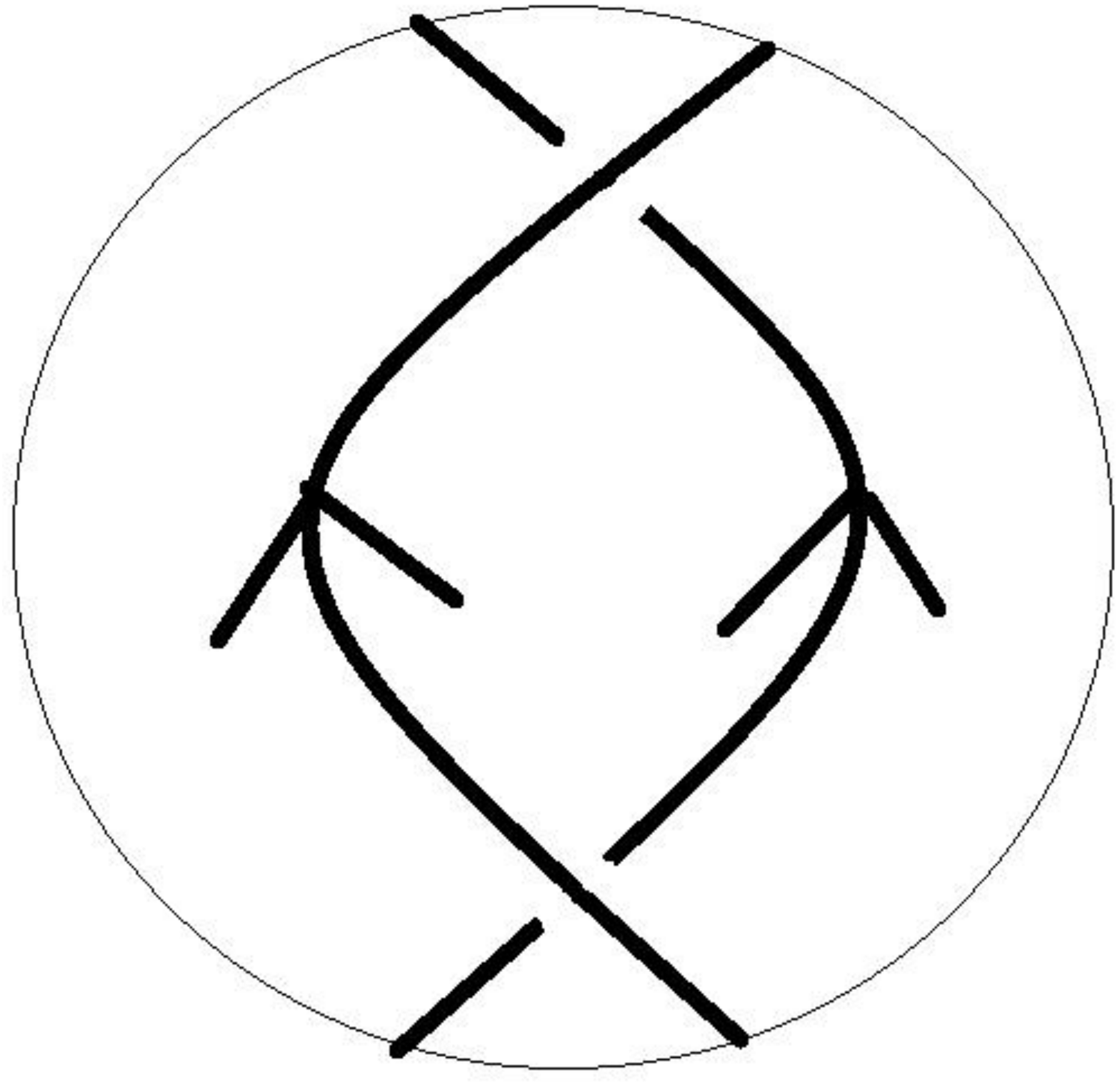}}
\put(60,30){\LARGE{$\leftrightarrow$}}
\put(130,30){\LARGE{$\leftrightarrow$}}
\end{picture}
\caption{RII: Reidemeister move II-i}
\label{fig:RIIi}
\end{figure}

\begin{figure}
\begin{picture}(300,80)
\put(0,0){\includegraphics[width=60pt]{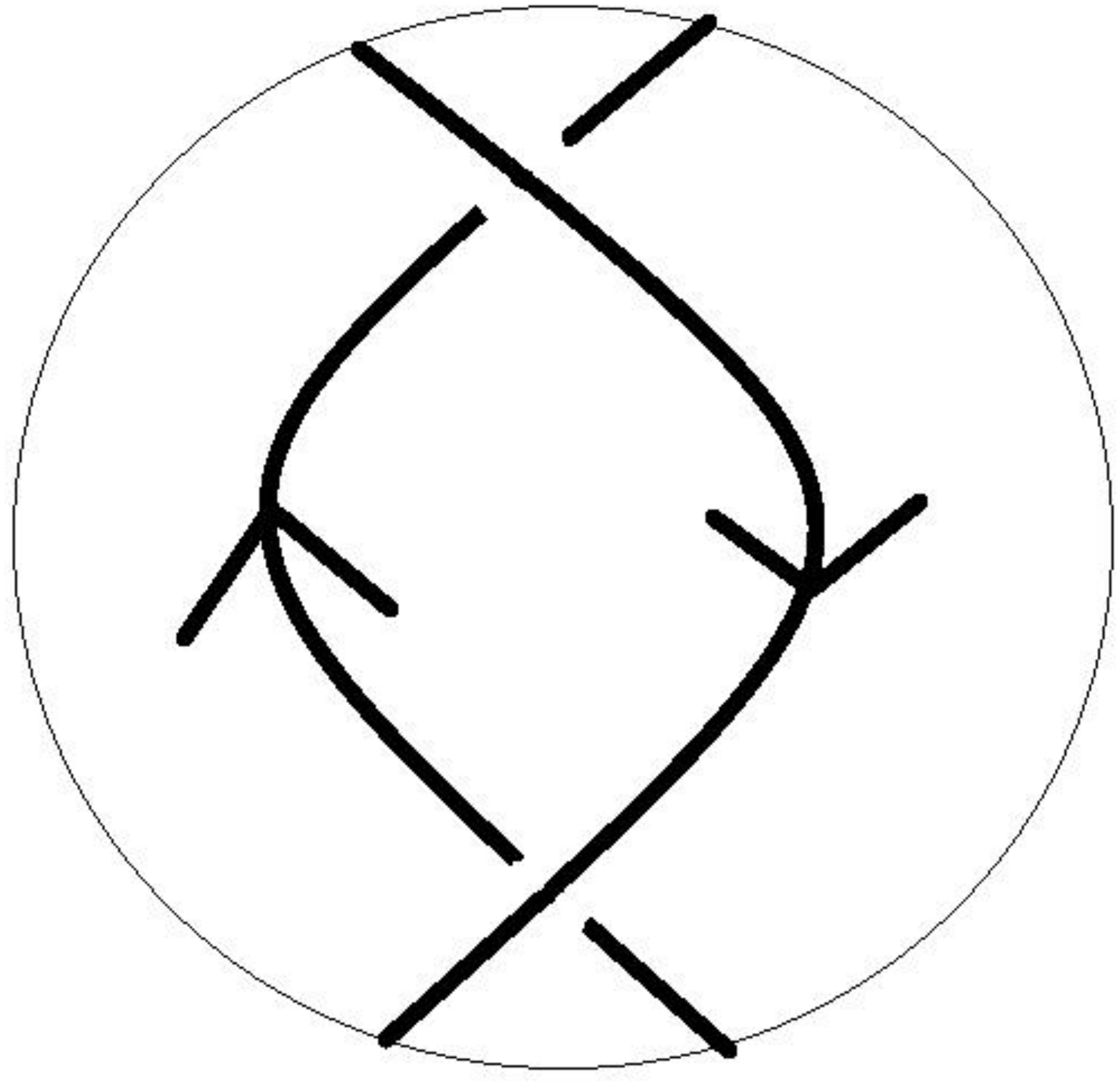}}
\put(70,0){\includegraphics[width=60pt]{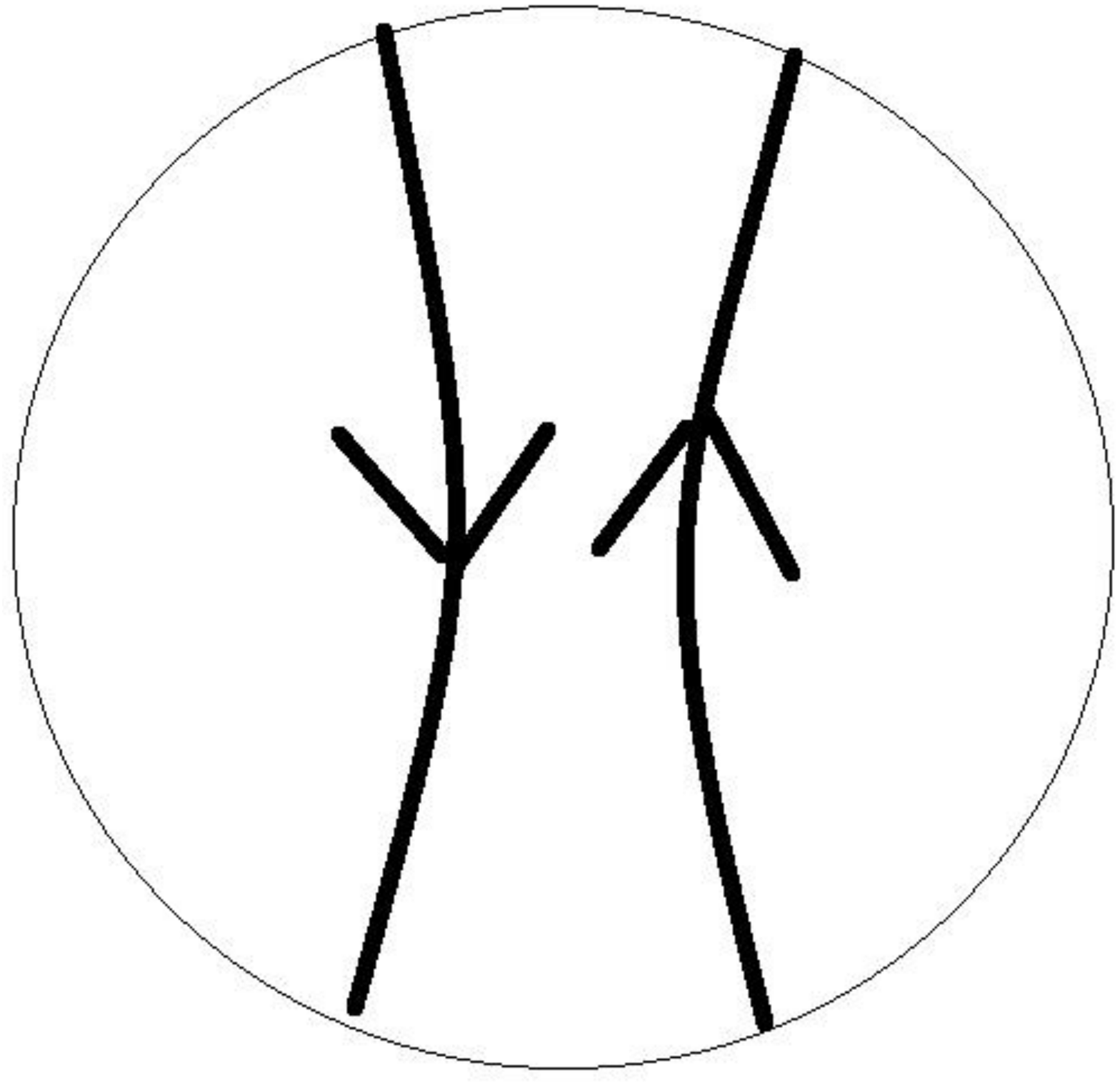}}
\put(140,0){\includegraphics[width=60pt]{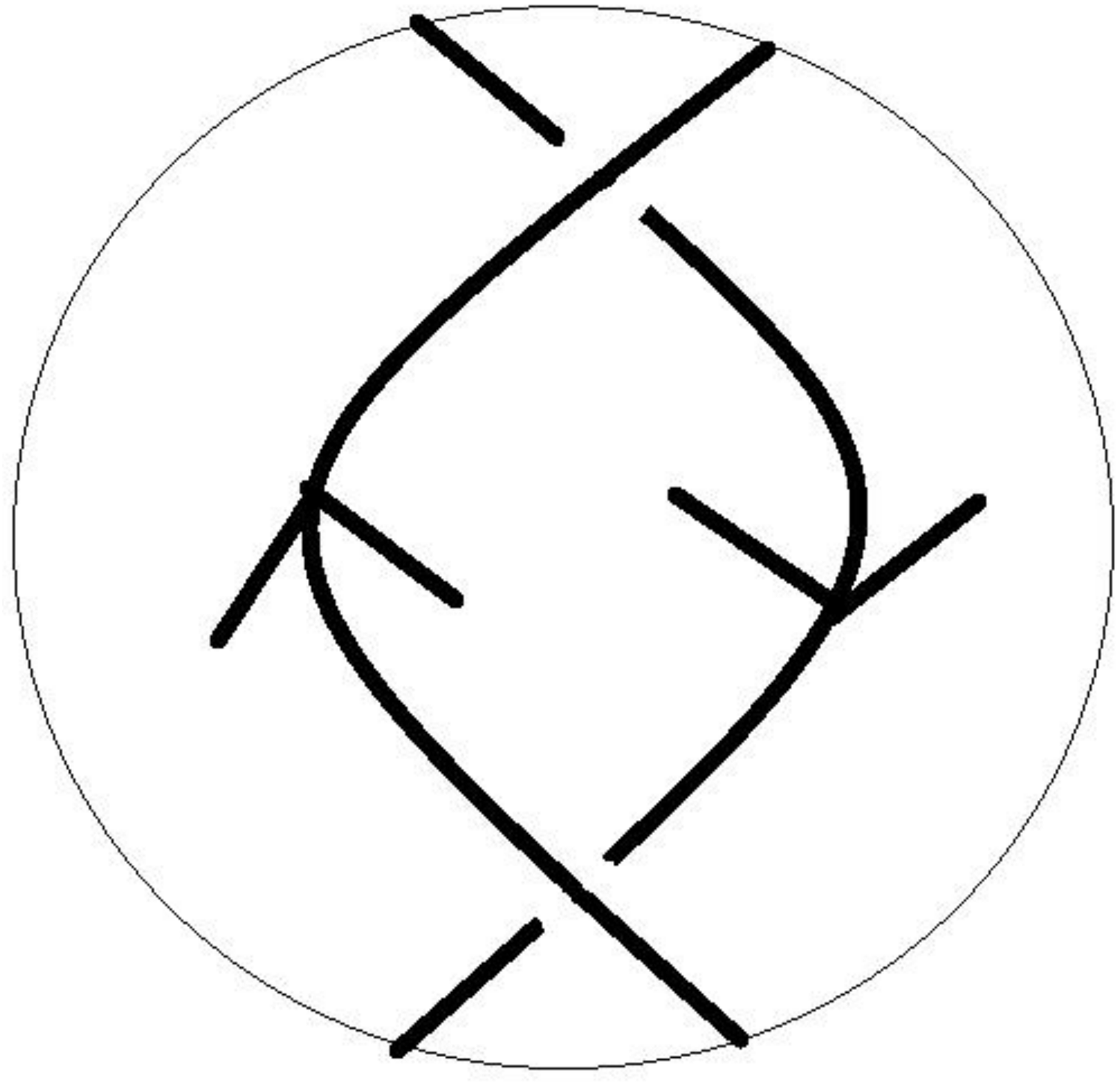}}
\put(60,30){\LARGE{$\leftrightarrow$}}
\put(130,30){\LARGE{$\leftrightarrow$}}
\end{picture}
\caption{RII: Reidemeister move II-ii}
\label{fig:RIIii}
\end{figure}

\begin{figure}
\begin{picture}(300,80)
\put(0,0){\includegraphics[width=60pt]{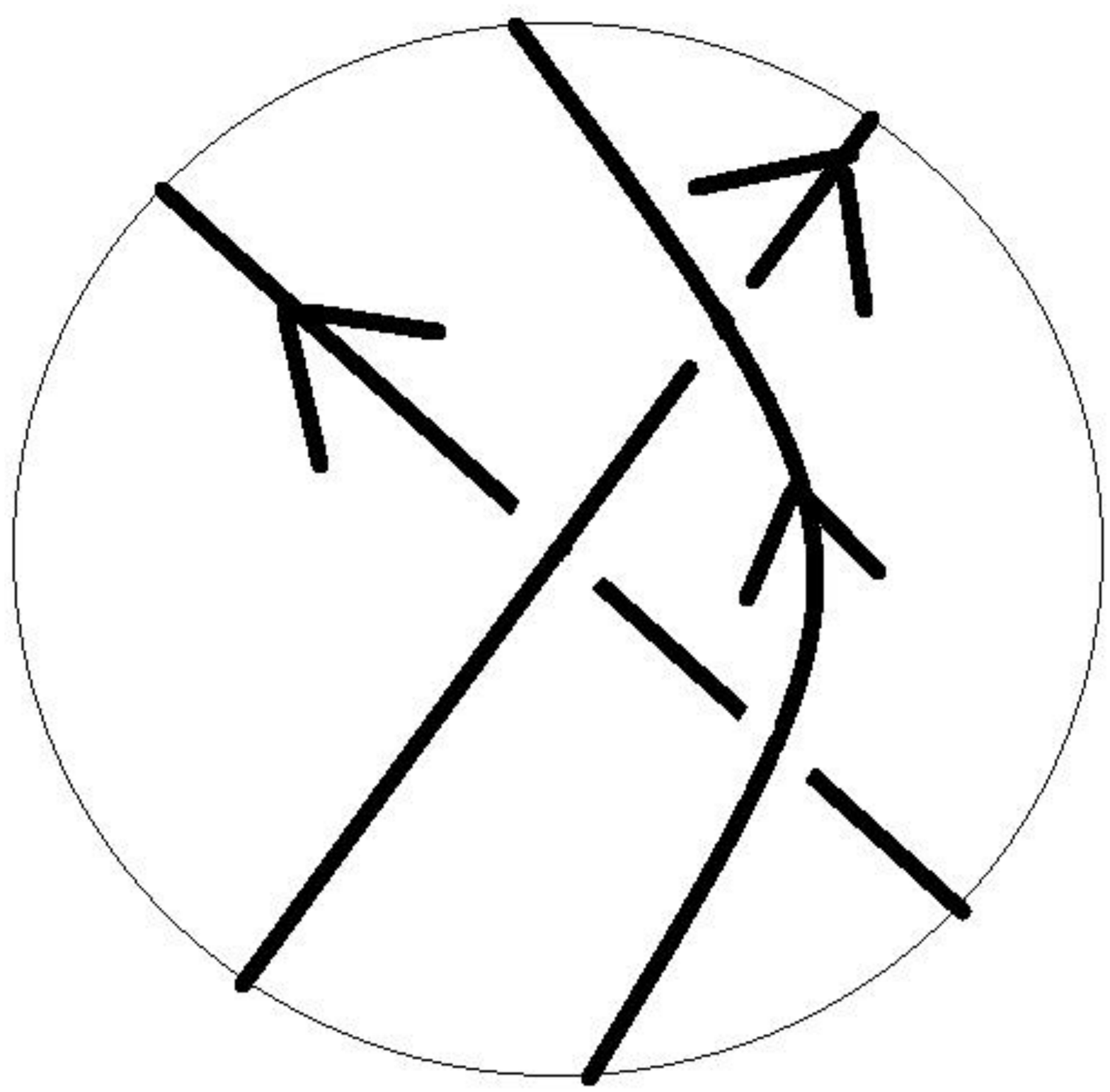}}
\put(70,0){\includegraphics[width=60pt]{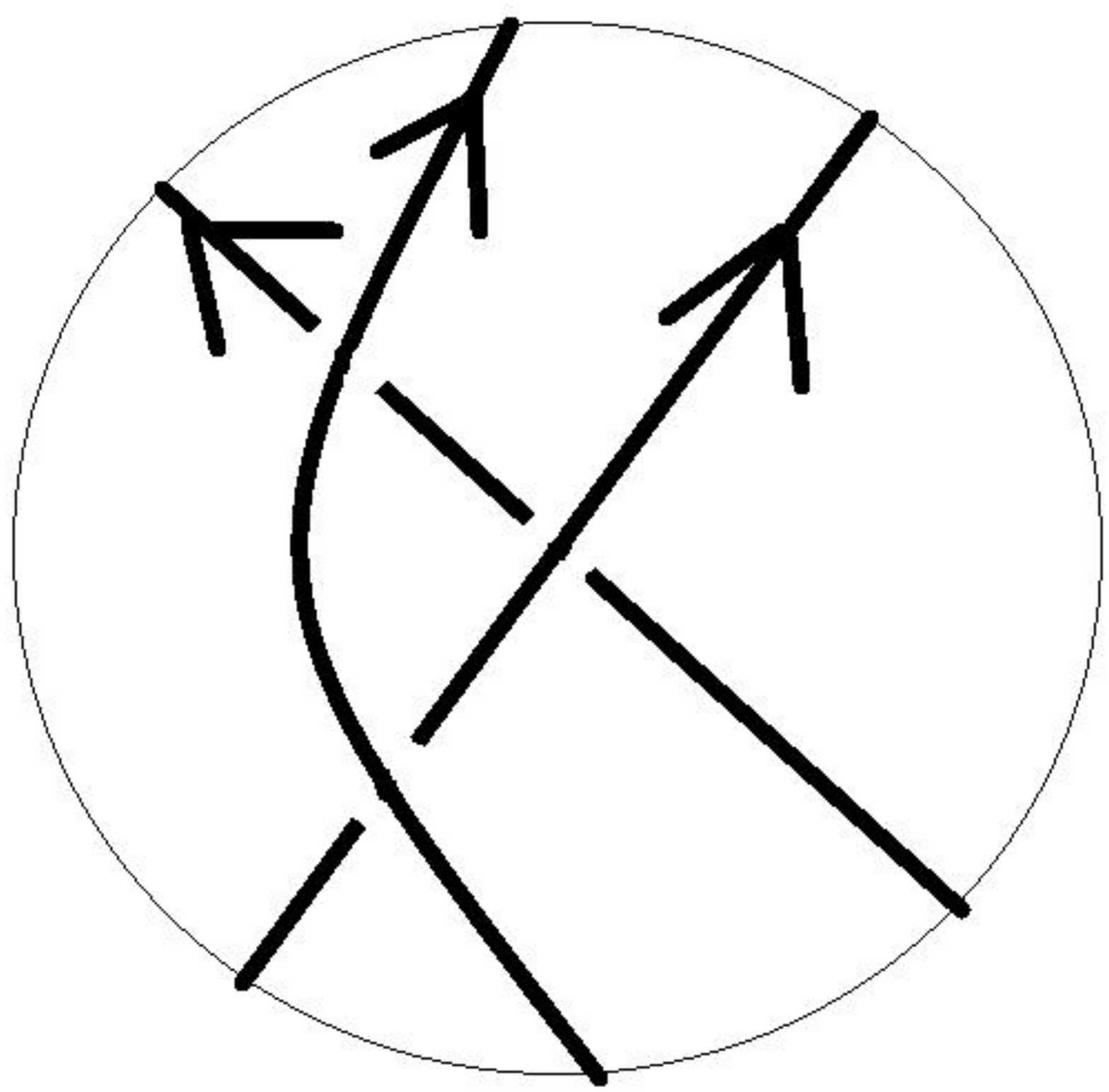}}
\put(60,30){\LARGE{$\leftrightarrow$}}
\end{picture}
\caption{RIII: Reidemeister move III}
\label{fig:RIII}
\end{figure}

We remark that these moves induce the  other Reidemeister moves
such as Figure \ref{fig:RIIIs}.

\begin{figure}
\begin{picture}(400,350)
\put(-70,0){\includegraphics[width=500pt]{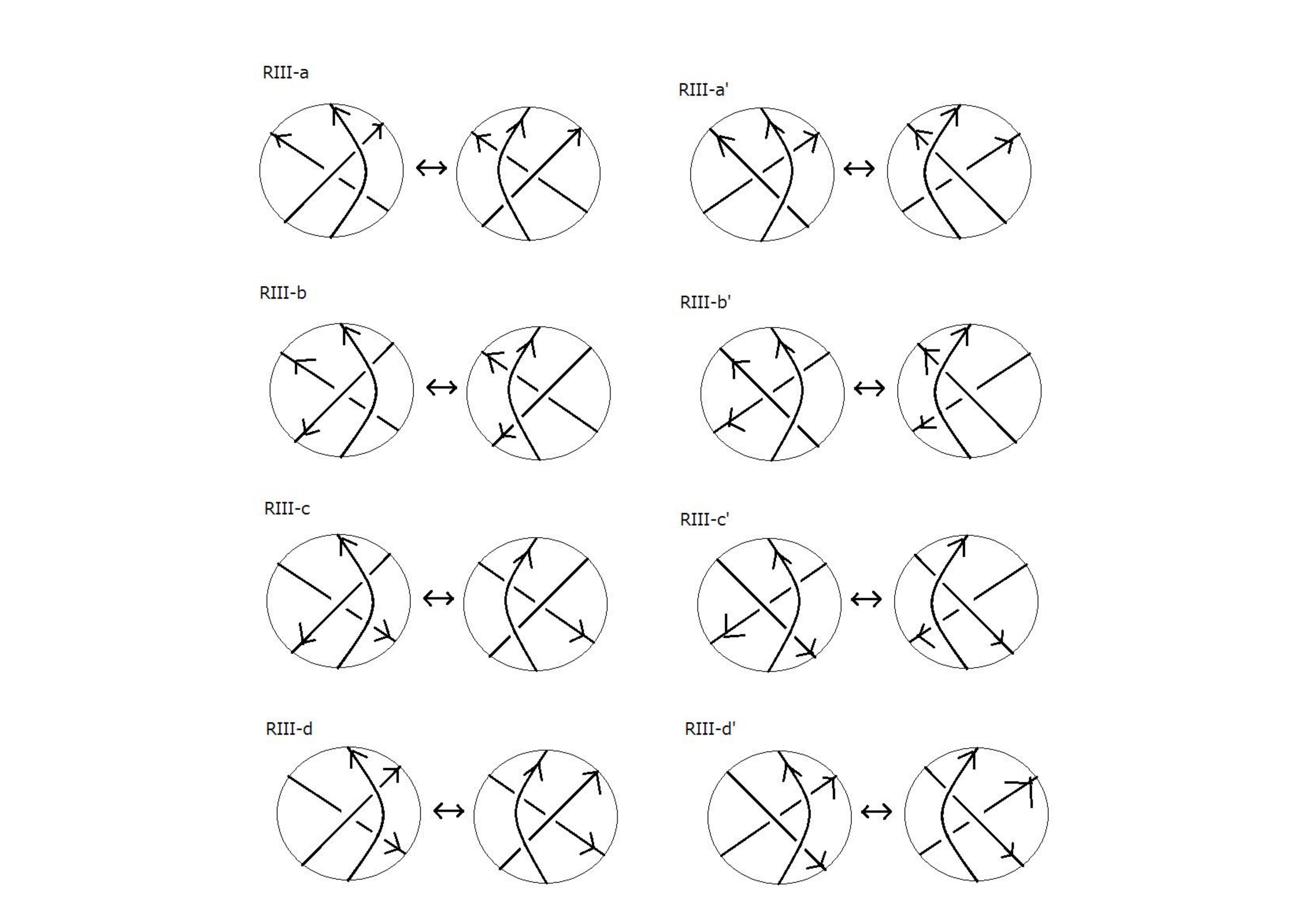}}
\end{picture}
\caption{RIII: Reidemeister moves}
\label{fig:RIIIs}
\end{figure}

Let $J^+,J^-,J'^+$ and $J'^-$ be mutually disjoint finite subsets of $\partial \Sigma$
with $\sharp (J^+) =\sharp (J^-)$ and 
$\sharp(J'^+) =\sharp(J'^-)$.
Here $e_1$ and $e_2$ denote the embedding maps from $\Sigma \times I$ to
$\Sigma \times I$ defined by $e_1(x,t) =(x,\frac{t+1}{2})$ and
$e_2(x,t)=(x,\frac{t}{2})$, respectively.
 We define $\boxtimes : \mathcal{T}(\Sigma,J^-,J^+)
\times \mathcal{T}(\Sigma, J'^-,J'^+) \to \mathcal{T}(\Sigma, J^- \cup J'^-,
J^+ \cup J'^+)$ by
\begin{equation*}
\kukakko{E} \boxtimes \kukakko{E'} \defeq \kukakko{e_1 \circ E \sqcup e_2 \circ E'}
\end{equation*}
for $E \in \mathcal{E}(\Sigma,J^-,J^+)$ and $E' \in \mathcal{E}(\Sigma, J'^-,J'^+)$.

Let $J_1,J_2$ and $J_3$ be mutually disjoint finite subsets of $\partial \Sigma$
with $\sharp (J_1) =\sharp (J_2) = \sharp (J_3)$.
Let $d_1$ and $d_2$ be  tangle diagrams presenting of 
elements of $\mathcal{T}(\Sigma,J_2,J_1)$ and $\mathcal{T}(\Sigma,J_3,J_2)$,
respectively, such that the intersections of $d_1$ and $d_2$ consist of
transverse double points except for $J_2$.
We denote by $d_1 \circ d_2$ the tangle diagram
satisfying the following.
\begin{itemize}
\item The tangle diagram $d_1 \circ d_2$ equals
$d_1 \cup d_2$ with the same height-information as $d_1$ and $d_2$
except for a neighborhoods of any intersection of $d_1$ and $d_2$.
\item 
In a neighborhood of each intesection
of $(d_1\cap d_2)\backslash J_2$,
the branches of $d_1 \circ d_2$ 
belonging to $d_1$ are over-crossings.
\item In a neighborhood of each point 
of $J_2$, the tangle diagram $d_1 \circ d_2$ is as shown in 
Figure \ref{fig_boundary_gousei}.
\end{itemize}
We define $\circ :\mathcal{T}(\Sigma,J_2,J_1) \times
\mathcal{T}(\Sigma,J_3,J_2) \to \mathcal{T}(\Sigma,J_3,J_1)$
by $T(d_1) \circ T(d_2) \defeq T(d_1 \circ d_2)$.

\begin{figure}
\begin{picture}(160,80)
\put(0,-40){\includegraphics[width=160pt]{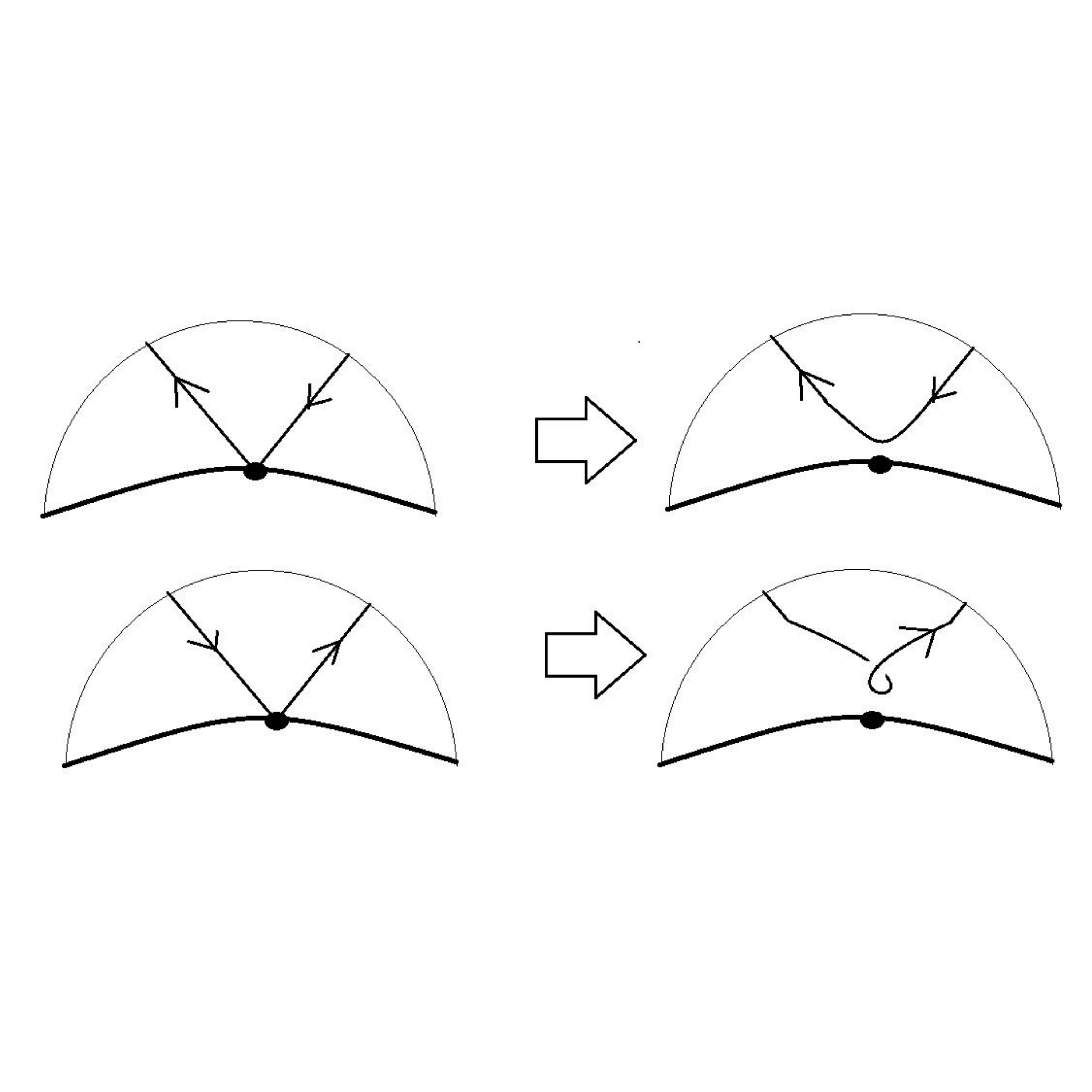}}
\put(10,49){\includegraphics[width=10pt]{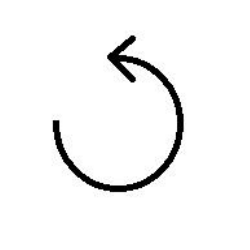}}
\end{picture}
\caption{}
\label{fig_boundary_gousei}
\end{figure}

Let $J^+$ and $J^-$ be  disjoint finite subsets of $\partial \Sigma $
with $\sharp (J^+) =\sharp (J^-)$,
$T$ an element of $\mathcal{T}(\Sigma,J^-,J^+)$ 
represented by $E \in \mathcal{E}(\Sigma,J^-,J^+)$ and 
$\xi$ an element of $\mathcal{M}(\Sigma)$ represented by 
a diffeomorphism $\mathcal{X}_\xi$.
We denote by $\xi T$ an element of 
$\mathcal{T}(\Sigma,J^-.J^+)$ represented by
$(\mathcal{X}_\xi \times \id_I)\circ E \in \mathcal{E}(\Sigma,J^-,J^+)$.
This defines  a left action of the mapping class group
$\mathcal{M} (\Sigma)$ on the set 
$\mathcal{T} (\Sigma, J^-, J^+)$.

\section{HOMFLY-PT type skein modules}

Throughout this section, 
let $\Sigma$ be a compact connected oriented surface.

\subsection{Definition}
First of all, we define a Conway triple.

\begin{df}
Let $J^+$ and $J^-$ be  finite disjoint subsets of $\partial \Sigma $
with $\sharp (J^+) =\sharp (J^-)$.
A triple of three tangles $T_+$, $T_-$ and $T_0$ is a Conway triple
if there exist tangle diagrams $d_+$, $d_-$ and $d_0$
presenting $T_+$, $T_-$ and $T_0$, respectively,
such that $d_+$, $d_-$ and $d_0$ are identical
except for a closed disk, where they differ as shown in 
Type C(+), Type C(-) and Type C(0)
in Figure \ref{fig_Conway_triples}, respectively.
\end{df}

\begin{figure}
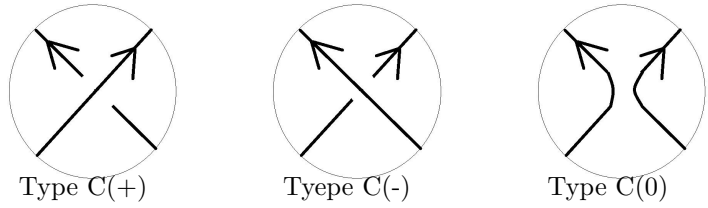

\begin{picture}(300,83)
\put(0,-20){\includegraphics[width=80pt]{conway_plus_PNG.pdf}}
\put(100,-20){\includegraphics[width=80pt]{conway_minus_PNG.pdf}}
\put(200,-20){\includegraphics[width=80pt]{conway_zero_PNG.pdf}}
\put(10,-10){Type C(+)}

\put(110,-10){Tyepe C(-)}

\put(210,-10){Type C(0)}

\end{picture}
\caption{Conway triples}
\label{fig_Conway_triples}
\end{figure}

\begin{df}
\label{df_skein_algebra}
Let $J^+$ and $J^-$ be  finite disjoint subsets of $\partial \Sigma $
with $\sharp (J^+) =\sharp (J^-)$.
We define $\tskein{\Sigma,J^+,J^-}$ to be
the quotient of $\Q [\rho][[h]]\mathcal{T}(\Sigma,J^+,J^-)$ modulo the skein relation,
 the framing relation and the trivial knot relation.
Here the skein relation is given by
\begin{equation*}
T_+-T_- =hT_0
\end{equation*}
for $(T_+,T_-,T_0)$ a Conway triple of $\mathcal{T}(\Sigma,J^+,J^-)$.
The framing relation is given by
\begin{equation*}
T(d_1) =\exp (\rho h) T(d_0)
\end{equation*}
for $d_1$ and $d_0$ which are identical except for a disk,
where they are a positive (right handed) twist and a straight line.
We mean by a positive (right handed) twist
the diagram as shown in Figure \ref{fig:RI}.
Let $d_\mathrm{triv.knot}$ and $d_\mathrm{empty}$
be two diagrams
which are identical except for a disk,
where they are the boundary of a closed disk and
empty, respectively.
The trivial knot relation is
\begin{equation*}
T(d_\mathrm{triv. knot}) = \frac{2 \sinh (\rho h)}{h}T(d_\mathrm{empty}).
\end{equation*}
We denote by $[T]$ the element of $\mathcal{A}
(\Sigma, J^-,J^+)$ represented by a tangle $T$.
We simply denote $\tskein{\Sigma} \defeq \tskein{\Sigma, \emptyset, \emptyset}$.

\end{df}

Let $J^+,J^-,J'^+$ and $J'^-$ be mutually disjoint finite
subsets of $\partial \Sigma$
with $\sharp (J^+) =\sharp (J^-)$ and 
$\sharp(J'^+) =\sharp(J'^-)$.
We define the $\Q [\rho][[h]]$-bilinear homomorphism
$\boxtimes :\tskein{\Sigma,J^-,J^+} \times \tskein{\Sigma,J'^-,J'^+}
\to \tskein{\Sigma,J^- \cup J'^-, J^+ \cup J'^+}$
by $[T_1] \boxtimes [T_2]  =[T_1 \boxtimes T_2]$.
The skein module $\mathcal{A} (\Sigma)$
is the associative algebra over  $\Q[k][[h]]$ 
with product defined $ab \defeq a \boxtimes b$ for $a, b \in 
\tskein{\Sigma}$.
The skein module $\tskein{\Sigma,J^-,J^+}$ is the
$\tskein{\Sigma}$-bimodule given by
$av \defeq a\boxtimes v$ and $va \defeq v \boxtimes a$
for $a \in \tskein{\Sigma}$ and $v \in \tskein{\Sigma,J^-.J^+}$.

We fix points $*_1, *_2, \cdots ,*_b$, $*'_1, *'_2, \cdots ,*'_b$,
$*''_1, *''_2, \cdots ,*''_b$ and paths $r_1, r_2, \cdots, r_b$,
$r'_1, r'_2, \cdots, r'_b$ as Figure \ref{fig_boundary_points}.
We denote $\mathcal{T}(\Sigma, *_\alpha,*,j) \defeq
\mathcal{T}(\Sigma,\shuugou{*_\alpha},\shuugou{*'_\beta})$.
We define $\circ :\mathcal{T} (\Sigma,*_\beta,*_\gamma) \times \mathcal{T}
(\Sigma, *_\alpha,*_\beta) \to \mathcal{T}(\Sigma,*_\alpha,*_\gamma)$
by $T_1 \circ T_2 \defeq   (r'_\gamma)^{-1} \circ(((r'_\gamma \circ T_1)
 \circ (r_\beta)^{-1}) \circ T_2)$
where we denote by $r_\alpha$, $\gyaku{(r_\alpha)}$, $r'_\alpha$ and $\gyaku{(r'_\alpha)}$ the tangles
presented by $r_\alpha$, $\gyaku{(r_\alpha)}$, $r'_\alpha$ and $\gyaku{(r'_\alpha)}$,
 respectively.
We denote $\tskein{\Sigma,*_\alpha,*_\beta} \defeq \tskein{\Sigma,\shuugou{*_\alpha},\shuugou{*'_\beta}}$.
The map $\circ :\mathcal{T} (\Sigma,*_\beta,*_\gamma) \times \mathcal{T}
(\Sigma, *_\alpha,*_\beta) \to \mathcal{T}(\Sigma,*_\alpha,*_\gamma)$
induces
$\circ :\tskein{\Sigma,*_\beta,*_\gamma} \times \tskein{\Sigma,*_\alpha,*_\beta} \to
\tskein{\Sigma,*_\alpha,*_\gamma}$.
In particular, $\tskein{\Sigma,*_\alpha,*_\alpha} \defeq \tskein{\Sigma,*_\alpha}$
is the associative algebra over $\Q [\rho][[h]]$ with unit $1  \defeq [r_\alpha]$.
There exists a natural map $\zettaiti{\cdot}:\tskein{\Sigma,*_\alpha}
\to \tskein{\Sigma}$ defined by Fig \ref{fig_closure}.

\begin{figure}
\begin{picture}(300,200)
\put(0,-80){\includegraphics[width=330pt]{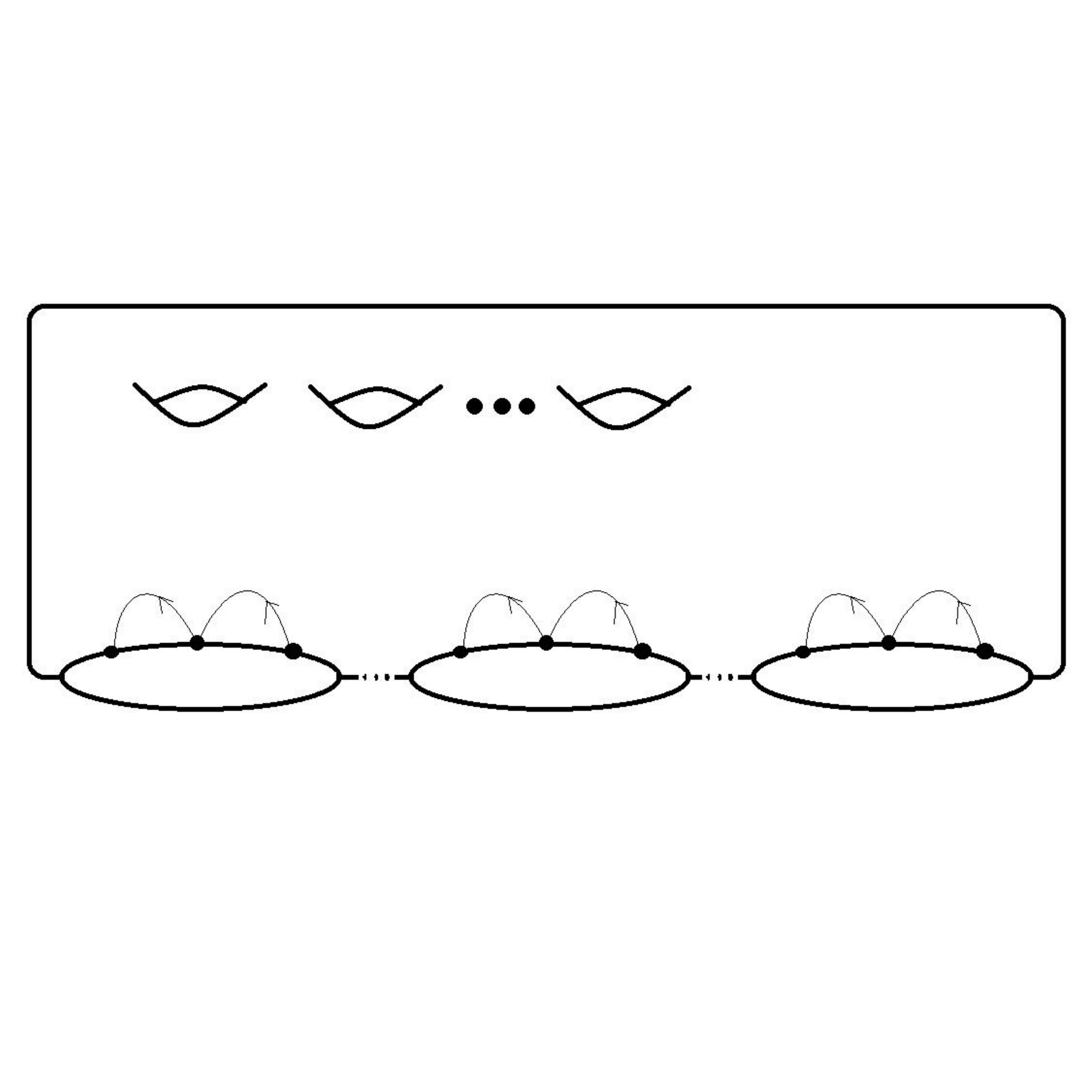}}
\put(33,44){$*''_1 \ \ \ \ \ *'_1 \ \ \ \ *_1$}
\put(136,44){$*''_\alpha \ \ \ \ \ *'_\alpha \ \ \ \ *_\alpha$}
\put(238,44){$*''_b \ \ \ \ \ *'_b \ \ \ \ *_b$}
\put(46,80){$r'_1\ \ \ \ \ \ r_1$}
\put(149,80){$r'_\alpha \ \ \ \ \ \ r_\alpha$}
\put(251,80){$r'_b\ \ \ \ \ \ r_b$}

\put(233,100){\includegraphics[width=20pt]{anorientation_PNG.pdf}}
\end{picture}
\caption{$*_\alpha,*'_\alpha,*''_\alpha,r_\alpha,r'_\alpha$}
\label{fig_boundary_points}
\end{figure}

\begin{figure}
\begin{picture}(330,50)
\put(0,-150){\includegraphics[width=330pt]{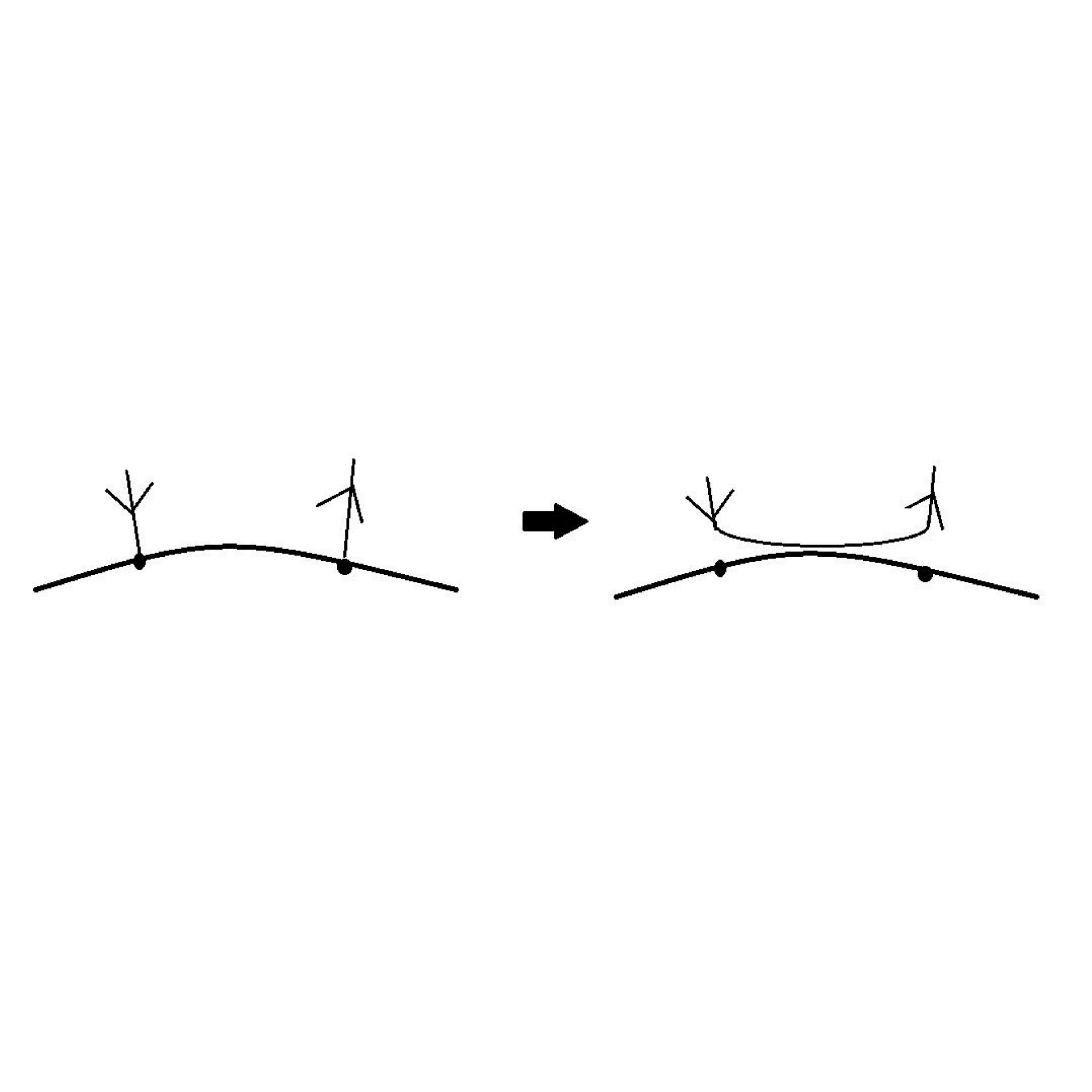}}

\put(0,20){\includegraphics[width=20pt]{anorientation_PNG.pdf}}
\end{picture}
\caption{$\zettaiti{\cdot}: \tskein{\Sigma,*_\alpha} \to \tskein{\Sigma}$}
\label{fig_closure}
\end{figure}

\begin{prop}
\label{prop_tskein_disk}
We define  a $\Q [\rho][[h]]$-algebra homomorphism
$\mathcal{N}:\tskein{D}  \to \Q [\rho][[h]]$
by 
\begin{equation*}
\mathcal{N} ([L]) \defeq \exp (w(L) \rho h) \homfly (L)(\exp(\rho h),h)
\end{equation*}
where $D$ is the closed disk and $w(L)$ is the self linking number of
$L$, i.e., $w(L)= \sharp \shuugou{\mathrm{the \ positive \ crossing \ of \ } L}-
\sharp \shuugou{\mathrm{the \ negative \ crossing \ of \ } L}$.
Then $\mathcal{N}$ is well-defined and bijective.

\end{prop}

\begin{proof}
Let $L_+$, $L_-$ and $L_0$ be elements of 
$\mathcal{T}(D)$ such that $(L_+,L_-, L_0)$
is a Conway triple.
Since $w(L_+)-1=w(L_-)+1=w(L_0)$ and 
\begin{equation*}
\exp(\rho h) \homfly (L_+)-\exp (-\rho h) \homfly (L_-) =h \homfly (L_0),
\end{equation*}
we have 
\begin{equation*}
\mathcal{N}  ([L_+])-\mathcal{N} ([L_-]) =h \mathcal{N} ([L_0]).
\end{equation*}
If $d_1$ and $d_0$ 
are tangle diagrams which are identical except for a dsik,
where they are a positive twist and a straight line, respectively,
then,
we have $\homfly (T(d_1)) =\homfly (T(d_0))$ and 
$w(T(d_1))-1=w(T(d_0))$. Then we obtain
\begin{equation*}
\mathcal{N}([T(d_1)]) =\exp (\rho h) \mathcal{N} ([T (d_0)]).
\end{equation*}
If two $d_{\mathrm{triv.knot}}$ and $d_\mathrm{empty}$
 diagrams are identical except for a disk,
where they are the boundary of a closed disk and 
empty, respectively,
we have
\begin{equation*}
\mathcal{N}([T(d_\mathrm{triv.knot})]) 
= \frac{\sinh (\rho h)}{h} \mathcal{N}[T(d_\mathrm{empty})].
\end{equation*}
These formulas prove that $\mathcal{N}$ is well-defined.
We define $\mathcal{L'}: \Q [\rho][[h]]  \to \tskein{D}$
by $f(\rho,h) \in \Q [\rho][[h]] \mapsto f(\rho,h) [\emptyset]$.
We have $\mathcal{L'} \circ \mathcal{N} =\id_{\Q [\rho][[h]]}$ and
$\mathcal{N} \circ \mathcal{L'} =\id_{\tskein{D}}$.
This proves the proposition.

\end{proof}

\subsection{Regular homotopy path and Turaev skein algebra}
Let $\pi^+_1 (\Sigma, {*}, {*'})$ be
the regular homotopy set of free immersed paths from $*$
to $*'$ on $\Sigma$ for two points $*,*' \in \partial \Sigma$
satisfying $* \neq *'$.
We define a composite 
$\cdot : \pi^+_1 (\Sigma, {*'},{*''})
\times \pi^+_1 (\Sigma, {*},{*'})
\to \pi^+_1 (\Sigma,{*}, {*''})$
as Figure \ref{fig_boundary_gousei}
for $*,*',*'' \in \partial \Sigma$ satisfying
$* \neq *'$, $*' \neq *''$ and $*'' \neq *$. 
We denote $\pi^+_1 (\Sigma,*_\alpha,*_\beta)
\defeq \pi^+_1 (\Sigma,{*_\alpha},{*'_\beta})$.
We define a composite $\cdot :\pi^+_1 (\Sigma,*_\beta,*_\gamma) \times
\pi^+_1 (\Sigma,*_\alpha,*_\beta) \to \pi^+_1 (\Sigma,*_\alpha,*_\gamma)$
by
\begin{equation*}
R_1 \cdot R_2 \defeq   (r'_\gamma
)^{-1} \cdot(((r'_\gamma \circ R_1) \circ (r_\beta)^{-1}) \circ R_2).
\end{equation*}
In particular, $\pi^+_1 (\Sigma,*_\alpha,*_\alpha) \defeq
\pi^+_1 (\Sigma,*_\alpha)$ is a group with unit $\ r_\alpha$.
Let $\mathcal{P} (\Sigma,*_\alpha,*_\beta)$ be the quotient of  
$\Q [\rho][[h]] \pi^+_1 (\Sigma,*_\alpha,*_\beta)$ modulo the relation
\begin{equation*}
R_1 = \exp (\rho h) R_0.
\end{equation*}
Here $R_1$ and $R_0$ are identical except for a closed disk,
where they differ  as shown in Figure \ref{fig_regular_relation}.

\begin{figure}
\begin{picture}(160,80)
\put(0,0){\includegraphics[width=70pt]{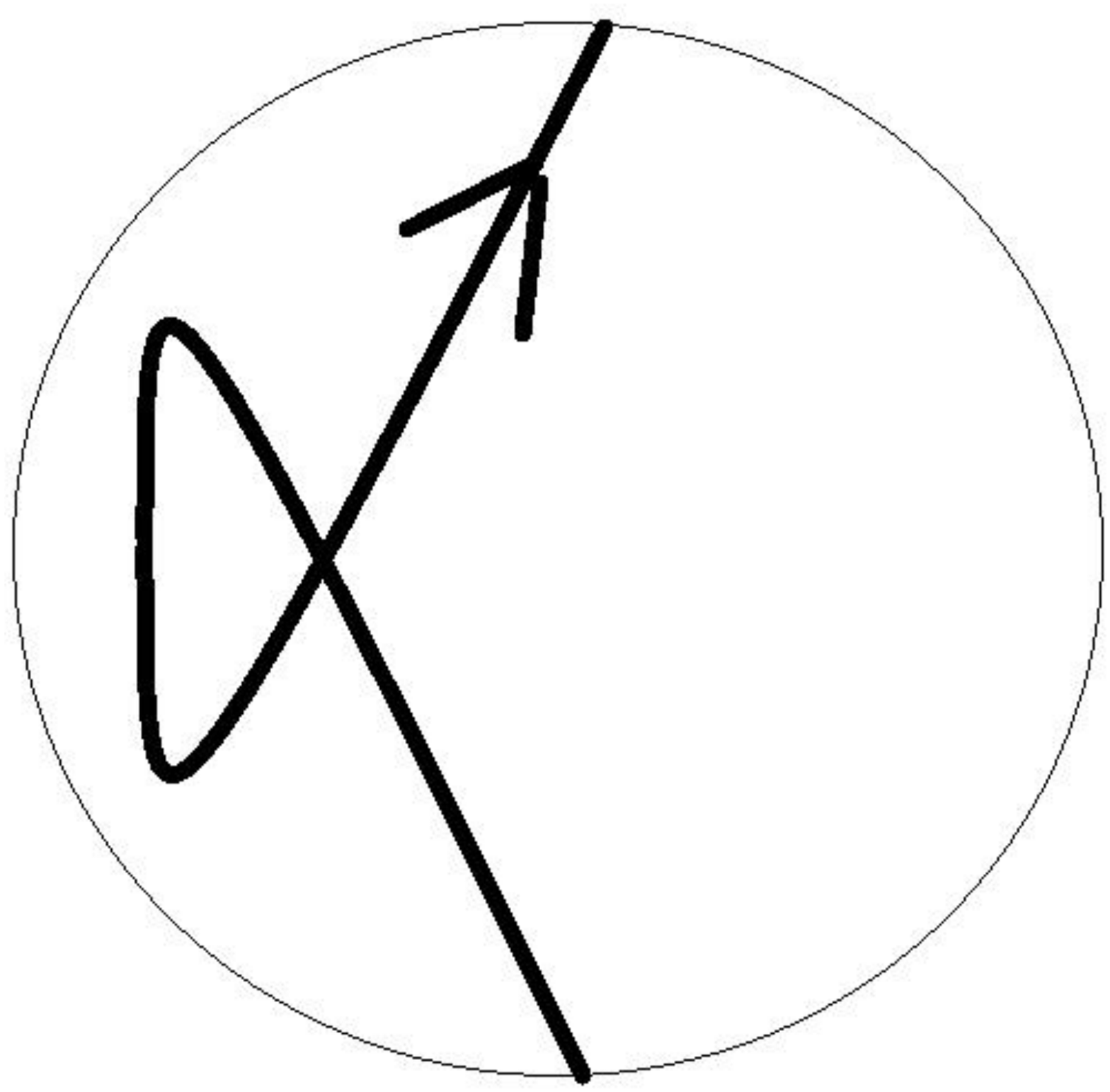}}
\put(35,0){$R_1$}
\put(80,0){\includegraphics[width=70pt]{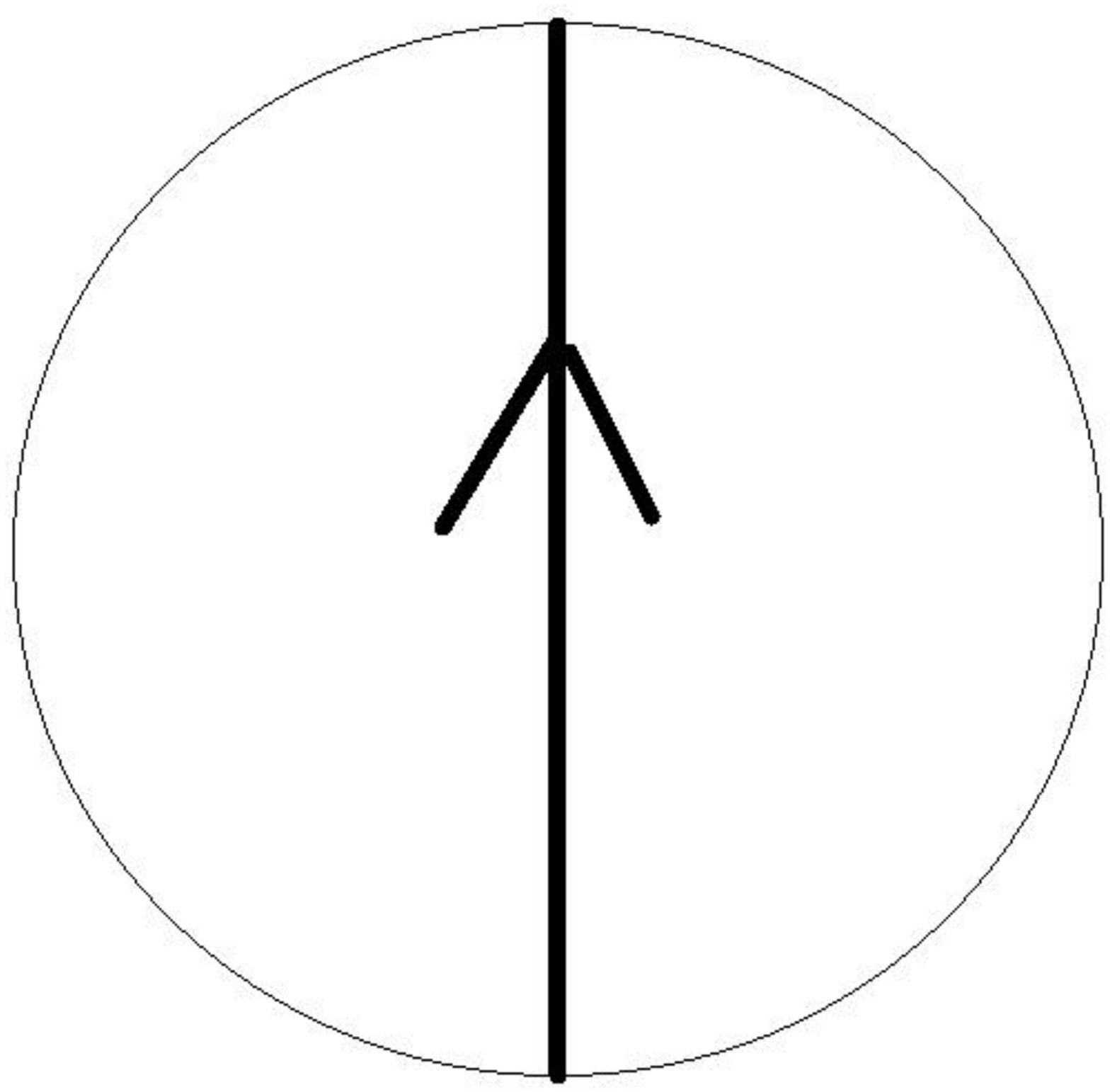}}
\put(115,0){$R_0$}
\end{picture}
\caption{}
\label{fig_regular_relation}
\end{figure}

There exists a $\Q [\rho][[h]]$-module homomorphism
$\psi : \mathcal{P}  (\Sigma,\shuugou{*}, \shuugou{*'})  \to
  \tskein{\Sigma,\shuugou{*}, \shuugou{*'}}$
defined by $R \mapsto \tilde{R}$ where
$\tilde{R} : I \times I
\to \Sigma \times I$ is defined by $ (t_1, t_2) \mapsto (R(t_1),t_1+\delta t_2)$
for a map $R :I \to \Sigma$. Here $\delta > 0$ is a sufficiently small number.
Let $J^- $ and $J^+$ be finite disjoint subsets of $\partial \Sigma$
with $\sharp J^-= \sharp J^+ $,
and $*$ an element of $J^-$.
We define a $\Q [\rho][[h]]$-module homomorphism
\begin{equation*}
\psi_* : \bigoplus_{ *' \in J^+} \tskein{\Sigma, J^- \backslash 
\shuugou{*}, J^+ \backslash \shuugou{*'}} \otimes_{\Q [\rho][[h]]}
\mathcal{P} (\Sigma,\shuugou{*}, \shuugou{*'}) \to \tskein{\Sigma,J^-, J^+}
\end{equation*}
by $\psi_* (a \otimes R) \defeq a \boxtimes \psi(R)$.
The aim of this subsection is  to prove the following.

\begin{prop}
\label{prop_psi_isom}
The  $\Q [\rho] [[h]]$-module homomorphism $\psi_*$ is an
isomorphism.
\end{prop}

It is obvious that $\psi_*$ is surjective.
We will construct $\bar{\psi} :\tskein{\Sigma,J^-,J^+} 
\to  \bigoplus_{ *' \in J^+} \tskein{\Sigma, J^- \backslash 
\shuugou{*}, J^+ \backslash \shuugou{*'}} \otimes_{\Q [\rho][[h]]}
\mathcal{P} (\Sigma,\shuugou{*}, \shuugou{*'})$ satisfying
$\psi_* \circ \bar{\psi} =\id_{\tskein{\Sigma,J^-,J^+} }$ and
$\bar{\psi} \circ \psi_* = \id_{\bigoplus_{ *' \in J^+} \tskein{\Sigma, J^- \backslash 
\shuugou{*}, J^+ \backslash \shuugou{*'}} \otimes_{\Q [\rho][[h]]}
\mathcal{P} (\Sigma,\shuugou{*}, \shuugou{*'})}$. 

For a tangle diagram $d$ of an element of
$\mathcal{T} (\Sigma,J^-, J^+)$,
we denote by $\mathrm{Cross} (d)$
the set of crossings of $d$
and by $\mathrm{Cross}_1 (d)$
the set of crossings
of types shown in 
\ref{fig_shoumei_group_ring}.
Let 
$\bar{d}$  be the components of $d$
not including $*$ and
$d_{\mathrm{path}}$ the component 
of $d$ including $*$.
If $\mathrm{Cross}_1 (d) = \emptyset$,
$\psi ([T (\bar{d})] \otimes d_{\mathrm{path}})
=[T(d)]$.
By induction on 
$( \sharp \mathrm{Cross} (d), \sharp \mathrm{Cross}_1 (d))$,
we define a map $\bar{\psi}$ from
the set of tangle diagrams to
$\bigotimes_{ *' \in J^+} \tskein{\Sigma, J^- \backslash 
\shuugou{*}, J^+ \backslash \shuugou{*'}} \otimes_{\Q [\rho][[h]]}
\mathcal{P} (\Sigma,\shuugou{*}, \shuugou{*'})$.
If three tangle diagrams $d_1$, $d_2$ and $d_3$ are identical except
for a sufficiently small neighborhood such that $(T(d_1), T(d_2), T(d_3))$
is a Conway triple, then we have $\sharp \mathrm{Cross}_1 (d_1) \pm
1 = \sharp \mathrm{Cross}_1 (d_2)$ and
$\sharp \mathrm{Cross} (d_1) = \sharp \mathrm{Cross} (d_2)
=\sharp \mathrm{Cross} (d_3)+1$.
Using the skein relation, we define $\bar{\psi} (d_1) = \bar{\psi} (d_2) + h 
\bar{\psi} (d_3)$ or $\bar{\psi} (d_2)= \bar{\psi}(d_1)-h\bar{\psi} (d_3)$.

For a tangle diagram $d$ with $\sharp (\mathrm{Cross}_1 (d))=0$,
we define $\bar{\psi} (d)
\defeq [T(\bar{d})] \otimes d_{\mathrm{path}}$.

Let $d$ be  a tangle diagram satisfying 
$\sharp \mathrm{Cross} =N $ and $\sharp \mathrm{Cross_1} =N' \geq 1$.
We choose  one of the crossings $P \in
\mathrm{Cross}_1(d)$.
We assume this crossing $P$ is shown as
Type 1, Type 3 or Type 5 in 
Figure \ref{fig_shoumei_group_ring}.
Let $d'$ and $d''$ be the tangle diagrams
which are identical except for a sufficiently small neighborhood of $P$
such that $(T(d), T(d'), T(d''))$ is a Conway triple.
We define 
\begin{equation*}
\bar{\psi} (d) \defeq \bar{\psi} (d')-h\bar{\psi}(d'').
\end{equation*}
We assume this crossing $P$ is shown as
 Type 2, Type 4 or Type 6
in Figure \ref{fig_shoumei_group_ring}.
Let $d'$ and $d''$ be the tangle diagrams
which are identical except for a sufficiently small neighborhood of $P$
such that $(T(d'), T(d), T(d''))$
is a Conway triple.
We define 
\begin{equation*}
\bar{\psi} (d) \defeq \bar{\psi} (d')+h\bar{\psi}(d'').
\end{equation*}

\begin{figure}
\begin{picture}(300,160)
\put(0,90){\includegraphics[width=2cm]{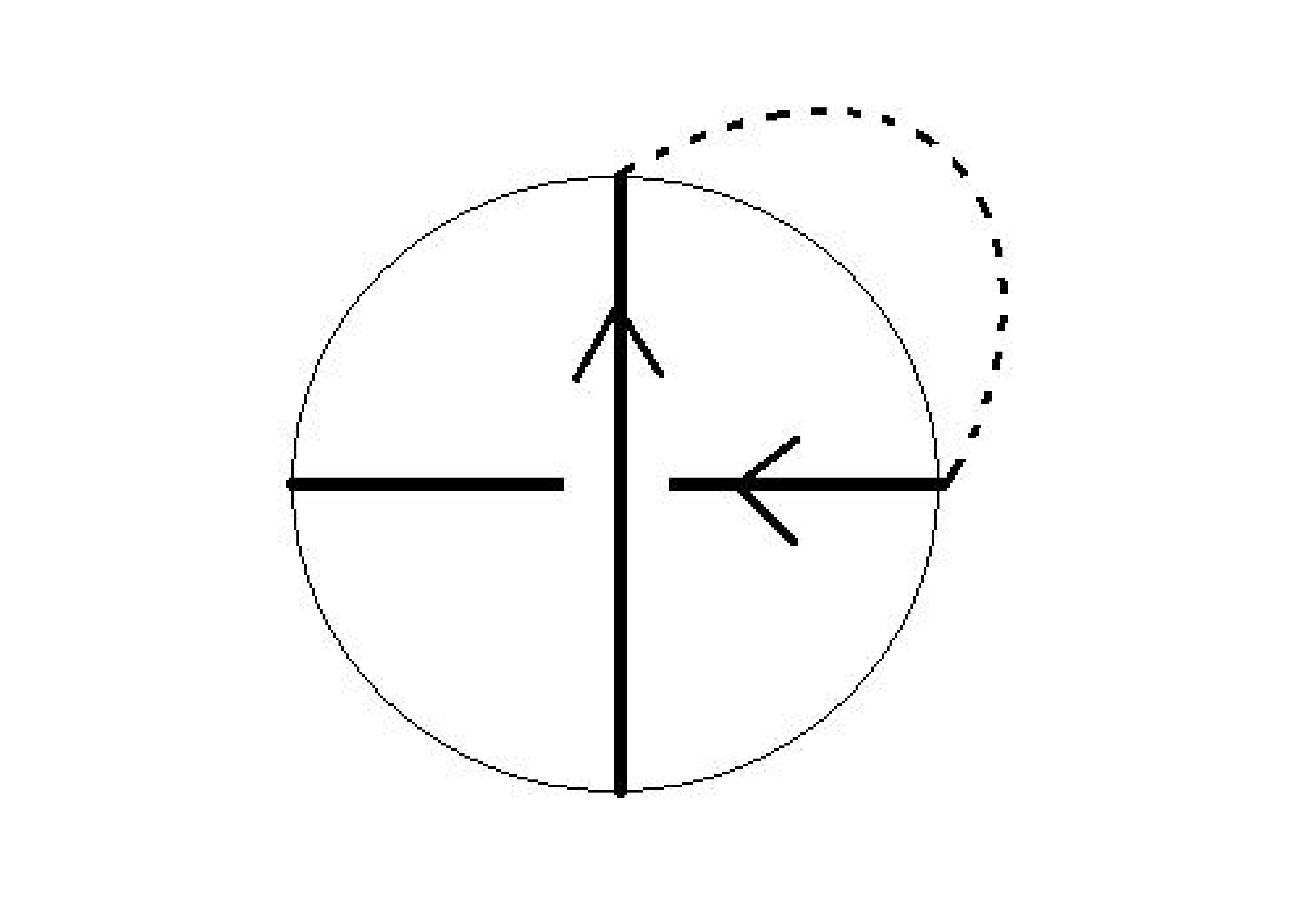}}
\put(75,90){\includegraphics[width=2cm]{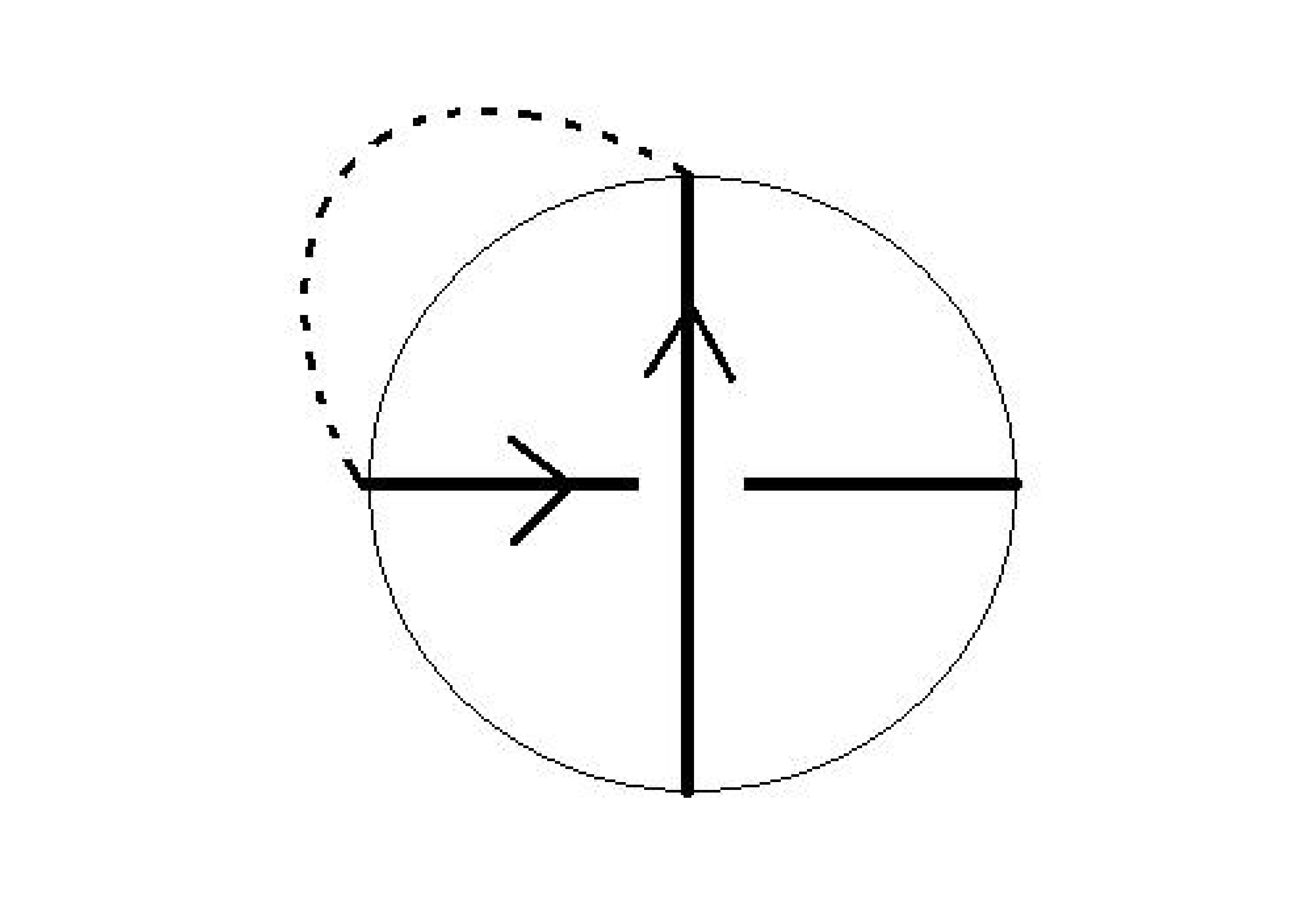}}
\put(150,90){\includegraphics[width=2cm]{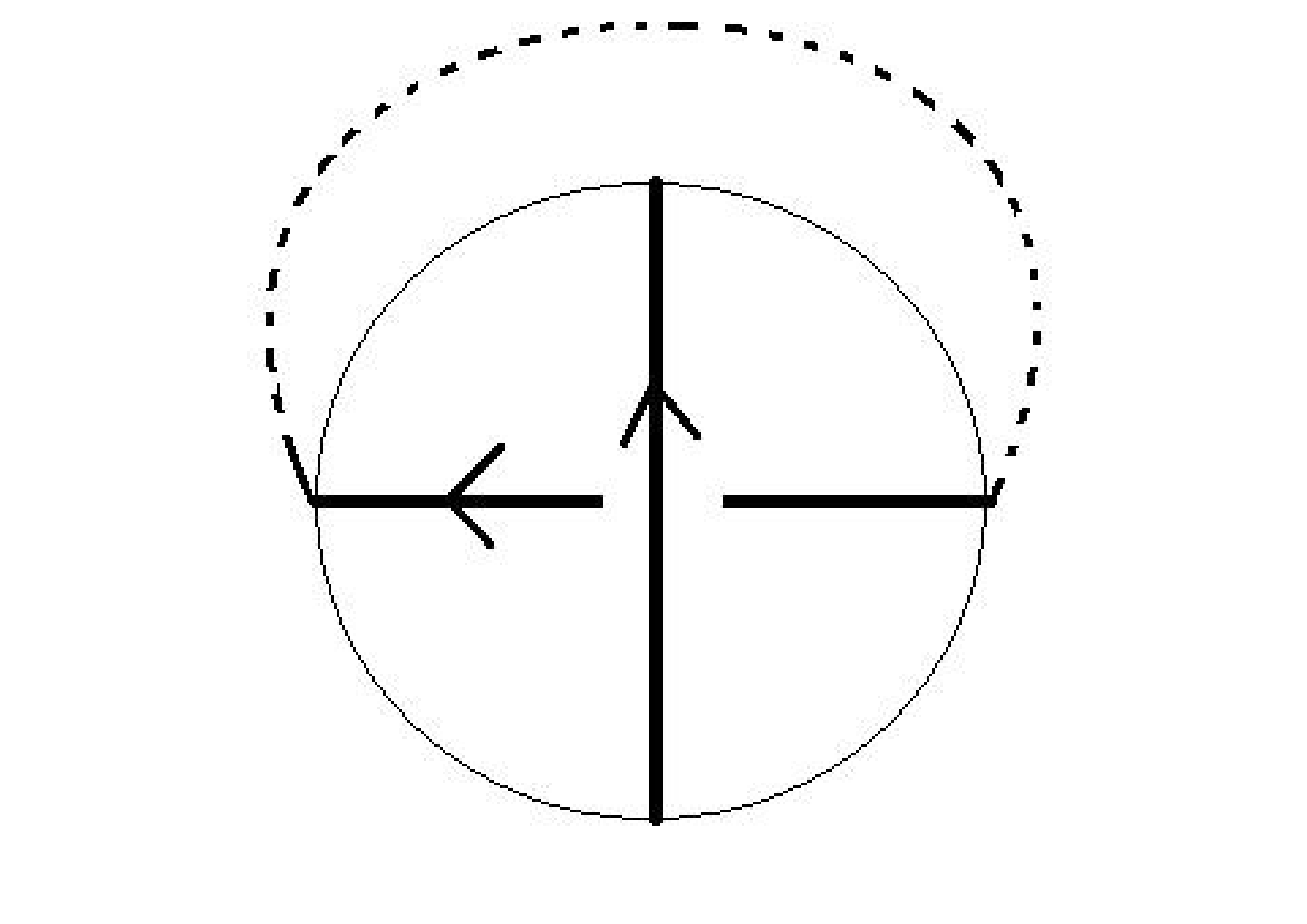}}
\put(225,90){\includegraphics[width=2cm]{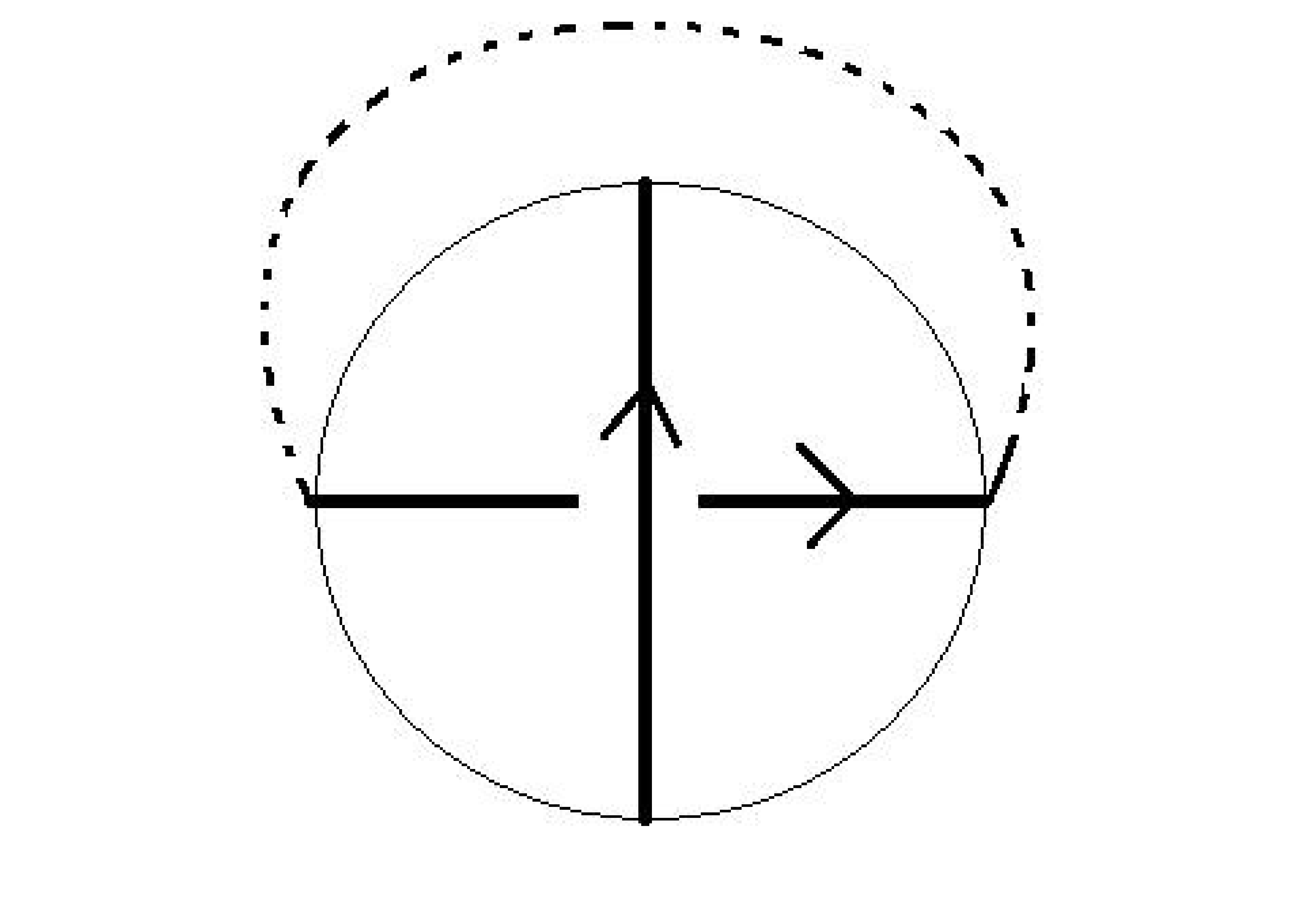}}
\put(75,10){\includegraphics[width=2cm]{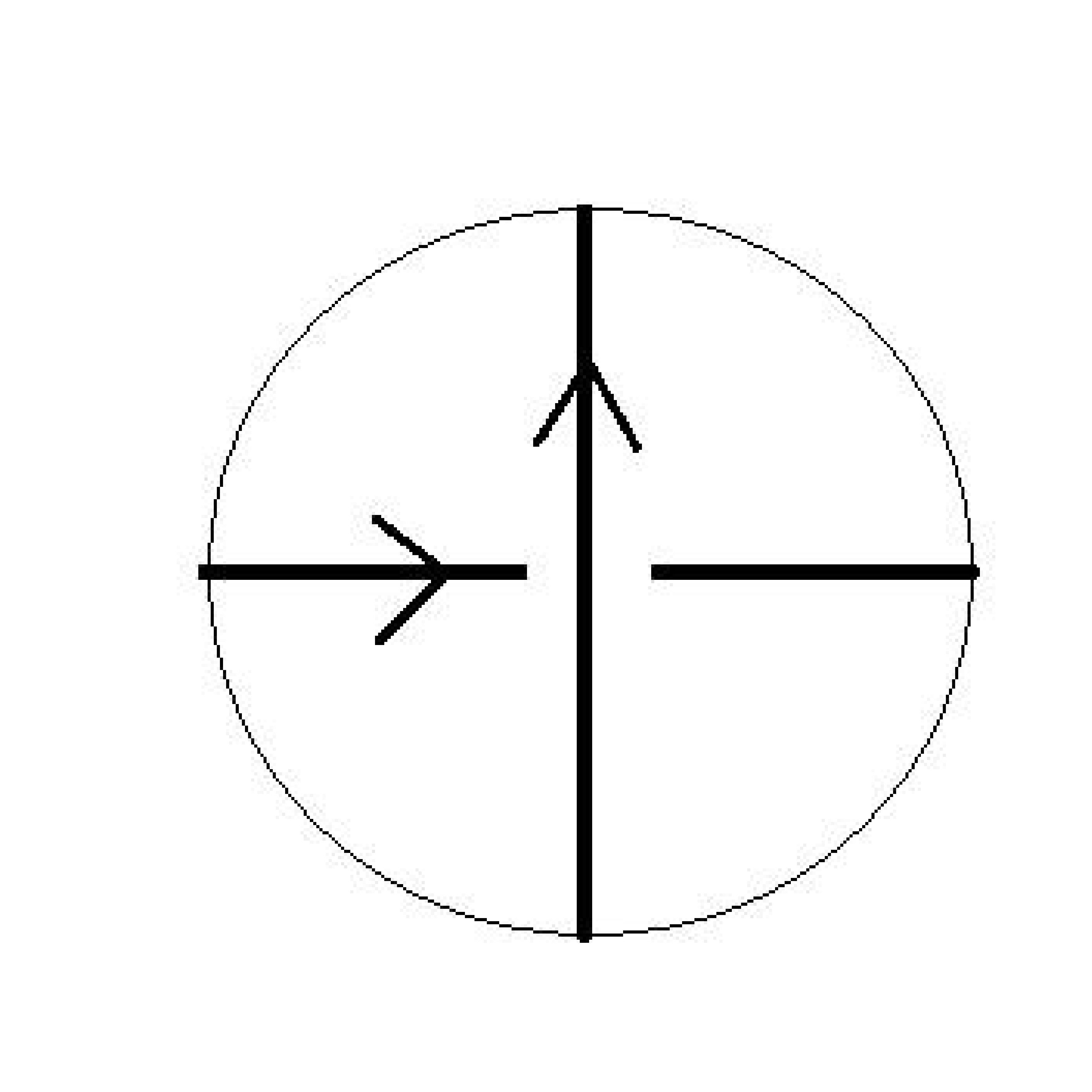}}
\put(150,10){\includegraphics[width=2cm]{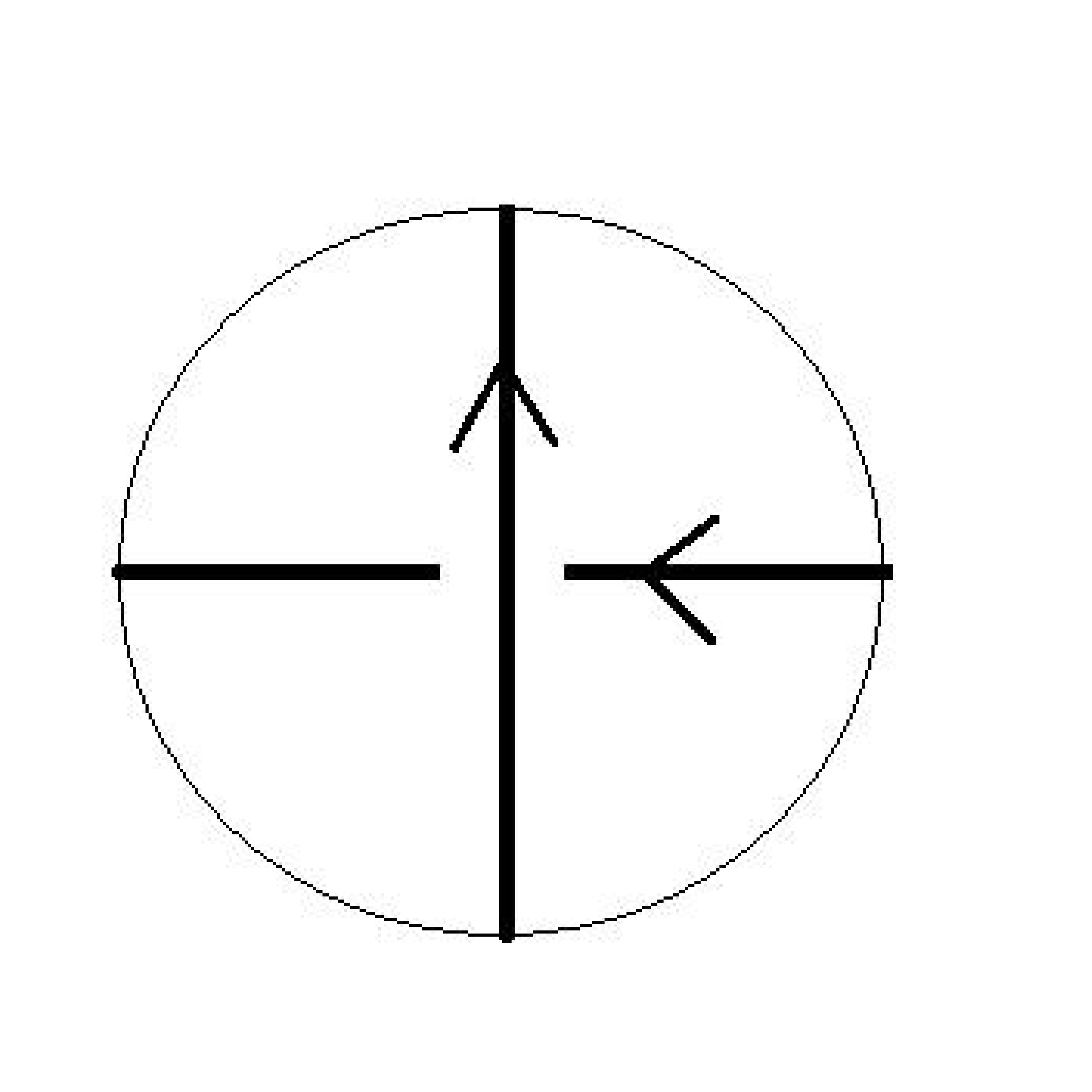}}

\put(25,90){$*$}
\put(100,90){$*$}
\put(175,90){$*$}
\put(250,90){$*$}

\put(25,80){Type 1}
\put(100,80){Type 2}
\put(175,80){Type 3}
\put(250,80){Type 4}

\put(100,10){$*$}
\put(175,10){$*$}

\put(100,0){Type 5}
\put(175,0){Type 6}

\put(25,33){$*'' \in J^- \backslash \shuugou{*}$}
\put(205,33){$*'' \in J^- \backslash \shuugou{*}$}

\end{picture}
\caption{}
\label{fig_shoumei_group_ring}
\end{figure}

\begin{prop}[($N$, $N'$)]
\label{prop_induction}
Let $\mathcal{D} (N,N')$ be the set 
consisting of all diagram $d$ satisfying 
one of the followings.
\begin{itemize}
\item $\sharp (\mathrm{Cross}(d)) < N$.
\item $\sharp (\mathrm{Cross}(d)) =N $ and $\sharp (\mathrm{Cross}_1(d))<N'$.
\end{itemize}
Then we have the followings.
\begin{enumerate}
\item For $d \in \mathcal{D} (N,N')$, $\bar{\psi} (d)$ is well-defined,
in other words, $\bar{\psi} (d)$ is independent from the choice of
$\mathrm{Cross}_1 (d)$.
\item Let $d_+,d_-$ and $d_0$ be three elements of $\mathcal{D}(N,N')$
such that $d_+$, $d_-$ and $d_0$ are identical
except for a closed disk, where they differ as shown in 
Type C(+), Type C(-) and Type C(0)
in Figure \ref{fig_Conway_triples}
respectively. Then we have $\bar{\psi} (d_+)-\bar{\psi} (d_-)
=h \bar{\psi} (d_0)$.
\item If $d_1$ and $d_0 \in \mathcal{D} (N,N')$ are identical except for a disk,
where they are one of Figure \ref{fig:RI} and  a straight line, respectively, then,
we have $\bar{\psi} (d_1) =\exp (\rho h) \bar{\psi} (d_0)$.
\item If $d_1$ and $d_0 \in \mathcal{D} (N,N')$ are identical except for a disk,
where they are the boundary of a closed disk and
empty, respectively, then, we have  $\bar{\psi} (d_1) =
\frac{\sinh {\rho h}}{h} \bar{\psi} (d_0)$.
\item If $d$ and $d' \in \mathcal{D} (N,N')$ are diagrams
obtained by
RI move Figure \ref{fig:RI}, we have $\bar{\psi} (d) =\bar{\psi} (d')$.
\item If $d$ and $d' \in \mathcal{D} (N,N')$  are
diagrams  obtained by
RII move Figure \ref{fig:RIIi} or Figure \ref{fig:RIIii}, we have $\bar{\psi} (d) =\bar{\psi} (d')$.
\item If $d$ and $d' \in \mathcal{D} (N,N')$ are
diagrams obtained by one of
RIII moves Figure \ref{fig:RIIIs}, then,
we have $\bar{\psi} (d) =\bar{\psi} (d')$.
\end{enumerate}
\end{prop}

\begin{lemm}[Proposition  \ref{prop_induction}
($N$, $N'+1$ ) (1)(2)(4)]
The element of 
\begin{equation*}
\bar{\psi} (d) \in \bigotimes_{ *' \in J^+} \tskein{\Sigma, J^- \backslash 
\shuugou{*}, J^+ \backslash \shuugou{*'}} \otimes_{\Q [\rho][[h]]}
\mathcal{P} (\Sigma,\shuugou{*}, \shuugou{*'}) 
\end{equation*}
 is independent of the choice of the crossings.
In particular, we have Proposition ($N$,$N'$) (2)(4).
\end{lemm}

\begin{proof}
We choose  two crossings $P_1$ and $P_2
\in \mathrm{Cross}_1 (d)$.
Let $d(\epsilon'_1, \epsilon'_2)$ be a tangle diagram
which is identical to $d$ except for sufficiently small neighborhoods
of $P_i$.
In the neighborhood of $P_i$,
it coincides with
Type C(+), Type C(-) and Type C(0)
in Figure \ref{fig_Conway_triples}
if $\epsilon'_i =1,-1,0$, respectively,
for $i=1,2$.
We assume $d(\epsilon_1, \epsilon_2) =d$.
It suffices to show 
\begin{equation*}
(\bar{\psi}(d(-\epsilon_1, \epsilon_2))+\epsilon_1 \bar{\psi}(d(0, \epsilon_2)))
-
(\bar{\psi}(d(\epsilon_1, -\epsilon_2))+\epsilon_2 \bar{\psi}(d(\epsilon_1, 0)))
=0.
\end{equation*}
By Proposition \ref{prop_induction}
($N$,$N'$)(2), we have
\begin{align*}
&(\bar{\psi}(d(-\epsilon_1, \epsilon_2))
+\epsilon_1 \bar{\psi}(d(0, \epsilon_2)))
-
(\bar{\psi}(d(\epsilon_1, -\epsilon_2))
+\epsilon_2 \bar{\psi}(d(\epsilon_1, 0))) \\
&=(\bar{\psi}(d(-\epsilon_1, \epsilon_2))
-\bar{\psi}(d(-\epsilon_1, -\epsilon_2))
+\epsilon_1 \bar{\psi}(d(0, \epsilon_2))) \\
&-
(\bar{\psi}(d(\epsilon_1, -\epsilon_2))
-\bar{\psi}(d(-\epsilon_1, -\epsilon_2))
+\epsilon_2 \bar{\psi}(d(\epsilon_1, 0)))  \\
&=(\epsilon_2\bar{\psi}(d(-\epsilon_1, 0))
+\epsilon_1 \bar{\psi}(d(0, \epsilon_2)))
-
(\epsilon_1\bar{\psi}(d(0, -\epsilon_2))
+\epsilon_2 \bar{\psi}(d(\epsilon_1, 0))) \\
&=(\epsilon_2\bar{\psi}(d(-\epsilon_1, 0))
-\epsilon_2 \bar{\psi}(d(\epsilon_1,0)))
+(\epsilon_1 \bar{\psi}(d(0,\epsilon_2))
-\epsilon_1 \bar{\psi}(d(0, -\epsilon_2))) \\
&=-\epsilon_1 \epsilon_2 \bar{\psi} (d(0,0))
+\epsilon_1 \epsilon_2 \bar{\psi} (d(0,0))=0.
\end{align*}
This proves
Proposition \ref{prop_induction}
($N$,  $N'+1$ )(1).

Since the map $\bar{\psi}$ is defined using the skein relation,
we have Proposition( $N$,  $N'+1$ )(2).

Since a trivial knot is included by
$\bar{d}$,
we have Proposition \ref{prop_induction} ($N$,  $N'+1$ )(4).

This completes the proof of  the lemma.
\end{proof}

\begin{lemm}[Proposition \ref{prop_induction} ($N$, $N'+1$ ) (3)(5)]
\label{lemm_RI}
Let $d_1$, $d'_1$ and $d_0 \in \mathcal{D} (N,N'+1)$ be three
tangle diagrams identical except for a disk $D$.
If, in the disk, the diagrams $d_1$, $d'_1$ and $d_0$ are as shown in
the right of Figure \ref{fig:RI}, the left of Figure \ref{fig:RI}
and a straight line, respectively, then,
we have $\bar{\psi} (d_1) =\bar{\psi} (d'_1) =\exp (\rho h) \bar{\psi} (d_0)$.
In other words, Proposition \ref{prop_induction} ($N$, $N'+1$ )(3)
and Proposition \ref{prop_induction} ($N$, $N'+1$ )(5) hold.

\end{lemm}

\begin{proof}
We assume $D \cap \mathrm{Cross}_1 (d'_1) = \emptyset$.
We choose $P \in \mathrm{Cross_1} (d'_1)$.
We assume $P$ is a positive crossing.
We denote by $d_1(-1)$, $d'_1(-1)$ and $d_0(-1)$
 tangle diagrams they are identical except for a sufficiently
small neighborhood of $P$, where they are as shown in
Type C(-) in Figure \ref{fig_Conway_triples}
and by $d_1 (0)$, $d'_1(0)$ and $d_0(0)$
tangle diagrams they are identical except for the
 neighborhood of $P$, where they are as shown in
Type C(0) in Figure \ref{fig_Conway_triples}.
We remark that the diagrams
$d_1(-1)$, $d'_1(-1)$, $d_0(-1)$
$d_1 (0)$, $d'_1(0)$ and $d_0(0)$
are element of $\mathcal{D} (N, N')$.
By Proposition \ref{prop_induction} (3) (5) (N,N'),
we have
\begin{align*}
\bar{\psi} (d_1(i))= \bar{\psi} (d'_1 (i))= \exp (\rho h) \bar{\psi} (d_1 (i))
\end{align*}
for $i=0,-1$.
By Proposition \ref{prop_induction} (2)(N,N'+1),
we have
\begin{equation*}
\bar{\psi} (d)= \bar{\psi} (d (-1))+h \bar{\psi} (d(0))
\end{equation*}
for $d=d_1,d'_1, d_0$.
Using the above equations, we have
\begin{equation*}
\bar{\psi} (d_1) =\bar{\psi} (d'_1) =\exp (\rho h) \bar{\psi} (d_0).
\end{equation*}

We assume $P$ is a negative crossing.
We denote by $d_1(1)$, $d'_1(1)$ and $d_0(1)$
 tangle diagrams they are identical except for a sufficiently
small neighborhood of $P$, where they are as shown in
Type C(+) in Figure \ref{fig_Conway_triples}
and by $d_1 (0)$, $d'_1(0)$ and $d_0(0)$
tangle diagrams they are identical except for the
 neighborhood of $P$, where they are as shown in
Type C(0) in Figure \ref{fig_Conway_triples}.
Proceeding in the same way upto sign,
we have
\begin{equation*}
\bar{\psi} (d_1(i))= \bar{\psi} (d'_1 (i))= \exp (\rho h) \bar{\psi} (d_1 (i)), \ \
\bar{\psi} (d)= \bar{\psi} (d (1))-h \bar{\psi} (d(0)),
\end{equation*}
for $i=0,1$ and $d=d_1,d'_1, d_0$.
Using the above equations, we have
\begin{equation*}
\bar{\psi} (d_1) =\bar{\psi} (d'_1) =\exp (\rho h) \bar{\psi} (d_0).
\end{equation*}

We assume $D \cap \mathrm{Cross}_1 (d'_1) \neq \emptyset$.
By Proposition \ref{prop_induction} ($N$,$N'$)(3), we have
$\bar{\psi}(d_1) =\exp(\rho h) \bar{\psi}(d_0)$.
Let $d''_1$ be the diagram which is identical except for the disk,
where it is the mirror of the right of Figure \ref{fig:RI}.
By Proposition \ref{prop_induction} ($N$,$N'$)(3), we have
$\bar{\psi}(d''_1) =\exp(-\rho h) \bar{\psi}(d_0)$.
By Proposition \ref{prop_induction} ($N$, $N'+1$ )(2)(4), we have
\begin{align*}
&\bar{\psi}(d'_1)=\bar{\psi}(d''_1)+h \frac{\sinh ( \rho h)}{h} \bar{\psi}(d_0) \\
&=\exp (-\rho h) \bar{\psi} (d_0)+(\exp(\rho h)-\exp(-\rho 
h)) \bar{\psi} (d_0) =\exp(\rho h)(d_0).
\end{align*}
This proves the lemma.
\end{proof}

\begin{lemm}[Proposition \ref{prop_induction} ($N$, $N'+1$ )(6)]
\begin{enumerate}
\item Let $d_1, d_2$ and $d_3$ be three tangle diagrams
which are identical except for a disk $D$. We suppose that they differ as shown in
the right, the center  and the left of Figure \ref{fig:RIIi}, respectively,
in the disk.
Then we have $\bar{\psi} (d_1) =\bar{\psi}(d_2) = \bar{\psi}(d_3)$.
\item  Let $d_1, d_2$ and $d_3$ be three tangle diagrams
which are identical except for a disk $D$. We suppose that they differ as shown in
the right, the center  and the left of Figure \ref{fig:RIIii}, respectively,
in the disk.
Then we have $\bar{\psi} (d_1) =\bar{\psi}(d_2) = \bar{\psi}(d_3)$.
\end{enumerate}
These imply Proposition \ref{prop_induction} ($N$, $N'+1$ )(6).
\end{lemm}

\begin{proof}
If $D \cap \mathrm{Cross}_1 (d_1) =D \cap \mathrm{Cross}_1 (d_3) =\emptyset$,
we can prove this lemma in a similar way to the proof of  
Lemma \ref{lemm_RI}
in the case $D \cap \mathrm{Cross}_1 (d_1) =\emptyset$.

We assume $D \cap \mathrm{Cross}_1 (d_1) \neq \emptyset$
or $D \cap \mathrm{Cross}_1 (d_3) \neq \emptyset$.
The claims (1) and (2) are proved simultaneously.
Since $d_1,d_2 \in \mathcal{D}(N,N')$ or 
$d_2, d_3 \in \mathcal{D}(N,N')$,
we have $\bar{\psi}(d_1)=\bar{\psi}(d_2)$ or
$\bar{\psi}(d_2)=\bar{\psi}(d_3)$ from
Proposition \ref{prop_induction} ($N$,$N'$)(6).
By Proposition \ref{prop_induction} ($N$, $N'+1$ )(2)
and Proposition \ref{prop_induction} ($N$, $N+1$)(3),
we can use  the skein relation and the trivial knot relation.
By computation using the skein relation and the trivial knot relation,
we obtain $\bar{\psi} (d_1) =\bar{\psi} (d_3)$.
This proves the lemma.

\end{proof}

\begin{lemm}[Proposition  \ref{prop_induction} ($N$, $N'+1$ )(7)]
Let $d_1$ and $d_2$  be two tangle diagrams
which are identical except for a disk $D$.
If  they differ as shown in
the right and the left one of 
RIII-a, RIII-a', RIII-b, RIII-b', 
RIII-c, RIII-c', RIII-d and RIII-d' in 
Figure \ref{fig:RIIIs}, respectively, in the disk,
we have $\bar{\psi}(d_1) = \bar{\psi}(d_2)$.
\end{lemm}

\begin{proof}
If $D \cap \mathrm{Cross}_1 (d_1) =D \cap \mathrm{Cross}_1 (d_3) =\emptyset$,
we can prove this lemma in a similar way to the proof of  
Lemma \ref{lemm_RI}
in the case $D \cap \mathrm{Cross}_1 (d_1) =\emptyset$.

We prove in the case
$D \cap \mathrm{Cross}_1 (d_1) \neq \emptyset$.
We assume that
$d_1$ and $d_2$  are two tangle diagrams
which are identical for the disk $D$, where they differ as shown in
the right and the left one of 
RIII-a. The other cases are proved similarly.
Let $P_1,P_2$ and $P_3$ be  the crossings of $d_1$ in this disk as Figure
\ref{fig_RIIIs}.
Since $P_1, P_2 \in \mathrm{Cross}(d_1) \backslash \mathrm{Cross}_1 (d_1)$
implies $ P_3 \in \mathrm{Cross}(d_1) \backslash \mathrm{Cross}_1 (d_1)$,
it suffices to show $P_1 \in \mathrm{Cross}_1(d_1)$ or $P_2 \in \mathrm{Cross}_1(d_2)$.

We assume $P_2 \in \mathrm{Cross}_1(d_1)$.
Let $d'_1, d'_2$ and $ d'_3$ be tangle diagrams
which are identical $d_1$ except for the disk,
 where they are  as shown in
Figure \ref{fig_RIIIs}.
By Proposition \ref{prop_induction} ($N$,$N'$)(7), we have
$
\bar{\psi}(d'_1)=\bar{\psi}(d'_2).
$
By Proposition \ref{prop_induction} ($N$, $N'+1$ )(2), we have
\begin{equation*}
\bar{\psi}(d_1) =\bar{\psi}(d'_1)-h\bar{\psi}(d'_3)=\bar{\psi}(d'_2)-h\bar{\psi}(d'_3)=
\bar{\psi}(d_2).
\end{equation*}

We assume $P_1 \in \mathrm{Cross}_1(d_1)$.
Let $d''_1, d''_2, d''_3$ be tangle diagrams
which are identical to $d_1$ except for
 the disk, where they are  
as shown in
Figure \ref{fig_RIIIs}.
By Proposition  \ref{prop_induction} ($N$,$N'$)(7),
we have $\bar{\psi}(d''_2) = \bar{\psi}(d''_4)$.
By Proposition \ref{prop_induction}
 ($N$, $N'+1$ )(2) (6), 
we have
\begin{equation*}
\bar{\psi}(d_1) =\bar{\psi}(d''_2)+h\bar{\psi}(d''_1)=\bar{\psi}(d''_4)+h\bar{\psi}(d''_3)=
\bar{\psi}(d_2).
\end{equation*}
This proves the lemma.
\end{proof}

\begin{figure}
\begin{picture}(300,160)
\put(0,90){\includegraphics[width=1.6cm]{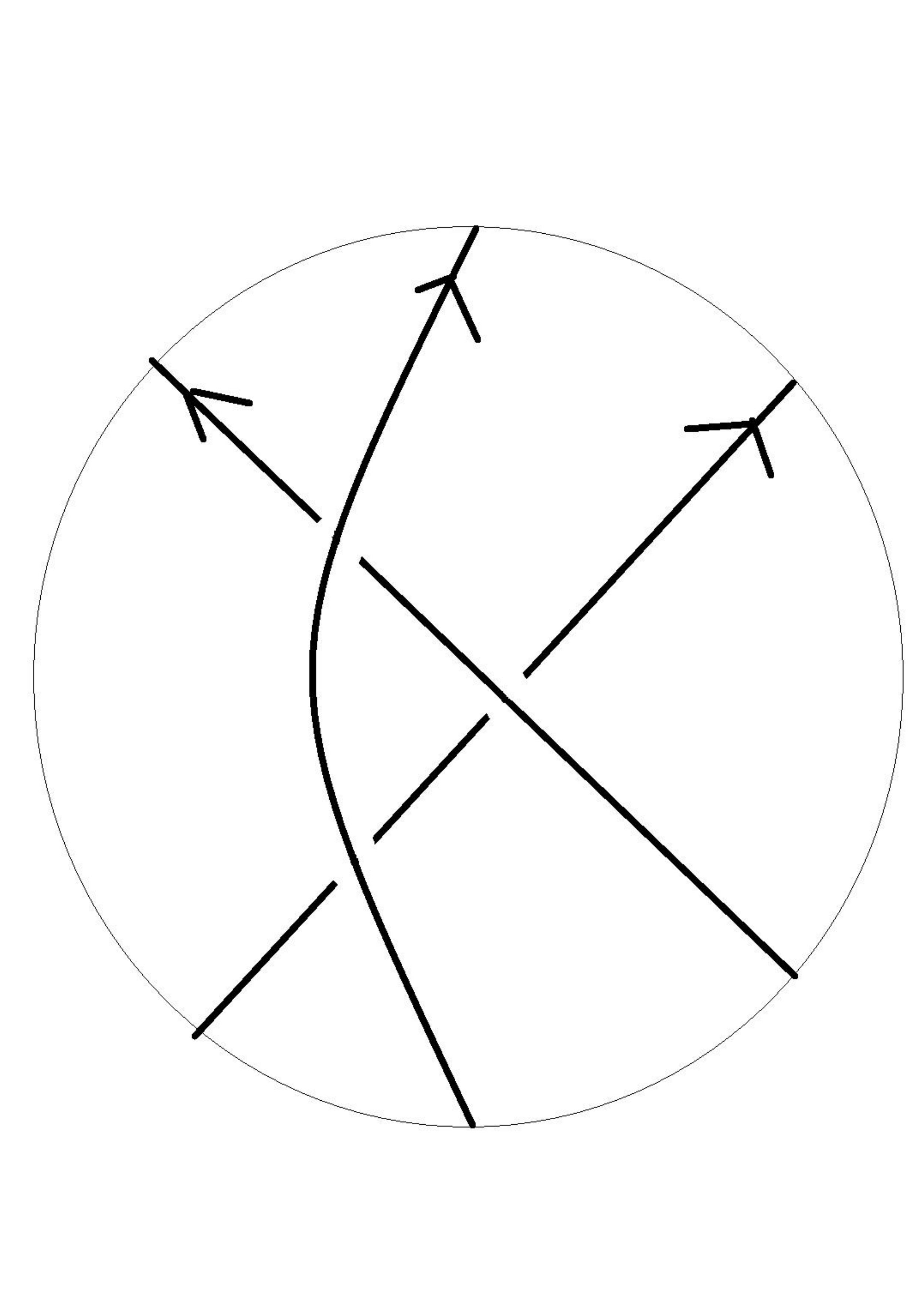}}
\put(75,90){\includegraphics[width=1.6cm]{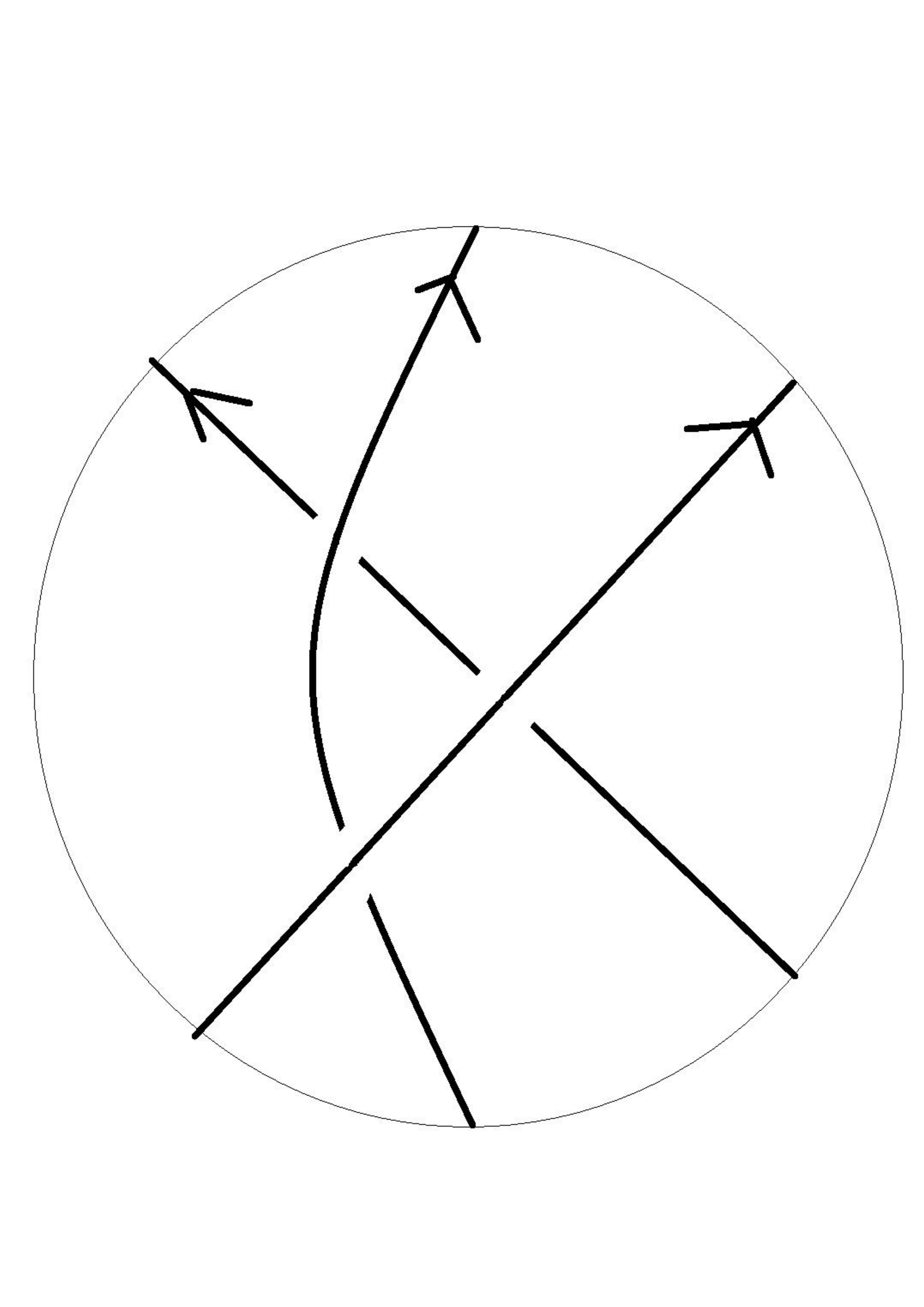}}
\put(150,90){\includegraphics[width=1.6cm]{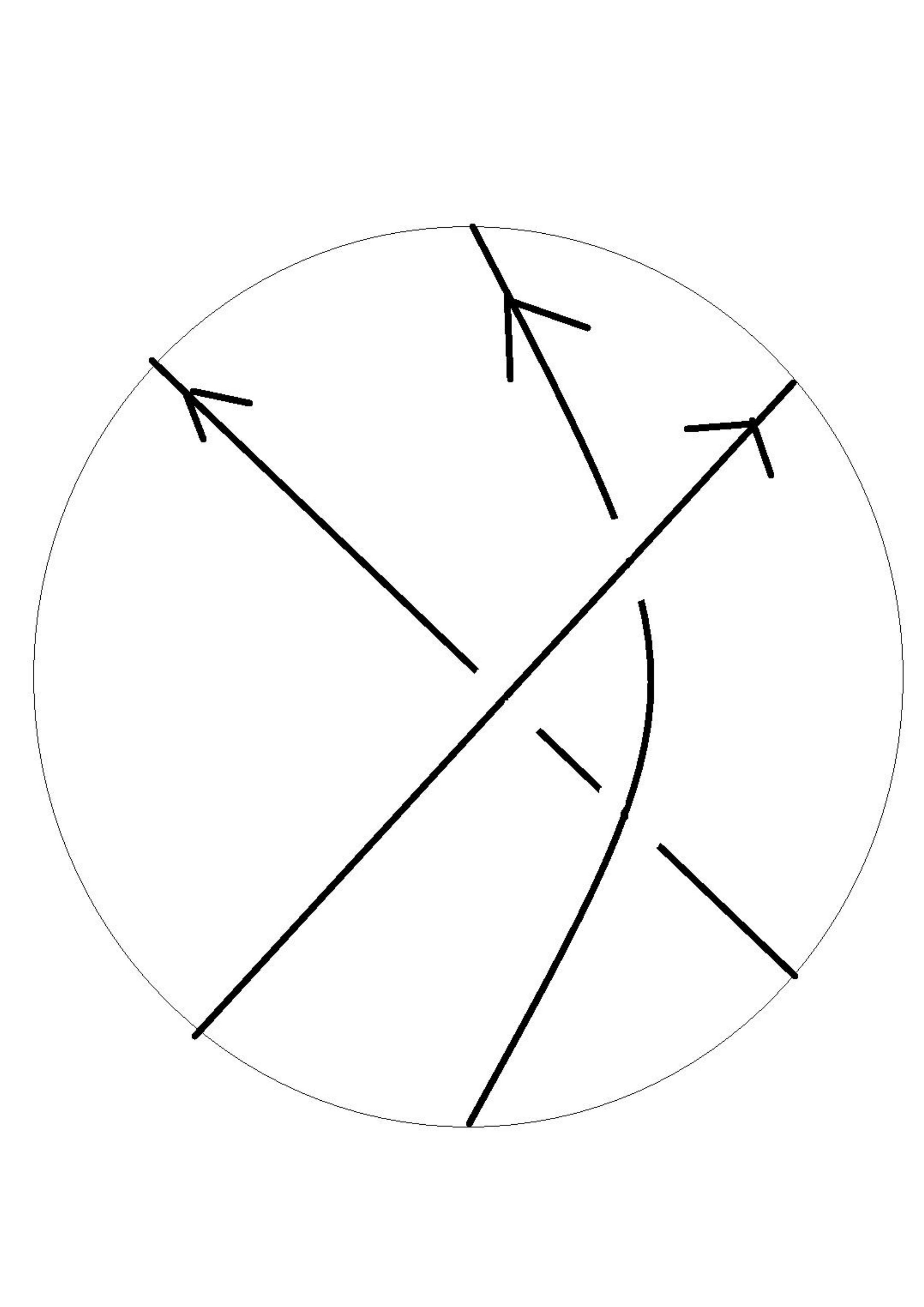}}
\put(225,90){\includegraphics[width=1.6cm]{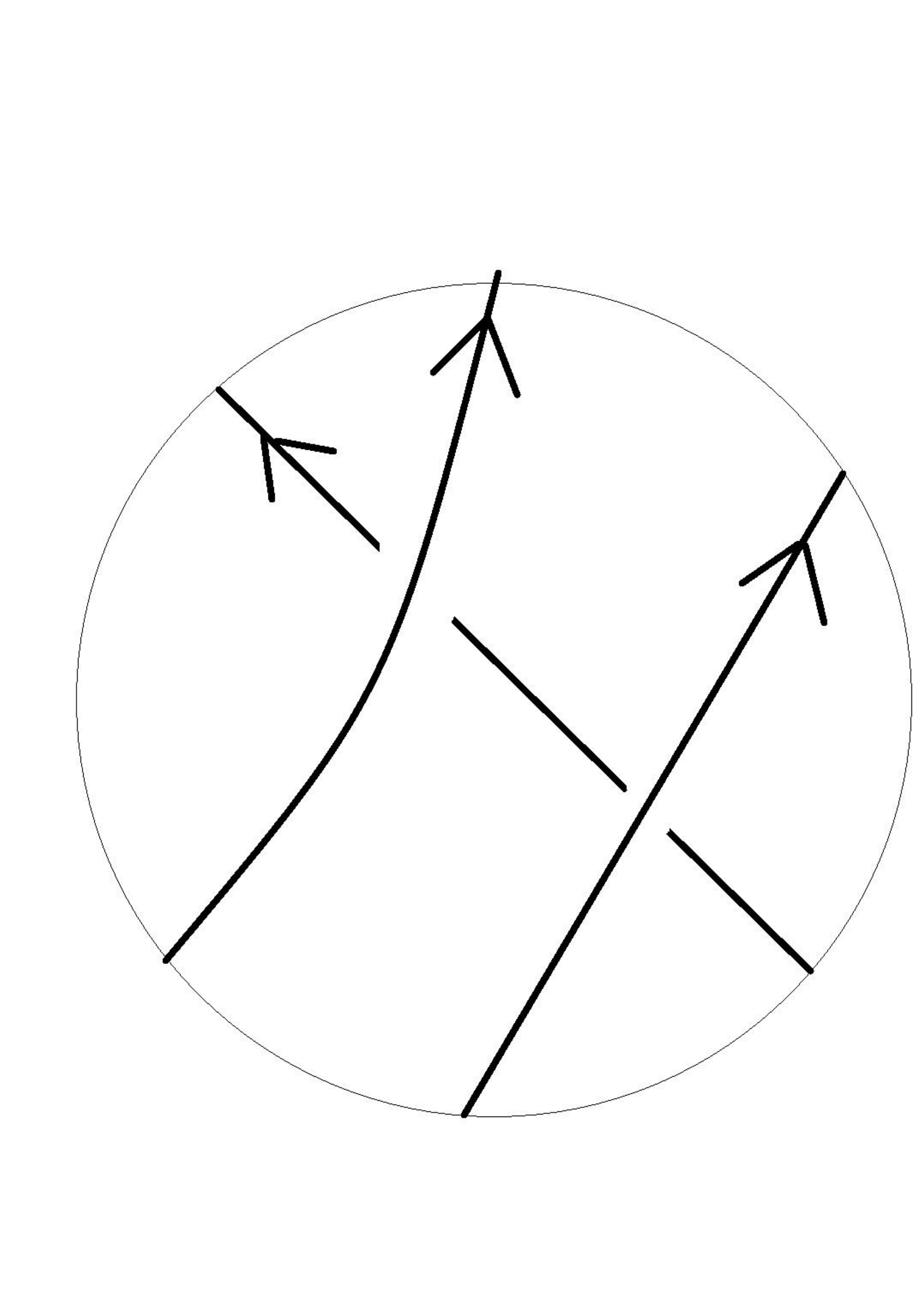}}

\put(0,10){\includegraphics[width=1.6cm]{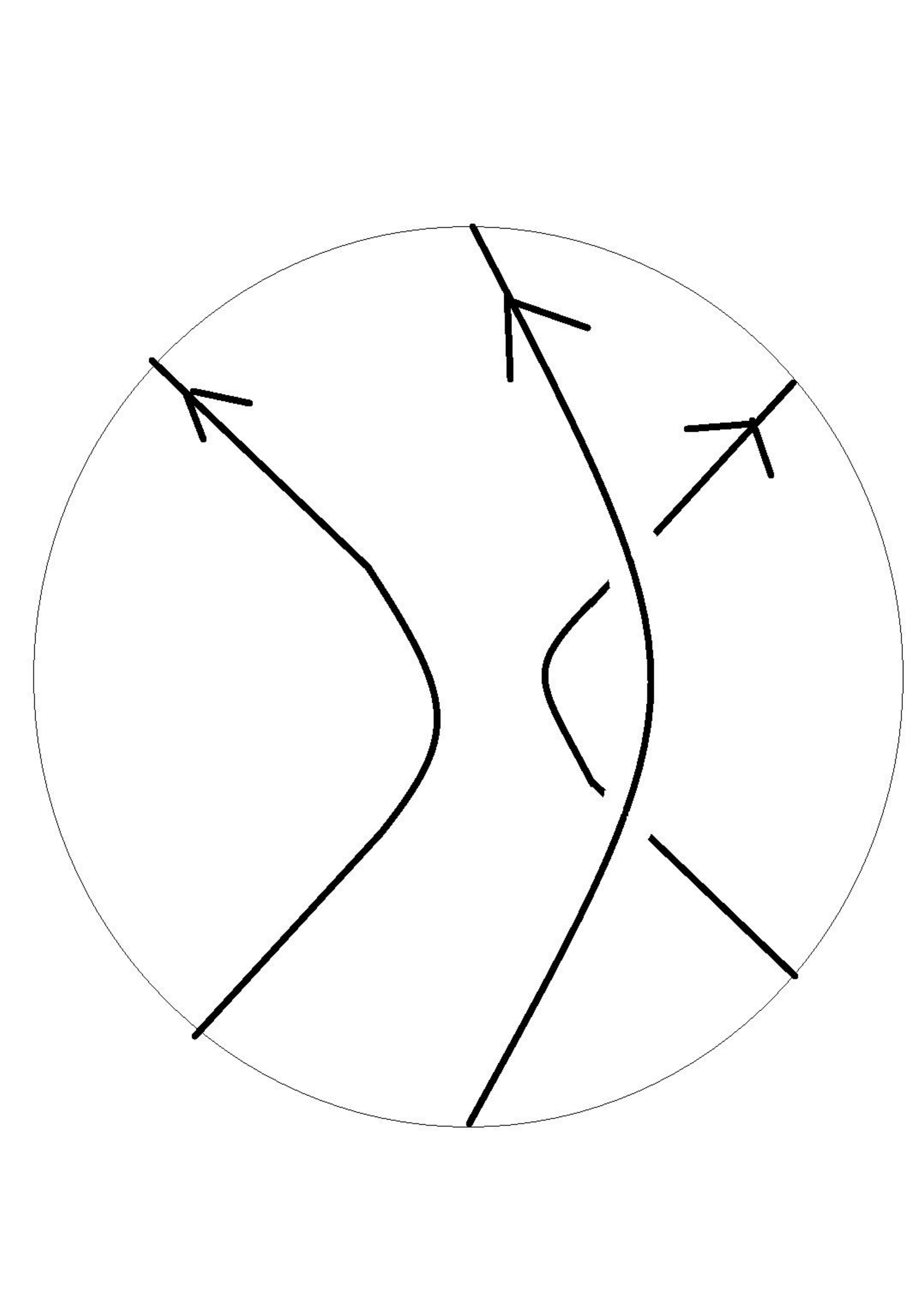}}
\put(75,10){\includegraphics[width=1.6cm]{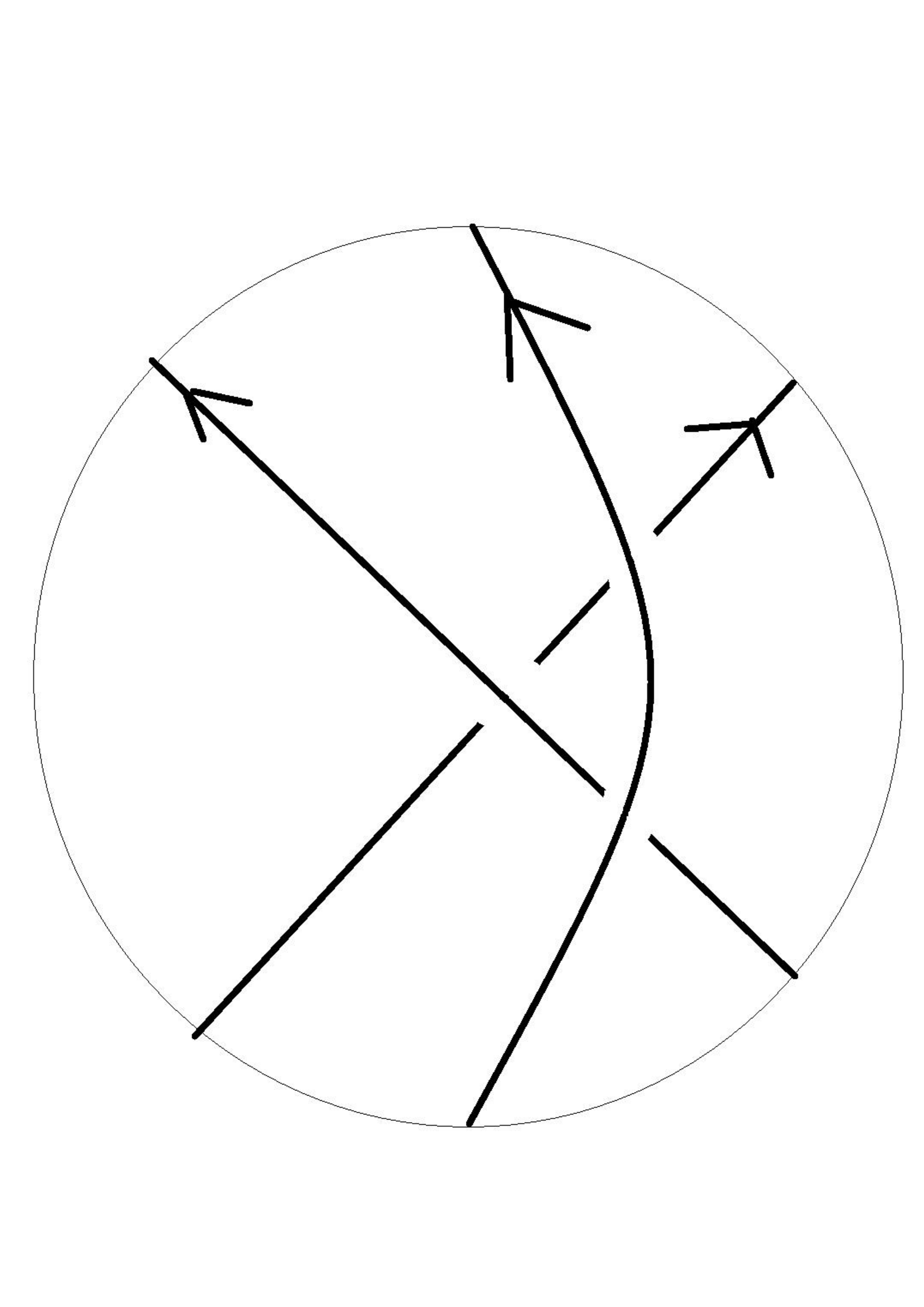}}
\put(150,10){\includegraphics[width=1.6cm]{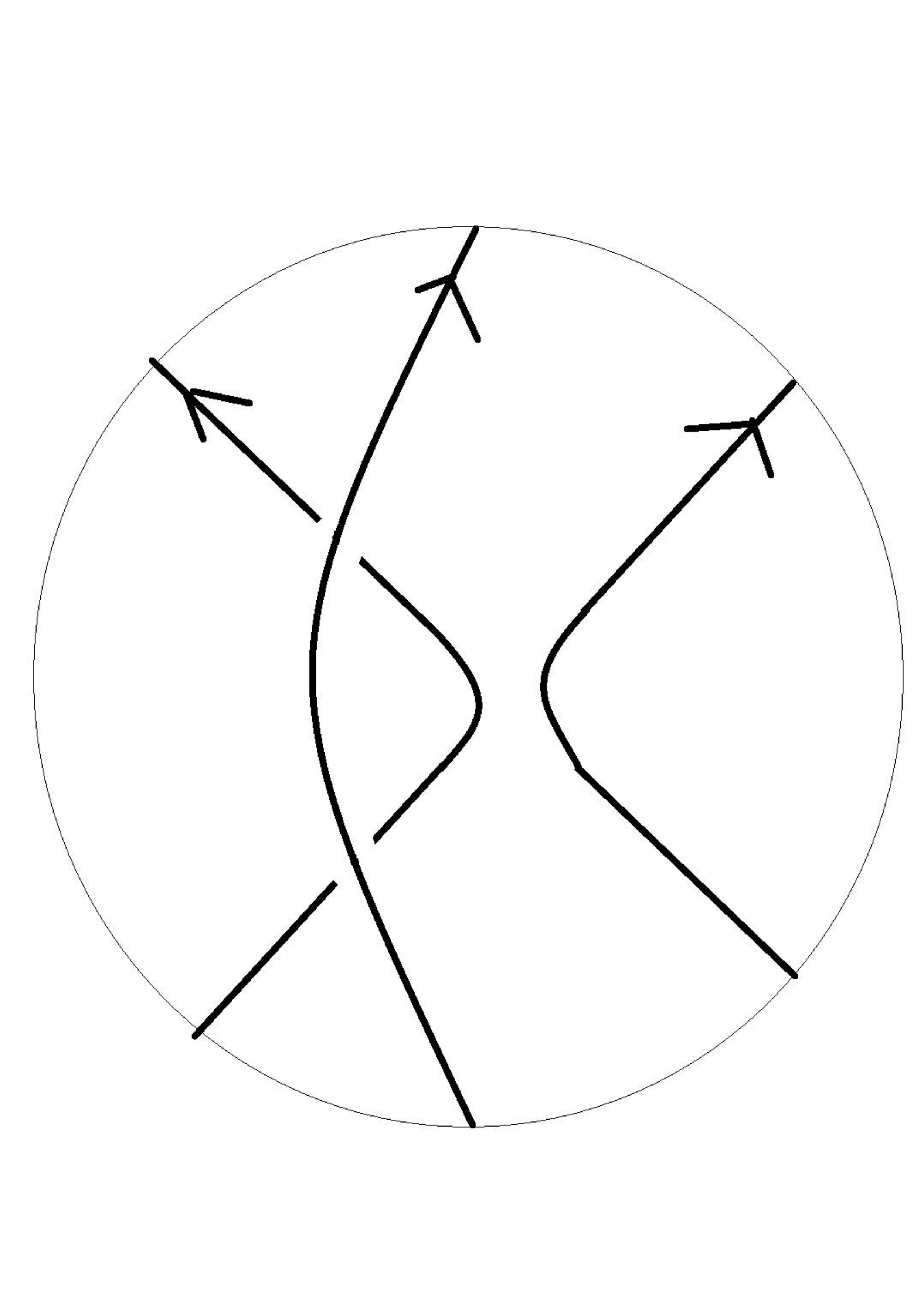}}
\put(225,10){\includegraphics[width=1.6cm]{RIII_o22_PNG.pdf}}

\put(25,80){$d_1$}
\put(100,80){$d'_1$}
\put(175,80){$d'_2$}
\put(250,80){$d'_3$}

\put(25,0){$d''_1$}
\put(100,0){$d''_2$}
\put(175,0){$d''_3$}
\put(250,0){$d''_4$}

\put(27,118){\tiny{$P_1$}}
\put(7,112){\tiny{$P_3$}}
\put(7,124){\tiny{$P_2$}}
\end{picture}
\caption{}
\label{fig_RIIIs}
\end{figure}

\begin{proof}[Proof of Proposition \ref{prop_psi_isom}]
Using the above lemmas, $\bar{\psi}(d)$ induces
\begin{equation*}
\bar{\psi} : \tskein{\Sigma,J^-, J^+} \to
\bigotimes_{ *' \in J^+} \tskein{\Sigma, J^- \backslash 
\shuugou{*}, J^+ \backslash \shuugou{*'}} \otimes_{\Q [\rho][[h]]}
\mathcal{P} (\Sigma,\shuugou{*}, \shuugou{*'}) 
\end{equation*}
by $\bar{\psi} ([T(d)]) =\bar{\psi}(d)$.
By definition, we have
\begin{equation*}
\psi \circ \bar{\psi} =\id_{\tskein{\Sigma,J^-, J^+}},
\bar{\psi} \circ \psi = \id_{\bigotimes_{ *' \in J^+} \tskein{\Sigma, J^- \backslash 
\shuugou{*}, J^+ \backslash \shuugou{*'}} \otimes_{\Q [\rho][[h]]}
\mathcal{P} (\Sigma,\shuugou{*}, \shuugou{*'}) }.
\end{equation*}
This proves the proposition.
\end{proof}

Using Proposition \ref{prop_tskein_disk} and Proposition 
\ref{prop_psi_isom}, we have the following.

\begin{cor}
The map $\frac{h}{2 \sinh{\rho h}} \zettaiti{\cdot}:
\tskein{D, *_1} \to \tskein{D} \simeq \Q [\rho][[h]]$
is an algebra isomorphism.
\end{cor}

\begin{prop}[\cite{Turaev} Theorem 5.2.]
The skein algebra $\tskein{S^1 \times I}$
is the free $\Q [\rho] [[h]]$-algebra
generated by $\shuugou{\zettaiti{\psi (r^n)}|n \in \Z \backslash \shuugou{0}}$,
where $r$ is a generator of the fundamental group of $S^1 \times I$.
\end{prop}

Using this proposition and Proposition \ref{prop_psi_isom}, 
we have the following.

\begin{cor}
The skein module $ \tskein{S^1 \times I, J^-,J^+}$
is a free $\Q [\rho][[h]]$-module.
\end{cor}

\subsection{Poisson-like structure} 
In this section, we define a Lie bracket $[ \ \ , \ \ ]:\tskein{\Sigma} \times
\tskein{\Sigma} \to \tskein{\Sigma}$
and an action $\sigma :\tskein{\Sigma} \times \tskein{\Sigma,J^-,J^+}
\to \tskein{\Sigma,J^-,J^+}$
which makes $\tskein{\Sigma,J^-,J^+}$ a $\tskein{\Sigma}$-module.
Here we consider $\tskein{\Sigma}$ as a Lie algebra.

\begin{df}
Let $J^+,J^-,J'^+$ and $J'^-$ be four disjoint finite subsets of $\partial \Sigma$
with $\sharp (J^+) =\sharp (J^-)$ and 
$\sharp(J'^+) =\sharp(J'^-)$.
We choose tangle diagrams $d_1$ and $d_2$ which present elements of 
$\mathcal{T}(\Sigma,J^-.J^+)$ and $ \mathcal{T}
(\Sigma,J'^-,J'^+)$, respectively, and
suppose $d_1\cup d_2$ has at worst transverse
double points. 
Choose an order of $d_1\cap d_2= \shuugou{P_1, P_2, \cdots , P_m}$.
We denote by $d (\epsilon'_1, \epsilon'_2, \cdots , \epsilon'_m)$
the tangle diagram which are identical except for
sufficiently small neighborhoods such that $d (\epsilon'_1, \cdots ,\epsilon'_m)$
as shown in Type C(+), Type C(-) or Type C(0)
in Figure \ref{fig_Conway_triples}
in the neighborhood of $P_i$
if $\epsilon'_i =1,-1,0$, respectively, for $i =1,2, \cdots,m$.
We denote by $\epsilon_j$ 
the local intersection number of $d_1$ and $d_2$ at $P_j$.
We define
\begin{equation*}
\sigma (d_1, d_2)  \defeq \sum_{j=1}^m \epsilon_j [T(d(-\epsilon_1, \cdots ,-\epsilon_{j-1},0,
\epsilon_{j+1}, \cdots, \epsilon_m))]
\in \tskein{\Sigma,J^-\cup J'^-, J^+ \cup J'^+}.
\end{equation*}

\end{df}

We remark that $h \sigma (d_1,d_2) =[T(d_1)] \boxtimes [T(d_2)]-
[T(d_2)] \boxtimes [T(d_1)]$.

\begin{lemm}
\begin{enumerate}
\item The element $\sigma (d_1,d_2)$ is independent of the choice of the order 
of $d_1 \cap d_2$.
\item We have $\sigma (d_1,d_2) =-\sigma (d_2,d_1)$.
\item Let  $d'_1$ and $d'_2$ be tangle diagrams which are identical to
$d_1$ and $d_2$, respcectively,
 except for a disk, where they are as shown in Fig \ref{fig_regular_RIIa}
or Fig \ref{fig_regular_RIIb}.
This move corresponds to the Reidemeister move II in knots or tangles.
Then we have $\sigma (d_1,d_2) =\sigma (d'_1,d'_2)$.
\item Let  $d'_1$ and $d'_2$ be tangle diagrams which are identical to
$d_1$ and $d_2$, respcectively,
 except for a disk, where they are as shown in Fig \ref{fig_regular_RIII}.
This move corresponds to the Reidemeister move III in knots or tangles.
Then we have $\sigma (d_1,d_2) =\sigma (d'_1,d'_2)$.
\item Let $J^+,J^-,J'^+,J'^-,J''^+$ and $J''^-$  be
 mutually disjoint finite subsets of $\partial \Sigma$
with $\sharp (J^+) =\sharp (J^-)$, 
$\sharp(J'^+) =\sharp(J'^-)$ and 
$\sharp(J''^+) =\sharp(J''^-)$.
We choose $d_1$, $d_2$ and $d_3$ which present elements
$\mathcal{T}(\Sigma,J^-.J^+)$, $ \mathcal{T}
(\Sigma,J'^-,J'^+)$ and $\mathcal{T}(\Sigma,J''^-,J''^)$, respectively, and
suppose $d_1\cup d_2 \cup d_3$
 has at worst transverse
double points. 
Then we have $\sigma (\sigma(d_1,d_2),d_3) =\sigma (d_1, \sigma(d_2,d_3))-
\sigma (d_2, \sigma (d_1,d_3))$.
\end{enumerate}
\end{lemm}

\begin{figure}[htbp]
	\begin{tabular}{rrr}
	\begin{minipage}{0.5\hsize}
\begin{picture}(160,80)
\put(0,0){\includegraphics[width=70pt]{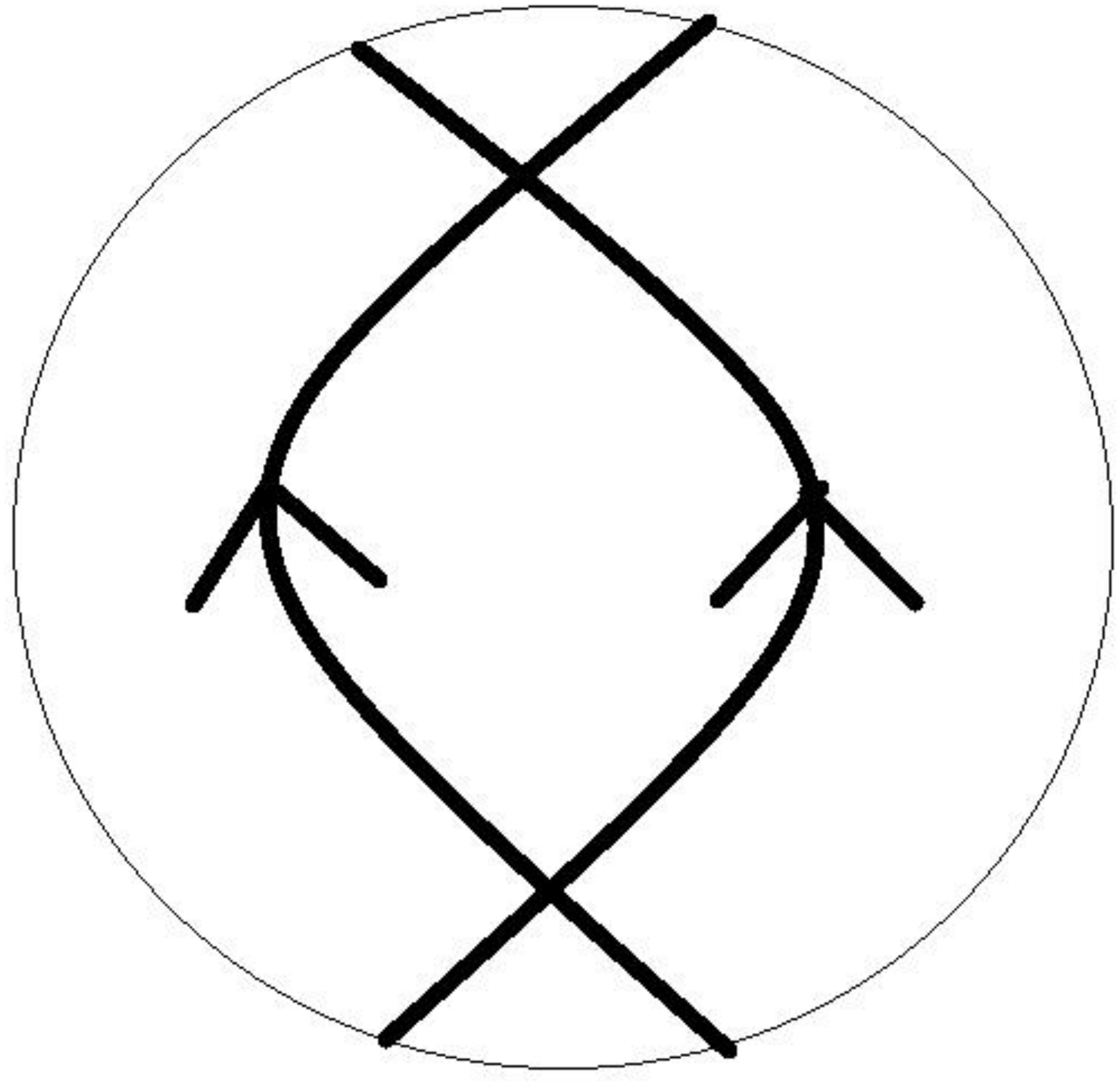}}
\put(80,0){\includegraphics[width=70pt]{oar20_PNG.pdf}}
\put(43,25){$P_1$}
\put(43,65){$P_2$}
\put(24,40){$d_1 d_2$}
\put(94,40){$d'_1 \ \ \ \ \ \ d'_2$}
\end{picture}
\caption{}
\label{fig_regular_RIIa}
	\end{minipage}
	\begin{minipage}{0.5\hsize}
\begin{picture}(160,80)
\put(0,0){\includegraphics[width=70pt]{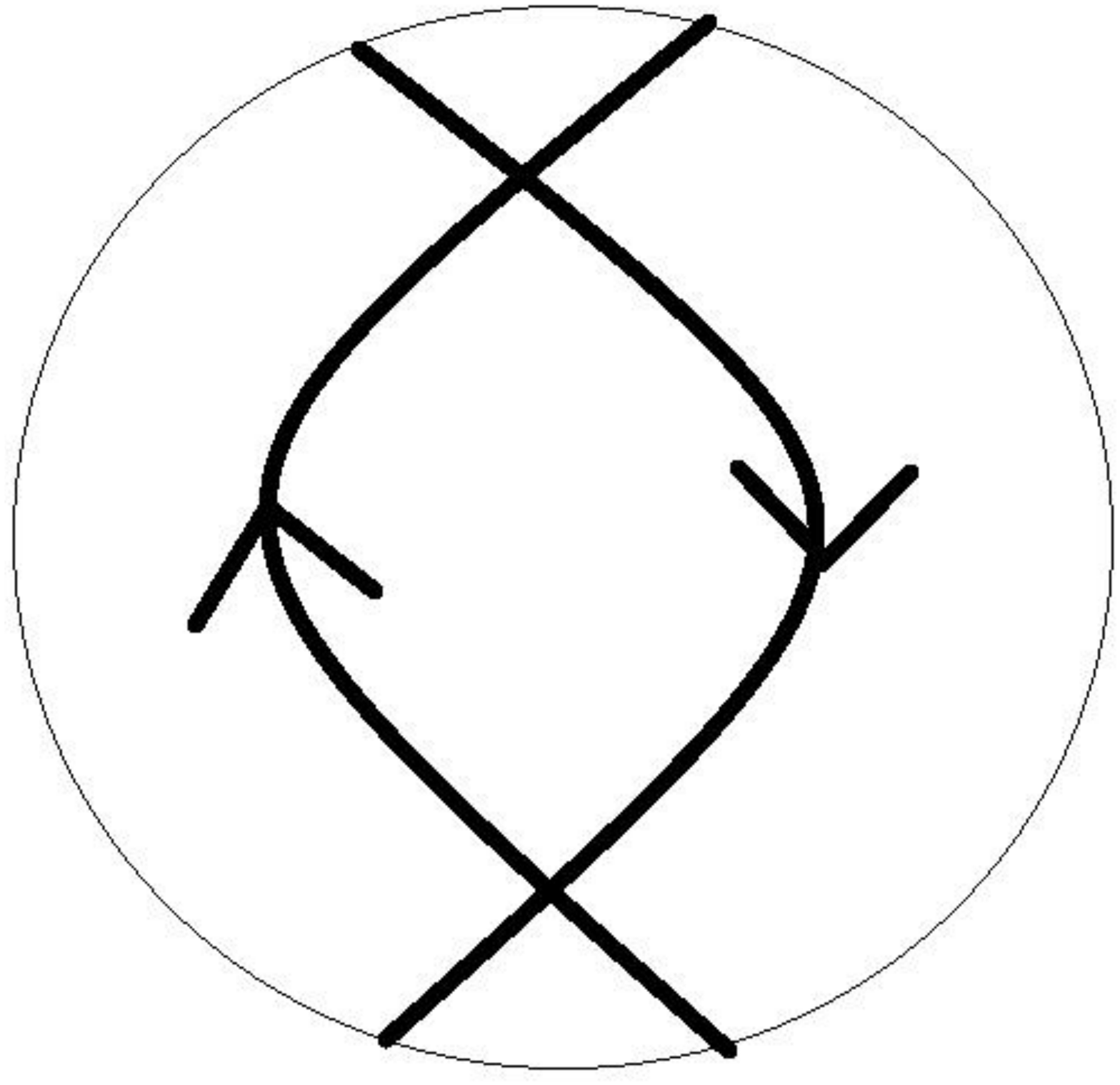}}
\put(80,0){\includegraphics[width=70pt]{obr20_PNG.pdf}}
\put(43,25){$P_2$}
\put(43,65){$P_1$}
\put(24,40){$d_1 d_2$}
\put(94,40){$d'_1 \ \ \ \ \ \ d'_2$}
\end{picture}
\caption{}
\label{fig_regular_RIIb}
	\end{minipage}
	\end{tabular}
\end{figure}

\begin{figure}
\begin{picture}(160,80)
\put(0,0){\includegraphics[width=70pt]{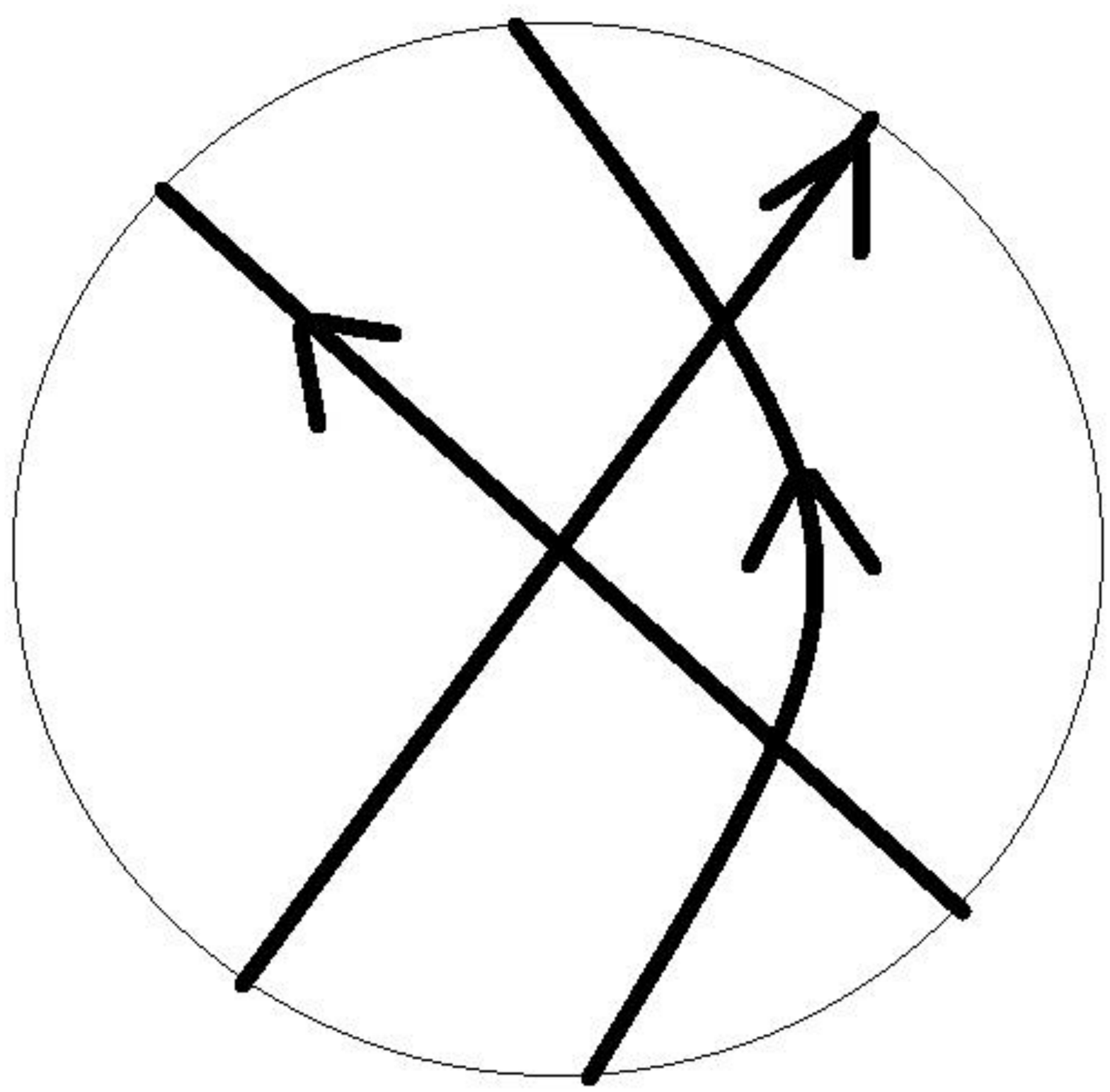}}
\put(80,0){\includegraphics[width=70pt]{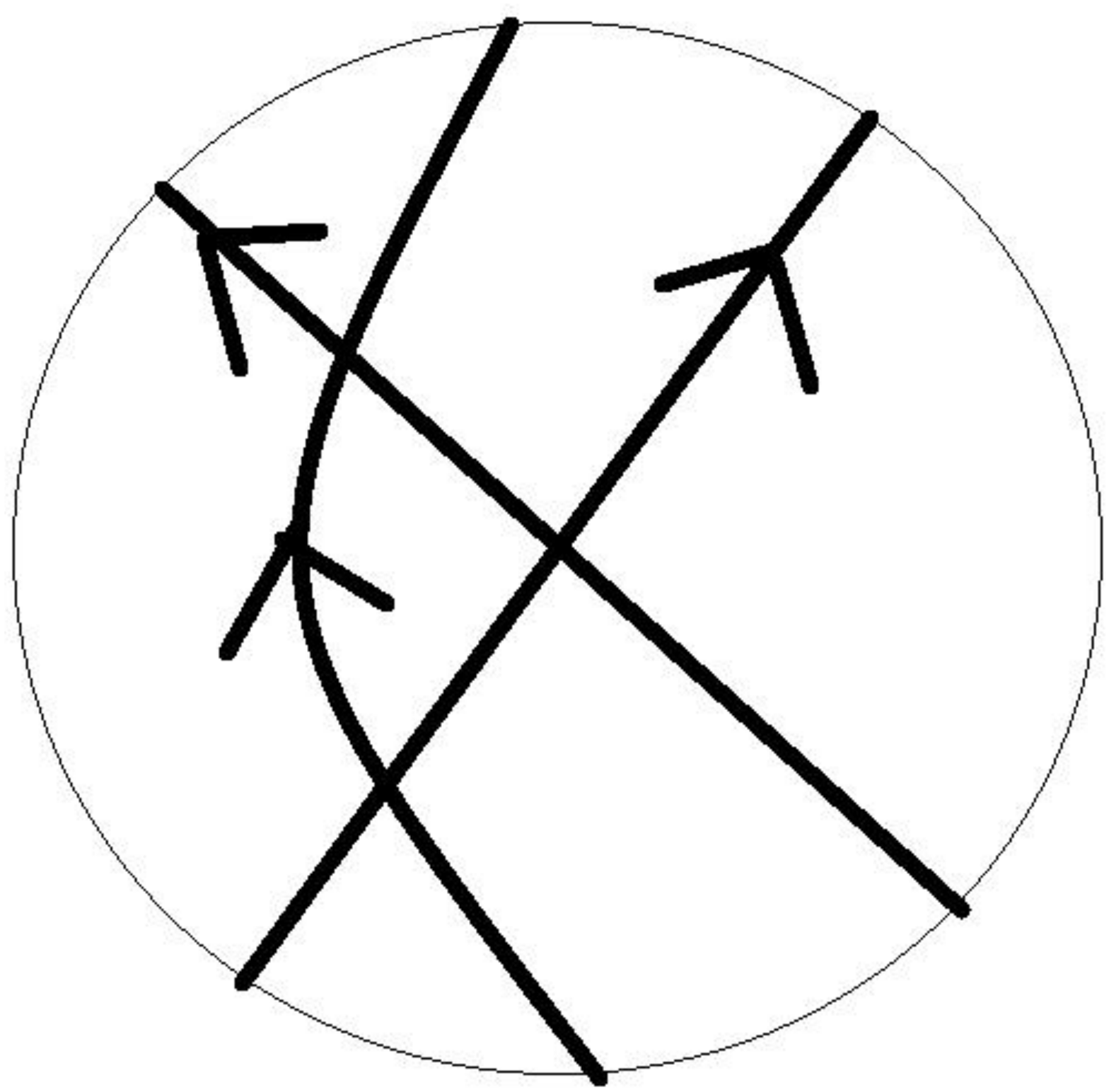}}
\put(53,35){$P_1$}
\put(53,55){$P_2$}
\put(24,45){$d_2$}
\put(35,10){$d_1$}
\put(90,35){$P'_2$}
\put(90,52){$P'_1$}
\put(123,45){$d'_2$}
\put(115,10){$d'_1$}
\end{picture}
\caption{}
\label{fig_regular_RIII}
\end{figure}

\begin{proof}[Proof of (1)]
It suffices to show that 
\begin{align*}
&\epsilon_j [T(d (\epsilon'_1, \cdots ,\epsilon'_m)_{|\epsilon'_i =\epsilon_i, \epsilon'_j=0})]
+
\epsilon_i [T(d (\epsilon'_1, \cdots ,\epsilon'_m)_{|\epsilon'_i =0, \epsilon'_j=-\epsilon_j})] \\
&=\epsilon_j [T(d (\epsilon'_1, \cdots ,\epsilon'_m)_{|\epsilon'_i =-\epsilon_i, \epsilon'_j=0})]
+
\epsilon_i [T(d (\epsilon'_1, \cdots ,\epsilon'_m)_{|\epsilon'_i =0, \epsilon'_j=\epsilon_j})] .\\
\end{align*}
By skein relation, we have
\begin{align*}
&\epsilon_j [T(d (\epsilon'_1, \cdots ,\epsilon'_m)_{|\epsilon'_i =\epsilon_i, \epsilon'_j=0})]
-
\epsilon_j [T(d (\epsilon'_1, \cdots ,\epsilon'_m)_{|\epsilon'_i =-\epsilon_i, \epsilon'_j=0})] \\
&=\epsilon_i [T(d (\epsilon'_1, \cdots ,\epsilon'_m)_{|\epsilon'_i =0, \epsilon'_j=\epsilon_j})]
-
\epsilon_i [T(d (\epsilon'_1, \cdots ,\epsilon'_m)_{|\epsilon'_i =0, \epsilon'_j=-\epsilon_j})] \\
&=\epsilon_i \epsilon_j
[T(d (\epsilon'_1, \cdots ,\epsilon'_m)_{|\epsilon'_i =0, \epsilon'_j=0})].
\end{align*}
This prove the lemma (1).
\end{proof}

\begin{proof}[Proof of (2)]
Choosing the inverse order $d_2 \cap d_1 
=\shuugou{P_m, \cdots , P_1}$,
$\sigma (d_2, d_1)$ can be written as
\begin{equation*}
\sigma (d_2, d_1)  \defeq \sum_{j=1}^m -\epsilon_j [T(d(-\epsilon_1, \cdots ,-\epsilon_{j-1},0,
\epsilon_{j+1}, \cdots, \epsilon_m))]
\in \tskein{\Sigma,J^-\cup J'^-, J^+ \cup J'^+}.
\end{equation*}
Since
\begin{equation*}
\sigma (d_1, d_2)  \defeq \sum_{j=1}^m \epsilon_j [T(d(-\epsilon_1, \cdots ,-\epsilon_{j-1},0,
\epsilon_{j+1}, \cdots, \epsilon_m))]
\in \tskein{\Sigma,J^-\cup J'^-, J^+ \cup J'^+},
\end{equation*}
we obtain $\sigma (d_1, d_2) =-\sigma (d_2, d_1)$.
This proves the lemma (2).
\end{proof}

\begin{proof}[Proof of (3)]
By this lemma (1), we can suppose that
$P_1$ and $P_2$ are in the disk such as
in Figure \ref{fig_regular_RIIa}
or Figure \ref{fig_regular_RIIb}.
It suffices to show 
\begin{align*}
 [T(d (0,1, \epsilon'_3, \cdots ,\epsilon'_m))]
= [T(d (1,0, \epsilon'_3, \cdots ,\epsilon'_m))]. \\
\end{align*}
The equation is obvious.
We remark  that we need the framing relation
if $d_1$ and $d_2$ are as shown in Fig \ref{fig_regular_RIIb}.
This proves the lemma (3).
\end{proof}

\begin{proof}[Proof of (4)]
By this lemma (1), we can suppose that
$P_1$ and  $P_2$ are in the disk
such as in Figure \ref{fig_regular_RIII}.
Let $P'_1$ and $P'_2$ be two points
such as in Figure \ref{fig_regular_RIII}.
Fix the order  $d'_1\cap d'_2 = \shuugou{P'_1, P'_2,P'_3 \cdots , P'_m}$
where $P_i= P'_i$ for $i=3, \cdots, m$.
We denote by $d' (\epsilon'_1, \epsilon'_2, \cdots , \epsilon'_m)$
the tangle diagram which are identical except for
the neighborhoods such that $d' (\epsilon'_1, \cdots ,\epsilon'_m)$
as shown in 
Type C(+), Type C(-) or Type C(0)
in Figure \ref{fig_Conway_triples}
in the neighborhood of $P'_i$
if $\epsilon'_i =1,-1,0$, respectively, for $i =1,2, \cdots,m$.
We have
\begin{align*}
&T(d (0,-1,\epsilon'_3, \cdots , \epsilon'_m))=
T(d' (0,-1, \epsilon'_3\cdots , \epsilon'_m)), \\
&T(d (-1,0,\epsilon'_3, \cdots , \epsilon'_m))=
T(d' (-1,0, \epsilon'_3, \cdots , \epsilon'_m)), \\
&T(d (-1,1,\epsilon'_3, \cdots , \epsilon'_m))=
T(d' (-1,1,\epsilon'_3, \cdots , \epsilon'_m)), \\
\end{align*}
for any $\epsilon'_i \in  \shuugou{-1,0,1}$.
This proves the lemma (4).

\end{proof}

\begin{proof}[Proof of (5)]
Let  $P_1, P_2, \cdots , P_m$  be the crossings of $d_1 \cap d_2$,
$P_{m+1}, \cdots ,P_{m+m'}$  the crossings of $d_1 \cap d_3$
and $P_{m+m'+1}, \cdots ,P_{m+m'+m''}$
the crossings of $d_2 \cap d_3$.
We denote by $d' (\epsilon'_1, \epsilon'_2, \cdots , \epsilon'_m)$
the tangle diagram which are identical except for
the neighborhoods such that $d' (\epsilon'_1, \cdots ,\epsilon'_m)$
as shown in
Type C(+), Type C(-) or Type C(0)
in Figure \ref{fig_Conway_triples}
in the neighborhood of $P'_i$
if $\epsilon'_i =1,-1,0$, respectively, for $i =1,2, \cdots,m$.
We denote by $\epsilon_j$ 
the local intersection number of $d_1$ and $d_2$
, $d_1$ and $d_3$ or $d_2$ and $d_3$ at $P_j$
for $j =1,2, \cdots ,m+m'+m''$.

In this proof, we simply denote
\begin{align*}
T(i, \emptyset , j) \defeq
&T(d(-\epsilon_1, \cdots, -\epsilon_{i-1},0, 
\epsilon_{i+1}, \cdots, \epsilon_m, 
\epsilon_{m+1}, \cdots, \epsilon_{m+m'}, \\
&-\epsilon_{m+m'+1}, \cdots,
-\epsilon_{j-1}, 0, \epsilon_{j+1}, \cdots, 
 \epsilon_{m+m'+m''})), \\
T(i,j ,\emptyset) \defeq
&T(d(-\epsilon_1, \cdots, -\epsilon_{i-1},0, 
\epsilon_{i+1}, \cdots, \epsilon_m, \\
&-\epsilon_{m+1}, \cdots, -\epsilon_{j-1},0, 
\epsilon_{j+1}, \cdots, \epsilon_{m+m'}, -\epsilon_{m+m'+1}, \cdots, 
-\epsilon_{m+m'+m''})), \\
T( \emptyset, i,j) \defeq
&T(d(-\epsilon_1, \cdots, - \epsilon_m, 
-\epsilon_{m+1}, \cdots,-\epsilon_{i-1}, 0,\epsilon_{i+1}, \cdots,
 \epsilon_{m+m'}, \\
&-\epsilon_{m+m'+1}, \cdots,
-\epsilon_{j-1}, 0, \epsilon_{j+1}, \cdots, 
 \epsilon_{m+m'+m''})).
\end{align*}
Choose the order $d_1 \cap d_2 = \shuugou{P_1, P_2, \cdots ,P_m}$
and the order $(d_1 \cup d_2) \cap d_3 = \shuugou{P_{m+m'+1}, P_{m+2}, \cdots ,P_{m+m'+m'},
P_{m+1}, \cdots, P_{m+m'}}$
By definition, we have
\begin{align*}
\sigma(\sigma(d_1,d_2),d_3) =&\sum_{i  \in \shuugou{1, \cdots, m}, j \in \shuugou{m+m'+1, \cdots, m+m'+m''}} \epsilon_i \epsilon_j [T(i, \emptyset, j)] \\
&+\sum_{i  \in \shuugou{1, \cdots, m}, j \in \shuugou{m+1, \cdots, m+m'}} \epsilon_i \epsilon_j [T(i,j \emptyset)]. \\
\end{align*}

Choose the order $d_2 \cap d_3 = \shuugou{P_{m+m'+1}, P_{m+m'+2}, \cdots ,P_{m+m'+m''}}$
and the order $(d_2 \cup d_3) \cap d_1 = \shuugou{P_{1}, P_{2}, \cdots ,P_{m+m'}}$.
By definition, we have
\begin{align*}
\sigma(d_1,\sigma(d_2,d_3)) 
&=\sum_{i  \in \shuugou{1, \cdots, m}, j \in \shuugou{m+m'+1, \cdots, m+m'+n''}} \epsilon_i \epsilon_j [T(i, \emptyset ,j)] \\
&+\sum_{i  \in \shuugou{m+1, \cdots, m+m'}, j \in \shuugou{m+m'+1, \cdots, m+m'+m''}} \epsilon_i \epsilon_j [T(\emptyset,i,j)]. \\
\end{align*}

Choose the order $d_1 \cap d_3 = \shuugou{P_{m+1}, P_{m+1}, \cdots, P_{m+m'}}$ and
the order 
$(d_1 \cup d_3) \cap d_2 = \shuugou{P_{m+m'+1},  P_{m+m'+2}, \cdots, P_{m+m'+m''},
P_m, P_{m-1}, \cdots, P_2, P_1}$.
By definition, we have
\begin{align*}
\sigma(d_2,\sigma(d_1,d_3)) 
&=\sum_{i  \in \shuugou{m+1, \cdots, m+m'}, j \in \shuugou{m+m'+1, \cdots, m+m'+m''}} \epsilon_i \epsilon_j [T(\emptyset,i,j)] \\
&-\sum_{i  \in \shuugou{1, \cdots, m}, j \in \shuugou{m+1, \cdots, m+m'}} \epsilon_i \epsilon_j [T(i,j, \emptyset)]. \\
\end{align*}

From the above equations, we obtain
\begin{equation*}
\sigma(\sigma(d_1,d_2),d_3)=\sigma(d_1,\sigma(d_2,d_3))
-\sigma (d_2, \sigma(d_1,d_3)).
\end{equation*}
This proves the lemma (5).
\end{proof}

By the above lemma(1)(3)(4),  we define
\begin{align*}
[[T(d_1)],[T(d_2)]]=\sigma(d_1,d_2), \\
\sigma ([T(d_1)])([T(d_3)]) = \sigma (d_1,d_3).
\end{align*}
for tangle diagrams $d_1 $ and $d_2$
presenting elements of $\mathcal{T}(\Sigma)$
and a tangle diagram $d_3$
presenting an element of $\mathcal{T}(\Sigma,J^-,J^+)$.
By the above lemma (2) (5), we have
\begin{align*}
&[x_1,x_2]=-[x_2,x_1], \\
&[x_1,[x_2,x_3]]+[x_2,[x_3,x_1]]+[x_3,[x_1,x_2]]=0, \\
&\sigma ([x_1,x_2])(z)= \sigma (x_1)(\sigma(x_2)(z))-\sigma (x_2)(\sigma(x_1)(z)). \\
\end{align*}
So the bracket $[ \ \ , \ \ ]:\tskein{\Sigma} \times
\tskein{\Sigma} \to \tskein{\Sigma}$ is a Lie bracket and
the action $\sigma :\tskein{\Sigma} \times \tskein{\Sigma,J^-,J^+}
\to \tskein{\Sigma,J^-,J^+}$ makes
 $\tskein{\Sigma,J^-,J^+}$ a $\tskein{\Sigma}$-module.
Here we consider $\tskein{\Sigma}$ as a Lie algebra.

Let $d_1$, $d_2$ and $d_3$ be tangle diagrams presenting knots in
$\Sigma \times I$.
We choose an order $d_1 \cap (d_2 \cup d_3)= \shuugou{P_1, \cdots , P_m,P'_1, 
\cdots, P'_{m'}}$ such that $P_i \in d_1 \cap d_2$ and $P'_i \in d_1 \cap d_3$.
Then we have the Leibniz rule:
\begin{equation*}
[x,yz]=[x,y]z+y[x,z].
\end{equation*}
In a similar way, we obtain the Leibniz rules:
\begin{align*}
&[xy,z]=x[y,z]+[x,z]y, \\
&\sigma (x)(y v_1)=[x,y]v_1+y \sigma(x)(v_1), \\
&\sigma(x)(v_1 y)=\sigma (x)(v_1)y+v_1 [x,y], \\
&\sigma (x)(v_1 v_2) =v_1 \sigma(x)(v_2)+\sigma(x)(v_1)v_2,
\end{align*}
for $x,y,z \in \tskein{\Sigma}$, $v_1 \in \tskein{\Sigma,*_\beta,*_\gamma}$ and
$v_2 \in \tskein{\Sigma,*_\alpha,*_\beta}$.
Let $J^+,J^-,J'^+$ and $J'^-$ be mutually disjoint finite subsets of $\partial \Sigma$
with $\sharp (J^+) =\sharp (J^-)$ and 
$\sharp(J'^+) =\sharp(J'^-)$.
Furthermore, we obtain the Leibniz rule:
\begin{equation*}
\sigma (x)(v_1 \boxtimes v_2)=\sigma (x)(v_1) \boxtimes v_2 +
v_1 \boxtimes \sigma (x)(v_2)
\end{equation*}
for $x \in \tskein{\Sigma}$, $v_1 \in \tskein{\Sigma,J^-,J^+}$
and $v_2 \in \tskein{\Sigma, J'^-,J'^+}$.

\section{Filtration and completion}
Let $T \mathcal{P}(\Sigma,*_\alpha)$ be the tensor algebra  $\oplus_{k=0}^\infty
\mathcal{P}(\Sigma,*_\alpha)^{\otimes k}$ of 
$\mathcal{P}(\Sigma,*_\alpha)$
over $\Q [\rho] [[h]]$.
We remark that $\mathcal{P}(\Sigma,*_\alpha)^{\otimes 0} =\Q [\rho][[h]]$.
We consider the augmentation map $\epsilon:
\mathcal{P}(\Sigma,*_\alpha) \to \Q$ defined by
$\epsilon (R) =1, \epsilon (h)=0$
and  the surjective maps 
\begin{align*}
&\psi_{T}: T \mathcal{P}(\Sigma,*_\alpha) \to \tskein{\Sigma}, \\
&\psi_{T,*_\alpha}: T\mathcal{P}(\Sigma,*_\alpha) \otimes 
\mathcal{P}(\Sigma,*_\alpha) \to \tskein{\Sigma, *_\alpha}, \\
&\psi_{T,*_\alpha, *_\beta}:T\mathcal{P}(\Sigma,*_\alpha) \otimes 
\mathcal{P}(\Sigma,*_\alpha, *_\beta) \to \tskein{\Sigma, *_\alpha, *_\beta}, \\
\end{align*}
defined by
\begin{align*}
&\psi_{T} (a_1 \otimes a_2 \otimes \cdots \otimes a_j) =\zettaiti{\psi(a_1)}
\zettaiti{\psi (a_2)} \cdots \zettaiti{\psi (a_j)}, \\
&\psi_{T,*_\alpha}(A \otimes a) =\psi_{T} (A) \psi(a), \\
&\psi_{T,*_\alpha, *_\beta}(A \otimes a) =\psi_{T} (A) \psi(a).
\end{align*}
We define  filtrations  $\filtn{F^n \tskein{\Sigma}}$,
$\filtn{F^n \tskein{\Sigma,*_\alpha}}$ and 
$\filtn{F^n \tskein{\Sigma, *_\alpha, *_\beta}}$
by 
\begin{align*}
&F^n \tskein{\Sigma} = \sum_{2i_0+i_1+i_2+ \cdots +i_j \geq n}
h^{i_0} \psi_{T}( (\ker \epsilon)^{i_1} \otimes (\ker \epsilon)^{i_2}
\otimes \cdots \otimes (\ker \epsilon)^{i_j}), \\
&F^n \tskein{\Sigma,*_\alpha} = \sum_{i_0+i_1 \geq n}
F^{i_0} \tskein{\Sigma} \psi ((\ker \epsilon)^{i_1}), \\
&F^n \tskein{\Sigma,*_\alpha, *_\beta} = \sum_{i_0+i_1 \geq n}
F^{i_0} \tskein{\Sigma} \psi (\mathcal{P}
(\Sigma,*_\alpha, *_\beta)(\ker \epsilon)^{i_1}). \\
\end{align*}
Let $*, *'$ be points of $\partial \Sigma$
and $\chi$ a diffeomorphism
$\Sigma \to \Sigma$ satisfying 
that $\chi (*) =*_\alpha \in \shuugou{*_1, \cdots , *_b}$, $\chi (*')
=*'_\beta\in \shuugou{*'_1, \cdots ,*'_b}$
and $\chi (\partial) = \partial$ for any component $\partial$ of $\partial \Sigma$.
In the case that $* \in \partial \Sigma \backslash
\shuugou{*_1, \cdots , *_b}$ and $*' \in \partial \Sigma \backslash
\shuugou{*'_1, \cdots, *'_b}$, 
we define a filtration $\filtn{F^n\tskein{\Sigma,*,*'}}$ by
\begin{equation*}
F^n \tskein{\Sigma,\shuugou{*},\shuugou{*'}} =\chi (F^n \tskein{\Sigma,*_\alpha, *_\beta}).
\end{equation*}
Let $J^+$ and $J^-$ be finite disjoint subsets of
$\partial \Sigma$ with $\sharp J^+ =\sharp J^-$ and
$*$ an element of $J^-$.
In the case that $\sharp J^-=\sharp J^+ \geq 2$, we define the filtration
$\filtn{\tskein{\Sigma,J^-,J^+}}$ by
\begin{equation*}
F^n \tskein{\Sigma,J^-,J^+} =\sum_{i_0+i_1 \geq n}
\bigoplus_{*' \in J^+} (F^{i_0}\tskein{\Sigma, J^- \backslash \shuugou{*},
J^+ \backslash \shuugou{*'}}) (F^{i_1} \tskein{\Sigma, \shuugou{*}, \shuugou{*'}}).
\end{equation*}

\subsection{The filtrations depend only on the under lying $3$-manifolds}

Let $\Sigma_{0,b+1}$  be a compact connected surface
of genus $0$ with $b+1$ boundary components
and $R_1, R_2, \cdots , R_b$ the elements of $\pi^+_1 (\Sigma,*_1)$
as in Figure \ref{fig_sigma_0_b_r}.
We denote $\eta_{b} \defeq \psi ((R_1-1)(R_2- 1) \cdots
(R_{b-1}- 1)(R_b-1))$. In particular $\eta_0= 1 \in \tskein{\Sigma_{0,1}, *_1}$. Any embedding 
\begin{equation*}
\iota :
\coprod_{1 \leq j \leq M+N} \Sigma_{b_j+1,1} \times I  
\sqcup \coprod_{i} \Sigma_{0,2} \times I \to 
\Sigma \times I
\end{equation*}
satisfying $\iota (\shuugou{*_1 \in \Sigma_{b_j+1,0}|1 \leq j \leq M} \times I)
= J^+ \times [\frac{1}{4}, \frac{3}{4}]$ and 
$\iota (\shuugou{*'_1 \in \Sigma_{b_0+1,0}|1 \leq j \leq M} \times I)
=  J^- \times [\frac{1}{4}, \frac{3}{4}]$
induces $\iota_* : (\oplus_{j=1}^M \tskein{\Sigma_{b_0+1,0}, *_1}) \oplus 
(\oplus_{j=M+1}^N \tskein{\Sigma_{b_j+1,0}}) 
\oplus (\oplus_i \tskein{\Sigma_{2,0}}) \to
\tskein{\Sigma, *_\alpha}$.

\begin{figure}
\begin{picture}(160,80)
\put(-70,-10){\includegraphics[width=270pt]{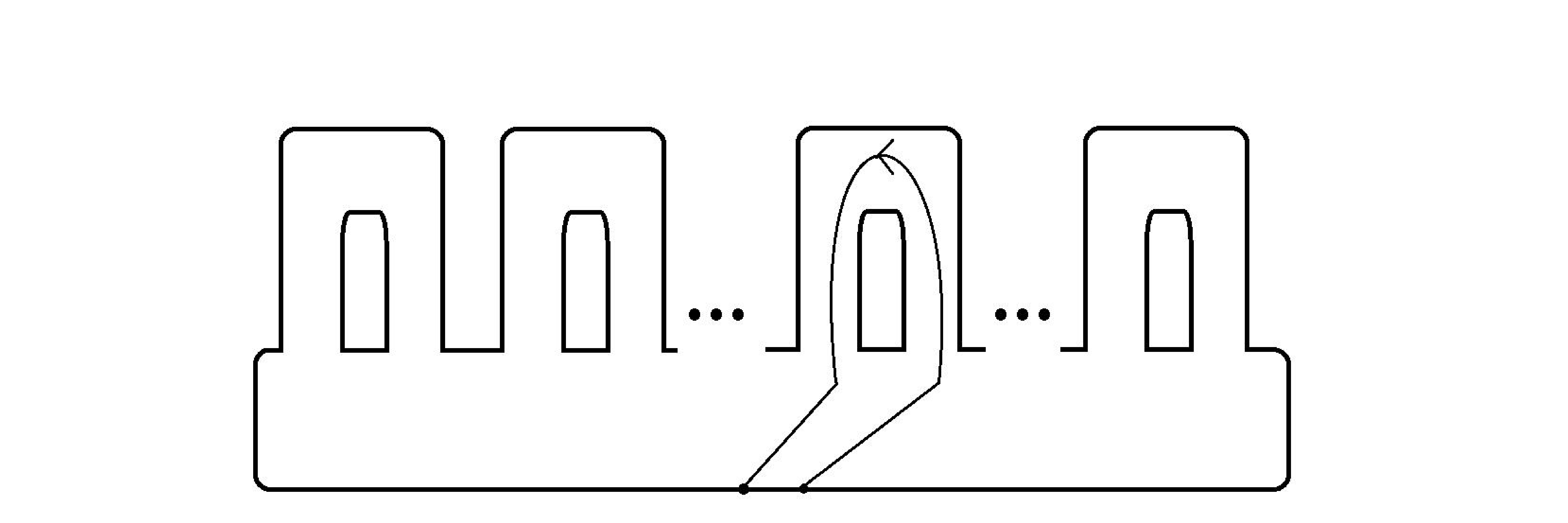}}

\put(52,-12){$*'_1$}
\put(70,-12){$*_1$}

\put(52,10){$R_i$}

\put(-10,60){$1$}
\put(30,60){$2$}
\put(80,60){$i$}
\put(130,60){$b$}
\end{picture}
\caption{}
\label{fig_sigma_0_b_r}
\end{figure}

\begin{df}
Let $J^-$ and $J^+$ be disjoint finite subsets
of $\partial \Sigma$ with  $\sharp (J^+) =\sharp (J^-) =M$.
We define $F^{\star (-1)} \tskein{\Sigma,J^-, J^+}
=F^{\star 0} \tskein{\Sigma,J^-,J^+} \defeq \tskein{\Sigma, J^-,J^+}$.
For $n \geq 1$, a submodule $F^{\star n}  \tskein{\Sigma, J^-,J^+}$
is the submodule generated by $h F^{\star (n-2)} \tskein{\Sigma,
J^-,J^+}$ and
\begin{align*}
&\shuugou{\kukakko{\iota} \defeq \iota_* ((\oplus_{j}^N \eta_{b_1}) \oplus
(\oplus_{j=M+1}^N \zettaiti{\eta_{b_j}}) \oplus (\oplus_i \zettaiti{R_1})) \\
&
|\iota :
\coprod_{1 \leq j \leq M+N} \Sigma_{b_j+1,1} \times I  
\sqcup \coprod_{i} \Sigma_{0,2} \times I \to 
\Sigma \times I \\
&\mathrm{embedding},b_1+ \cdots +b_{M+N} \geq n, \\
&\iota (\shuugou{*_1 \in \Sigma_{b_j+1,0}|1 \leq j \leq M} \times I)
= J^+ \times [\frac{1}{4}, \frac{3}{4}] \\
&\iota (\shuugou{*'_1 \in \Sigma_{b_0+1,0}|1 \leq j \leq M} \times I)
=  J^- \times [\frac{1}{4}, \frac{3}{4}]}.
\end{align*}
\end{df}

The aim of this section is to prove the lemma.

\begin{lemm}
\label{lemm_filt_douti}
We have $F^{\star n} \tskein{\Sigma,J^-,J^+} =F^n \tskein{\Sigma,J^-,J^+}$
for any $n$.
\end{lemm}

Since the filtration $\filtn{F^{\star n} \tskein{\Sigma,J^-,J^+}}$
depends only on the under lying $3$-manifolds,
the filtration $\filtn{F^n \tskein{\Sigma,J^-,J^+}}$
also depends only on the under lying $3$-manifolds
(Theorem \ref{thm_filt_underlying3mfd}).

\begin{prop}
We have $F^{\star n} \tskein{\Sigma,J^-,J^+} \supset F^n \tskein{\Sigma,J^-,J^+}$
for any $n$.
\end{prop}

\begin{proof}
It suffices to show $\psi (X_0(X_1-1)(X_2-1) \cdots
(X_{j-1}-1)(X_j-1)) \in F^{\star j} \tskein{\Sigma,*_\alpha, *_\beta}$
for $X_1, \cdots , X_j \in \pi^{+}_1 (\Sigma, *_\alpha)$ and
$R_0 \in \pi^+_1 (\Sigma, *_\alpha, *_\beta)$.
There is a properly embedding 
$e : (\Sigma_{0,j+1} \times I, *_1 \times I, *'_1 \times I)
\to (\Sigma \times I, *_\alpha \times I , *'_\beta \times I)$
as shown in Figure \ref{fig_embedding_construct},
such that
$e_* (\psi (R_1^{\epsilon_1}R_2^{\epsilon_2} 
\cdots R_{j-1}^{\epsilon_{j-1}}R_j^{\epsilon_j}  ))
=\psi(X_0 X_1^{\epsilon_1} X_2^{\epsilon_2} \cdots X_{j-1}^{\epsilon_{j-1}}X_j^{\epsilon_j})$
for $\epsilon_1, \epsilon_2, \cdots, \epsilon_j \in \shuugou{0,1}$
where we denote by $e_* \tskein{\Sigma_{0,j+1},*_1} \to \tskein{\Sigma,*_\alpha, *_\beta}$
 the $\Q[ \rho ][[h]]$-module homomorphism
induced by $e$.
By definition, we have $\psi (X_0(X_1-1)(X_2-1) \cdots
(X_{j-1}-1)(X_j-1) ) \in F^{\star j} \tskein{\Sigma,*_\alpha, *_\beta}$.
This proves the lemma.
\end{proof}

\begin{figure}
\begin{picture}(400,150)
\put(-50,0){\includegraphics[width=450pt]{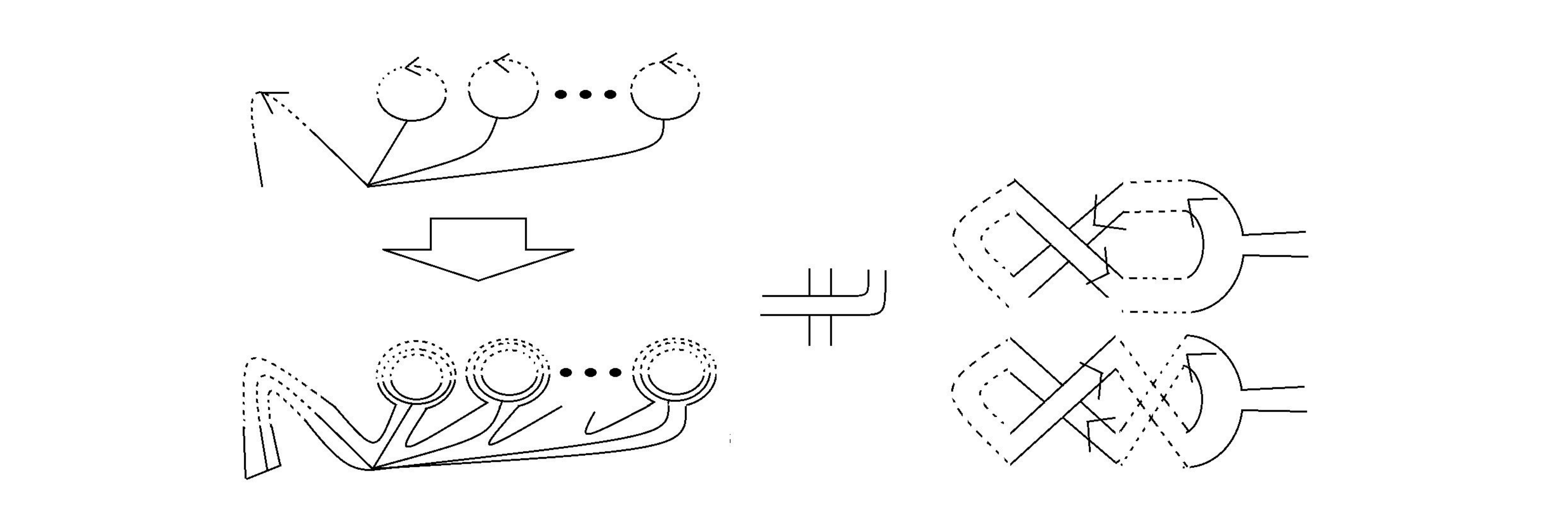}}

\put(20,136){$X_0$}
\put(60,136){$X_1$}
\put(90,136){$X_2$}
\put(135,136){$X_j$}

\put(60,90){$*_\alpha $}
\put(20,90){$*_\beta$}

\put(25,56){$0$}
\put(65,56){$1$}
\put(95,56){$2$}
\put(140,56){$j$}

\put(60,7){$*_1 $}
\put(20,7){$*'_1$}

\put(180,80){$i >i'$}
\end{picture}
\caption{}
\label{fig_embedding_construct}
\end{figure}

To prove Lemma \ref{lemm_filt_douti}, we need the following lemmas.
By straightforward calculations, we have the following two lemmas.

\begin{lemm}
We fix $i , j  \in \shuugou{1, 2, \cdots, b}$.
Let $\Sigma_{\star, i, j}$ be the surface in Figure \ref{fig_sigma_star_i_j},
$\iota_1: \Sigma_{0,b+1} \times I \to \Sigma_{\star,i,j} \times I$ and
$\iota_2 : \Sigma_{0,b+1} \times I \to \Sigma_{\star,i,j} \times I$ the embeddings
as in  Figure \ref{fig_sigma_star_i_j_emb_1},
respectively.
Then there exists an embedding such as Figure \ref{fig_sigma_star_i_j_emb_3}
$\iota_3 :\Sigma_{0,b-j+i+1} \times I  \sqcup \Sigma_{0,j-i+1} \times I \to
\Sigma_{\star,i,j} \times I$ satisfying
\begin{align*}
\kukakko{\iota_1}-\kukakko{\iota_2}=& \iota_{3*} (h((R_1-1)(R_2-1) \cdots(R_{i-1}-1)R_i
(R_{i+1}-1)\cdots(R_{b-j+i}-1)) \oplus  \\ &(\zettaiti{R_1(R_2-1) \cdots (R_{j-i}-1)})
)\in hF^{\star b-2} \tskein{\Sigma_{\star,i,j},*_1}.
\end{align*}
\end{lemm}

\begin{figure}
\begin{picture}(160,150)
\put(-120,-20){\includegraphics[width=400pt]{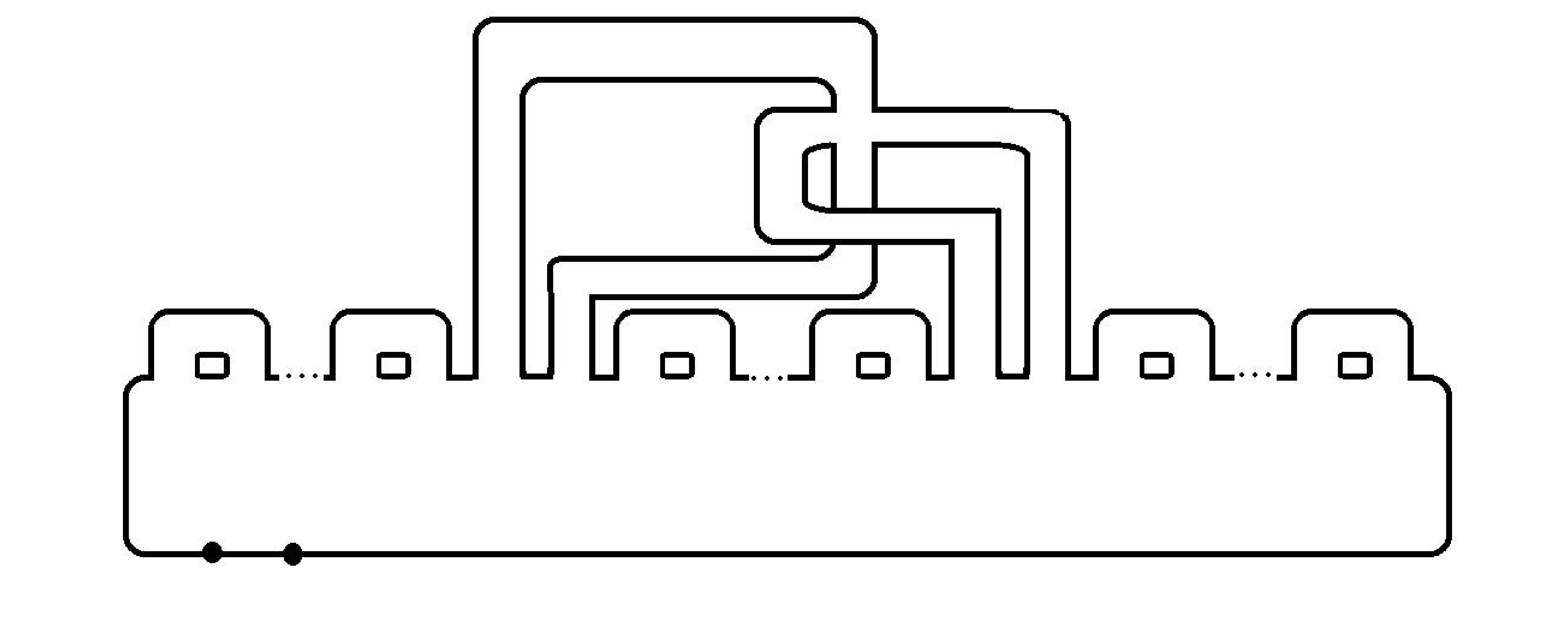}}
\put(-73,-12){$*_1$}
\put(-45,-12){$*'_1$}
\put(-69,36){$1$}
\put(-27,36){$i-1$}
\put(14,36){$i$}
\put(44,36){$i+1$}
\put(93,36){$j-1$}
\put(134,36){$j$}
\put(164,36){$j+1$}
\put(222,36){$b$}
\end{picture}
\caption{$\Sigma_{\star,i,j}$}
\label{fig_sigma_star_i_j}
\end{figure}

\begin{figure}
\begin{picture}(160,150)
\put(-160,-20){\includegraphics[width=480pt]{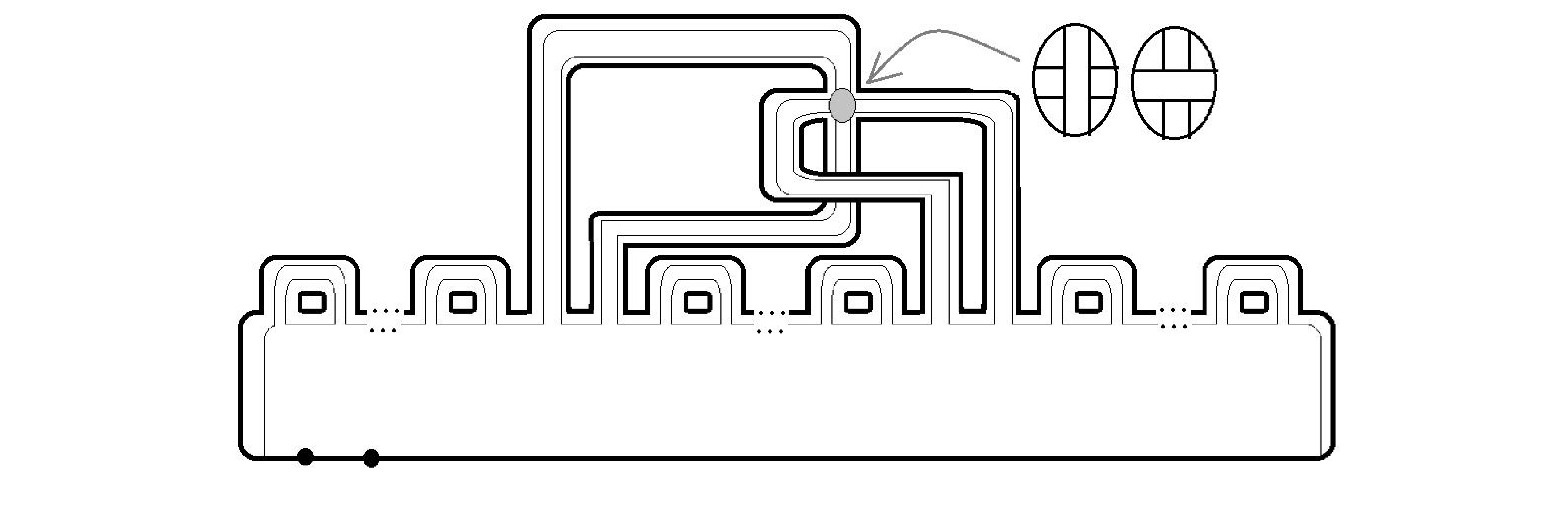}}
\put(-73,-12){$*_1$}
\put(-45,-12){$*'_1$}
\put(-69,36){$1$}
\put(-27,36){$i-1$}
\put(14,36){$i$}
\put(44,36){$i+1$}
\put(93,36){$j-1$}
\put(134,36){$j$}
\put(164,36){$j+1$}
\put(222,36){$b$}

\put(170,80){$\iota_1$}
\put(200,80){$\iota_2$}
\end{picture}
\caption{$\iota_1$ and $\iota_2$}
\label{fig_sigma_star_i_j_emb_1}
\end{figure}

\begin{figure}
\begin{picture}(160,150)
\put(-120,-20){\includegraphics[width=400pt]{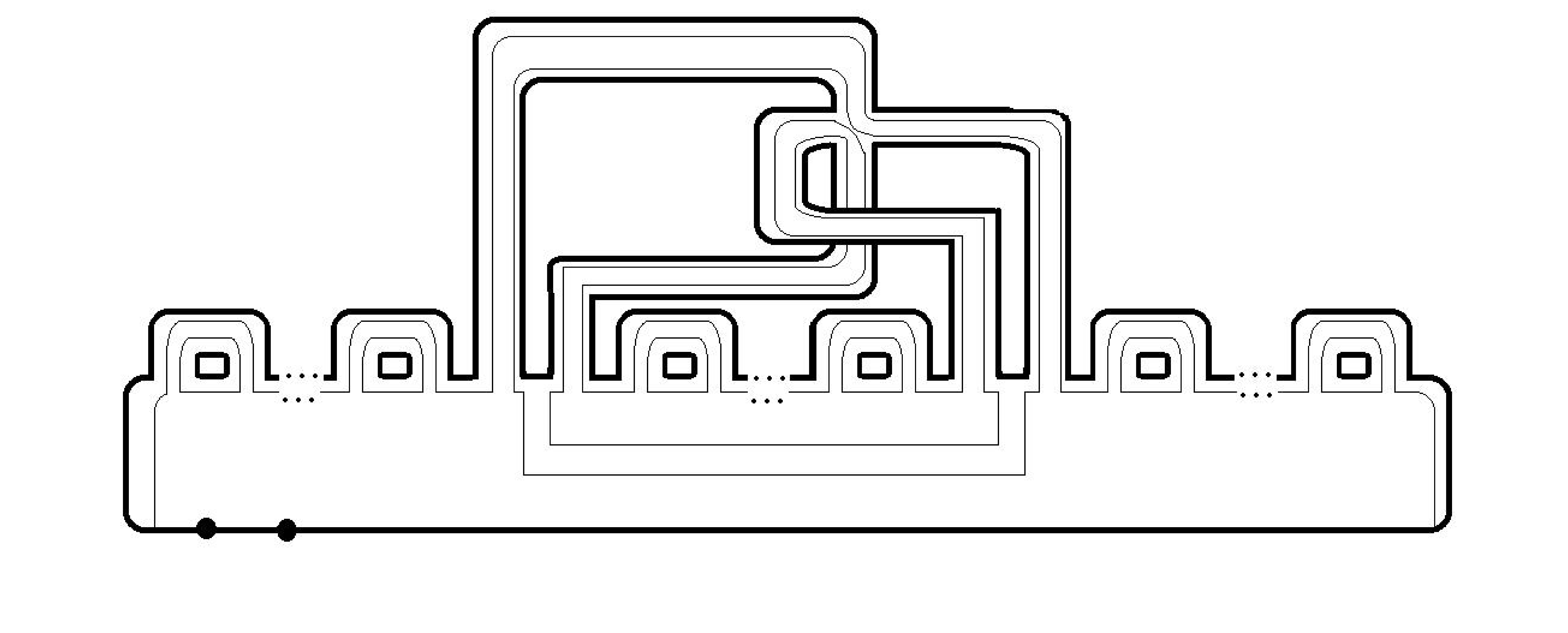}}
\put(-73,-12){$*_1$}
\put(-45,-12){$*'_1$}
\put(-69,36){$1$}
\put(-27,36){$i-1$}
\put(24,110){$i$}

\put(24,80){$1$}
\put(44,36){$2$}
\put(93,36){$j-i$}
\put(134,36){$$}
\put(164,36){$i+1$}
\put(209,36){$b-j+i$}
\end{picture}
\caption{$\iota_3$}
\label{fig_sigma_star_i_j_emb_3}
\end{figure}

\begin{lemm}
We fix $i  \in \shuugou{1, 2, \cdots, b_1}$ and $ j \in \shuugou{1,2, \cdots, b_2}$.
Let $\Sigma_{\star \star, i, j}$ be the surface in Figure \ref{fig_sigma_star_star_i_j},
$\iota_1: \Sigma_{0,b_1+1} \times I \sqcup \Sigma_{0,b_2+1} \times I
\to \Sigma_{\star \star,i,j} \times I$ and
$\iota_1: \Sigma_{0,b_1+1} \times I \sqcup \Sigma_{0,b_2+1} \times I
\to \Sigma_{\star \star,i,j} \times I$ the embeddings
as in Figure \ref{fig_sigma_star_star_i_j_emb_1},
respectively.
Then there exists an embedding such as Figure \ref{fig_sigma_star_star_i_j_emb_3}
$\iota_3 :\Sigma_{0,b_2-j+i} \times I  \sqcup \Sigma_{0,b_1-i+j} \times I \to
\Sigma_{\star \star,i,j} \times I$ satisfying
\begin{align*}
\iota_{1*} (\eta_{b_1} \oplus \eta_{b_2})-\iota_{2*} (\eta_{b_1} \oplus {\eta_{b_2}})
 = \iota_{3* }(h(\eta_{b_2-j+i-1} \oplus \eta_{b_1+j-i-1}))
\in hF^{\star b_1+b_2-2}\tskein{\Sigma_{\star \star,i,j}, *_1}.
\end{align*}

\end{lemm}

\begin{figure}
\begin{picture}(160,150)
\put(-120,-20){\includegraphics[width=400pt]{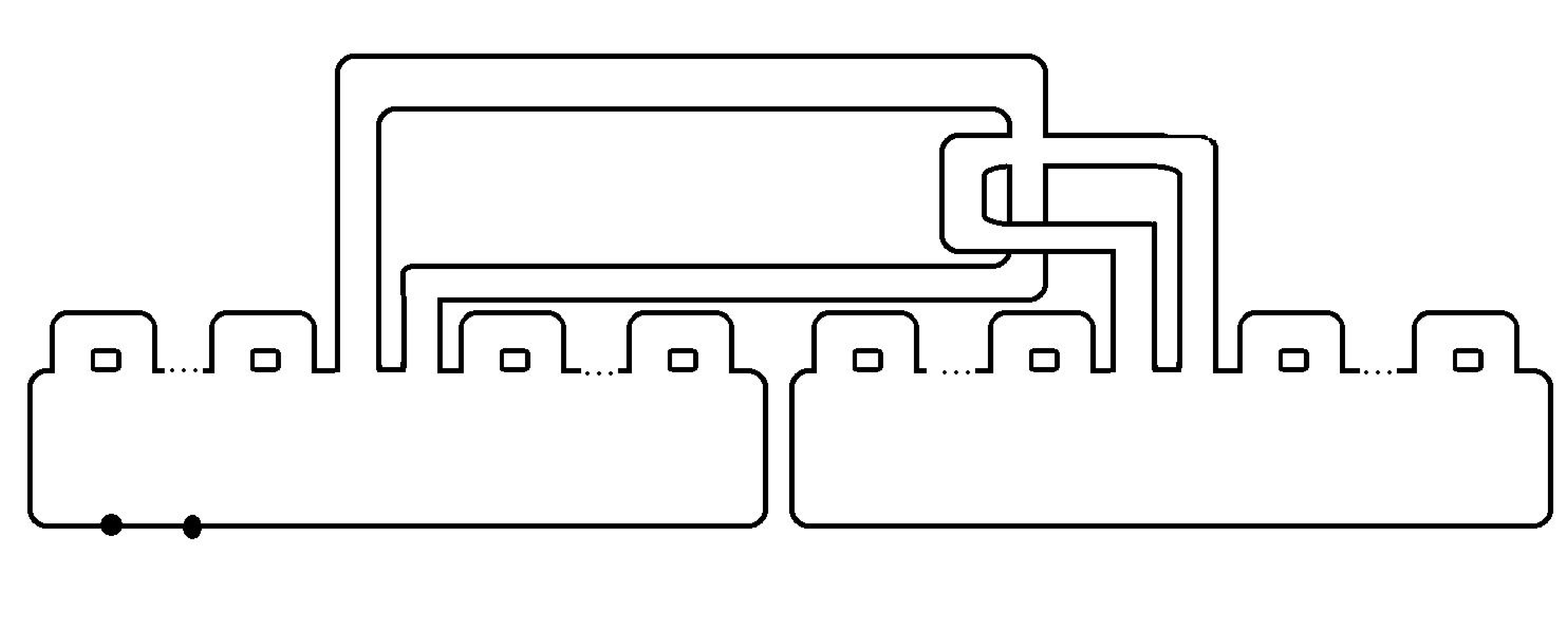}}
\put(-98,-2){$*'_1$}
\put(-70,-2){$*_1$}
\put(-98,36){$1$}
\put(-63,36){$i-1$}
\put(-22,36){$i$}
\put(0,36){$i+1$}
\put(52,36){$b_1$}
\put(98,36){$1$}
\put(133,36){$j-1$}
\put(174,36){$j$}
\put(196,36){$j+1$}
\put(248,36){$b_2$}
\end{picture}
\caption{$\Sigma_{\star \star,i,j}$}
\label{fig_sigma_star_star_i_j}
\end{figure}

\begin{figure}
\begin{picture}(160,150)
\put(-130,-10){\includegraphics[width=420pt]{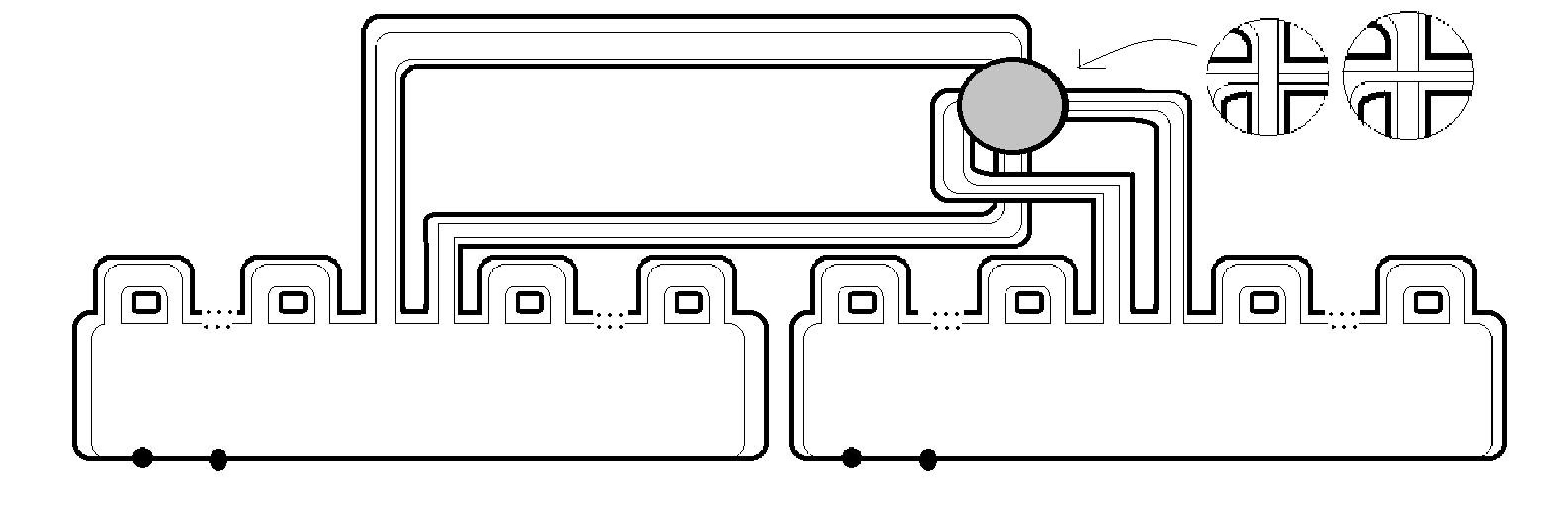}}
\put(-98,-2){$*'_1$}
\put(-70,-2){$*_1$}
\put(100,-2){$*'_1$}
\put(120,-2){$*_1$}
\put(-98,36){$1$}
\put(-63,36){$i-1$}
\put(-22,36){$i$}
\put(0,36){$i+1$}
\put(52,36){$b_1$}
\put(98,36){$1$}
\put(133,36){$j-1$}
\put(174,36){$j$}
\put(196,36){$j+1$}
\put(248,36){$b_2$}

\put(208,78){$\iota_1$}
\put(238,78){$\iota_2$}
\end{picture}
\caption{$\iota_1$ and $\iota_2$}
\label{fig_sigma_star_star_i_j_emb_1}
\end{figure}

\begin{figure}
\begin{picture}(160,150)
\put(-120,-20){\includegraphics[width=400pt]{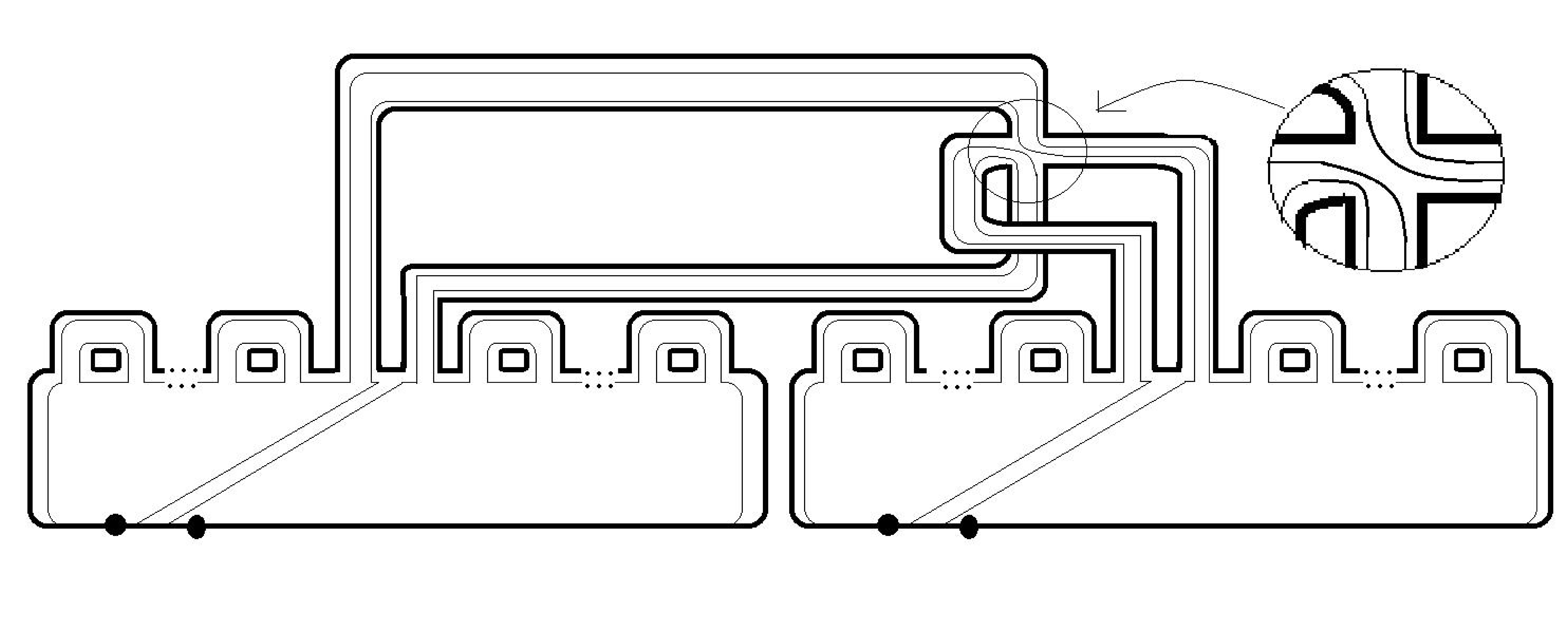}}
\put(-98,-2){$*'_1$}
\put(-70,-2){$*_1$}
\put(100,-2){$*'_1$}
\put(120,-2){$*_1$}
\put(-98,36){$1$}
\put(-63,36){$i-1$}
\put(-22,36){$$}
\put(10,36){$j$}
\put(28,36){{\tiny $b_1-i+j-1$}}
\put(98,36){$1$}
\put(133,36){$j-1$}
\put(174,36){$$}
\put(206,36){$i$}
\put(234,36){{\tiny $b_2-j+i-1$}}
\end{picture}
\caption{$\iota_3$}
\label{fig_sigma_star_star_i_j_emb_3}
\end{figure}

Using the above two lemmas, we have the following.

\begin{lemm}
\label{lemm_filt_douti_move}
  Let $\iota$ and $\iota'$ be two embeddings
\begin{align*}
&(\Sigma_{b_0+1,0}  \times I \sqcup
\coprod_{j \in \shuugou{1, \cdots, N}} \Sigma_{b_j+1,1} \times I  
\sqcup \coprod_{i} \Sigma_{0,2} \times I, \shuugou{*_1 \in \Sigma_{0,b_0+1}} 
\times I, \shuugou{*'_1 \in \Sigma_{0,b_0+1}} \times I) \\
&\to 
(\Sigma \times I, \shuugou{*_\alpha} \times I,
\shuugou{*_\beta}\times I)
\end{align*}
 which are only
differ in an closed ball in $\Sigma \times I$
as shown in Figure \ref{fig_sigma_move},
where $\sum_{i=0}^j b_i \geq N$.
Then we have
\begin{align*}
&\kukakko{\iota}-\kukakko{\iota'} \in h F^{\star (n-2)}
\tskein{\Sigma,*_\alpha, *_\beta}.
\end{align*}
\end{lemm}

\begin{figure}
\begin{picture}(100,50)
\put(0,0){\includegraphics[width=100pt]{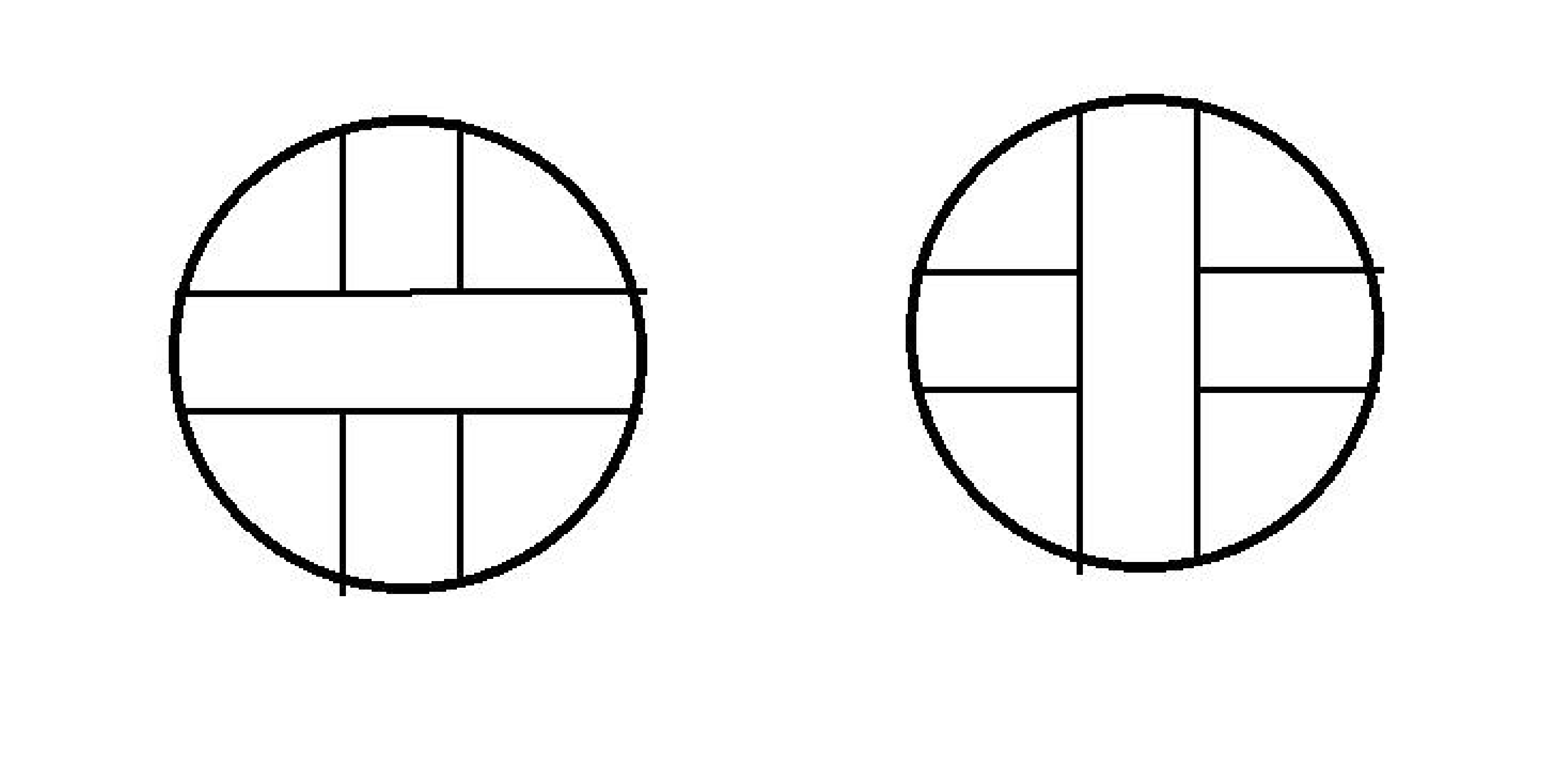}}
\put(0,40){$\iota$}
\put(45,40){$\iota'$}
\end{picture}
\caption{}
\label{fig_sigma_move}
\end{figure}

\begin{proof}[Proof of Lemma \ref{lemm_filt_douti}]
To prove it, we use the induction on $n$.
If $n=0,-1$, the claim follows from definition, 
where we denote $F^{-1} \tskein{\Sigma, J^-,J^+} \defeq 
\tskein{\Sigma, J^-,J^+}$.
We assume $F^k \tskein{\Sigma,J^-,J^+} 
=F^{\star k} \tskein{\Sigma,J^-,J^+}$
for any $k \in \shuugou{1,2, \cdots, n-1}$  and
any finite disjoint subset $J^-, J^+ \subset \partial \Sigma$
with $\sharp J^-
=\sharp J^+$.
Let $\iota$ be an embedding
\begin{align*}
&(\coprod_{j \in \shuugou{1, \cdots, N}}
\Sigma_{b_j+1,0}  \times I \sqcup
\coprod_{j \in \shuugou{1, \cdots, N'}} \Sigma_{b'_j+1,1} \times I  
\sqcup \coprod_{i} \Sigma_{0,2} \times I, \shuugou{*_1|*_1 \in \Sigma_{0,b_j+1}} 
\times I, \\
&\shuugou{*'_1 |*'_1 \in \Sigma_{0,b_j+1}} \times I) \to 
(\Sigma \times I, J^-\times I,
J^+\times I).
\end{align*}
with $\sum b_j +\sum b'_j \geq k+1$ and $(b_1, b'_1) \neq (0,0)$.

If $b_1 \neq 0$, We denote
$\iota_1 \defeq \iota_{|\Sigma_{b_1+1,0} \times I}$
and
\begin{equation*}
\iota' \defeq \iota_{|
(\coprod_{j \in \shuugou{2, \cdots, N}}
\Sigma_{b_j+1,0}  \times I \sqcup
\coprod_{j \in \shuugou{1, \cdots, N'}} \Sigma_{b'_j+1,1} \times I  
\sqcup \coprod_{i} \Sigma_{0,2} \times I)}.
\end{equation*}
Let $\iota'_1$ be
an embedding $\Sigma_{b_1+1,0} \times I \to \Sigma \times I$
such that
$\iota'_{1*} (\psi (R_1^{\epsilon_1}R_2^{\epsilon_2} 
\cdots R_{b_1-1}^{\epsilon_{b_1-1}}R_{b_1}^{\epsilon_{b_1}}  ))
=\psi( X_1^{\epsilon_1} X_2^{\epsilon_2} \cdots X_{b_1}^{\epsilon_{b_1}})$
for $\epsilon_1, \epsilon_2, \cdots, \epsilon_{b_1} \in \shuugou{0,1}$
where $X_i =\iota_1 (R_i)$.
For example, see Figure \ref{fig_embedding_construct}.
We remark $\kukakko{\iota'_1} \in  F^{b_1} \tskein{\Sigma,J^-,J^+}$. 
Using Lemma \ref{lemm_filt_douti_move} repeatedly,
we obtain
\begin{equation*}
\kukakko{\iota} -\kukakko{\iota_1}\kukakko{\iota'}  \in h F^{\star k-1} \tskein{\Sigma,J^-,
J^+}.
\end{equation*}
Using Lemma \ref{lemm_filt_douti_move} repeatedly,
we obtain
\begin{equation*}
\kukakko{\iota_1}\kukakko{\iota'} -
\kukakko{\iota'_1}\kukakko{\iota'}  \in h F^{\star k-1} \tskein{\Sigma,J^-,
J^+}.
\end{equation*}
We remark 
\begin{equation*}
 \kukakko{\iota'_1}\kukakko{\iota'}
\in 
F^{b_1} \tskein{\Sigma,\iota_1 (*_1),
\iota_1 (*'_1)}
\boxtimes
F^{\star k+1-b_1} \tskein{\Sigma, J^-
\backslash \iota_1 (*_1), J^+
\backslash \iota_1 (*'_1)}.
\end{equation*}

If $b'_1 \neq 0$, We denote
$\iota_1 \defeq \iota_{|\Sigma_{b'_1+1,0} \times I}$
and
\begin{equation*}
\iota' \defeq \iota_{|
(\coprod_{j \in \shuugou{1, \cdots, N}}
\Sigma_{b_j+1,0}  \times I \sqcup
\coprod_{j \in \shuugou{2, \cdots, N'}} \Sigma_{b'_j+1,1} \times I  
\sqcup \coprod_{i} \Sigma_{0,2} \times I)}.
\end{equation*}
Let $\iota'_1$ be
an embedding $\Sigma_{b'_1+1,0} \times I \to \Sigma \times I$
such that
$\iota'_{1*} (\psi (R_1^{\epsilon_1}R_2^{\epsilon_2} 
\cdots R_{b'_1}^{\epsilon_j}  ))
=\psi( X_1^{\epsilon_1} X_2^{\epsilon_2} \cdots X_{b'_1}^{\epsilon_j})$
for $\epsilon_1, \epsilon_2, \cdots, \epsilon_{b'_1} \in \shuugou{0,1}$
where $\zettaiti{X_i} =\zettaiti{\iota_1 (R_i)}$.
We remark $\kukakko{\iota'_1} \in  F^{b_1} \tskein{\Sigma,J^-,J^+}$. 
For example, see Figure \ref{fig_embedding_construct}.
Using Lemma \ref{lemm_filt_douti_move}  repeatedly,
we obtain
\begin{equation*}
\kukakko{\iota} -\kukakko{\iota_1}\kukakko{\iota'}  \in h F^{\star k-1} \tskein{\Sigma,J^-,
J^+}.
\end{equation*}
Using Lemma \ref{lemm_filt_douti_move}  repeatedly,
we obtain
\begin{equation*}
\kukakko{\iota_1}\kukakko{\iota'} -
\kukakko{\iota'_1}\kukakko{\iota'}  \in h F^{\star k-1} \tskein{\Sigma,J^-,
J^+}.
\end{equation*}
We remark 
\begin{equation*}
 \kukakko{\iota'_1}\kukakko{\iota'}
\in 
F^{b_1} \tskein{\Sigma}
\boxtimes
F^{\star k+1-b_1} \tskein{\Sigma, J^-,J^+}
\end{equation*}

Hence we obtain
\begin{align*}
&F^{\star k+1} \tskein{\Sigma, J^-, J^+} \\
&\subset \sum_{* \in J^+, *' \in J^-}
\sum_{i+j \geq k+1, i\geq 1}F^i \tskein{\Sigma,*,*'} 
\boxtimes F^{\star j} 
\tskein{\Sigma,J^+ \backslash *, J^- \backslash *'} \\
&+
\sum_{i+j \geq k+1, j \geq 1} F^i \tskein{\Sigma} F^{\star j} 
\tskein{\Sigma,J^+, J^- }
+ h F^{\star (k-1)} \tskein{\Sigma,*_\alpha, *_\beta}\\
&=\sum_{* \in J^+, *' \in J^-}
\sum_{i+j \geq k+1, i\geq 1}F^i \tskein{\Sigma,*,*'} 
\boxtimes F^{j} 
\tskein{\Sigma,J^+ \backslash *, J^- \backslash *'} \\
&+
\sum_{i+j \geq k+1, j \geq 1} F^i \tskein{\Sigma} F^{j} 
\tskein{\Sigma,J^+, J^- }
+ h F^{(k-1)} \tskein{\Sigma,*_\alpha, *_\beta}\\\\
&\subset F^{k+1} \tskein{\Sigma}.
\end{align*}
This proves the lemma.

\end{proof}

By Lemma \ref{lemm_filt_douti}, we have the following.

\begin{thm}
\label{thm_filt_underlying3mfd}
Let $\chi$ be an diffeomorphism
$(\Sigma \times I, J^- \times I , J^+ \times I) \to (\Sigma' \times I, J^- \times I, J^+
\times I)$.
Then we have
\begin{equation*}
\chi_* (F^n \tskein{\Sigma, J^-,J^+})
= F^n \tskein{\Sigma,J^-,J^+}.
\end{equation*}
\end{thm}

\begin{proof}
By definition and Lemma \ref{lemm_filt_douti}, we have
\begin{equation*}
\chi_* (F^n \tskein{\Sigma, J^-,J^+})
=\chi_* (F^{\star n} \tskein{\Sigma, J^-,J^+})
= F^{\star n} \tskein{\Sigma,J^-,J^+}
= F^n \tskein{\Sigma,J^-,J^+}.
\end{equation*}
This proves the lemma.
\end{proof}

The filtration $\filtn{F^{\star n} \tskein{\Sigma}}=\filtn{F^n \tskein{\Sigma}}$
is independent of the choice of $*_\alpha$.

\begin{prop}
The filtration $\filtn{F^n \tskein{\Sigma}}$ is independent of the choice of $*_\alpha$.
\end{prop}

\subsection{Corollaries}
Lemma \ref{lemm_filt_douti} gives some 
useful propositions.

\begin{prop}
\label{prop_filt_bracket}
Let $J^-$ and $J^+$  be disjoint finite subsets of $\partial \Sigma$ 
with  $\sharp J^- = \sharp J^+$.
We have 
\begin{align*}
&[F^n \tskein{\Sigma}, F^m \tskein{\Sigma}] \subset F^{n+m-2} \tskein{\Sigma}, \\
&\sigma (F^n \tskein{\Sigma})(F^m \tskein{\Sigma,J^-,J^+}) 
\subset F^{n+m-2} \tskein{\Sigma,J^-,J^+}
\end{align*}
for any $n, m \geq 0$.
\end{prop}

\begin{proof}
By the Leibniz rule, it suffices to show
\begin{align*}
\sigma (\zettaiti{\psi ((X_j-1) \cdots(X_2-1)(X_1-1))})(\psi (Y))
\in F^{j-1} \tskein{\Sigma}
\end{align*}
for $X_1, X_2 , \cdots, X_j \in \pi_1^+ (\Sigma,*_\alpha)$ and
$Y \in \pi_1^+ (\Sigma, *_\alpha, *_\beta)$.
By definition, we have
\begin{align*}
\sigma (\zettaiti{\psi ((X_j-1) \cdots(X_2-1)(X_1-1))})(\psi (Y))
\in F^{\star(j-1)} \tskein{\Sigma}=F^{j-1} \tskein{\Sigma}.
\end{align*}
This proves the proposition.
\end{proof}

\begin{prop}
\label{prop_filt_product}
Let $J^+,J^-,J'^+$ and $J'^-$ be mutually disjoint finite subsets of $\partial \Sigma$
with $\sharp (J^+) =\sharp (J^-)$ and 
$\sharp(J'^+) =\sharp(J'^-)$.
We have 
\begin{align*}
F^n \tskein{\Sigma,J^-, J^+} \boxtimes 
F^m \tskein{\Sigma, J'^-,J'^+}
\subset F^{n+m} \tskein{\Sigma,J^- \cup 
J'^-,J^+ \cup J'+}
\end{align*}
for any $n,m$. 
In particular, we have
\begin{align*}
&F^n \tskein{\Sigma} F^m \tskein{\Sigma} 
\subset F^{n+m} \tskein{\Sigma}, \\
&F^n \tskein{\Sigma} F^m \tskein{\Sigma, *_\alpha, *_\beta}
\subset F^{n+m} \tskein{\Sigma, *_\alpha, *_\beta}, \\
&F^m \tskein{\Sigma, *_\alpha, *_\beta}
F^n \tskein{\Sigma} 
\subset F^{n+m} \tskein{\Sigma, *_\alpha, *_\beta}.
\end{align*}
for any $n,m \geq 0$.
\end{prop}

\begin{proof}
We have
\begin{align*}
F^n \tskein{\Sigma,J^-, J^+} \boxtimes 
F^m \tskein{\Sigma, J'^-,J'^+}
&\subset 
F^{\star (n+m)} \tskein{\Sigma,J^- \cup 
J'^-,J^+ \cup J'+} \\
&=F^{n+m} \tskein{\Sigma,J^- \cup 
J'^-,J^+ \cup J'+}.
\end{align*}
This proves the proposition.
\end{proof}

\begin{prop}
\label{prop_filt_disk}
We have $F^{\star n} \tskein{D} =F^n \tskein{D} =h^{\gauss{(n+1)/2}} \Q [\rho][[h]]
[\emptyset]$,
where $D$ is a closed disk and $\gauss{x}$ is the greatest integer 
not greater than $x$ for $x \in \Q$.
\end{prop}
\begin{proof}
Since $(\ker \epsilon)^n =(\exp (\rho h)1_{\pi^+(D,*_1)} -1_{\pi^+(D,*_1)})^n \mathcal{P} (D,*_1)$,
we have 
$(\ker \epsilon)^n  \subset h^n \Q [\rho][[h]] 1_{\pi^+(D,*_1)}$.
This proves the proposition.
\end{proof}

\subsection{Completion}
We consider the topology on $\tskein{\Sigma}$ induced by
the 
 $\filtn{F^n \tskein{\Sigma}}$, and denote its completion
by $\widehat{\tskein{\Sigma}} \defeq
\underleftarrow{\lim}_{i \rightarrow \infty}\tskein{\Sigma}/
F^i \tskein{\Sigma}$.
We call $\widehat{\tskein{\Sigma}}$ the completed skein algebra.
We also consider the topology on $\tskein{\Sigma,J^-,J^+}$ induced by
the filtration $\filtn{F^n \tskein{\Sigma,J^-,J^+}}$, and denote its completion
by $\widehat{\tskein{\Sigma,J^-,J^+}} \defeq
\underleftarrow{\lim}_{i \rightarrow \infty}\tskein{\Sigma,J^-,J^+} /
F^i \tskein{\Sigma,J^-,J^+}$.
We call $\widehat{\tskein{\Sigma,J^-,J^+}}$ the completed skein module.
The completed skein algebra $\widehat{\tskein{\Sigma}}$ 
has a filtration $\widehat{\tskein{\Sigma}} = F^0 \widehat{\tskein{\Sigma}} \supset 
F^1 \widehat{\tskein{\Sigma}} \supset F^2 \widehat{\tskein{\Sigma}} \supset \cdots$
such that $\widehat{\tskein{\Sigma}}/F^n \widehat{\tskein{\Sigma}} \simeq
\tskein{\Sigma}/F^n \tskein{\Sigma}$ for $n \in \Z_{\geq 0}$.
The completed skein module $\widehat{\tskein{\Sigma,J^-,J^+}}$ 
also has a filtration $\widehat{\tskein{\Sigma,J^-,J^+}} = 
F^0 \widehat{\tskein{\Sigma,J^-,J^+}} \supset 
F^1\widehat{\tskein{\Sigma,J^-,J^+}} \supset F^2 \widehat{\tskein{\Sigma,J^-,J^+}} \supset \cdots$
such that $\widehat{\tskein{\Sigma,J^-,J^+}}/F^n \widehat{\tskein{\Sigma,J^-,J^+}} 
\simeq
\tskein{\Sigma,J^-,J^+}/F^n \tskein{\Sigma,J^-,J^+}$ for $n \in \Z_{\geq 0}$.

Let $\pi_1 (\Sigma, *)$ be the fundamental group of $\Sigma$, $\Q [\rho] \pi_1 (\Sigma,*)$
the group ring of $\pi_1 (\Sigma, *)$ over $\Q [\rho]$ and $\epsilon_{\pi} :
\Q[\rho] \pi_1 (\Sigma, *) \to \Q[\rho]$ the augmentation map
defined by $x \in \pi_1 (\Sigma, *) \mapsto 1$.
We consider the natural surjective map $\epsilon_{\mathcal{A}} :
\tskein{\Sigma} \to \Q [\rho]$ defined by 
$\epsilon_{\mathcal{A}} (h) =0$ and $\epsilon_{\mathcal{A}}
([L]) =(2\rho)^{\zettaiti{L}}$ for $L \in \mathcal{T}(\Sigma)$,
 where we denote 
by $\zettaiti{L}$ the number of components of $L$.
We remark that $\ker \epsilon_{\mathcal{A}} =F^1 \tskein{\Sigma}$.
The map $\varpi_{\mathcal{P}} :\mathcal{P}(\Sigma , *_i, *_j)\to
\Q [\rho] \pi_1 (\Sigma, *_i, *_j)$ is defined by
$h \mapsto 0$ and the natural surjection
$\pi^+_1 (\Sigma, *_i,*,j) \to \pi_1 (\Sigma, *_i, *_j)$, 
where $\Q [\rho] \pi_1 (\Sigma, *_i, *_j)$
is the free $\Q [\rho]$-module with basis
$\pi_1 (\Sigma, *_i, *_j)$.
We denote $\varpi_{\mathcal{A}} \defeq
(\epsilon_{\mathcal{A}} \otimes \varpi_{\mathcal{P}}) \circ 
\psi_{*_i} : \tskein{\Sigma, *_i , *_j} \to \Q [\rho] \pi_1 (\Sigma, *_i, *_j)$.
Since $\varpi_{\mathcal{A}} (F^n \tskein{\Sigma, *_i, *_j})=
(\ker \epsilon_\pi)^n \Q [\rho] \pi_1 (\Sigma, *_i, *_j)$
for any $n$, the map $\varpi_{\mathcal{A}}$
induces the surjection
\begin{equation*}
\varpi_{\mathcal{A}} :
\widehat{\tskein{\Sigma, *_i, *_j}} \to 
\widehat{\Q [\rho ] \pi_1 (\Sigma, *_i, *_j)}\defeq
\comp{i} \Q [\rho] \pi_1 (\Sigma, *_i, *_j)/
(\ker \epsilon_\pi)^i \Q [\rho] \pi_1 (\Sigma, *_i, *_j) .
\end{equation*}
It is well-known that $\cap^{\infty}_{j=0}
(\ker \epsilon_\pi)^j \Q [\rho] \pi_1 (\Sigma, *_i, *_j)= \shuugou{0}
$. See, for example, Bourbaki \cite{Bourbaki} Exercise 4.6. 
Since $\varpi_{\mathcal{A}} (\xi (x)) =\xi (\varpi_{\mathcal{A}}(x))$
for any $\xi \in \mathcal{M} (\Sigma)$ and any $x \in
\widehat{\tskein{\Sigma, *_i, *_j}}$, 
we have the following.

\begin{prop}
\label{prop_MCG_tskein_faithful}
For any $\xi \in \mathcal{M} ( \Sigma)$, 
$\xi = \id \in \mathcal{M} (\Sigma)$ if and only if
$\xi (x) =x $ for any $x \in \widehat{\tskein{\Sigma, *_i, *_j}}$ 
and any $*_i, *_j \in \shuugou{*_1, *_2, \cdots, *_b}$.
\end{prop}

We denote by $\check{\mathcal{M}} (\Sigma)
\subset \mathcal{M} (\Sigma)$  the set of consisting of elements $\xi$
satisfying $(1-\xi)^2 (H_1 (\Sigma)) = \shuugou{0}$, where 
$H_1 (\Sigma) $ is the first homology group of $\Sigma$.
For example, the Dehn twist along a simple closed curve is
an element of $\check{\mathcal{M}}(\Sigma)$.

There exists a natural $\Q$-module surjective homomorphism
\begin{equation*}
\varphi : \ker \epsilon/ (\ker \epsilon)^2 = \Q \oplus \Q [\rho] \otimes H^1 (\Sigma)
\end{equation*}
defined by $h \mapsto 1 \in \Q$ and $R \in \pi^+_1 (\Sigma, *)
\to [R] \in H_1 (\Sigma)$.

\begin{prop}
Let $\xi$ be an element of $\check{\mathcal{M}}(\Sigma)$.
\begin{enumerate}
\item Fix the base point $* \in \shuugou{*_1, \cdots ,*_b}$.
We have $(1-\xi)^2 (\ker \epsilon) \subset (\ker \epsilon)^2
\subset \mathcal{P} (\Sigma, *)$.
\item Fix the base point $* \in \shuugou{*_1, \cdots ,*_b}$. 
We have $(1-\xi)^{N+1} ((\ker \epsilon)^N)
\subset (\ker \epsilon)^{N+1}\subset (\ker \epsilon)^2$
for and any $N$.
\item Let $*_\alpha$ and $*_\beta$ be two elements of $\shuugou{*_1, \cdots ,*_b}$.
We have $(1-\xi)^{N+1}((\ker \epsilon)^N \mathcal{P}(\Sigma, *_\alpha, *_\beta)
) \subset (\ker \epsilon)^{N+1} \mathcal{P}(\Sigma, *_\alpha, *_\beta)$.
\end{enumerate}
\end{prop}

\begin{proof}
(1)Since $\varphi (\xi (x)) = \xi (\varphi(x))$ for any  $x \in
\ker \epsilon$, we have $\varphi ((1-\xi)^2(x)) =(1-\xi)^2
\varphi(x) =0$. So we have $(1-\xi)^2 (x) \in (\ker \epsilon)^2$.
This proves (1).

(2)The proof goes by induction in $N$.
If $N=0,1$, the claim is obvious.
We assume $N \in \Z_{\geq 1}$, $x \in (\ker \epsilon)^{N}$
and $y \in (\ker \epsilon)^1$.
By induction assumption, we have
\begin{align*}
&(1-\xi)^{N+1}(xy)=\sum_{i+j=N+1} \frac{(N+1)!}{i! j!} (1-\xi)^i \xi^{j} (x)
(1-\xi)^j(y) \in (\ker \epsilon)^{N+1} .
\end{align*}
This proves (2).

(3)The proof goes by induction in $N$. 
If $N=0$, the claim is obvious.
We assume $N \in \Z_{\geq 1}$, $x \in (\ker \epsilon)^{N}$ and  $v \in
\mathcal{P} (\Sigma, *_\alpha, *_\beta)$.
By induction assumption and (1), we have
\begin{align*}
&(1-\xi)^{N+1}(xv)=\sum_{i+j =N+1} \frac{(N+1)!}{i! j!} (1-\xi)^i \xi^{j} (v')
(1-\xi)^j(v) \in (\ker \epsilon)^{N+1} \mathcal{P}(\Sigma, *_\alpha, *_\beta).
\end{align*}
This proves (3).
\end{proof}

By the above proposition, we have the following.

\begin{prop}
Let $J^-$ and $J^+$  be disjoint finite subsets of $\partial \Sigma$ 
with $\sharp J^- = \sharp J^+$.
For any $\xi \in \check{\mathcal{M}}(\Sigma)$ and ant $N \in \Z_{\geq 0}$,
we have $(1-\xi)^{N+1}(F^N \tskein{\Sigma,J^-,J^+}) \subset F^{N+1}
\tskein{\Sigma, J^-,J^+}$.
\end{prop}

For $\xi \in \check{\mathcal{M}}(\Sigma)$ and
disjoint finite subsets $J^-$ and $J^+$ of $\partial \Sigma$ 
with  $\sharp J^- = \sharp J^+$,
we denote
\begin{equation*}
\log (\xi) :\widehat{\tskein{\Sigma,J^-,J^+}} \to 
\widehat{\tskein{\Sigma,J^-,J^+}},
x \mapsto \sum_{i=1}^\infty \frac{-1}{i} (1-\xi)^i (x).
\end{equation*}

\section{A formula for Dehn twists}
For a simple closed curve $c$ in $\Sigma$,
we choose a  simple path $\gamma \in \pi^+_1 (\Sigma, \shuugou{*_\alpha}, 
\shuugou{*'_\alpha})$ 
satisfying $\zettaiti{\psi (\gamma)}  \in 
\tskein{\Sigma}$ is presented by the diagram
$c$. We consider $\mathcal{P} (\Sigma, *_\alpha)^{\otimes n}$
as an associative algebra $\Q [\rho] [[h]]$-algebra
with unit $1 \defeq {1_{\pi^+(\Sigma,*_\alpha)}}^{\otimes n}$
by the product $(x_1 \otimes x_2 \otimes \cdots \otimes x_n)
\cdot (y_1 \otimes y_2 \otimes \cdots \otimes y_n)
=x_1 y_1 \otimes x_2 y_2 \otimes \cdots \otimes x_n y_n$.
We define  $\epsilon_{\mathcal{P}^{\otimes n}}: 
\mathcal{P} (\Sigma, *_\alpha)^{\otimes n} \to \Q [\rho]$ by
$h \mapsto 0$ and $x_1 \otimes \cdots \otimes x_n \mapsto 1$
for $x_1, \cdots , x_n \in \pi_1^+ (\Sigma , *_\alpha)$,
and  
\begin{equation*}
\widehat{\mathcal{P} (\Sigma, *_\alpha)^{\otimes n}}
\defeq \comp{i} \mathcal{P} (\Sigma, *_\alpha)^{\otimes n}/
(\ker \epsilon_{\mathcal{P}^{\otimes n}})^i.
\end{equation*}
The $\Q [\rho][[h]]$-module homomorphism
$\psi_n \defeq \psi_{T|\mathcal{P} (\Sigma, *_\alpha)^{\otimes n}} 
 :\mathcal{P} (\Sigma, *_\alpha)^{\otimes n}
\to \tskein{\Sigma}$
induces
$\psi_n :\widehat{\mathcal{P} (\Sigma, *_\alpha)^{\otimes n}}
\to \widehat{\tskein{\Sigma}}$.
For $f(X) =f^{[1]}(X)\in \Q [[X-1]]$,
we denote 
\begin{align*}
&f^{[k]} (X_1, X_2, \cdots, X_k)
\defeq \frac{f^{[k-1]}(X_1, X_2, \cdots ,X_{k-1})-
f^{[k-1]}(X_2, X_3, \cdots, X_{k})}{X_1-X_k} \\
&\in \Q [[X_1-1, X_2-1, \cdots, X_k-1]].
\end{align*}
We define
\begin{equation*}
\Lambda (c)  \defeq
(\frac{h/2}{\arcsinh (h/2)})^2(
\sum_{n=1}^\infty \frac{(-h)^{n-1}\exp (-n \rho h)}{n} \psi_n
(\gamma_{1,n} \cdots \gamma_{n,n}
\lambda^{[n]} (\gamma_{1,n}, \cdots, \gamma_{n,n}))
-\frac{1}{3}\rho^3 h^2),
\end{equation*}
where $\lambda (X) \defeq \frac{1}{2X} (\log (X))^2
\in \Q [[X-1]]$,
$\gamma_{i,n} \defeq {1_{\pi^+(\Sigma,*_\alpha)}}^{\otimes (i-1)} \otimes \exp (\rho h) \gamma
\otimes  {1_{\pi^+(\Sigma,*_\alpha)}}^{\otimes n-i}$.
Here, we denote
\begin{equation*}
F(x_1,x_2, \cdots , x_n) =\sum a_{i_1, i_2, \cdots, i_n}
(x_1-1)^{i_1} (x_2-1)^{i_2} \cdots (x_n-1)^{i_n}
\end{equation*}
for $F(X_1, X_2, \cdots ,X_n) =\sum
a_{i_1, i_2, \cdots, i_n} (X_1-1)^{i_1} (X_2-1)^{i_2} \cdots (X_n-1)^{i_n}
\in \Q [[X_1-1, X_2-1, \cdots, X_n-1]]$
and $X_1, X_2, \cdots, X_n \in 1 +\ker \epsilon_{\mathcal{P}^{\otimes n}}$.
We remark $\Lambda (c) =\frac{1}{2} \zettaiti{\psi((\gamma-1)^2)}
\mod F^3 \widehat{\tskein{\Sigma}}$.

The aim of this section is to prove the first main theorem as follows.

\begin{thm}
\label{thm_main_Dehn}
Let $\Sigma$ be a compact connected surface, $c$ a simple
closed curve, $t_c$ the Dehn twist alog $c$ and
$J^-$ and $J^+$  disjoint finite subsets of $\partial \Sigma$ 
with $\sharp J^- = \sharp J^+$.
Then we have
\begin{equation*}
\log (t_c) = \sigma (\Lambda (c)):
\widehat{\tskein{\Sigma,J^-,J^+}}
\to \widehat{\tskein{\Sigma,J^-,J^+}}.
\end{equation*}
\end{thm}

\subsection{Well-definedness of $\Lambda (c)$}
The aim of this subsection is to explain that
$\Lambda (c)$ is independent of the choice of $\gamma$.

We denote by $\gamma \in \pi_1^+ (S^1 \times I, *_1)$
as Figure \ref{fig_tskein_gamma} and
by $\gamma_\star \in \pi_1^+ (S^1 \times I, *_2)$
as Figure \ref{fig_tskein_gamma2}.
We remark that 
$\gamma^{-1}$ and $\gamma_\star^{-1}$
are shown as Figure \ref{fig_tskein_gamma_inverse}
and Figure \ref{fig_tskein_gamma2_inverse}.

\begin{figure}[htbp]
	\begin{tabular}{rrr}
	\begin{minipage}{0.25\hsize}
		\centering
\begin{picture}(75,75)
		\put(0,0){\includegraphics[width=75 pt]{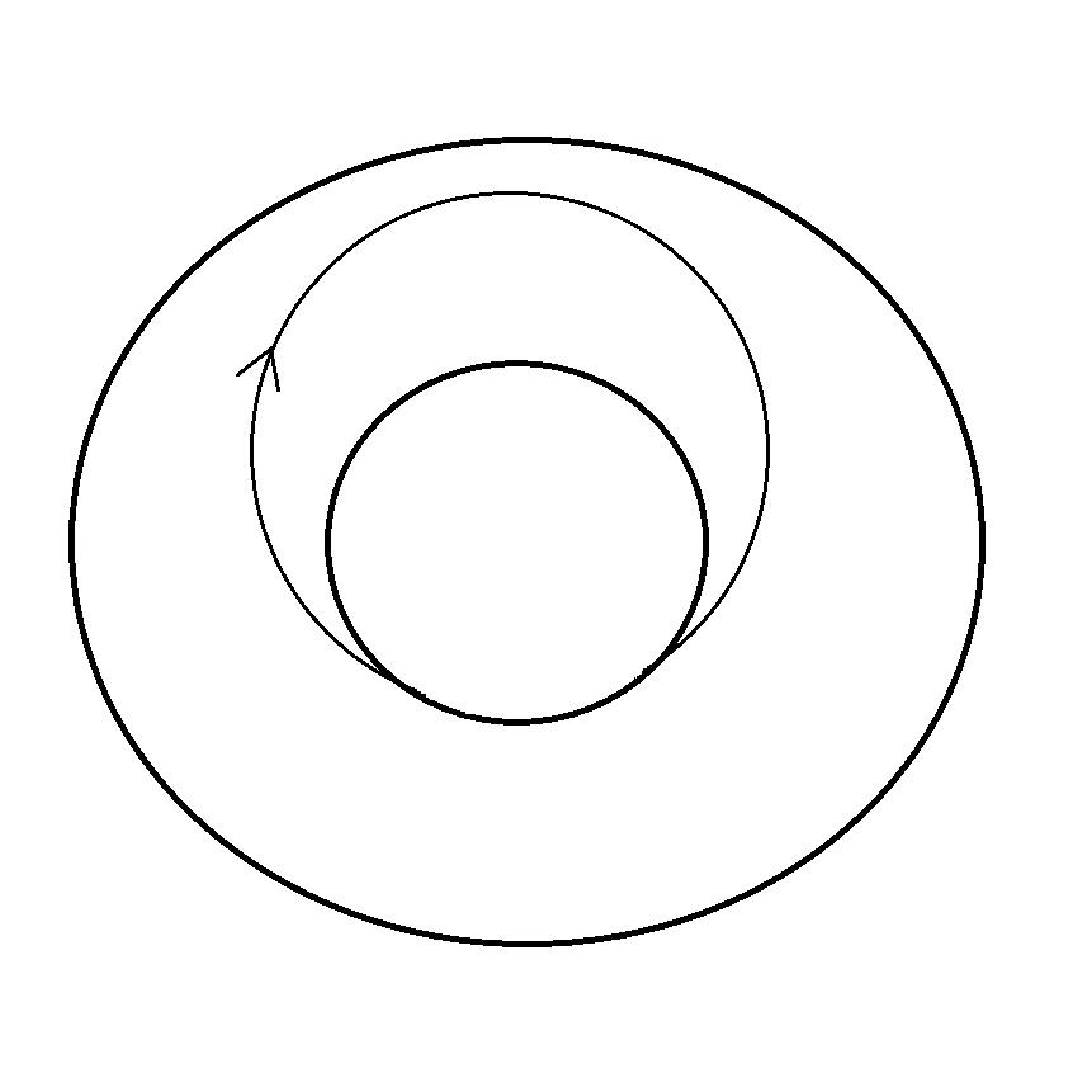}}
\end{picture}
		\caption{$\gamma$}
		\label{fig_tskein_gamma}
	\end{minipage}
	\begin{minipage}{0.25\hsize}
		\centering
\begin{picture}(75,75)
		\put(0,0){\includegraphics[width=75 pt]{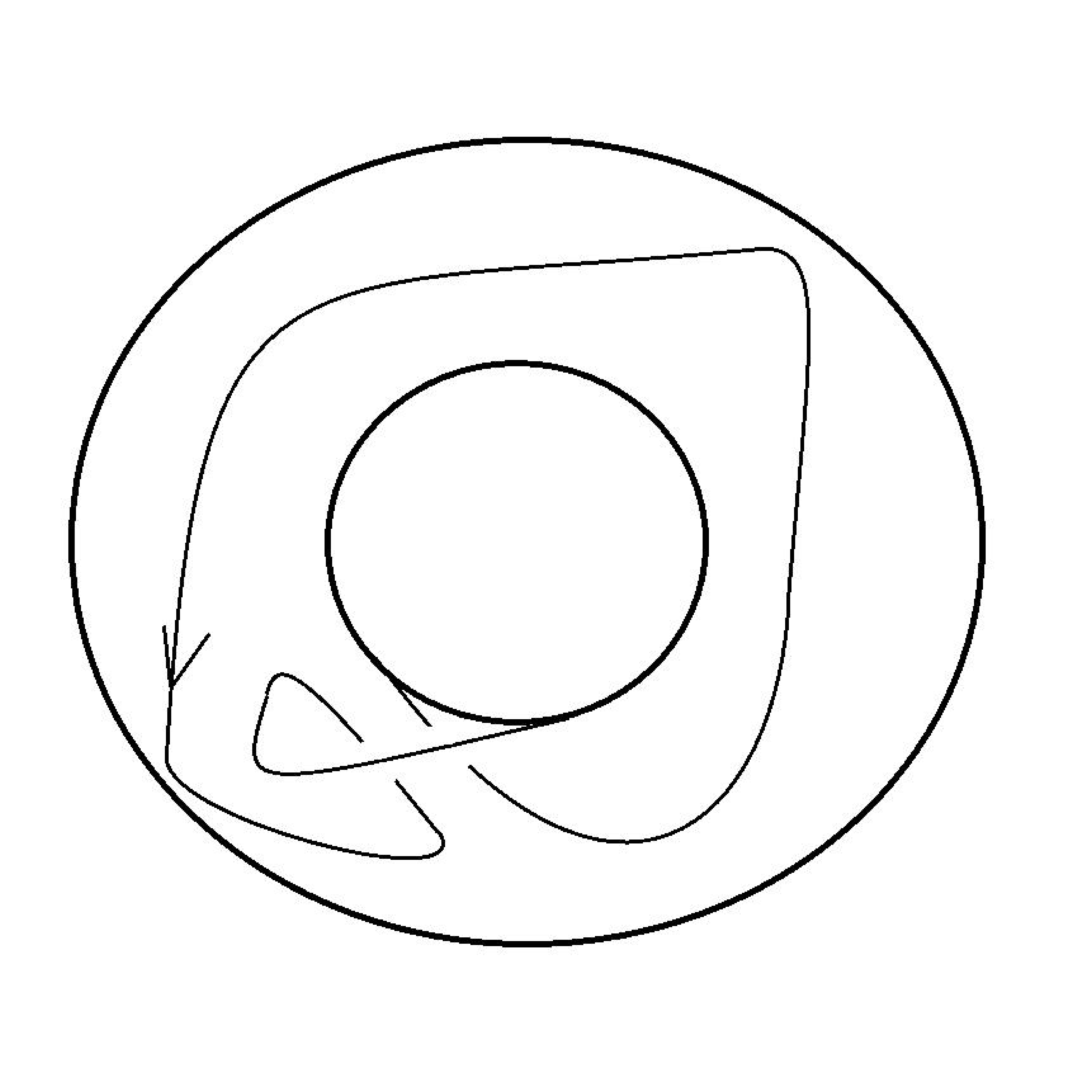}}
\end{picture}
		\caption{$\gamma^{-1}$}
		\label{fig_tskein_gamma_inverse}
	\end{minipage}
	\begin{minipage}{0.25\hsize}
		\centering
\begin{picture}(75,75)
		\put(0,0){\includegraphics[width=75 pt]{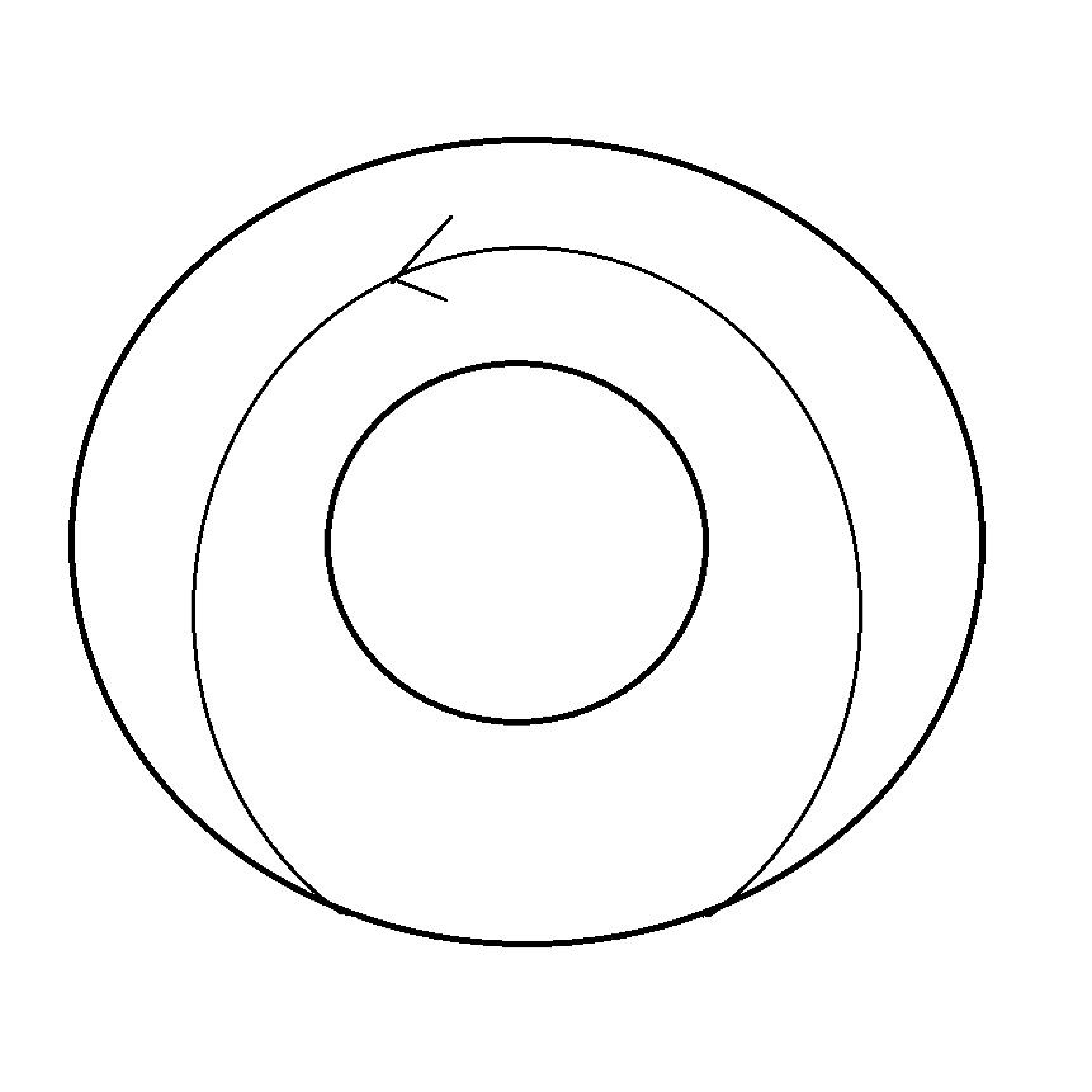}}
\end{picture}
		\caption{$\gamma_\star$}
		\label{fig_tskein_gamma2}
	\end{minipage}
	\begin{minipage}{0.25\hsize}
		\centering
\begin{picture}(75,75)
		\put(0,0){\includegraphics[width=75 pt]{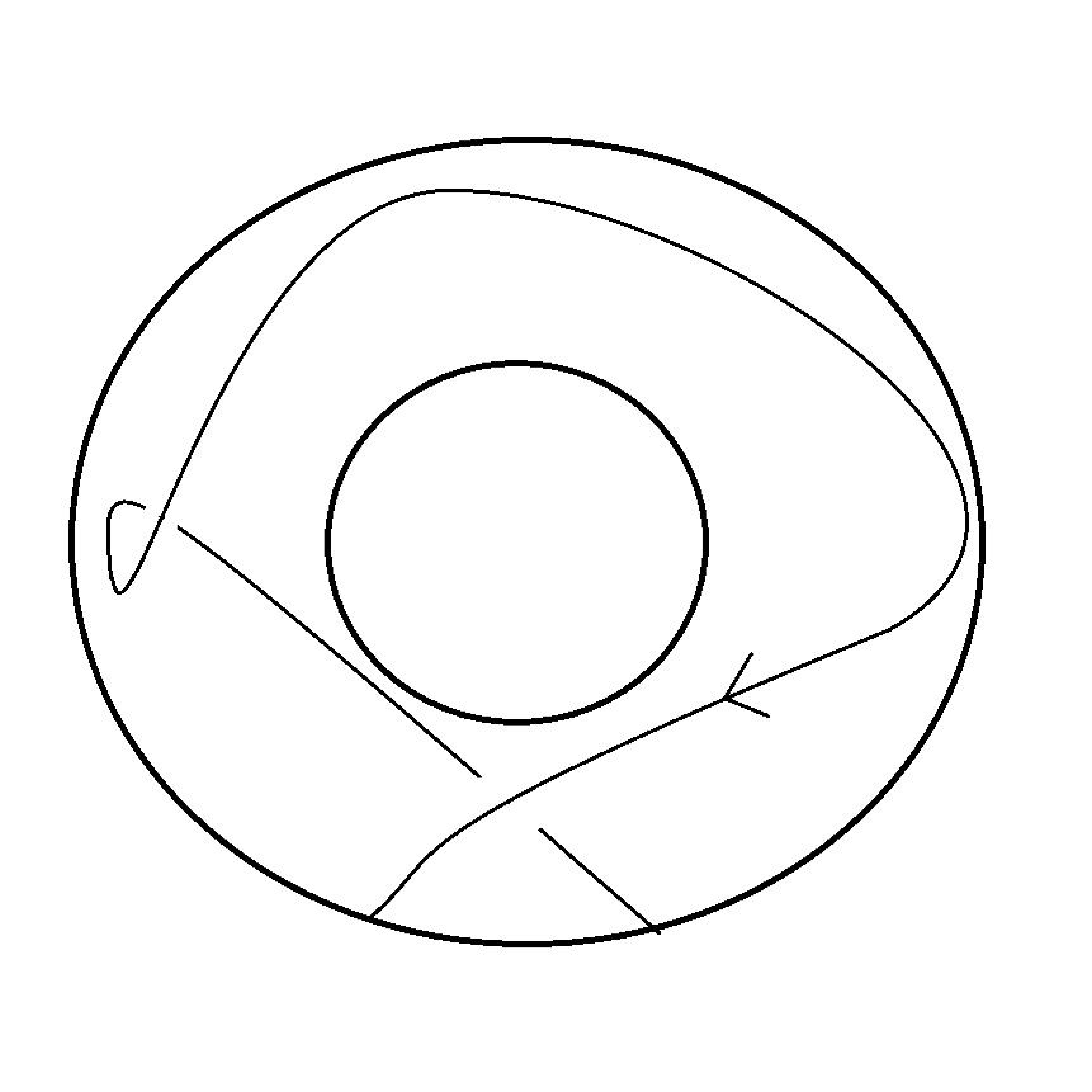}}
\end{picture}
		\caption{$\gamma_\star^{-1}$}
		\label{fig_tskein_gamma2_inverse}
	\end{minipage}
	\end{tabular}
\end{figure}

We denote $l_n \defeq \zettaiti{\psi (\exp((n-1)\rho h) \gamma^n)}$,
 $l_{\star n} \defeq \zettaiti{\psi (\exp(-(n-1) \rho h) \gamma_\star^{-n})}$,
$l_{-n} \defeq \zettaiti{\psi(\exp((n-1) \rho h) \gamma_\star^n)}$
and $l_{\star -n} \defeq \zettaiti{\psi(\exp(-(n-1) \rho h) \gamma^{-n})}$
for any $n \in \Z_{\geq 1}$.
We remark that $l_{\star n}$ are represented by  a knot
presented by the diagram as shown in Figure \ref{figure_l_ij}
when $i=n-1$ and $j=0$
and that $l_n$ are represented by a knot
presented by the diagram as shown in Figure \ref{figure_l_ij}
when $i=0$ and $j=n-1$.

\begin{figure}
\begin{picture}(240,200)
\put(0,0){\includegraphics[width=240pt]{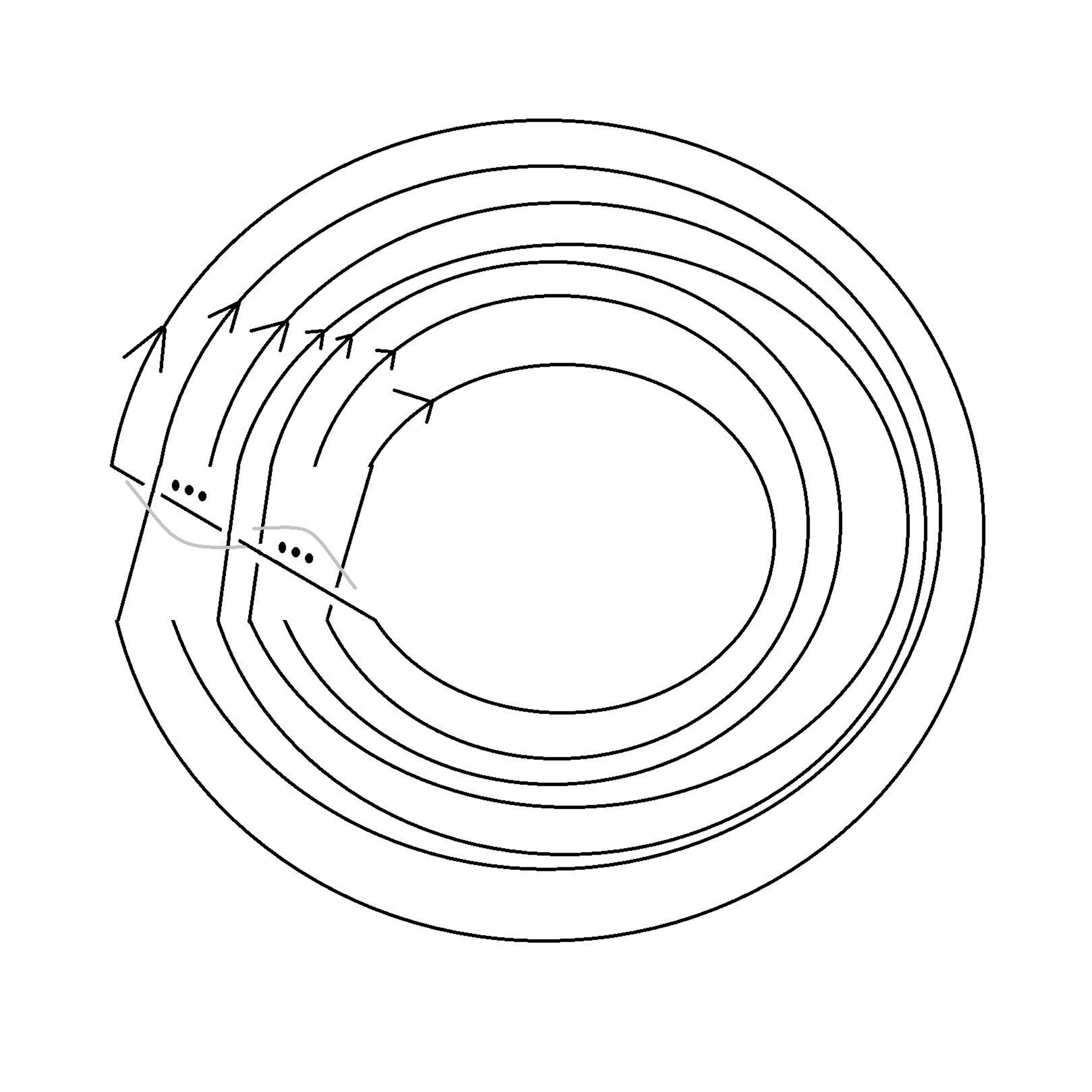}}
\put(40,110){$i$}
\put(70,130){$j$}
\end{picture}
\caption{$l_{i,j}$}
\label{figure_l_ij}
\end{figure}

The aim of this subsection is to prove the theorem.

\begin{thm}
The element
\begin{equation*}
\Lambda (c)(1) =\Lambda (c)  \defeq
(\frac{h/2}{\arcsinh (h/2)})^2(
\sum_{n=1}^\infty \frac{(-h)^{n-1}\exp (-n \rho h)}{n} \psi_n
(\gamma_{1,n} \cdots \gamma_{n,n}
\lambda^{[n]} (\gamma_{1,n}, \cdots, \gamma_{n,n}))
-\frac{1}{3}\rho^3 h^2)
\end{equation*}
equals the following elements
\begin{align*}
&\Lambda (c)(2)   \defeq
(\frac{h/2}{\arcsinh (h/2)})^2(
\sum_{n=1}^\infty \frac{(-h)^{n-1}\exp (-n \rho h)}{n} \psi_n
(\gamma_{1,n} \cdots \gamma_{\star n,n}
\lambda^{[n]} (\gamma_{\star 1,n}, \cdots, \gamma_{ \star n,n}))
-\frac{1}{3}\rho^3 h^2) \\
&\Lambda (c)(3)   \defeq
(\frac{h/2}{\arcsinh (h/2)})^2(
\sum_{n=1}^\infty \frac{h^{n-1}\exp (n \rho h)}{n} \psi_n
(\gamma^{-1}_{1,n} \cdots \gamma^{-1}_{ n,n}
\lambda^{[n]} (\gamma^{-1}_{1,n}, \cdots, \gamma^{-1}_{n,n}))
-\frac{1}{3}\rho^3 h^2) \\
&\Lambda (c)(4)   \defeq
(\frac{h/2}{\arcsinh (h/2)})^2(
\sum_{n=1}^\infty \frac{h^{n-1}\exp (n \rho h)}{n} \psi_n
(\gamma^{-1}_{\star 1,n} \cdots \gamma^{-1}_{\star n,n}
\lambda^{[n]} (\gamma^{-1}_{\star 1,n}, \cdots, \gamma^{-1}_{ \star n,n}))
-\frac{1}{3}\rho^3 h^2) \\
&\Lambda(c) (5) \defeq (\frac{h/2}{\arcsinh (h/2)})^2
(\sum_{n=2}^\infty u_n ((\sum_{i=1}^k  \frac{k!}{i!(k-i)!} (-1)^{k-i}
l'(i)) + 2 (-1)^k \rho) 
-\frac{1}{3}\rho^3 h^2)\\
&\Lambda (c)(6) \defeq (\frac{h/2}{\arcsinh (h/2)})^2
(\sum_{n=2}^\infty u_n ((\sum_{i=1}^k  \frac{k!}{i!(k-i)!} (-1)^{k-i}
l'_\star (-i)) + 2 (-1)^k \rho)
-\frac{1}{3}\rho^3 h^2) \\
&\Lambda(c) (7)  \defeq (\frac{h/2}{\arcsinh (h/2)})^2
(\sum_{n=2}^\infty u_n ((\sum_{i=1}^k  \frac{k!}{i!(k-i)!} (-1)^{k-i}
l'(-i)) + 2 (-1)^k \rho)
-\frac{1}{3}\rho^3 h^2) \\
&\Lambda (c)(8) \defeq (\frac{h/2}{\arcsinh (h/2)})^2
(\sum_{n=2}^\infty u_n ((\sum_{i=1}^k  \frac{k!}{i!(k-i)!} (-1)^{k-i}
l'_\star (i)) + 2 (-1)^k \rho) 
-\frac{1}{3}\rho^3 h^2). \\
\end{align*}
where $\shuugou{u_n}_{\geq 2}$ is defined by $\frac{1}{2} (\log X )^2 =
\sum_{n \geq 2 } u_n (X-1)^n$.
Here we denote
\begin{align*}
&l' (m)= \sum_{k=1}^m \frac{(-h)^{k-1}}{k} \sum_{j_1 + \cdots j_k=m, j_i \geq 1} l_{j_1} l_{j_2} \cdots l_{j_k} \\
&l'_\star (-m)= \sum_{k=1}^m \frac{(-h)^{k-1}}{k} \sum_{j_1 + \cdots j_k=m, j_i \geq 1} l_{\star -j_1} l_{ \star -j_2} \cdots l_{ \star -j_k} \\
&l' (-m)= \sum_{k=1}^m \frac{h^{k-1}}{k} \sum_{j_1 + \cdots j_k=m, j_i \geq 1} l_{-j_1} l_{-j_2} \cdots l_{-j_k} \\
&l'_\star (m)= \sum_{k=1}^m \frac{h^{k-1}}{k} \sum_{j_1 + \cdots j_k=m, j_i \geq 1} l_{\star j_1} l_{ \star j_2} \cdots l_{ \star j_k} \\
\end{align*}
for any $m  \in \Z_{\geq 1}$.
Furthermore, $\Lambda (c) (1)=\Lambda (c) (2)$ implies that
$\Lambda(c)$ does not depend only on the choice of the orientation of 
$c$.

\end{thm}

By definition, we will prove $\Lambda (c) (i) =\Lambda (c) (i+4)$ for $i=1,2,3,4$.
The proof of $\Lambda (c)(1) =\Lambda (c) (3)$ and
$\Lambda (c) (2)= \Lambda (4)$  requires some analytic considerations.
By the skein relation, we will prove $\Lambda (c) (5)=
\Lambda (c)(8)$.

First of  all, we prove $\Lambda (c) (i) =\Lambda (c) (i+4)$ for 
$i =1,2,3,4$.

When $g_{-1} (X) \defeq \dfrac{1}{X} \in \Q [[X-1]]$,
we have
${g_{-1}}^{[n]}(X_1, \cdots , X_n)=\dfrac{(-1)^{n-1}}{X_1 \cdots X_n}$.
Hence we have 
\begin{align*}
&\psi_n (\exp (n \rho h)\gamma_{1,n} \cdots \gamma_{n,n}
g_{-1}^{[n]} (
\gamma_{1,n}, \cdots, \gamma_{n,n}) =l(0,n) =(-1)^n(\exp (\rho h)
\frac{2 \sinh (\rho h)}{h})^n, \\
&\psi_n (\exp (-n \rho h)\gamma_{\star 1,n}^{-1} \cdots \gamma_{\star n,n}^{-1}
g_{-1}^{[n]} (
\gamma_{\star 1,n}^{-1}, \cdots, \gamma_{\star n,n}^{-1}) =l_\star (0,n) =(-1)^{n-1}(\exp (-\rho h)
\frac{2 \sinh (\rho h)}{h})^n.
\end{align*}
For $i \in \Z_{\geq 1}$, when
$g_i (X) \defeq X^i$,
we have
\begin{align*}
&g_i^{[n]} (X_1, \cdots , X_n) =\sum_{i_1+i_2+ \cdots +i_n =n-k, i_j \geq 0}
X_1^{i_1} \cdots X_n^{i_n} \\
\end{align*}
for  $k =1, \cdots , n$.
If $k+1 \leq n$, we have
\begin{align*}
&g_i^{[n]} (X_1, \cdots , X_n) =0.
\end{align*}
Using the equations,
we obtain
\begin{align*}
&\psi_n (\exp (-n \rho h)\gamma_{1,n} \cdots \gamma_{n,n}
g_{-1}^{[n]} (
\gamma_{1,n}, \cdots, \gamma_{n,n}) =l(0,n) =(-1)^n(\exp (-\rho h)
\frac{2 \sinh (\rho h)}{h})^n, \\
&\psi_n (\exp (n \rho h)\gamma_{\star 1,n}^{-1} \cdots \gamma_{\star n,n}^{-1}
g_{-1}^{[n]} (
\gamma_{\star 1,n}^{-1}, \cdots, \gamma_{\star n,n}^{-1}) =l_\star (0,n) =(-1)^{n-1}(\exp (\rho h)
\frac{2 \sinh (\rho h)}{h})^n.
\end{align*}
If $n \leq k$, we obtain
\begin{align*}
&\psi_n (\exp (-n \rho h) \gamma_{1,n} \cdots \gamma_{n,n} g_k^{[n]} (\gamma_{1,n}
, \cdots, \gamma_{n,n}))=
l(k,n) \defeq
\sum_{i_1+i_2+ \cdots +i_n =k, i_j \geq 1}
l_{i_1} l_{i_2} \cdots l_{i_n}, \\
&\psi_n (\exp (n \rho h) \gamma_{\star 1,n}^{-1} \cdots \gamma_{\star n,n}^{-1}
 g_k^{[n]} (\gamma_{\star 1,n}^{-1}
, \cdots, \gamma_{\star n,n}^{-1}))=
l_\star (k,n) \defeq
\sum_{i_1+i_2+ \cdots +i_n =k, i_j \geq 1}
l_{\star i_1} l_{\star i_2} \cdots l_{\star i_n}.
\end{align*}
If $k +1 \leq n$, we obtain
\begin{align*}
&\psi_n (\exp (-n \rho h) \gamma_{1,n} \cdots \gamma_{n,n} g_k^{[n]} (\gamma_{1,n}
, \cdots, \gamma_{n,n}))=
0, \\
&\psi_n (\exp (n \rho h) \gamma_{\star 1,n}^{-1} \cdots \gamma_{\star n,n}^{-1}
 g_k^{[n]} (\gamma_{\star 1,n}^{-1}
, \cdots, \gamma_{\star n,n}^{-1}))=
0.
\end{align*}

Hence, where $g'_k (X) = X^{-1} (X-1)^k$, we have
\begin{align*}
&\sum_{n=1}^\infty \frac{(-h)^{n-1}}{n}
\psi_n (\exp (-n \rho h) \gamma_{1,n} \cdots \gamma_{n,n} {g'_k}^{[n]} (\gamma_{1,n}
, \cdots, \gamma_{n,n})) \\
&=
(\sum_{n=1}^k \sum_{i=n}^k  \frac{h^{n-1}}{n}\frac{k!}{i!(k-i)!} (-1)^{k-i}
l(i,n)) + (-1)^k \sum_{i=1}^\infty \frac{(-h)^{i-1}}{i} l(0,i), \\
&=
(\sum_{i=1}^k  \frac{k!}{i!(k-i)!} (-1)^{k-i}
l'(i)) + (-1)^k \sum_{i=1}^\infty \frac{(-h)^{i-1}}{i} l(0,i)\\
\end{align*}
and  
\begin{align*}
&\sum_{n=1}^\infty \frac{h^{n-1}}{n}
\psi_n (\exp (n \rho h) \gamma_{\star 1,n} \cdots \gamma_{\star n,n} {g'_k}^{[n]} (\gamma_{\star 1,n}
, \cdots, \gamma_{\star n,n})) \\
&=
(\sum_{n=1}^k \sum_{i=n}^k  \frac{h^{n-1}}{n}\frac{k!}{i!(k-i)!} (-1)^{k-i}
l_\star (i,n)) + (-1)^k \sum_{i=1}^\infty \frac{h^{i-1}}{i} l(0,i), \\
&=
(\sum_{i=1}^k  \frac{k!}{i!(k-i)!} (-1)^{k-i}
l'(i)) + (-1)^k \sum_{i=1}^\infty \frac{h^{i-1}}{i} l(0,i).\\
\end{align*}

Here we have
\begin{align*}
&\sum_{i=1}^\infty \frac{(-h)^{i-1}}{i} l(0,i) 
=\sum_{i=1}^\infty \frac{h^{i-1}}{i}(1-\exp (-2 \rho h))^i = -\frac{\log (\exp (-2 \rho
h))}{h} = 2 \rho, \\
&\sum_{i=1}^\infty \frac{h^{i-1}}{i} l_\star (0,i) 
=\sum_{i=1}^\infty \frac{(-h)^{i-1}}{i}(\exp (2 \rho h)-1)^i = \frac{\log (\exp (2 \rho
h))}{h} = 2 \rho. \\
\end{align*} 

Using the formulas, we obtain the following.

\begin{lemm}
\label{prop_def_Lambda_c}
We have
\begin{align*}
&\Lambda (c) (1) =\sum_{n=1}^\infty \frac{(-h)^{n-1}\exp ( -n \rho h)}{n} \psi_n
(\gamma_{1,n} \cdots \gamma_{n,n}
\lambda^{[n]} (\gamma_{1,n}, \cdots, \gamma_{n,n})) \\
&=\sum_{n=2}^\infty u_n ((\sum_{i=1}^k  \frac{k!}{i!(k-i)!} (-1)^{k-i}
l'(i)) + 2 (-1)^k \rho)=\Lambda (c) (5), \\
&\sum_{n=1}^\infty \frac{h^{n-1}\exp ( n \rho h)}{n} \psi_n
(\gamma_{\star 1,n}^{-1} \cdots \gamma_{\star n,n}^{-1}
\lambda^{[n]} (\gamma_{\star 1,n}^{-1}, \cdots, \gamma_{\star n,n}^{-1})) \\
&=\sum_{n=2}^\infty u_n ((\sum_{i=1}^k \frac{k!}{i!(k-i)!} (-1)^{k-i}
l'_\star (i)) + 2(-1)^k \rho)= \Lambda (c) (6).
\end{align*}
These equations imply that
$\Lambda (c) (1) =\Lambda (c) (5)$
and $\Lambda (c) (2) = \Lambda (c) (6)$.
Furthermore, using the orientation preserving diffeomorphism
$S^1 \times I \to S^1 \times I,
(s,t) \mapsto (1-s,1-t)$,
we obtain $\Lambda (c) (3) =\Lambda (c) (7)$
and $\Lambda (c) (4) = \Lambda (c) (8)$.
\end{lemm}

Using the following, we will prove
$\Lambda (c) (1) = \Lambda (c) (3)$
in Lemma \ref{lemm_Lambda_well_defined_2}

\begin{lemm}
We have
\begin{align*}
&\lambda^{[n]} (X_1^{-1}, X_2^{-1}, \cdots, X_n^{-1}) \\
&=(-1)^{n-1} X_1 \cdots X_n \sum_j 
\sum_{1 < i_1 < i_2
<  \cdots <
 i_j \leq n}X_{i_1} X_{i_2}  \cdots X_{i_j} \lambda^{[j]} (X_{i_1},
X_{i_2}, \cdots, X_{i_j}),
\end{align*}
where $\lambda (X) \defeq \frac{1}{2X} (\log (X))^2
\in \Q [[X-1]]$,
\end{lemm}

\begin{proof}
We prove the lemma by induction on $n$.
If $n=1$, the claim is obvious.
By induction hypothesis, we have
\begin{align*}
&\lambda^{[n]} (\gyaku{X_1}, \cdots, \gyaku{X_n}) \\
&=\frac{\lambda^{[n-1]}(\gyaku{X_1}, \cdots, \gyaku{X_{n-1}})
-\lambda^{[n-1]}(\gyaku{X_2}, \cdots, \gyaku{X_n})}{X_1-X_n} \\
&=(-1)^{n-2}\sum_j \sum_{2 \leq i_1 <i_2< \cdots < i_j \leq n-1}( 
\frac{(X_1X_2 \cdots X_{n-1} -X_2 X_3 \cdots X_n)X_{i_1} \cdots X_{i_j}
\lambda^{[j]} (X_{i_1}, \cdots , X_{i_j})}{\gyaku{X_1}-\gyaku{X_{n}}} \\
&+\frac{X_1^2 X_2 \cdots X_{n-1} X_{i_1} \cdots X_{i_j} \lambda^{[j+1]}(X_1,X_{i_1}, \cdots,
X_{i_j})
-X_2 X_3 \cdots X_{n}^2 X_{i_1} \cdots X_{i_j}
\lambda^{[j+1]} (X_{i_1}, \cdots, X_{i_j}, X_n)}{X_1^{-1}-X_n^{-1}} ) \\
&=(-1)^{n-1}X_1 \cdots X_n\sum_j \sum_{2 \leq i_1 <i_2< \cdots < i_j \leq n-1}( 
 X_{i_1}X_{i_2} \cdots X_{i_j}
\lambda^{[j]} (X_{i_1}, \cdots, X_{i_j})\\
&+X_1 X_{i_1} \cdots X_{i_j} \lambda^{[j+1]}(X_1, X_{i_1}, \cdots, X_{i_j})
+X_{i_1} \cdots X_{i_j} X_n \lambda^{[j+1]}(X_{i_1}, \cdots, X_{i_j}, X_n)\\
&+X_1 X_{i_1} \cdots X_{i_j} X_n \frac{\lambda^{[j+1]}(X_1, X_{i_1}, \cdots, X_{i_j})
-\lambda^{[j+1]} (X_{i_1}, \cdots, X_{i_j}, X_n)}{X_1-X_n} )\\
&=(-1)^{n-1}X_1 \cdots X_n\sum_j \sum_{2 \leq i_1 <i_2< \cdots < i_j \leq n-1}( 
 X_{i_1}X_{i_2} \cdots X_{i_j}
\lambda^{[j]} (X_{i_1}, \cdots, X_{i_j})\\
&+X_1 X_{i_1} \cdots X_{i_j} \lambda^{[j+1]}(X_1, X_{i_1}, \cdots, X_{i_j})
+X_{i_1} \cdots X_{i_j} X_n \lambda^{[j+1]}(X_{i_1}, \cdots, X_{i_j}, X_n)\\
&+X_1 X_{i_1} \cdots X_{i_j} X_n \lambda^{[j+2]}
(X_1,X_{i_1}, \cdots, X_{i_j}, X_n) )\\
&=(-1)^{n-1}X_1 \cdots X_n\sum_j \sum_{1 \leq i_1 <i_2< \cdots < i_j \leq n}(
 X_{i_1}X_{i_2} \cdots X_{i_j}
\lambda^{[j]} (X_{i_1}, \cdots, X_{i_j}))
\end{align*} 

This proves the lemma.

\end{proof}

\begin{lemm}
\label{lemm_Lambda_well_defined_2}
We have
\begin{align*}
&\sum_{n=1}^\infty \frac{(-h)^{n-1}}{n}
\exp (-n \rho h)\psi_n (\gamma_{1,n} \cdots \gamma_{n,n}
\lambda^{[n]}(\gamma_{1,n}, \cdots, \gamma_{n,n})) \\
&=\sum_{n=1}^\infty \frac{h^{n-1}}{n}
\exp (n \rho h)\psi_n (\gamma_{1,n}^{-1} \cdots \gamma_{n,n}^{-1}
\lambda^{[n]}(\gamma_{1,n}^{-1}, \cdots, \gamma_{n,n}^{-1})). \\
\end{align*}
This equation implies that $\Lambda (c) (1) =\Lambda (c) (3)$.
Furthermore, using the diffeormorphism
$S^1 \times I \to S^1 \times I, (s,t)\mapsto (1-s,1-t)$,
we obtain $\Lambda (c) (2) =\Lambda (c) (4)$.
\end{lemm}

\begin{proof}
We simply denote $e \defeq \exp (\rho h)$.
By the above lemma, we have
\begin{align*}
&\sum_{n=1}^\infty (\frac{h^{n-1}}{n}e^{n}
\psi_n(\gamma_{1,n}^{-1} \cdots \gamma_{n,n}^{-1}
\lambda^{[n]}(\gamma_{1,n}^{-1}, \cdots, \gamma_{n,n}^{-1})) \\
&=\sum_{n=1}^\infty
(\frac{(-h)^{n-1}}{n}e^{n} \psi_n (\gamma_{1,n}^{-1} \cdots \gamma_{n,n}^{-1}
\gamma_{1,n} \cdots \gamma_{n,n}
\sum_j \sum_{1 \leq i_1 <i_2< \cdots < i_j \leq n}
\gamma_{i_1,n} \cdots \gamma_{i_j,n}
\lambda^{[n]}(\gamma_{i_1,n}, \cdots, \gamma_{i_j,n})) \\
&=\sum_{n=1}^\infty
(\frac{(-h)^{n-1}}{n}e^{n} \psi_n (
\sum_j \sum_{1 \leq i_1 <i_2< \cdots < i_j \leq n}
\gamma_{i_1,n} \cdots \gamma_{i_j,n}
\lambda^{[n]}(\gamma_{i_1,n}, \cdots, \gamma_{i_j,n})) \\
&=\sum_{n=1}^\infty
(\frac{(-h)^{n-1}}{n}e^{n} 
\sum_j (\frac{e-e^{-1}}{h})^{n-j} \sum_{1 \leq i_1 <i_2< \cdots < i_j \leq n}
\psi_j (\gamma_{1,j} \cdots \gamma_{j,j}
\lambda^{[j]}(\gamma_{1,j}, \cdots, \gamma_{j,j})) \\
&=\sum_{n=1}^\infty
(\frac{(-h)^{n-1}}{n}e^{n} 
\sum_j (\frac{e-e^{-1}}{h})^{n-j} \frac{n!}{(n-j)!j!}
\psi_j (\gamma_{1,j} \cdots \gamma_{j,j}
\lambda^{[j]}(\gamma_{1,j}, \cdots, \gamma_{j,j})) \\
&=\sum_{n=1}^\infty
(\frac{(-h)^{n-1}}{n}e^{n} 
\sum_j (\frac{e-e^{-1}}{h})^{n-j} \frac{n!}{(n-j)!j!}
\psi_j (\gamma_{1,j} \cdots \gamma_{j,j}
\lambda^{[j]}(\gamma_{1,j}, \cdots, \gamma_{j,j})) \\
&=\sum_{n=1}^\infty
(\frac{(-h)^{n-1}}{n}e^n+\sum_{k=1}^\infty 
\frac{(-h)^{n+k-1}}{n+k} e^{n+k} (\frac{-e+e^{-1}}{h})^k
\frac{(n+k)!}{n! k!})
\psi_n (\gamma_{1,n} \cdots \gamma_{n,n}
\lambda^{[n]} (\gamma_{1,n}, \cdots, \gamma_{n,n})) \\
&=\sum_{n=1}^\infty
(\frac{(-h)^{n-1}}{n}e^n-\sum_{k=1}^\infty 
h^{-1} (-he)^{n} (1-e^2)^k
\frac{(n+k-1)!}{n! k!})
\psi_n (\gamma_{1,n} \cdots \gamma_{n,n}
\lambda^{[n]} (\gamma_{1,n}, \cdots, \gamma_{n,n})) \\
\end{align*}

By the Taylor expansion $\frac{1}{n} (x^{-n}-1) =\sum_{k=1}^\infty 
\frac{(n+k-1)!}{n!k!} (-X+1)^k
\in \Q [[X-1]]$ at $X=1$, we have
\begin{align*}
&\sum_{n=1}^\infty
(\frac{(-h)^{n-1}}{n}e^n-\sum_{k=1}^\infty 
h^{-1} (he)^{n} (e^2-1)^k
\frac{(n+k-1)!}{n! k!})
\psi_n (\gamma_{1,n} \cdots \gamma_{n,n}
\lambda^{[n]} (\gamma_{1,n}, \cdots, \gamma_{n,n})) \\
&=\sum_{n=1}^\infty
(\frac{(-h)^{n-1}}{n}e^n-
h^{-1} (-he)^{n} \frac{1}{n} (e^{-2n}-1))
\psi_n (\gamma_{1,n} \cdots \gamma_{n,n}
\lambda^{[n]} (\gamma_{1,n}, \cdots, \gamma_{n,n})\\
&=\sum_{n=1}^\infty
\frac{(-h)^{n-1}}{n}e^{-n}
\psi_n (\gamma_{1,n} \cdots \gamma_{n,n}
\lambda^{[n]} (\gamma_{1,n}, \cdots, \gamma_{n,n})\\
\end{align*}
This proves the lemma.

\end{proof}

\begin{lemm}
We have $l'(n)= l'_\star (n)$ for any $n$.
Hence we obtain $\Lambda (c)(5)=\Lambda (c) (8)$.
\end{lemm}

\begin{proof}
Let $l_{i,j}$ be an element of $\tskein{S^1 \times I}$
represented by a knot presented by the diagram
as shown in Figure \ref{figure_l_ij}.
We remark $l_{n,0}= l_{\star n+1}$ and $l_{0,n} = l_{n+1}$.
First, we will prove
\begin{equation*}
l_{n,0} = \sum_{k'=0}^k \sum_{i_1+ \cdots +i_{k'}=k-k'+1, i_l \geq 0}
(-h)^{k'} l_{n-k,i_1} l_{0,i_2} \cdots l_{0,i_{k'}}
\end{equation*}
by induction on $k$.
If $k=0$, the equation is trivial.
We remark that $l_{n-k,i}=l_{n-k-1,i+1}-h l_{n-k-1,0} l_{0,i}$.
By inductive assumption, we obtain
\begin{align*}
&l_{n,0} = \sum_{k'=0}^k \sum_{i_1+ \cdots +i_{k'}=k-k'+1, i_l \geq 0}
(-h)^{k'} l_{n-k,i_1} l_{0,i_2} \cdots l_{0,i_{k'}} \\
&=\sum_{k'=0}^k \sum_{i_1+ \cdots +i_{k'}=k-k'+1, i_l \geq 0}
(-h)^{k'} l_{n-k,i_1} l_{0,i_2} \cdots l_{0,i_{k'}} \\
&=\sum_{k'=0}^k \sum_{i_1+ \cdots +i_{k'}=k-k'+1, i_l \geq 0}
(-h)^{k'} (l_{n-k-1,i_1+1}-h l_{n-k-1,0}l_{0,i_1}) l_{0,i_2} \cdots l_{0,i_{k'}} \\
&=\sum_{k'=0}^{k+1} \sum_{i_1+ \cdots +i_{k'}=k-k'+1, i_l \geq 0}
(-h)^{k'} l_{n-k-1,i_1} l_{0,i_2} \cdots l_{0,i_{k'}} \\.
\end{align*}
This proves the above equation.
The above equation implies
\begin{align*}
l_{\star n} =\sum_{i=1}^n (-h)^{i-1} l (n,i).
\end{align*}

Second, we will prove 
\begin{align*}
l_\star (n,i)=\sum_{j=i}^n (-h)^{j-i} \frac{(j-1)!}{(i-1)! (j-i)!} l (n,j).
\end{align*}
We have 
\begin{equation*}
l_\star (n,i) = \sum_{j_1 + \cdots +j_i=n, i_l \geq 1}  \prod_{l=1}^i l_{\star i_l}
=  = \sum_{j_1 + \cdots +j_i=n, i_l \geq 1}  \prod_{l=1}^i 
\sum_{k_l=1}^{i_l} (-h)^{k_l-1} l (i_l, k_l).
\end{equation*}
Since
\begin{equation*}
\sum_{k_1+ \cdots +k_m=n, k_l \geq j_l}
l(j_1,k_1) l(j_2,k_2) \cdots l(j_m,k_m)=
l(k,j_1+ \cdots +j_m),
\end{equation*}
we obtain
\begin{align*}
l_\star (n,i)=\sum_{j=i}^n (-h)^{j-i} \frac{(j-1)!}{(i-1)! (j-i)!} l (n,j).
\end{align*}

Hence we have
\begin{align*}
&l'_\star (n) =\sum_{i=1}^n \frac{h^{i-1}}{i} l_\star (n,j) 
=\sum_{i=1}^n \frac{h^{i-1}}{i}
\sum_{j=i}^n (-h)^{j-i} \frac{(j-1)!}{(i-1)! (j-i)!} l (n,j) \\
&=\sum_{j=1}^n h^{j-1}(\sum_{i=1}^j \frac{(j-1)!}{(i-1)! (i-j)!} \frac{(-1)^{j-i}}{i})
l (n,j).
\end{align*}
Since 
\begin{align*}
\sum_{i=1}^j \frac{(j-1)!}{(i-1)! (i-j)!} \frac{(-1)^{j-i}}{i}
=\int_0^1(X-1)^{j-1} dX=[\frac{1}{j} (X-1)^j]^1_0 =(-1)^{j-1} \frac{1}{j},
\end{align*}
we have
\begin{align*}
l'_\star (n)=\sum_{j=1}^n \frac{(-h)^{j-1}}{j} l (n,j)=l' (n).
\end{align*}
This proves the lemma.

\end{proof}

\begin{cor}
\label{cor_def_Lambda_trivial}
If $c$ is a contractible simple closed curve,
we have $\Lambda (c) =0$.
\end{cor}

\begin{proof}
We fix the embedding $e:S^1 \times I \to D^2$ and
denote by $e_*$ the $\Q [\rho][[h]]$-algebra homomorphism
$ \tskein{S^1 \times I} \to \tskein{D^2} \simeq
\Q [\rho][[h]]$ induced by $e$.
We have
\begin{align*}
&e_* (l_\star'(n))= \sum_{k=1}^n \frac{h^{k-1}}{k} e_* (\sum_{i_1+ \cdots +i_k =n, i_j \geq 1}
l_{\star i_1} \cdots l_{\star i_k}) \\
&= \sum_{k=1}^n \frac{h^{k-1}}{k} \frac{(n-1)!}{(k-1)!(n-k)!}h^{k-1}
(\frac{2\sinh (\rho h)}{h})^{k}
(\exp (-\rho h))^{n-k}. \\
\end{align*}
Since 
\begin{align*}
&\sum_{k=1}^n \frac{(n-1)!}{(n-k)! (k-1)!} \frac{1}{k}X^k Y^{n-k}=\int_0^X (Z+Y)^{n-1} dZ \\
&=[\frac{1}{n}(Z+Y)^n]^X_0 =\frac{1}{n}((X+Y)^n-Y^n),
\end{align*}
we have
\begin{align*}
&e_* (l_\star'(n))=\frac{1}{n h} (\exp (n \rho h)- \exp ( -n \rho h))= \frac{2 \sinh (n\rho h)}{n h}.
\end{align*}
Using this, we obtain
\begin{align*}
&e_* (\sum_{n=2}^\infty u_n (\sum_{i=1}^n  \frac{n!}{i!(n-i)!} (-1)^{n-i}
l_\star'(i))+2(-1)^n \rho ) \\
&=\sum_{n=2}^\infty u_n (\sum_{i=1}^n  \frac{n!}{i!(n-i)!} (-1)^{n-i}
\frac{2 \sinh (i \rho h)}{ih})+2(-1)^n \rho) \\
&=\frac{1}{h}f(\rho h),
\end{align*}
where we denote
\begin{equation*}
f(X) =\sum_{n=2}^\infty u_n(\sum_{i=1}^n \frac{n!}{i!(n-i)!}(-1)^{n-i} \frac{2 \sinh(iX)}
{i} +2(-1)^n X) \in \Q [[X]].
\end{equation*}
Since $f(0)=0$ and
\begin{align*}
&f' (X) =\sum_{n=2}^\infty u_n (\sum_{i=1}^n \frac{n!}{i!(n-i)!}(-1)^{n-i} 2 \cosh (iX) +2(-1)^n ) 
\\
&=\sum_{n=2}^\infty u_n ((\exp (X)-1)^n+(\exp (-X)-1)^n) \\
&=\frac{1}{2}( (\log (\exp (X)))^2 +(\log (\exp(-X)))^2)= X^2
\end{align*}
we get $f(X)=\frac{1}{3}X^3$.
This proves the corollary.
\end{proof}

\subsection{Dehn twist on an annulus }
We denote by $S^1  \defeq \R/\Z$, by $c_l$ a simple closed
curve $S^1 \times \shuugou{\frac{1}{2}}$, by
$t$ the Dehn twist along $c_l$.
We denote by $r^0$ the element of $\tskein{S^1 \times I, \shuugou{(0,0)},
\shuugou{(0,1)}}$ and by $r^n \defeq t^n (r^0)$ for any $n \in \Z$.
We need  a $\Q [\rho][[h]]$-bilinear map
$(\cdot)  (\cdot) :\tskein{S^1 \times I, \shuugou{(0,0)}, \shuugou{(0,1)}}
\times \tskein{S^1 \times I, \shuugou{(0,0)}, \shuugou{(0,1)}}
\to \tskein{S^1 \times I, \shuugou{(0,0)}, \shuugou{(0,1)}}$ defined by
$[T_1][T_2] =[T_1T_2]$.
Here we denote by $T_1  T_2$ the tangle presented by
$\mu_1 (D_1) \cup \mu_2 (D_2)$ where 
we choose tangle diagrams $D_1$ and $D_2$
presenting $T_1$ and $T_2$, respectively,
and embeding maps $\mu_1$ and $\mu_2 :S^1 \times I 
\to S^1\times I$ defined by $\mu_1(\theta,t) =
(\theta,\frac{t+1}{2})$ and $\mu_2(\theta,\frac{t}{2})$.
We remark that $\mu_1 (D_1) \cup \mu_2 (D_2)$ must be
smoothed out in the neighborhood of $c_l$.
We consider $\tskein{S^1 \times I, \shuugou{(0,0)}, \shuugou{(0,1)}}$
as a commutative associative algebra by the bilinear function.
The aim of this subsection is  to prove the lemma.

\begin{lemm}
\label{lemm_annulus_Dehn}
We have $\sigma (\Lambda (c_l))(r^0)= \log (t)(r^0)$.
\end{lemm}

First of all, we prove the following.

\begin{lemm}
\label{lemm_l_n_prepare_0}
We have 
\begin{align*}
\sigma (l_{\star n})(r_0)=nr^n-h(n-1) r^{n-1}  l_{\star 1}-h
(n-2) r^{n-2}l_{\star 2} \cdots -hr^1l_{\star n-1}.
\end{align*}
\end{lemm}

\begin{proof}
For $\epsilon_1, \epsilon_2, \cdots, \epsilon_{2n-1}$,
let $D(\epsilon_1, \cdots, \epsilon_{2n-1})$ be
the diagram as shown in the Figure \ref{fig_annulus_l_n_r}
which is 
Type C(+), Type C(-) or Type C(0)
in Figure \ref{fig_Conway_triples}
in the box $i$ for $\epsilon_i =1,-1,0$, respectively.
We remark 
\begin{align*}
&[T(D(\overbrace{1, \cdots,1}^{j},0, \overbrace{-1, \cdots,-1}^{i},
\overbrace{1,\cdots,1}^{i},\overbrace{-1, \cdots, -1}^{j}))] =r^n, \\
&[T(D(\overbrace{1, \cdots,1}^{j},0, \overbrace{-1, \cdots,-1}^{i},
\overbrace{-1,\cdots,-1}^{i_2},0,
\overbrace{1,\cdots, 1}^{i_1},\overbrace{-1, \cdots, -1}^{j}))] =r^{i_1+j}l_{\star i_2+1}.
\end{align*}
for $i,j,i_1,i_2 \in \Z_{\geq 0}$
satisfying $i+j=n-1$ and $i_1+i_2=i-1$.
Using this, we have
\begin{align*}
&[T(D(\overbrace{1, \cdots,1}^{j},0, \overbrace{-1, \cdots,-1}^{i},
,\overbrace{-1, \cdots, -1}^{n}))] , \\
&=[T(D(\overbrace{1, \cdots,1}^{j},0, \overbrace{-1, \cdots,-1}^{i},
\overbrace{1,\cdots,1}^{i},\overbrace{-1, \cdots, -1}^{j}))] \\
&-h\sum_{i_1=0}^{i-1}[T(D(\overbrace{1, \cdots,1}^{j},0, \overbrace{-1, \cdots,-1}^{i},
\overbrace{-1,\cdots,-1}^{i_2},0,
\overbrace{1,\cdots, 1}^{i_1},\overbrace{-1, \cdots, -1}^{j}))] \\
&=r^n-h\sum_{j=1}^i r^{n-j} l_{\star j}.
\end{align*}

Hence we have
\begin{align*}
&\sigma(l_{\star n})(r^0) \\
&=\sum_{i=1}^n [T(D(\overbrace{1, \cdots,1}^{n-i},0, \overbrace{-1, \cdots,-1}^{i},
,\overbrace{-1, \cdots, -1}^{n}))] \\
&=nr^n-h(n-1) r^{n-1}  l_{\star 1}-h
(n-2) r^{n-2}l_{\star 2} \cdots -hr^1l_{\star n-1}.
\end{align*}
This proves the lemma.

\begin{figure}
\includegraphics[width=350pt]{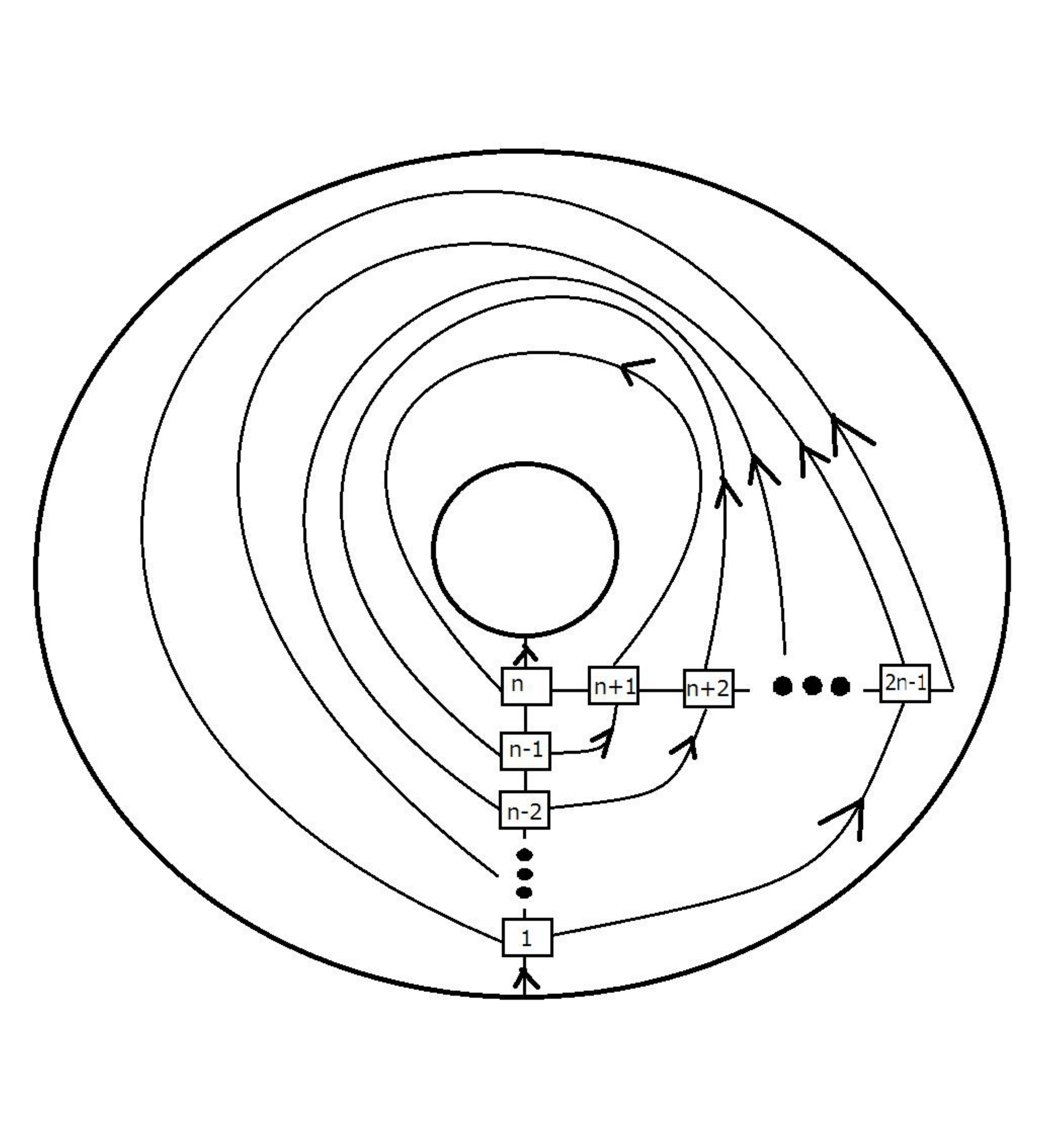}
\caption{$D(\epsilon_1, \cdots, \epsilon_{2n-1})$}
\label{fig_annulus_l_n_r}
\end{figure}
\end{proof}

The following lemma is proved in Morton \cite{Mor02} Theorem 4.2,
using the Murphy operator.
In this paper, we prove the lemma
by elementary methods.

\begin{lemm}
\label{lemm_l_n}
We have $\sigma (l(n))(r^0)
=\sigma (l_\star (n))(r^0)=
(\sum_{i=1}^{n}\frac{1}{i} \frac{(n+i-1)!}{(n-i)!(2i-1)!}h^{2(i-1)})r^n$.
\end{lemm}
 
To prove this lemma, we need the followings.

\begin{lemm}
\label{lemm_l_n_prepare_1}
For $x_1, x_2, \cdots, x_m  \in \tskein{S^1 \times I }$,
we have 
\begin{equation*}
\sigma (x_1 x_2 \cdots x_m)(r^0)=\sum_{j=1}^m h^{j-1}
\sum_{1 \leq i_1< \cdots <i_j \leq m}
\prod_{i \in \shuugou{i_1, \cdots, i_j} }(\sigma(x_{i})(r^0)
\prod_{i \in \shuugou{1, \cdots, m} \backslash
\shuugou{i_1, \cdots ,i_j}}(r^0 x_i).
\end{equation*}
\end{lemm}

\begin{proof}
We prove the lemma by induction on $m$.
If $m=1$, the lemma is obvious.
By Leibniz rule and induction hypothesis, we have
\begin{align*}
&
\sigma (x_1 x_2 \cdots x_m)(r) \\
&= (x_1r^0 )\sigma ( x_2 \cdots x_m)(r^0)
+\sigma (x_1)(r^0) \prod_{i \in \shuugou{2, \cdots, m }}
(r_0 x_i) \\
&=(r_0 x_1) \sigma ( x_2 \cdots x_m)(r^0)
+h \sigma (x_1)(r_0) \sigma (x_2 \cdots x_m)(r^0)
+\sigma (x_1)(r^0) \prod_{i \in \shuugou{2, \cdots, m }}
(r_0 x_i) \\
&=\sum_{j=1}^m h^{j-1}
\sum_{1 \leq i_1< \cdots <i_j \leq m}
\prod_{i \in \shuugou{i_1, \cdots, i_j} }(\sigma(x_{i})
\prod_{i \in \shuugou{1, \cdots, m} \backslash
\shuugou{i_1, \cdots ,i_j}}(r x_i).
\end{align*}
This proves the lemma.
\end{proof}

\begin{lemm}
\label{lemm_l_n_prepare_2}
We denote $C_k =a_k-h\sum_{i_1+i_2 =k,i_j \geq 1}a_{i_1}  b_{i_2}+b_k  \in \Z[h,a_1, \cdots , a_n,
b_1, \cdots, b_n]$.
Then we have
\begin{align*}
&D_n\defeq C_n+\frac{1}{2}h (\sum_{i_1 +i_2 =n, i_j \geq 1}C_{i_1} C_{i_2})+
\cdots + \frac{1}{n}h^{n-1}(\sum_{i_1+i_2 + \cdots+i_n=n, i_j \geq 1}C_{i_1} \cdots C_{i_n}) \\
&=a_n+\frac{1}{2}h (\sum_{i_1 +i_2 =n, i_j \geq 1}a_{i_1}a_{i_2})+
\cdots + \frac{1}{n}h^{n-1}(\sum_{i_1+i_2 + \cdots+i_n=n, i_j \geq 1}a_{i_1} \cdots a_{i_n}) \\
&+b_n+\frac{1}{2}h (\sum_{i_1 +i_2 =n, i_j \geq 1}b_{i_1}b_{i_2})+
\cdots + \frac{1}{n}h^{n-1}(\sum_{i_1+i_2 + \cdots+i_n=n, i_j \geq 1}b_{i_1} \cdots b_{i_n})
\end{align*}
\end{lemm}

\begin{proof}
It suffices to show $\frac{\partial^2}{\partial b_s \partial a_t}D_n=0$. 
We denote
\begin{align*}
&E_0 \defeq 1, \\
&E_i  \defeq C_i+ h(\sum_{i_1 +i_2 =i, i_j \geq 1}C_{i_1} C_{i_2})+
\cdots + h^{i-1}(\sum_{i_1+i_2 + \cdots+i_i=i, i_j \geq 1}C_{i_1} \cdots C_{i_i}), \\
&F_0 \defeq 1, \\
&F_i  \defeq C_i+ 2h(\sum_{i_1 +i_2 =i, i_j \geq 1}C_{i_1} C_{i_2})+
\cdots + ih^{i-1}(\sum_{i_1+i_2 + \cdots+i_i=i,i_j \geq 1}C_{i_1} \cdots C_{i_i}) \\.
\end{align*}
If $t=n$ or $s=n$, the claim is obvious. We assume 
$t<n$ and $s<n$. We have
\begin{align*}
&\frac{\partial}{\partial a_t} D_n=hE_{n-t}- h^2\sum_{k=1}^{n-t} b_k E_{n-k-t}, \\
&\frac{\partial^2}{\partial b_s \partial a_t}D_n \\
&=h^2F_{n-t-s}-h^2E_{n-t-s}-h^3\sum_{i=1}^{n-t-s}(a_i+b_i- h
\sum_{i_1+i_2=i}a_{i_1}b_{i_2})F_{n-t-s-i} \\
&=h^2F_{n-t-s}-h^2E_{n-t-s}-h^3\sum_{i=1}^{n-t-s}C_i F_{n-t-s-i} \\
&=h^2F_{n-t-s}-h^2E_{n-t-s} \\
&-h^3(\sum_{i_1 +i_2 =i, i_j \geq 1}C_{i_1} C_{i_2})-h^4 (2\sum_{i_1+i_2+i_3=i, i_j \geq 1}C_{i_1}
C_{i_2}C_{i_3})
\cdots -h^{i+1}((i-1) \sum_{i_1+i_2 + \cdots+i_i=i}C_{i_1} \cdots C_{i_i}) \\
&=0
\end{align*}
This proves the lemma.
\end{proof}

\begin{proof}[Proof of \ref{lemm_l_n}]
By Lemma \ref{lemm_l_n_prepare_0} and Lemma \ref{lemm_l_n_prepare_1}, 
we have 
\begin{align*}
&\sigma(l_n)(r^0) = \\
&\frac{1}{h}(C_n-b_n+\frac{h}{2} \sum_{i_1 +i_2 =n, i_j \geq 1}(C_{i_1} C_{i_2}-b_{i_1}b_{i_2})+
\cdots + \frac{h^{n-1}}{n}\sum_{i_1+i_2 + \cdots+i_n=n,i_j \geq 1}(C_{i_1} \cdots C_{i_n}-b_{i_1} \cdots b_{i_n}) ),
\end{align*}
where $b_i =r^0 l_{\star i}$, $a_i=hi r^i$ and $C_i =a_i+b_i-h\sum_{i_1 +i_2 =i, i_j \geq 1}a_{i_1}b_{i_1}$.
By Lemma \ref{lemm_l_n_prepare_2}, we have
\begin{align*}
&\sigma (l_{\star n})(r_0)==a_n+\frac{1}{2}h (\sum_{i_1 +i_2 =n, i_j \geq 1}a_{i_1}a_{i_2})+
\cdots + \frac{1}{n}h^{n-1}(\sum_{i_1+i_2 + \cdots+i_n=n,i_j \geq 1}a_{i_1} \cdots a_{i_n}) \\
&=r^n (n+\frac{1}{2}h (\sum_{i_1 +i_2 =n, i_j \geq 1}{i_1}{i_2})+
\cdots + \frac{1}{n}h^{n-1}(\sum_{i_1+i_2 + \cdots+i_n=n,i_j \geq 1}{i_1} \cdots {i_n})). \\
\end{align*}
Since 
\begin{equation*}
\sum_{i_1+i_2+\cdots+i_k=n, i_j \geq 0} i_1 i_2 \cdots i_k =\frac{ (n+k-1)!}{(2k-1)!(n-k)!},
\end{equation*}
we have $\sigma (l(n))=(\sum_{i=1}^{n}\frac{1}{i} \frac{(n+i-1)!}{(n-i)!(2i-1)!}h^{2(i-1)})r^n$.
This proves the lemma.
\end{proof}

The polynomial $\sum_{i=0}^{n-1}\frac{1}{i+1} \frac{(n+i)!}{(n-i-1)!(2i-1)!}h^{2i}$
is
written by the quantum integer.

\begin{prop}
We have
\begin{equation*}
\sum_{i=0}^{n-1}\frac{1}{i+1} \frac{(n+i)!}{(n-i-1)!(2i-1)!}(q-q^{-1})^{2i}=
\frac{1}{n}(\frac{q^n-g^{-n}}{q-q^{-1}})^2.
\end{equation*}
\end{prop}

\begin{proof}
First of all, we prove
\begin{equation*}
F_n ((q-q^{-1})^2) \defeq \sum_{i=0}^{n-1} \frac{(n+i)!}{(n-i-1)!(2i-1)!}(q-q^{-1})^{2i}=
\frac{q^{2n}-q^{-2n}}{q^2-q^{-2}}.
\end{equation*}
It suffices to show
\begin{align*}
F_n (t)(t+2)-F_{n-1} (t)=F^{n+1} (t).
\end{align*}
We have
\begin{align*}
&F_n (t)(t+2)-F_{n-1} (t)=t^n+t^0(2n-(n-1)) \\
&+\sum_{i=1}^{n-1}t^i (2 \frac{(n+i)!}{(n-i-1)!(2i+1)!}+
\frac{(n+i-1)!}{(n-i)!(2i-1)!}-\frac{(n+i-1)!}{(n-i-2)!(2i+1)!}) \\
&=t^n+t^0(2n-(n-1)) \\
&+\sum_{i=1}^{n-1}t^i \frac{(n+i-1)!}{(n-i)!(2i+1)!}
(2(n-i)(n+i)+2i(2i+1)-(n-i)(n-i-1))\\
&=t^n+t^0(2n-(n-1)) \\
&+\sum_{i=1}^{n-1}t^i \frac{(n+i-1)!}{(n-i)!(2i+1)!}
(2n^2-2i^2+4i^2+2i-n^2+2ni-i^2+n-i)\\
&=\sum_{i=0}^{n} \frac{(n+1+i)!}{(n-i)!(2i-1)!}t^i.
\end{align*}

We compute
\begin{align*}
&\int \frac{q^{2n}-q^{-2n}}{q^2-q^{-2}}dt \mathrm{\ \ \ where }(q-q^{-1})^2 =t 
=\int \frac{q^{2n}-q^{-2n}}{q^2-q^{-2}}2(q-q^{-3})dq \\
&=2 \int (q^{2n-1}-q^{-2n-1}) dq 
=\frac{1}{n}(q^{2n}+q^{-2n}) +C
\end{align*}
Using this computation, we have
\begin{align*}
&\sum_{i=0}^{n-1}\frac{1}{i+1} \frac{(n+i)!}{(n-i-1)!(2i-1)!}t^i 
=\frac{1}{t}\int_0^t \sum_{i=0}^{n-1} \frac{(n+i)!}{(n-i-1)!(2i-1)!}z^idz \\
&=\frac{1}{(q-\gyaku{q})^2}[\frac{1}{n}(r^{2n}+r^{-2n})]^q_0 
=\frac{1}{n}(\frac{q^n-g^{-n}}{q-q^{-1}})^2 \\.
\end{align*}

This proves the proposition.

\end{proof}

\begin{proof}[proof of Lemma \ref{lemm_annulus_Dehn}]
We denote $y \defeq \exp (\arcsinh \frac{h}{2})$, which satisfies
$y-y^{-1} =h$.
We have
\begin{align*}
&\sigma (\Lambda (c_l))(r^0) \\
&=\sigma ((\frac{h/2}{\arcsinh (h/2)})^2 \sum_{n=2}^\infty
u_n (\sum_{i=1}^n( \frac{n!}{i!(n-i)!}(-1)^{n-i} l(i))+ 2(-1)^n\rho))(r^0) \\
&=(\frac{h/2}{\arcsinh (h/2)})^2 \sum_{n=2}^\infty
u_n (\sum_{i=1}^n( \frac{n!}{i!(n-i)!}(-1)^{n-i} (\frac{y^i-y^{-i}}{y-y^{-1}})^2\frac{1}{i} r^i) \\
&=(\frac{h/2}{\arcsinh (h/2)})^2 \sum_{n=2}^\infty
u_n (\sum_{i=1}^n( \frac{n!}{i!(n-i)!}(-1)^{n-i} (\frac{1}{y-y^{-1}})^2\frac{(y^2r)^i+(y^{-2} r)^i -2 r^i}{i} r^i).
 \\
\end{align*}
Since 
\begin{align*}
&\frac{\partial}{\partial y} \sum_{n=1}^\infty
u_n  (\sum_{i=1}^n(\frac{n!}{i! (n-i)!} (-1)^{n-i} \frac{1}{i}(y^{2i}-1)x^i)+(-1)^n \log (y^2)) \\
&=\frac{2}{y} \sum_{n=1}^\infty u_n(y^2 x-1)^n =\frac{1}{y} (\log (y^2x))^2 \\
&=\frac{1}{y}((\log y^2)^2+ 2 \log y^2 \log x +(\log x)^2) \\
\end{align*}
we have
\begin{align*}
& \sum_{n=1}^\infty
u_n  (\sum_{i=1}^n(\frac{n!}{i! (n-i)!} (-1)^{n-i} \frac{1}{i}(y^{2i}-1)x^i)+(-1)^n \log (y^2)) \\
&=\frac{1}{6} (\log y^2)^3 +\frac{1}{2}(\log y^2)^2 \log x + \frac{1}{2}( \log x)^2 \log y^2 .\\
\end{align*}
Using this, we obtain
\begin{align*}
&\sigma (\Lambda (c_l))(r^0) \\
&=(\frac{h/2}{\arcsinh (h/2)})^2 
\frac{1}{(y-y^{-1})^2}(\frac{1}{6} (\log y^2)^3r^0 +\frac{1}{2}(\log y^2)^2 \log r + \frac{1}{2}
( \log r)^2 \log y^2 \\
&+\frac{1}{6} (\log y^{-2})^3r^0 +\frac{1}{2}(\log y^{-2})^2 \log r + \frac{1}{2}( \log r)^2 \log y^{-2}) \\
&=(\frac{h/2}{\arcsinh (h/2)})^2 (\frac{2 \log y}{y-y^{-1}})^2 \log r= \log r.
\end{align*}
where we denote $\log r = \sum_{i=1}^\infty \frac{(-1)^{i-1}}{i} (r^1-r^0)^i = \log (t)(r)$.
This proves the lemma.

\end{proof}

\subsection{Proof of main theorem}
In the subsection, we prove Theorem \ref{thm_main_Dehn}.

We fix an embedding $\iota :S^1 \times I \to \Sigma$ such that 
$\iota (c_l) =c$.

We assume that $\iota (S^1 \times I)$ separate $\Sigma$ into two surfaces
$\Sigma^1 $ and $\Sigma^2$.
For disjoint finite sets $J^1, J^2 \subset S^1$,
  we consider the trilinear map
\begin{align*}
\varpi_{J^1, J^2} :&\tskein{\Sigma_1, (J^- \cap \partial \Sigma_1) \cup
\iota (J^2 \times \shuugou{1}), (J^+ \cap \partial \Sigma_1) \cup
\iota (J^1 \times \shuugou{1})} \\
&\times \tskein{S^1 \times I, J^1 \times \shuugou{1} \cup
J^2 \times \shuugou{0}, J^1 \times \shuugou{0}  \cup
J^2 \times \shuugou{1}} \\
&\times  \tskein{\Sigma_1, (J^- \cap \partial \Sigma_2) \cup
\iota (J^1 \times \shuugou{0}), (J^+ \cap \partial \Sigma_1) \cup
\iota (J^2 \times \shuugou{0})} \\
& \to \tskein{\Sigma, J^-,J^+}
\end{align*}
defined by $\varpi_{J'} ([T_1],[T_2],[T_3]) =[T_1T_2T_3]$ for 
\begin{align*}
&T_1 \in \tknot{\Sigma_1, (J^- \cap \partial \Sigma_1) \cup
\iota (J^2 \times \shuugou{1}), (J^+ \cap \partial \Sigma_1) \cup
\iota (J^1 \times \shuugou{1})}, \\
& T_2 \in \tknot{S^1 \times I, J^1 \times \shuugou{1} \cup
J^2 \times \shuugou{0}, J^1 \times \shuugou{0}  \cup
J^2 \times \shuugou{1}}, \\
&T_3 \in \tknot{\Sigma_1, (J^- \cap \partial \Sigma_2) \cup
\iota (J^1 \times \shuugou{0}), (J^+ \cap \partial \Sigma_1) \cup
\iota (J^2 \times \shuugou{0})} .
\end{align*}
Here we denote by $T_1T_2T_3$ the tangle presented by
$d_1 \cup \iota (d_2) \cup d_3$, respectively, where 
$d_1$, $d_2$ and $d_3$ present $T_1$, $T_2$ and $T_3$, respectively.
We remark that $d_1 \cup \iota (d_2) \cup d_3$ must be
smoothed out in the neighborhood of $\iota (S^1 \times \shuugou{0,1})$.
Then we have the followings.
\begin{itemize}
\item The $\Q [\rho][[h]]$-module $\tskein{\Sigma,J^-,J^+}$
is generated by all images of $\varpi_{J_1,J_2}$ for any $J_1$ and $J_2$
as $\Q[\rho][[h]]$-module.
\item The map $\varpi_{J_1,J_2}$ preserves the filtrations, in other words,
\begin{align*}
\varpi_{J^1, J^2} (&\tskein{\Sigma_1, (J^- \cap \partial \Sigma_1) \cup
\iota (J^2 \times \shuugou{1}), (J^+ \cap \partial \Sigma_1) \cup
\iota (J^1 \times \shuugou{1})} \\
&\times F^n \tskein{S^1 \times I, J^1 \times \shuugou{1} \cup
J^2 \times \shuugou{0}, J^1 \times \shuugou{0}  \cup
J^2 \times \shuugou{1}} \\
&\times  \tskein{\Sigma_1, (J^- \cap \partial \Sigma_2) \cup
\iota (J^1 \times \shuugou{0}), (J^+ \cap \partial \Sigma_1) \cup
\iota (J^2 \times \shuugou{0})}) \\
& \subset F^n \tskein{\Sigma, J^-,J^+}
\end{align*}
for any $n$.
\item We have  $t_{c} \circ \varpi_{J_1,J_2} = \varpi_{J_1,J_2} \circ (\id ,t_{c_l},\id )$
and $\sigma(\iota(x)) \circ \varpi_{J_1,J_2} = \varpi_{J'} \circ (\id,\sigma (x), \id )$
for $x \in \tskein{S^1 \times I}$.
\end{itemize}

We assume that $\Sigma \backslash \iota (S^1 \times (0,1))$ is a connected surface 
$\Sigma^1$.
For  finite disjoint set $J^1,J^2 \subset S^1$, we consider the bilinear map
\begin{align*}
\varpi_{J^1,J^2} : 
&\tskein{\Sigma_1, 
J^- \cup \iota (J^1 \times \shuugou{0} \cup J^2 \times \shuugou{1}),
J^+ \cup \iota (J^1 \times \shuugou{1} \cup J^2 \times \shuugou{0})} \\
&\times 
\tskein{S^1 \times I,
J^1 \times \shuugou{1} \cup J^2 \times \shuugou{0} ,
J^1 \times \shuugou{0} \cup J^2 \times \shuugou{1} } \\
&\to  \tskein{\Sigma, J^-,J^+}
\end{align*}
defined by $\varpi_{J^1J^2} ([T_1],[T_2]) =[T_1T_2]$ for 
\begin{align*}
&T_1 \in \tknot{\Sigma_1, 
J^- \cup \iota (J^1 \times \shuugou{0} \cup J^2 \times \shuugou{1}),
J^+ \cup \iota (J^1 \times \shuugou{1} \cup J^2 \times \shuugou{0})} \\
&T_2 \in
\tknot{S^1 \times I,
J^1 \times \shuugou{1} \cup J^2 \times \shuugou{0} ,
J^1 \times \shuugou{0} \cup J^2 \times \shuugou{1} }. \\
\end{align*}
Here we denote by $T_1T_2$ the tangle presented by
$d_1 \cup \iota (d_2) $, respectively, where 
$d_1$ and $d_2$ present $T_1$ and $T_2$, respectively.
We remark that $d_1 \cup \iota (d_2)$ must be
smoothed out in the neighborhood of $\iota (S^1 \times \shuugou{0,1})$.
Then we have the followings.
\begin{itemize}
\item The $\Q [\rho][[h]]$-module $\tskein{\Sigma,J^-,J^+}$
is generated by all images of $\varpi_{J_1,J_2}$ for any $J_1$ and $J_2$
as $\Q[\rho][[h]]$-module.
\item The map $\varpi_{J_1,J_2}$ preserves the filtrations, in other words,
\begin{align*}
\varpi_{J^1,J^2} ( 
&\tskein{\Sigma_1, 
J^- \cup \iota (J^1 \times \shuugou{0} \cup J^2 \times \shuugou{1}),
J^+ \cup \iota (J^1 \times \shuugou{1} \cup J^2 \times \shuugou{0})} \\
&\times 
F^n \tskein{S^1 \times I,
J^1 \times \shuugou{1} \cup J^2 \times \shuugou{0} ,
J^1 \times \shuugou{0} \cup J^2 \times \shuugou{1} }) \\
&\subset F^n \tskein{\Sigma, J^-,J^+}
\end{align*}
for any $n$.
\item We have $t_{c} \circ \varpi_{J^1,J^2} = \varpi_{J^1,J^2} \circ (\id ,t_{c_l})$
and $\sigma(\iota(x)) \circ \varpi_{J^1,J^2} = \varpi_{J^1,J^2} \circ (\id,\sigma (x))$
for $x \in \tskein{S^1 \times I}$.
\end{itemize}
Hence, it suffices to show the following lemma. 

\begin{lemm}
\label{lemm_Dehn_twist}
Fix an positive integer $m$.
Choose points $p_1=\frac{1}{2m}$, $\cdots$, $p_i =
\frac{i}{2m}$, $\cdots$, $p_m =\frac{m}{2m}$ in $S^1$
and $\epsilon_1, \epsilon_2, \cdots, \epsilon_m \in \shuugou{0,1}$.
We denote by $r_i^0$ an element of
$\tskein{S^1 \times I, \shuugou{(p_i,1-\epsilon_1)},
\shuugou{(p_i,\epsilon_1)}}$ represented by the tangle
presented by $\shuugou{p_i} \times I$.
Then we have
\begin{equation*}
\sigma(\Lambda(c_l)) (r_1^0 \boxtimes r_2^0 \boxtimes \cdots \boxtimes 
r_m^0)= \log (t)(r_1^0 \boxtimes r_2^0 \boxtimes \cdots \boxtimes 
r_m^0).
\end{equation*}
\end{lemm}

\begin{proof}
By the Leibniz rule and Lemma \ref{lemm_annulus_Dehn},
we have
\begin{align*}
&\sigma(\Lambda(c_l)) (r_1^0 \boxtimes r_2^0 \boxtimes \cdots \boxtimes 
r_m^0) \\
&=\sum_{i=1}^m r_1^0 \boxtimes  \cdots \boxtimes r_{i-1}^0 
\boxtimes \sigma(\Lambda (c_l))(r_i^0) \boxtimes
r_{i+1}^0 \boxtimes \cdots \boxtimes r_m^0 \\
&=\sum_{i=1}^m r_1^0 \boxtimes  \cdots \boxtimes r_{i-1}^0 
\boxtimes \log(t)(r_i^0) \boxtimes
r_{i+1}^0 \boxtimes \cdots \boxtimes r_m^0 \\
&= \log (t)(r_1^0 \boxtimes r_2^0 \boxtimes \cdots \boxtimes 
r_m^0).
\end{align*}
This proves the lemma.
\end{proof}

\subsection{Corollary}
Let $\check{\tskein{\Sigma}}$
be the Lie subalgebra topologically generated by
\begin{equation*}
\shuugou{\sigma(\Lambda (c_1)) \circ \cdots \circ
\sigma (\Lambda (c_{j-1}))( \Lambda (c_j))|c_1, \cdots, c_j \mathrm{ \ \ s.c.c.}}.
\end{equation*}
We consider the action of $\zettaiti{\mathcal{P} (\Sigma, *_\alpha)}$
on $\mathcal{P}(\Sigma, *_\alpha, *_\beta)$ defined by the following.
Let $R$ and $R'$ be  an oriented immersed pass of $\pi^+_1 (\Sigma, *_\alpha,
*_\beta)$ and an oriented immersed loop of $\hat{\pi^+_1} 
(\Sigma)$. We define
\begin{equation*}
\sigma_{\mathcal{P}} (R')(R) \defeq \sum_{p \in R' \cap R} \epsilon(p,R',R)
R_{*_\beta p} R'_p R_{*_\alpha}
\end{equation*}
where  we denote by $R_{*_\beta p} $  the pass  along
$R$ from $p$ to $*_\beta$, by $R_{p *_\alpha}$
the pass along $R$ from $*_\alpha$ to $p$,
by $R'_p$ the pass along $R'$ from $p$ to $p$
and by $\epsilon(p,R',R)$  the local intersection number
of $R'$ and $R$ at $p$.
We consider $\zettaiti{\mathcal{P} (\Sigma, *_\alpha)}$ 
as a Lie algebra defined by $[R', \zettaiti{R}]=\zettaiti{\sigma (R')(R)}$.
We remark $\zettaiti{\mathcal{P} (\Sigma, *_\alpha)}$
is a $\zettaiti{\mathcal{P} (\Sigma, *_\alpha)}$-module.

We consider the topology on $\zettaiti{\mathcal{P}(\Sigma,*_\alpha)}$
induced by the filtration $\filtn{\zettaiti{(\ker \epsilon_{\mathcal{P}})^n}}$.
and denote its completion $\widehat{\zettaiti{\mathcal{P}
(\Sigma, *_\alpha)}} \defeq \comp{i} \zettaiti{\mathcal{P}(\Sigma, *_\alpha)}/
\zettaiti{(\ker \epsilon_{\mathcal{P}})^i})$. The completion
also has a filtration $\filtn{F^n \widehat{\zettaiti{\mathcal{P}(\Sigma, *_\alpha)}}}$
such that $\widehat{\zettaiti{\mathcal{P}(\Sigma, *_\alpha)}}/
F^n\widehat{\zettaiti{\mathcal{P}(\Sigma, *_\alpha)}}
=\zettaiti{\mathcal{P}(\Sigma, *_\alpha)}/(\ker \epsilon_{\mathcal{P}})^i$.
We remark $\sigma_{\mathcal{P}} (x) =0$ if and only if $x=0$ for 
$x \in \widehat{\zettaiti{\mathcal{P}(\Sigma, *_\alpha)}}$.
For details see \cite{Kawazumi} \cite{KK}.
\begin{prop}[\cite{Kawazumi} \cite{KK} \cite{MT}]
We have
\begin{equation*}
\log (t_c)=\sigma_{\mathcal{P}} (\zettaiti{\frac{1}{2} (\log (c))^2}):
\widehat{\mathcal{P}(\Sigma, *_\alpha, *_\beta)} \to
\widehat{\mathcal{P}(\Sigma, *_\alpha, *_\beta)}.
\end{equation*}
\end{prop}

Using this proposition and 
Theorem \ref{thm_main_Dehn}, we have the following.
We define the Lie algebra homomorphism
\begin{equation*}
\varrho :\check{\tskein{\Sigma}} \to 
F^2 \widehat{\zettaiti{\mathcal{P}(\Sigma, *_\alpha)}}
\end{equation*}
defined by $\Lambda (c) \mapsto \zettaiti{\frac{1}{2} (\log (c))^2}$.

\begin{prop}
The Lie algebra homomorphism $\varrho$ is well-defined and
the restriction
$\sigma (x)_{|\widehat{\mathcal{P}(\Sigma, *_\alpha, *_\beta)}}$
equals
$\sigma_{\mathcal{P}} (\varrho (x))$.
\end{prop}

Let $\varpi_{\mathcal{P}2} :
F^2 \widehat{\zettaiti{\mathcal{P}(\Sigma, *_\alpha)}}
\to F^2 \widehat{\zettaiti{\mathcal{P}(\Sigma, *_\alpha)}}/
F^3 \widehat{\zettaiti{\mathcal{P}(\Sigma, *_\alpha)}}
\simeq H\cdot H\oplus Q[\rho]$ defined by
$\varpi_{\mathcal{P}2} ((R_1-1)(R_2-1)) =[R_1]\cdot[R_2]$
and $\varpi_{\mathcal{P}2} (h)=1$, where
$H \defeq \Q [\rho] \otimes_{\Q} H_1 (\Sigma, Z)$ and
$H \cdot H$ is the symmetric tensor of $H$.
Using the above proposition, we have the following.

\begin{prop}
\label{prop_BCH_jouken}
Let $V_1$ and $V_2 \subset V_1$
be $\Q [\rho]$-subspaces of $H$ satisfying
$\mu (V_1, V_2) =\shuugou{0}$ where $\mu $ is the intersection form.
For any $s_1, s_2,s_3 \in \varrho^{-1} (\varpi_{\mathcal{P} 2}^{-1}(V_1 \cdot V_2 \oplus \Q 
[\rho]))$, we have
\begin{equation*}
\sigma (s_1) \circ \sigma (s_2) \circ \sigma (s_3) (\ker \epsilon_{\mathcal{P}}) \subset
(\ker \epsilon_{\mathcal{P}})^2.
\end{equation*}
Furthermore, for any $s_1, s_2, \cdots ,s_{2j-1}
\in \varrho^{-1} (\varpi_{\mathcal{P} 2}^{-1}(V_1 \cdot V_2 \oplus \Q 
[\rho]))$, we have
\begin{equation*}
\sigma (s_1) \circ \sigma (s_2)\circ
\cdots \circ \sigma (s_{2j-1}) (F^{j-1} \widehat{\tskein{\Sigma, J^-,J^+}}) \subset
F^j \widehat{\tskein{\Sigma, J^-,J^+}}.
\end{equation*}
\end{prop}

\subsection{The Baker-Campbell-Hausdorff series}
\label{subsection_bch}
In this subsection, we will explain the Baker-Campbell-Hausdorff series
(BCH series).
We choose $S \subset \widehat{\tskein{\Sigma}}$ such that,
for any $i \in \Z_{\geq 0}$, there exists $j_i\in \Z_{\geq 0}$
satisfying 
\begin{equation}
\label{equation_jouken_bch}
\sigma(a_{1})\circ
\sigma(a_{2}) \circ \cdots \sigma ( a_{j_i})(F^i
\tskein{\Sigma,J^-,J^+})\subset F^{i+1} \tskein{\Sigma,J^-,J^+}
\end{equation}
for $a_1, a_2, \cdots, a_{j_i} \in
S$ and $J^-,J^+ \subset \partial \Sigma$.
 The BCH series $\bch$ is defined by
\begin{align*}
&\bch (\epsilon_1 a_1, \epsilon_2 a_2, \cdots,\epsilon_m a_m) \\
&\defeq 
 \log (\exp ( \epsilon_1 a_1) \exp (\epsilon_2 a_2) \cdots \exp (\epsilon_m a_m))
\end{align*}
for $a_1, a_2, \cdots, a_{m} \in
S$ and $\epsilon_1, \epsilon_2, \cdots, \epsilon_m \in \shuugou{\pm 1}$.
For example, 
\begin{equation*}
\bch(x,y) = x+y+\frac{1}{2}[x,y]+\frac{1}{12}([x,[x,y]]+[y,[y,x]])+ \cdots.
\end{equation*}
We denote
\begin{equation*}
\zettaiti{S} \defeq
\shuugou{\bch (\epsilon_1 a_1, \cdots, \epsilon_m 
a_m)|m \in \Z_{\geq 1}, a_1,  \cdots, a_{m} 
\in S, \epsilon_1, \cdots, \epsilon_m  \in \shuugou{\pm1}}.
\end{equation*}
For $a_1, a_2, \cdots, a_{j_1}  \in \zettaiti{S}$, 
they satisfies the equation (\ref{equation_jouken_bch}).
The unit of the group $(\zettaiti{S}, \bch) $ is  $0$.
For example, $F^3 \widehat{\tskein{\Sigma}}$ satisfies the condition.
Furthermore, if the genus of $\Sigma$ is $0$,
$\widehat{\tskein{\Sigma}}$ satisfies the condition.

\section{Framed pure braid group}

In the section, let $\Sigma$ be a compact connected oriented surface of genus
$0$.
The set of Dehn twists $\shuugou{t_{ij}|1 \leq i<j \leq b} \cup \shuugou{t_i|1 \leq i \leq b}$,
where $t_{ij} \defeq t_{c_{ij}}$ and $t_i \defeq t_{c_i}$,
generate the mapping class group $\mathcal{M} (\Sigma)$.
Here the simple closed curve$ c_i$ is  presented by $\zettaiti{R_i}$
and the simple closed curve $c_{ij}$ is  presented by $\zettaiti{R_i R_j}$.
Furthermore $\mathcal{M} (\Sigma)$ is presented by the relations
\begin{align*}
&\ad (t_i)(t_j) =t_j, \\
&\ad (t_s)(t_{ij}) =t_{ij}, \\
&\ad(t_{rs})(t_{ij}) =t_{ij} & \mathrm{if\ \ } r<s<i<j, \\
&\ad(t_{rs})(t_{ij}) =t_{ij} & \mathrm{if\ \ } i<r<s<j, \\
&\ad(t_{rs}t_{rj})(t_{ij}) =t_{ij} & \mathrm{if\ \ } r<s=i<j, \\
&\ad(t_{rs}t_{ij}t_{sj})(t_{ij}) =t_{ij} & \mathrm{if\ \ } i=r<s<j, \\
&\ad(t_{rs}t_{rj}t_{sj}\gyaku{t_{rj}} \gyaku{t_{sj}})(t_{ij}) =t_{ij} & \mathrm{if\ \ } r<i<s<j, \\ 
\end{align*}
where $\ad (a)(b) \defeq a b \gyaku{a}$.
See, for example, \cite{Birman} p.20 Lemma 1.8.2.

\begin{df}
The group homomorphism $\zeta : \mathcal{M} (\Sigma) \to
(\widehat{\tskein{\Sigma}},\bch)$ is defined by
$t_i \mapsto \Lambda (c_i)$ and $t_{ij} \mapsto \Lambda (c_{ij})$.
\end{df}

\begin{thm}
\label{thm_zeta_pure_braid}
The group homomorphism $\zeta$ is well-defined.
Furthermore, $\zeta$ is injective.
\end{thm}

\begin{proof}
Let $c, c'_1, c'_2, \cdots ,c'_k$ be simple closed curves in $\Sigma$
and $\epsilon_1, \cdots,  \epsilon_k $ be elements of $\shuugou{1,-1}$
satisfying $t_{c'_1}^{\epsilon_1} t_{c'_2}^{\epsilon_2}
\cdots t_{c'_k}^{\epsilon_k} (c) =c$.
It is enough to check 
\begin{equation*}
\bch(\epsilon_1 \lambda(c'_1), \cdots,
\epsilon_2 \Lambda(c'_2), \cdots, \epsilon_k \Lambda (c'_k),\Lambda(c),
,-\epsilon_k \Lambda(c'_k), \cdots, -\epsilon_1 \Lambda(c'_1)) =\Lambda(c)
\end{equation*}
o check the well-definedness of $\zeta$.
By Theorem \ref{thm_main_Dehn}, we have
\begin{align*}
&\bch(\epsilon_1 \Lambda(c'_1), \cdots,
\epsilon_2 \Lambda (c'_2), \cdots, \epsilon_k \Lambda (c'_k), \Lambda (c),
,-\epsilon_k \Lambda (c'_k), \cdots, -\epsilon_1 \Lambda (c'_1)) \\
&=\exp(\sigma(\epsilon_1 \Lambda (c'_1))) \circ \cdots \circ \exp(\sigma(\epsilon_k 
\Lambda(c'_k)))(\Lambda (c)) \\
&=t_{c'_1}^{\epsilon_1} t_{c'_2}^{\epsilon_2}
\cdots t_{c'_k}^{\epsilon_k} (\Lambda (c)) =\Lambda (c).
\end{align*}
This finishes the proof of well-definedness of $\zeta$.

By definition of the BCH series $\bch$, we have $\xi (\cdot) = \exp (\sigma(\zeta(\xi)))
(\cdot): \widehat{\tskein{\Sigma,J^-,J^+}} \to \widehat{\tskein{\Sigma,J^-,J^+}}$
for any $\xi  \in \mathcal{M} (\Sigma)$ and  any finite subsets $J^-,J^+$ of 
$\partial \Sigma$.
By Proposition \ref{prop_MCG_tskein_faithful}, we have
$\xi =\id_\Sigma$ if and only if $\zeta(\xi) =0$.
This finishes the proof of injectivity of $\zeta$.
\end{proof}

\begin{rem}
Using the lantern realtion (Lemma \ref{lemm_lantern_relation}),
we obtain $\zeta (t_c) =\Lambda (c)$ for any simple closed curve $c$.
\end{rem}

\section{Torelli group on $\Sigma_{g,1}$}

\label{section_Torelli}
Through this section, let $\Sigma=\Sigma_{g,1}$ be a compact connected surface
with non-empty connected boundary.

\subsection{The definition of $I \tskein{\Sigma}$}
\label{subsection_definition_I}
The aim of this subsection is to define $I \tskein{\Sigma}$.

\begin{lemm}
Let $\shuugou{c_1, c_2}$ be a pair of curves whose algebraic intersection number
is $0$. 
Then the set $\shuugou{\Lambda (c_1),\Lambda (c_2)}$ satisfies
the conditions equation (\ref{equation_jouken_bch}).
\end{lemm}

\begin{proof}
Let $\gamma_1$ and $\gamma_2$ be elements of $H_1$ such that 
$\pm [c_1] =\gamma_1$ and $\pm [c_2] =\gamma_2$.
We remark that
\begin{equation*}
\varpi_{\mathcal{P}2}(\zettaiti{\frac{1}{2}(\log (c_i))^2}) =\frac{1}{2} \gamma_i \cdot \gamma_i
\end{equation*}
for $i =1,2$.
Since $\mu(\gamma_1,\gamma_2)=0$ and Proposition \ref{prop_BCH_jouken},
we have the claim of the Lemma.
This finishes the proof.
\end{proof} 

By this lemma, we define $C(c_1,c_2) =\bch (\Lambda(c_1),
\Lambda (c_2),-\Lambda (c_1),-\Lambda (c_2))$
for a pair $\shuugou{c_1, c_2}$ of curves whose algebraic intersection number
is $0$. We denote  by $\mathcal{N}_{\mathrm{comm}}(\Sigma)$
the set of all $C(c_1,c_2) $. Since $[\Lambda (c_1),\Lambda (c_2)] 
\in F^3 \widehat{\tskein{\Sigma}}$ for
a pair $\shuugou{c_1,c_2}$ of curves whose algebraic intersection number is $0$,
we have $\mathcal{N}_{\mathrm{comm}} (\Sigma) \subset F^3
\widehat{\tskein{\Sigma}}$.

\begin{lemm}
\label{lemm_bounding_pair}
Let $\shuugou{c_1,c_2}$ be a pair of non-isotopic disjoint homologous curves
which is presented by $\zettaiti{r[a_1,b_1]\cdots[a_m,b_m]}$ and $\zettaiti{r}$, 
respectively,
for $r, a_1, b_1, \cdots, a_m, b_m \in \pi_1^+(\Sigma)$.
Then we have $\Lambda (c_1)-\Lambda (c_2) \in F^3 \widehat{
\tskein{\Sigma} } $.
\end{lemm}

\begin{proof}
It suffices to show
\begin{equation*}
\Lambda (c_1)- \Lambda (c_2) \in F^3 \widehat{\tskein{\Sigma}},
\end{equation*}
for simple closed curves $c_1$ and $c_2$
presented by $\zettaiti{r[a,b]}$ and $\zettaiti{r}$,
respectively.
We have
\begin{align*}
&\Lambda (c_1)-\Lambda (c_2) =\zettaiti{\psi( (r-1)^2-(r[a,b]-1)^2)} \\
&=\zettaiti{\psi( \frac{1}{2}((r-r[a,b])(r+r[a,b]-2 )-(r+r[a,b]-2 )(r-r[a,b])))} 
\mod F^3 \widehat{\tskein{\Sigma}}.
\end{align*}
Since $r-r[a,b]=r((a-1)(b-1)-(b-1)(a-1))a^{-1} b^{-1}$, 
we have $\Lambda (c_1)- \Lambda (c_2) \in F^3 \widehat{\tskein{\Sigma}}$.
This proves the lemma.
\end{proof}

\begin{cor}
If $c$ is a null-homologous simple closed curve,
then we have $\Lambda (c) \in F^4 \widehat{\tskein{\Sigma}}$.
\end{cor}

We denote by $\mathcal{N}_{\mathrm{bp}}(\Sigma)$ and
$\mathcal{N}_{\mathrm{sep}}(\Sigma)$
the set of all $\Lambda(c_1)-\Lambda(c_2)$ for 
 a pair $\shuugou{c_1,c_2}$
 of non-isotopic disjoint homologous curves,
i.e. which is a bounding pair (BP)
and the set of all $\Lambda (c)$
for a null homologous simple closed curve $c$,
i.e. for separating s.c.c. $c$, respectively.
By the above lemma and corollary,
we have $\mathcal{N}_{\mathrm{bp}}(\Sigma) \cup
\mathcal{N}_{\mathrm{sep}}(\Sigma) \subset F^3 \widehat{\tskein{\Sigma}}$.
We define
\begin{equation*}
I \tskein{\Sigma} \defeq \zettaiti{\mathcal{N}_{\mathrm{gen}}(\Sigma)}
\end{equation*}
where we denote $\mathcal{N}_{\mathrm{gen}}(\Sigma) \defeq
\mathcal{N}_{\mathrm{comm}} (\Sigma)
\cup \mathcal{N}_{\mathrm{bp}}(\Sigma) \cup
\mathcal{N}_{\mathrm{sep}}(\Sigma)$.

We denote by
\begin{align*}
&\mathcal{I}_{\mathrm{comm}}(\Sigma) \defeq
\shuugou{C_{c_1c_2} \defeq [t_{c_1}, t_{c_2}]|\mu(c_1,c_2)=0}, \\
&\mathcal{I}_{\mathrm{sep}}(\Sigma) \defeq
\shuugou{t_{c_1c_2} \defeq t_{c_1}{t_{c_2}}^{-1}|\shuugou{c_1,c_2}: \mathrm{BP}}, \\
&\mathcal{I}_{\mathrm{sep}}(\Sigma) \defeq
\shuugou{t_{c}|c: \mathrm{separating \ \ s.c.c.}}.
\end{align*}
The set $\mathcal{I}_{\mathrm{gen}}(\Sigma) \defeq
\mathcal{I}_{\mathrm{comm}} (\Sigma)
\cup \mathcal{I}_{\mathrm{bp}}(\Sigma) \cup
\mathcal{I}_{\mathrm{sep}}(\Sigma)$ is the generator set
of the infinite presentation of $\mathcal{I} (\Sigma)$ in \cite{Pu2008}.

\begin{df}
The map $\theta :\mathcal{N}_{\mathrm{gen}}(\Sigma)  \to \mathcal{I}_{\mathrm{gen}}(\Sigma)$ is
defined by the following.
\begin{itemize}
\item  For
a pair $\shuugou{c_1,c_2}$ of curves whose algebraic intersection number is $0$,
$C(c_1,c_2) \mapsto C_{c_1c_2}$.
\item For a pair $\shuugou{c_1,c_2}$
 of non-isotopic disjoint homologous curves, $\Lambda(c_1)-\Lambda(c_2)
 \mapsto t_{c_1c_2}$.
\item For
a null homologous simple closed curve $c$,
$\theta (\Lambda (c)) =t_c$.
\end{itemize}

\end{df}

In subsection \ref{subsection_injectivity_of_zeta}
Proposition \ref{prop_theta}, we prove $\theta$ is well-defined.

\subsection{Well-definedness of $\theta$}
\label{subsection_injectivity_of_zeta}

The aim of this subsection is to prove the following.

\begin{lemm}
\label{lemm_zeta_inv_inj}
The map $\theta:\mathcal{N}_{\mathrm{gen}}(\Sigma) \to
\mathcal{I}_{\mathrm{gen}}(\Sigma)$ induces
$\theta:I \tskein{\Sigma} \to \mathcal{I}(\Sigma)$.
\end{lemm}

\begin{prop}
\label{prop_theta}
The map $\theta$ is well-defined.
\end{prop}

\begin{proof}
By Theorem \ref{thm_main_Dehn}, we have the followings.

\begin{itemize}
\item  For
a pair $\shuugou{c_1,c_2}$ of curves whose algebraic intersection number is $0$,
$\exp(\sigma(C(c_1,c_2)))(\cdot)= C_{c_1c_2}(\cdot)\in \Aut (\widehat{\tskein{\Sigma,
J^-,J^+}})$ for any finite sets $J^-,J^+ \subset \partial{\Sigma}$.
\item For a pair $\shuugou{c_1,c_2}$
 of non-isotopic disjoint homologous curves, 
$\exp(\sigma(\Lambda(c_1)-\Lambda(c_2))) (\cdot)= t_{c_1c_2}(\cdot)\in \Aut (\widehat{\tskein{\Sigma,J^-,J^+}})$ for any finite sets $J^-,J^+ 
\subset \partial{\Sigma}$.
\item For
a null homologous simple closed curve $c$,
$\exp(\sigma(\Lambda(c)))(\cdot) =t_c(\cdot) \in \Aut (\widehat{\tskein{\Sigma,J^-,J^+}})$
 for any finite set $J^-.J^+ \in \partial{\Sigma}$.
\end{itemize}

By Proposition \ref{prop_MCG_tskein_faithful}, $\theta$ is well-defined.
This finishes the proof.
\end{proof}

By Theorem \ref{thm_main_Dehn}, we have the following.

\begin{lemm}
\label{lemm_theta}
For $x_1,x_2, \cdots, x_j  \in \shuugou{k x|k \in \shuugou{\pm1},
x \in \mathcal{N}_{\mathrm{gen}}(\Sigma)}$,
 we have 
\begin{equation*}
\prod_{i=1}^j(\theta ( x_i)) (\cdot) =
\exp(\bch( x_1,  x_2, \cdots, x_j))(\cdot)  \in \Aut (\widehat{\tskein{\Sigma,J^-,J^+}})
\end{equation*} 
for any finite sets $J^-,J^+ \subset \partial \Sigma$.
\end{lemm}

\begin{proof}[Proof of Lemma \ref{lemm_zeta_inv_inj}]
For $x_1,x_2, \cdots, x_j ,
y_1,y_2, \cdots,y_k \in \shuugou{k' x|k' \in \shuugou{\pm1},
x \in \mathcal{N}_{\mathrm{gem}} (\Sigma)}$,
we assume \begin{equation*}
\bch(x_1, \cdots, x_j) = \bch (y_1, \cdots, y_k). 
\end{equation*}
By Lemma \ref{lemm_theta}, we have
\begin{equation*}
\prod_{i=1}^j\theta (x_i) (\cdot) =\prod_{i=1}^k\theta (y_i) (\cdot) 
\in \Aut (\widehat{\tskein{\Sigma,J^-,J^+}})
\end{equation*}
for any finite sets $J^-,J^+ \subset \partial \Sigma$.
By Proposition \ref{prop_MCG_tskein_faithful}, we have
\begin{equation*}
\prod_{i=1}^j\theta (x_i) =\prod_{i=1}^k \theta (y_i).
\end{equation*}
This finishes the proof.
\end{proof}

\subsection{Well-definedness of $\zeta$}
\label{subsection_well_defined_zeta}
In this section, we prove $\theta$ is injective.
In order to check Putman's relation \cite{Pu2008},
we need the following.

\begin{lemm}
\label{lemm_relation_key}
Let $\Sigma$ be a compact surface with non-empty boundary.
We choose an element $x \in \widehat{\tskein{\Sigma}}$ satisfying
$\sigma (x)(\widehat{\tskein{\Sigma,J^-,J^+}}) =
\shuugou{0}$ for any finite set $J^-,J^+ \subset \partial \Sigma$.
Then we have followings.
\begin{enumerate}
\item For any embedding $e' :\Sigma \to \tilde{\Sigma}$ and any element $\xi \in 
\mathcal{M}(\tilde{\Sigma})$, we have $\xi (e'(x)) =e'(x) \in 
\widehat{\tskein{\tilde{\Sigma}}}$ where 
we also denote by $e'$ the homomorphism $\widehat{\tskein{\Sigma}}
\to \widehat{\tskein{\tilde{\Sigma}}}$ induced by $e'$.
\item For any embedding $e'':\Sigma \times I  \to I^3 \to \Sigma \times I$,
we have $x =e''(x) \in \Q [[A+1]] [\emptyset] \subset \widehat{\tskein{\Sigma}}$
where we also denote by $e''$ the homomorphism $\widehat{\tskein{\Sigma}}
\to \widehat{\tskein{\Sigma}}$ induced by $e''$.
\end{enumerate}
\end{lemm}

\begin{proof}[Proof of (1)]
Since $\mathcal{M} (\tilde{\Sigma})$ is generated by Dehn twists,
it suffices to check $t_c (x) =x $ for  any simple closed curve $c$.
By assumption of $x$ and Theorem \ref{thm_main_Dehn}, we have
\begin{align*}
t_c(x) =\exp(\sigma(\Lambda(c))(x) =x.
\end{align*}
This finishes the proof.
\end{proof}

\begin{proof}[Proof of (2)]
It suffices to show there exists an embedding
$e'':\Sigma \times I  \to I^3 \to \Sigma \times I$
such that $x = e''(x)$ for any $x \in \tskein{\Sigma}$.
We choose some  compact connected surfaces $\tilde{\Sigma}$
 with non-empty connected boundary and
an embedding $e''': \tilde{\Sigma} \times I \to \Sigma \times I$
such that there
exists submanifolds $\Sigma'$, $\Sigma'' \subset \tilde{\Sigma}$
satisfying the following.

\begin{itemize}
\item The submanifolds $\Sigma' $ and $ \Sigma''$  are disjoint.
\item There are diffeomorphisms $\chi' : \Sigma \to \Sigma'$ and
$\chi'' : \Sigma \to \Sigma''$.
\item There is an embedding $e_{I^3}: I^3 \to \Sigma \times I$ satisfying
\begin{equation*}
e'''\circ (\chi'' \times \id_I) (\Sigma \times I) \subset e''(I^3).
\end{equation*}
\item The embedding $e''' \circ (\chi' \times \id_I)$ induces the identity map of
$\widehat{\tskein{\Sigma}}$.
\item There exists a diffeomorphism $s :\tilde{\Sigma} \to \tilde{\Sigma}$
such that $s \circ \chi' = \chi''$.
\end{itemize}

\begin{figure}
\begin{picture}(300,360)
\put(0,190){\includegraphics[width=310pt]{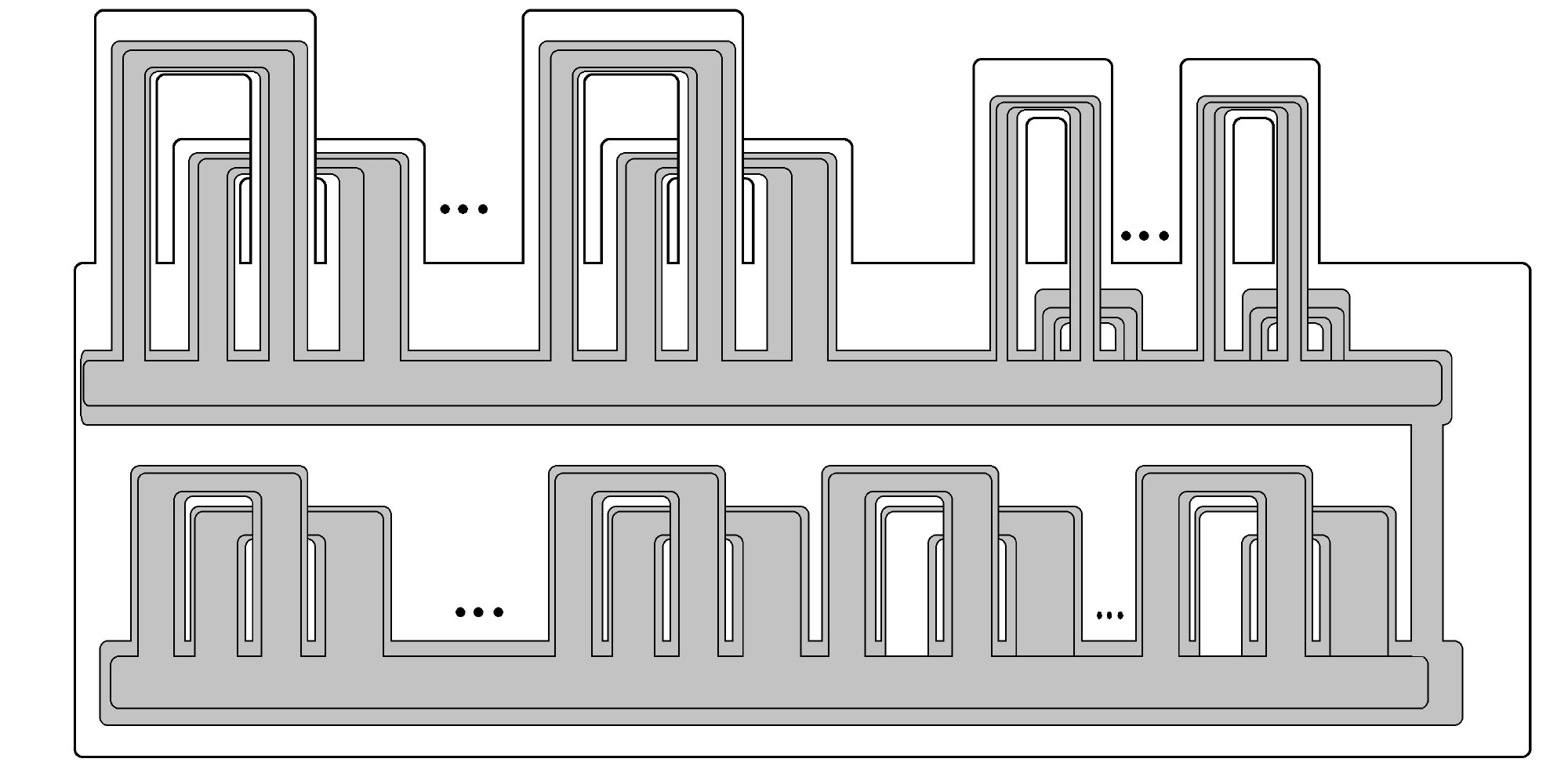}}
\put(0,10){\includegraphics[width=310pt]{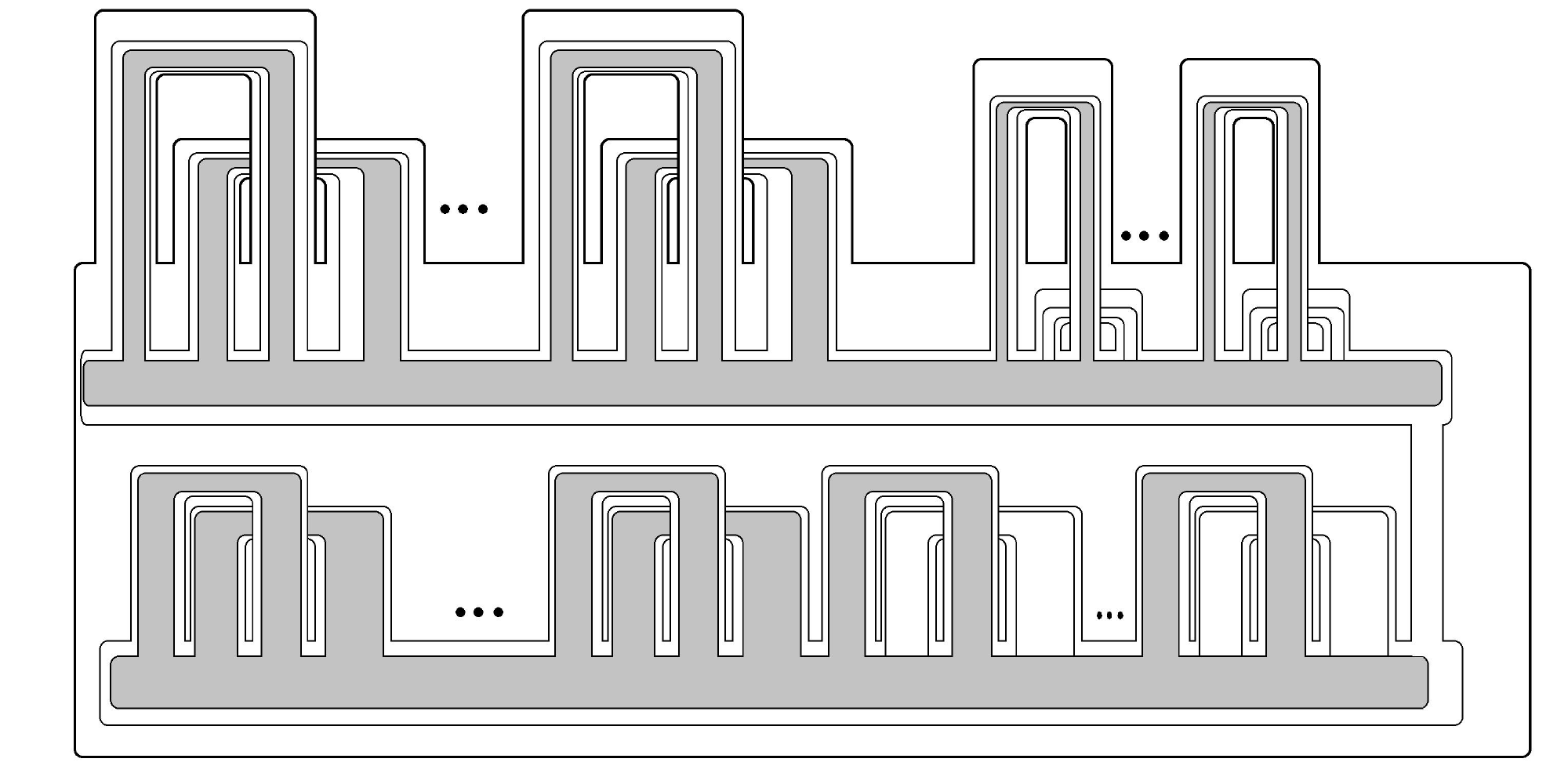}}
\put(120,205){$\tilde{\Sigma}$}
\put(270,100){{$\Sigma'$}}
\put(90,50){$\Sigma''$}
\end{picture}
\caption{$e''': \tilde{\Sigma}\times I \to \Sigma \times I$}
\label{figure_embedding_Sigma_tilde}
\end{figure}
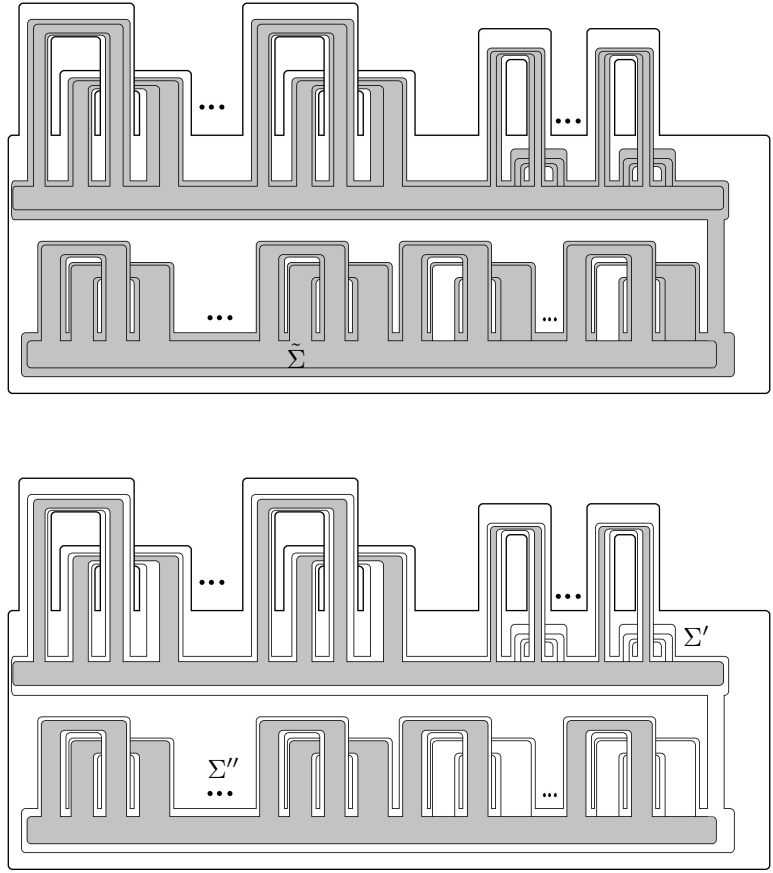

\begin{figure}
\begin{picture}(300,330)
\put(0,180){\includegraphics[width=310pt]{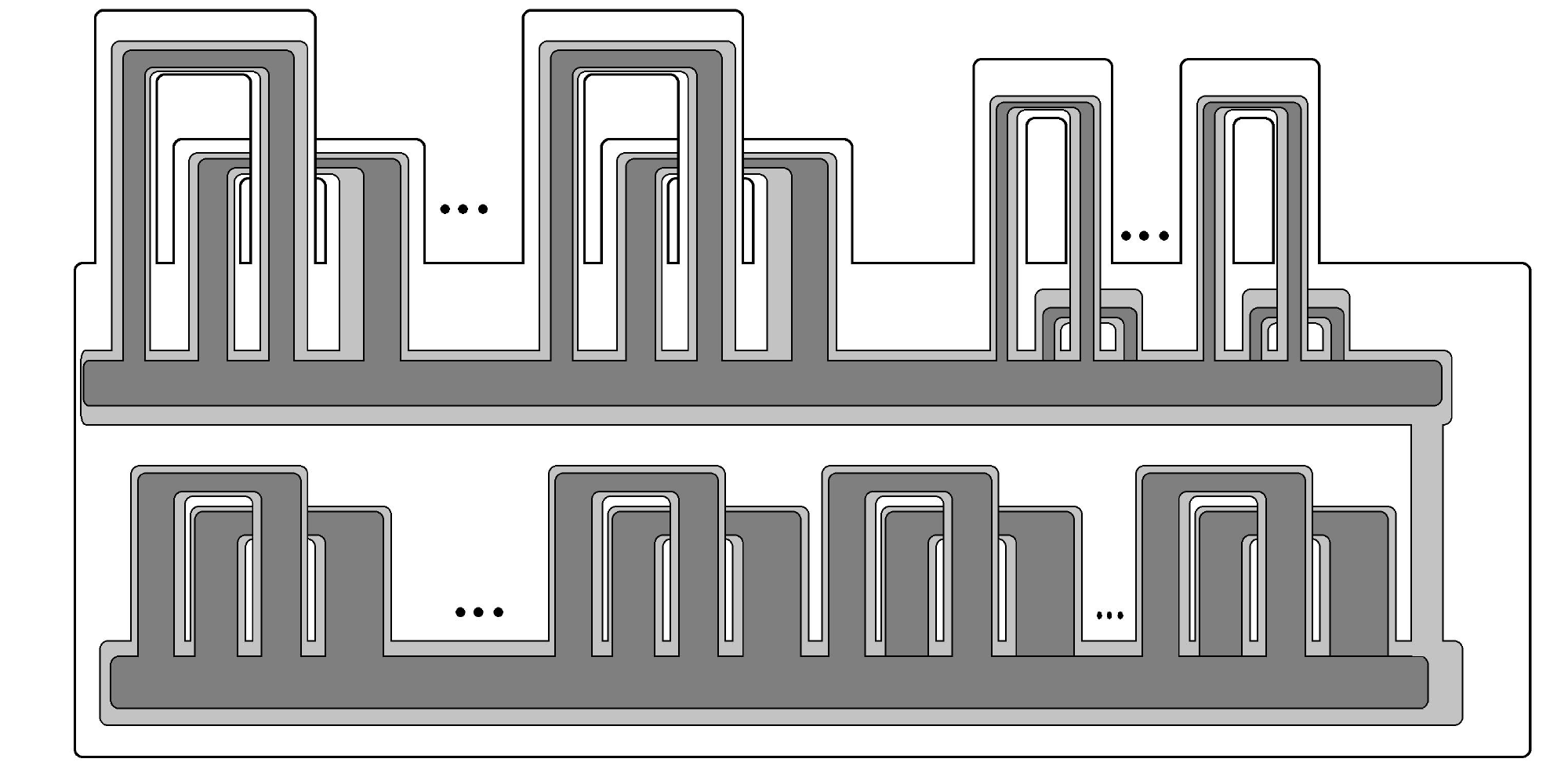}}
\put(0,0){\includegraphics[width=310pt]{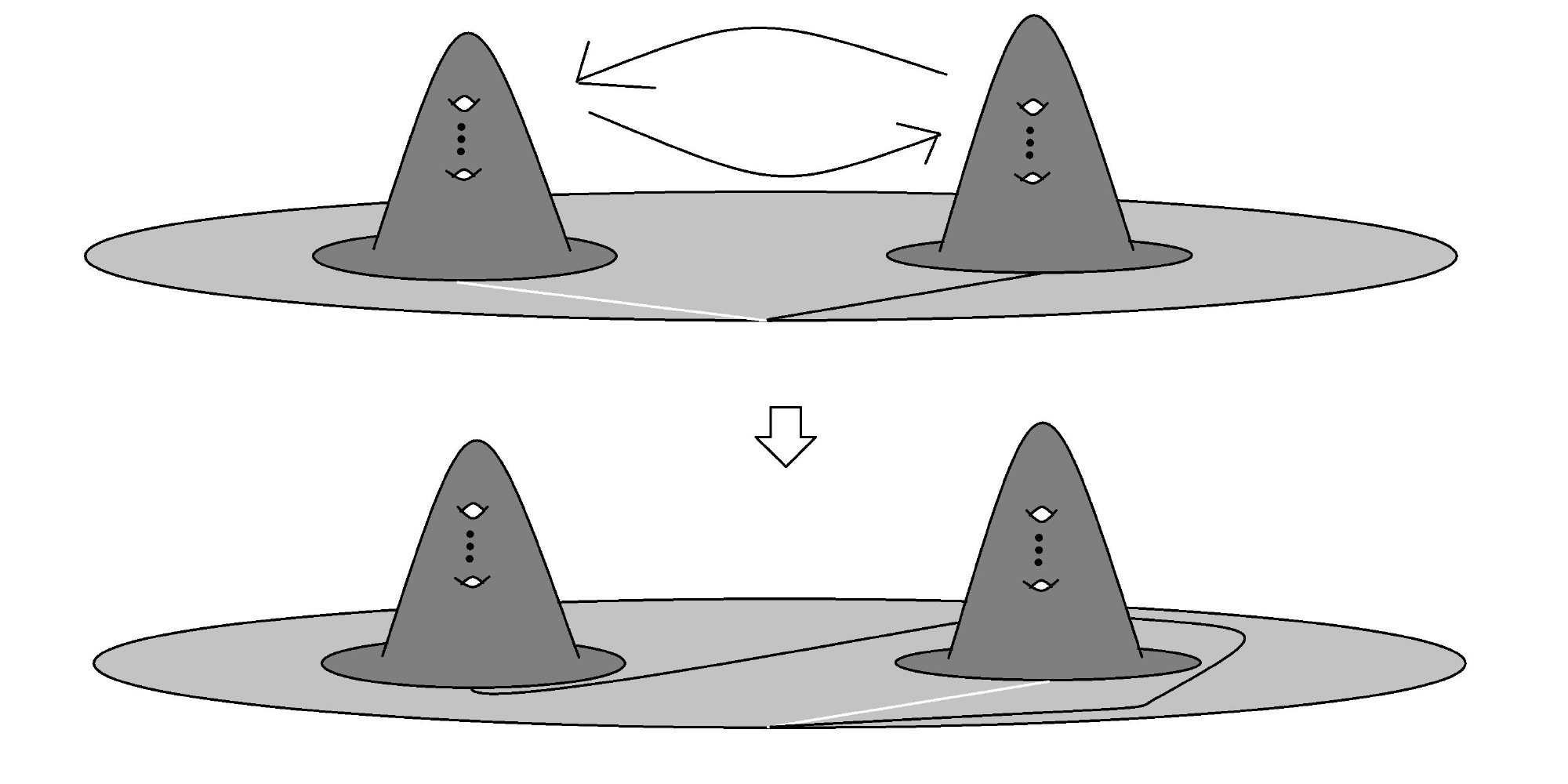}}
\put(163,65){{$s$}}
\end{picture}
\caption{$s : \tilde{\Sigma} \to \tilde{\Sigma}$}
\label{figure_embedding_Sigma_tilde_h}
\end{figure}

For example, see Figure \ref{figure_embedding_Sigma_tilde}
and Figure \ref{figure_embedding_Sigma_tilde_h}.
By (1), we have $\chi'(x) =\chi''(x)$
for any $x \in \tskein{\Sigma}$.
Hence, we obtain $x =e''' \circ (\chi' \times \id_I) (x) =e''' \circ 
(\chi'' \times \id_I)(x)$. The embedding $e''' \circ 
(\chi'' \times \id_I)$ satisfies the claim $e''$.
This finishes the proof.

\end{proof}

To check the relations (F.1), (F.2), ... (F.7) and (F.8)
in \cite{Pu2008}, we need the following.
We obtain the lemma by a straightforward calculation
and definition of the BCH series.

\begin{lemm}
\begin{enumerate}
\item When  $\shuugou{c_1,c_2}$ is a pair
 of non-isotopic disjoint homologous curves, 
we have
\begin{equation*}
\Lambda(c_1)-\Lambda(c_2) =-(\Lambda(c_2)-\Lambda(c_1)).
\end{equation*}
\item When  $\shuugou{c_1,c_2}$ is a pair
 of curves whose algebraic intersection number is $0$, we have
\begin{equation*}
\bch (\bch(\Lambda(c_1),\Lambda(c_2),
-\Lambda(c_1),-\Lambda(c_2)),\bch(\Lambda(c_2),\Lambda(c_1),
-\Lambda(c_2),-\Lambda(c_1)))=0.
\end{equation*}
\item When  $\shuugou{c_1,c_2,c_3}$ is a triple
of non-isotopic disjoint homologous curves, we have
\begin{equation*}
\bch(\Lambda(c_1)-\Lambda(c_2), \Lambda(c_2)-\Lambda(c_3)) 
= \Lambda(c_1)-\Lambda(c_3).
\end{equation*}
\item If $\shuugou{c_1,c_2}$ is a pair of
non-isotopic disjoint homologous curves such that $c_1$ and $c_2$
are separating curves, we have
\begin{equation*}
\Lambda(c_1)-\Lambda(c_2) = \bch (\Lambda(c_1),-\Lambda(c_2)).
\end{equation*}
\item If $\shuugou{c_1,c_2}$ is a pair 
of non-isotopic disjoint homologous curves and
$\shuugou{c_3,c_2}$ is a pair
of curves whose algebraic intersection number is $0$ such that $c_1$ and $c_3$ are disjoint,
we have
\begin{equation*}
\Lambda(t_{c_3}(c_2)) -\Lambda(c_{1}) 
=\bch (\bch(\Lambda(c_3),\Lambda(c_2),-\Lambda(c_3),
-\Lambda(c_2)),\Lambda(c_2)-\Lambda(c_1)).
\end{equation*}
\item  For
a pair $\shuugou{c_1,c_2}$ of curves whose algebraic intersection number is $0$
and $x \in \mathcal{N}_{\mathrm{gen}}(\Sigma)$, we have
\begin{equation*}
\bch(x,C(c_1,c_2),-x) =C(\theta(x)(c_1),\theta(x)(c_2)).
\end{equation*}
\item For a pair $\shuugou{c_1,c_2}$
 of non-isotopic disjoint homologous curves
and $x \in \mathcal{N}_{\mathrm{gen}}(\Sigma)$, we have
\begin{equation*}
\bch(x,\Lambda(c_1)-\Lambda (c_2),-x) =\Lambda (\theta(x)(c_1))
-\Lambda(\theta(x)(c_2)).
\end{equation*}
\item For
a null homologous simple closed curve $c$
and $x \in \mathcal{N}_{\mathrm{gen}}(\Sigma)$, we have
\begin{equation*}
\bch(x,\Lambda(c),-x) =\Lambda (\theta(x)(c))
\end{equation*}
\end{enumerate}
\end{lemm}

In order to check lantern relation
in \cite{Pu2008}, we need the following.

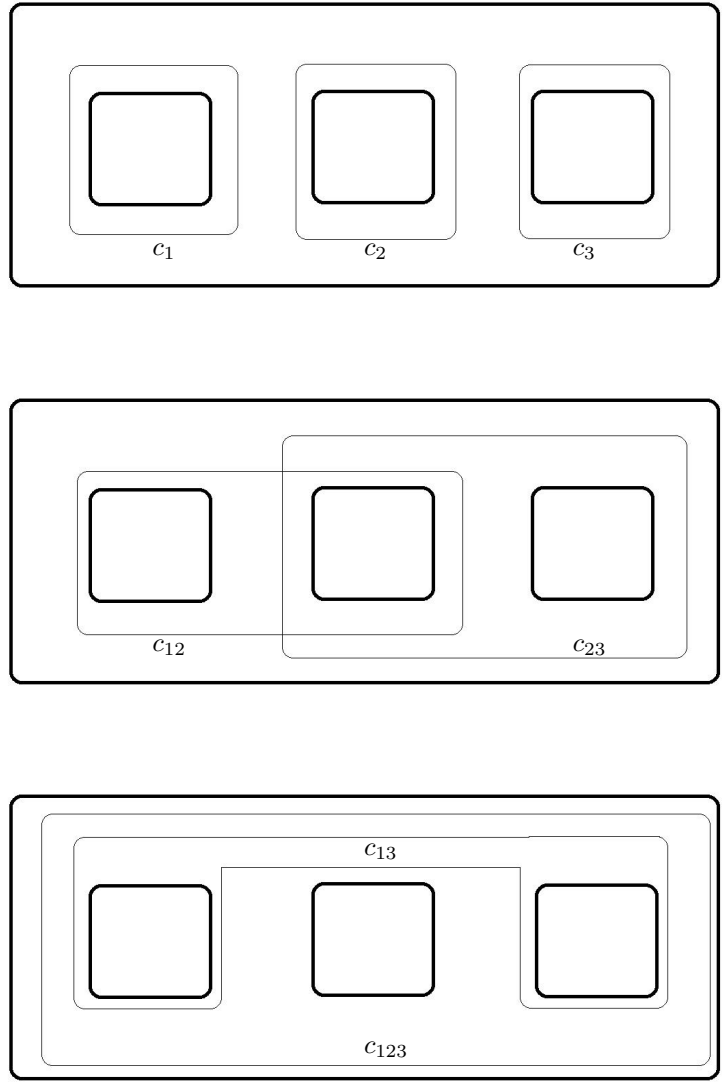
\begin{figure}
\begin{picture}(300,450)
\put(0,300){\includegraphics[width=310pt]{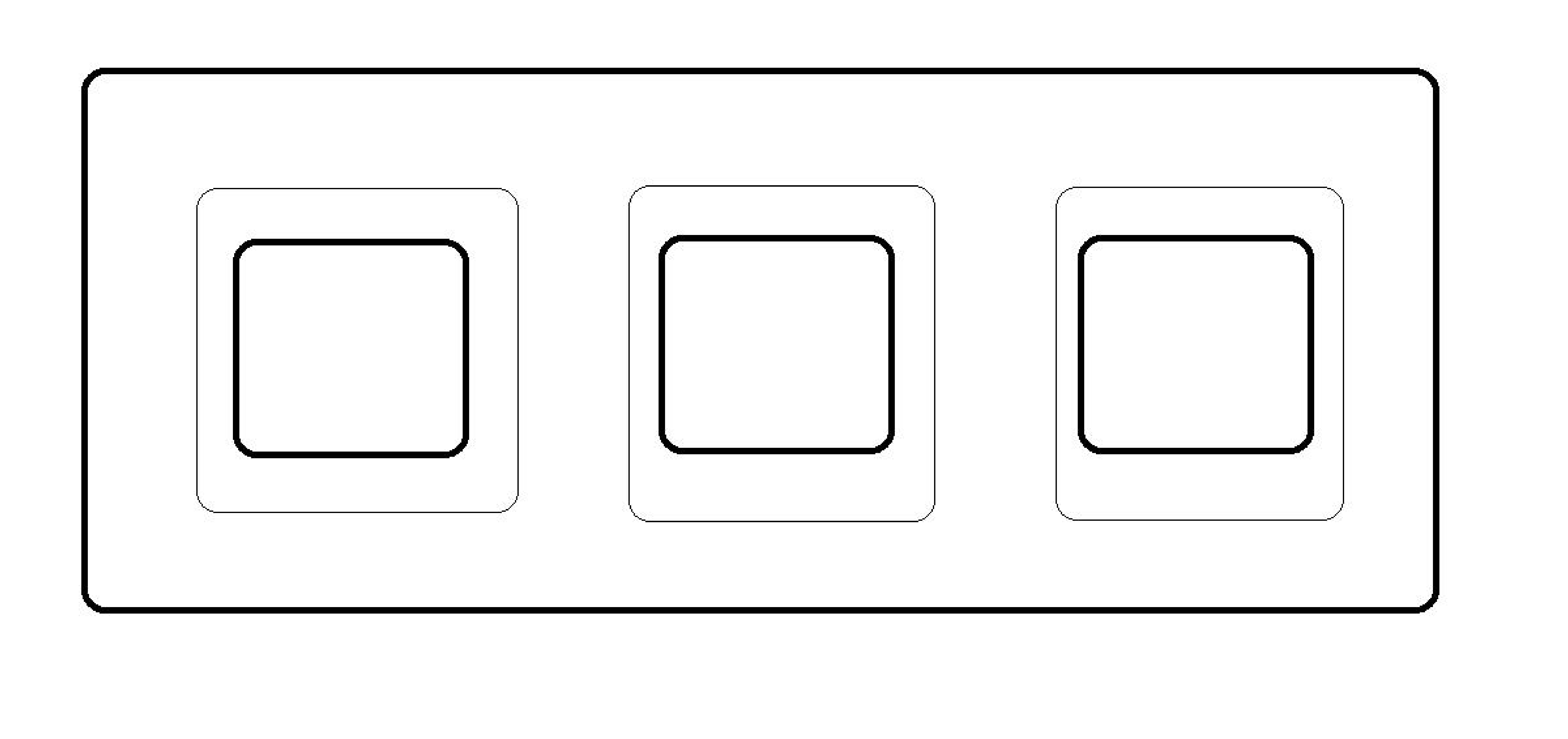}}
\put(0,150){\includegraphics[width=310pt]{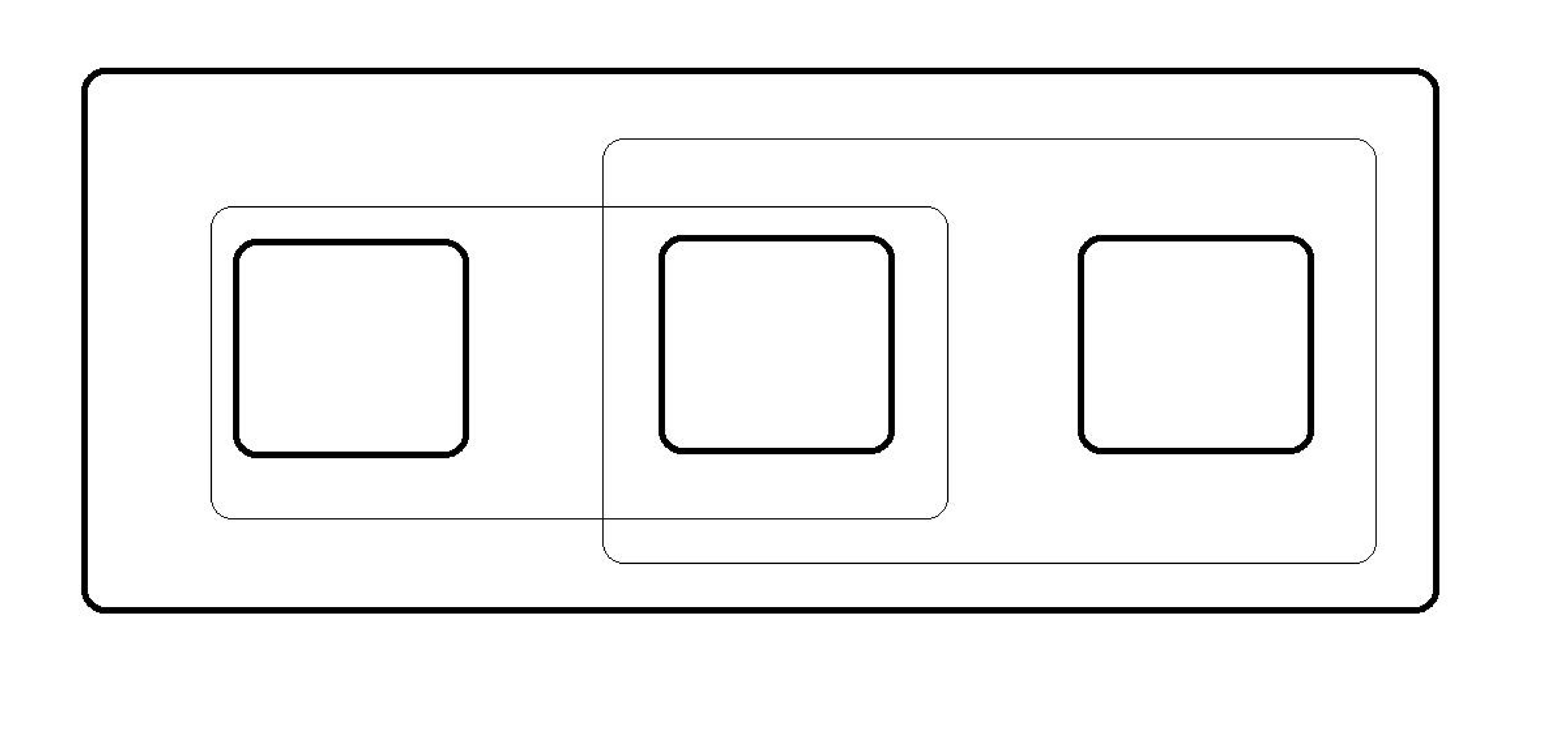}}
\put(0,0){\includegraphics[width=310pt]{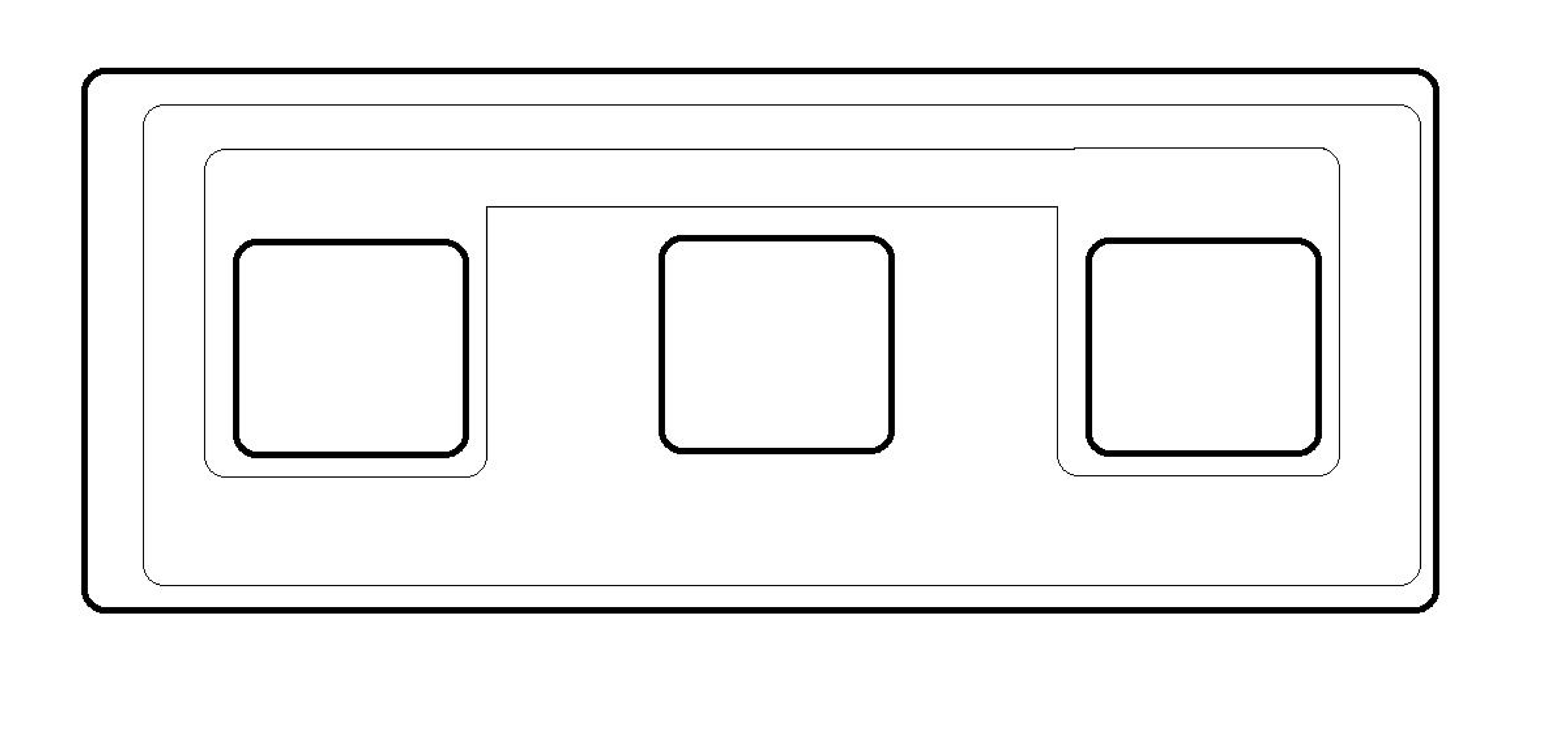}}
\put(70,337){$c_1$}
\put(150,337){$c_2$}
\put(229,337){$c_3$}
\put(70,187){$c_{12}$}
\put(229,187){$c_{23}$}
\put(150,35){$c_{123}$}
\put(150,110){$c_{13}$}
\end{picture}
\caption{The simple closed curves in $\Sigma_{04}$}
\label{figure_Sigma_0_4}
\end{figure}

\begin{lemm}
\label{lemm_lantern_relation}
Let $\Sigma_{0,4}$ be a compact connected oriented surface
of genus $0$ and $4$ boundary components and
$c_1,c_2,c_3,c_{12},c_{23},c_{13}$ and $c_{123}$
the simple closed curves as Figure \ref{figure_Sigma_0_4}.
We have $\bch(\Lambda(c_{123},
-\Lambda(c_{12}),-\Lambda(c_{23}),-\Lambda(c_{13}),
\Lambda(c_1),\Lambda(c_2),\Lambda(c_3)))=0\in
\widehat{\tskein{\Sigma_{0,4}}}$.
\end{lemm}

\begin{proof}
Let $e$ be an embedding $\Sigma_{0,4} \times I \to I^3$ as 
Figure \ref{figure_Sigma_0_4}.
Since 
\begin{align*}
&\exp (\bch(\Lambda (c_{123},-\Lambda(c_{12}),
-\Lambda(c_{23}),-\Lambda(c_{13}),\Lambda(c_1),\Lambda(c_2),
\Lambda(c_3))))(\cdot)
= \\
&t_{c_{123}}{t_{c_{12}}}^{-1}{t_{c_{23}}}^{-1}{t_{c_{31}}}^{-1}t_{c_{1}}t_{c_{2}}t_{c_3}(\cdot) = \id(\cdot)
\in \Aut (\widehat{\tskein{\Sigma_{0.4},J^-,J^+}})
\end{align*}
for all finite sets $J^-,J^+ \subset \partial \Sigma$, we have
\begin{equation*}
\sigma(\bch(\Lambda (c_{123},-\Lambda(c_{12}),
-\Lambda(c_{23}),-\Lambda(c_{13}),\Lambda(c_1),\Lambda(c_2),
\Lambda(c_3))))
(\widehat{\tskein{\Sigma_{0,4},J^-,J^+}}) = \shuugou{0}.
\end{equation*}
By Lemma \ref{lemm_relation_key},
it suffices to show $e(\bch(\Lambda (c_{123},-\Lambda(c_{12}),
-\Lambda(c_{23}),-\Lambda(c_{13}),\Lambda(c_1),\Lambda(c_2),
\Lambda(c_3))))
=0$. We remark $e([\widehat{\tskein{\Sigma_{0,4}}}, \widehat{\tskein{\Sigma_{0,4}}}])
=\shuugou{0}$. Using Corollary  \ref{cor_def_Lambda_trivial}, we have
\begin{align*}
&e(\bch(\Lambda (c_{123},-\Lambda(c_{12}),
-\Lambda(c_{23}),-\Lambda(c_{13}),\Lambda(c_1),\Lambda(c_2),
\Lambda(c_3))))=0. \\
\end{align*}
This finishes the proof.

\end{proof}

We will check the crossed lantern relation in \cite{Pu2008}.
Let $\Sigma_{1,2}$ be a connected compact surface
of genus $g=1$ with two boundary components
and $a$, $b$ and $v$ elements
$H_1(\Sigma_{1,2},\Q)$ represented by
$a \defeq [c'_a], b \defeq [c_b], v \defeq [c_v] \in H_1 (\Sigma_{1,2},\Q)$.
Here we denote simple closed curves in $\Sigma_{1,2}$
as in Figure \ref{figure_CL}.

In order to check the crossed lantern relation, 
we need the following.

\begin{figure}
\begin{picture}(300,140)
\put(0,-160){\includegraphics[width=340pt]{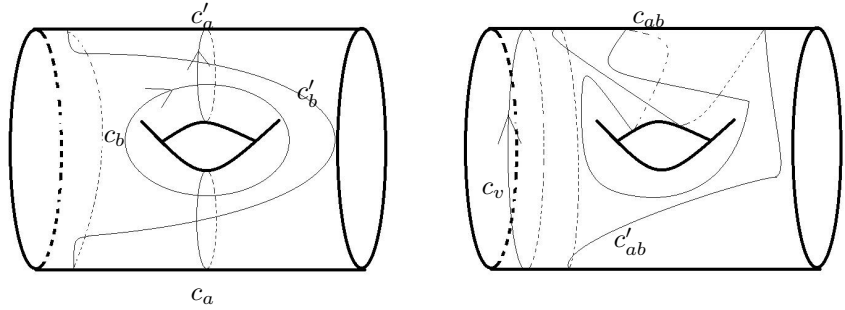}}
\put(47,70){$c_b$}
\put(120,86){$c'_b$}
\put(80,115){$c'_a$}
\put(80,10){$c_a$}
\put(190,50){$c_v$}
\put(240,30){$c'_{ab}$}
\put(247,115){$c_{ab}$}
\end{picture}
\caption{$\Sigma_{1,2}$}

\label{figure_CL}
\end{figure}

\begin{lemm}
\begin{enumerate}
\item The set $ \shuugou{\pm(\Lambda(c_a)-\Lambda(c'_a))
,\pm (\Lambda(c_b)-\Lambda(c'_b)), \pm (\Lambda(c_{ab})-\Lambda(c'_{ab}))}$
satisfies the conditions equation (\ref{equation_jouken_bch}).
\item We have $\bch (\Lambda(c_b)-\Lambda(c'_b),-\Lambda(c_a)
+\Lambda(c'_a),-\Lambda(c_{ab})+\Lambda(c'_{ab}))=0$.
\end{enumerate}
\end{lemm}

\begin{proof}[Proof of (1)]
We have 
\begin{align*}
&\varpi_{\mathcal{P}2} (\zettaiti{\frac{1}{2} (\log (c_a)^2-(\log (c'_a))^2)})
= \frac{1}{2}((a-v)\cdot(a-v)-a\cdot a)=-\frac{1}{2}v \cdot (2a-v) \\
& \varpi_{\mathcal{P}2} (\zettaiti{\frac{1}{2} (\log (c_b)^2-(\log (c'_b))^2)})
= \frac{1}{2}(b\cdot b-(b-v)\cdot(b-v))
=\frac{1}{2}v \cdot (2b-v) \\
& \varpi_{\mathcal{P}2} (\zettaiti{\frac{1}{2} (\log (c_{ab})^2-(\log (c'_{ab}))^2)})
= \frac{1}{2}((a+b)\cdot (a+b)-(a+b-v)\cdot(a+b-v)) \\
&=\frac{1}{2}v \cdot (2a+2b-v) \\
\end{align*}
By Proposition  \ref{prop_BCH_jouken}, this finishes the proof.
\end{proof}

\begin{proof}[Proof of (2)]
Let $e$ be an embedding $\Sigma_{0,4} \times I \to I^3$ as 
Figure \ref{figure_CL}.
Since 
\begin{align*}
&\exp (\sigma(\bch (\Lambda(c_b)-\Lambda(c'_b),-\Lambda(c_a)
+\Lambda(c'_a),-\Lambda(c_{ab})+\Lambda(c'_{ab}))))(\cdot) \\
&=t_{c_b c'_b} \circ t_{c'_a c_a}\circ {t_{c_{ab} c'_{ab}}}^{-1}(\cdot) = \id(\cdot)
\in \Aut (\widehat{\tskein{\Sigma_{1,2},J^-,J^+}})
\end{align*}
for all finite sets $J^-,J^+ \subset \partial \Sigma$, we have
\begin{equation*}
\sigma(\bch (\Lambda(c_b)-\Lambda(c'_b),-\Lambda(c_a)
+\Lambda(c'_a),-\Lambda(c_{ab})+\Lambda(c'_{ab})))
(\widehat{\tskein{\Sigma_{1,2},J}}) = \shuugou{0}.
\end{equation*}
By Lemma \ref{lemm_relation_key},
it suffices to show $e(\bch (\Lambda(c_b)-\Lambda(c'_b),-\Lambda(c_a)
+\Lambda(c'_a),-\Lambda(c_{ab})+\Lambda(c'_{ab})))=0$.
Since $e(x (\Lambda(c_i)-\Lambda(c'_i))) =e((\Lambda(c_i)-\Lambda(c_i)) x) =0 $ for $i \in \shuugou{a, b,ab}$ and
any $x \in \widehat{\tskein{\Sigma_{1,2}}}$, we have 
$e(\bch (\Lambda(c_b)-\Lambda(c'_b),-\Lambda(c_a)
+\Lambda(c'_a),-\Lambda(c_{ab})+\Lambda(c'_{ab})))=0$.
This finishes the proof.
\end{proof}

Since the Witt-Hall relation and the commutator shuffle relation
are relations between pushing maps.
it is enough  to show the following lemma
to check Witt-Hall relation and commutator shuffle relation in 
\cite{Pu2008}.

\begin{lemm}
Let $\Sigma'$ be a compact surface, $D$ a closed disk in $\Sigma'$
and $\Sigma''$ the surface $\Sigma' \backslash D$.
We denote by $\mathcal{S}_{\mathrm{push}}(\Sigma'', \partial D)$
the set of all pair $\shuugou{c_1,c_2}$ of simple closed curves
such that
$t_{c_1} \gyaku{t_{c_2}}$ is a pushing map of $D$
and that $[c_1] =[c_2] \in H_1 (\Sigma'').$

For $\shuugou{c_{11},c_{21}} \cdots \shuugou{c_{1k},c_{2k}} \in 
\mathcal{S}_{\mathrm{push}}(\Sigma'', \partial D)$
and $\epsilon_1 \cdots \epsilon_k \in \shuugou{\pm1}$,
if $ \prod_{i=1}^k (t_{c_{1i}c_{2i}})^{\epsilon_i} = \id \in \mathcal{I}(\Sigma'')$,
then we have
\begin{equation*}
\bch(\epsilon_1 (\Lambda(c_{11})-\Lambda(c_{21})), \cdots, 
\epsilon_k(\Lambda(c_{1k})-\Lambda(c_{2k}))) =0.
\end{equation*}
\end{lemm}


\begin{proof}
Since 
\begin{equation*}
\exp (\sigma(\bch(\epsilon_1 (\Lambda(c_{11})-\Lambda(c_{21})), \cdots, 
\epsilon_k(\Lambda(c_{1k})-\Lambda(c_{2k})))
))(\cdot)
=\prod_{i=1}^k (t_{c_{1i}c_{2i}})^{\epsilon_i}(\cdot) = \id(\cdot)
\in \Aut (\widehat{\tskein{\Sigma,J^-,J^+}})
\end{equation*}
for all finite set $J^-,J^+ \subset \partial \Sigma$, we have
\begin{equation*}
\sigma(\bch(\epsilon_1 (\Lambda(c_{11})-\Lambda(c_{21})), \cdots, 
\epsilon_k(\Lambda(c_{1k})-\Lambda(c_{2k}))))
(\widehat{\tskein{\Sigma,J^-,J^+}}) = \shuugou{0}.
\end{equation*}
We denote by $e'$ the embedding $\Sigma'' \to \Sigma'$.
We choose an embedding $e:\Sigma' \times I \to I^3$.
By Lemma \ref{lemm_relation_key},
it suffices to show  $e\circ e'(\bch(\epsilon_1 (\Lambda(c_{11})-\Lambda(c_{21})), \cdots, 
\epsilon_k(\Lambda(c_{1k})-\Lambda(c_{2k}))))=0$.
Since $e'(\Lambda(c_{1i})-\Lambda(c_{2i}))$ for $i \in \shuugou{1,2, \cdots,k}$,
we have $e'(\bch(\epsilon_1 (\Lambda(c_{11})-\Lambda(c_{21})), \cdots, 
\epsilon_k(\Lambda(c_{1k})-\Lambda(c_{2k}))))=0$.
This finishes the proof.
\end{proof}

By above lemmas, we have the main theorem.

\begin{thm}
\label{thm_main_embedding}
The group homomorphism $\zeta : \mathcal{I} (\Sigma)
\to I \tskein{\Sigma}$ defined by $\zeta (t_c) =\Lambda(c)$ for
$t_c \in \mathcal{I}_{\mathrm{sep}} (\Sigma)$,
$\zeta (t_{c_1 c_2}) = \Lambda(c_1)-\Lambda(c_2)$
for $t_{c_1,c_2} \in \mathcal{I}_{\mathrm{bp}} (\Sigma)$
and $\zeta (C_{c_1c_2}) = C(c_1, c_2)$
for $C_{c_1 c_2} \in \mathcal{I}_{\mathrm{com}} (\Sigma)$
is well-defined.
\end{thm}

\begin{proof}
Let $\Gamma_{g,1}$ be the group defined in \cite{Pu2008} Definition 4.1
which is generated by $\Gamma_{g,1 \mathrm{BP}} \cup
\Gamma_{g,1 \mathrm{sep}} \cup \Gamma_{g,1 \mathrm{comm}}$
where
\begin{align*}
&\Gamma_{g,1 \mathrm{BP}} \defeq \shuugou{t'_{c_1,c_2}|(c_1,c_2):\mathrm{BP}} \\ 
&\Gamma_{g,1 \mathrm{sep}} \defeq \shuugou{t'_{c}|c:\mathrm{sep. c. c. c.}} \\
&\Gamma_{g,1 \mathrm{comm}} \defeq \shuugou{c'_{c_1,c_2}|\mu (c_1,c_2)=0}. \\
\end{align*}
The relations of $\Gamma_{g,1}$ are
(F.1), (F.2), $\cdots$,  (F.8),
the lantern relation, the crossed lantern relation,
the Witt-Hall relation and the commutator shuffle relation.
We set the group homomorphism $\zeta'':\kukakko{
\Gamma_{g,1 \mathrm{BP}} \cup
\Gamma_{g,1 \mathrm{sep}} \cup \Gamma_{g,1 \mathrm{comm}}}\to
I \tskein{\Sigma}$ by
$t'_{c_1,c_2} \in \Gamma_{g,1 \mathrm{BP}} \mapsto \Lambda(c_1)-\Lambda(c_2)$,
$t'_{c} \in \Gamma_{g,1 \mathrm{sep}} \mapsto \Lambda(c)$ and
$c'_{c_1,c_2} \in \Gamma_{g,1 \mathrm{comm}} \mapsto C(c_1,c_2)$.
Using the above lemmas, the group homomorphism $\zeta''$ induces 
$\zeta':\Gamma_{g,1} \to I \tskein{\Sigma}$.
By \cite{Pu2008} Theorem 4.2, we obtain
$\mathcal{I} (\Sigma) \simeq \Gamma_{g,1}$ by
$t'_{c_1,c_2} \mapsto t_{c_1,c_2}$,
$t'_{c} \mapsto t_c$ and
$c'_{c_1,c_2} \mapsto C_{c_1,c_2}$.
This proves the theorem.

\end{proof}

Since $\mathcal{I} \tskein{\Sigma}$ is generated by the set
$\shuugou{t_{c_1c_2}|c_1,c_2: \mathrm{BP}}$, 
we have the following.

\begin{cor}
\label{cor_main_embedding}
We have $I \tskein{\Sigma} =\zettaiti{\mathcal{N}_{bp}(\Sigma)}$.
Furthermore $\zeta$ can be defined by
$\zeta(t_{c_1c_2})=\Lambda(c_1)-\Lambda(c_2)$.
\end{cor} 

\subsection{Some filtrations of the Torelli group of $\Sigma_{g,1}$}
\label{subsection_some_filtrations}
Let $\Sigma$ be a compact connected oriented surface with nonempty boundary.
We denote by $\skein{\Sigma}$ the Kauffman bracket skein algebra of $\Sigma$,
which is the quotient of $\Q [A^{\pm 1}]$-free module with basis
the set of unoriented framed links modulo the relations defining
the Kauffman bracket.
For details, see \cite{TsujiCSAI} \cite{Tsujipurebraid} \cite{TsujiTorelli} \cite{Tsujihom3}.
We will define the $\Q$-module homomorphism 
$\psi'_{\mathcal{A} \mathcal{S}}:\tskein{\Sigma} \to 
\skein{\Sigma}$, by $L \mapsto  (-A)^{w(L)} \bar{L}$, $h \mapsto
-A^2 +A^{-2}$ and $ \exp (\rho h) \mapsto A^4$.
Here $\bar{L}$ is the unoriented link of $L$ and 
$w (L)$ is the self-linking number of $L$,
which is the sum of the signs of crossings of $L$.

\begin{prop}
The $\Q$-module homomorphism
$\psi'_{\mathcal{A} \mathcal{S}}$
is well-defined.
\end{prop}

\begin{proof}
Let $(L^+,L^-, L^0)$ be a Conway triple.
We obtain
\begin{align*}
& \psi'_\mathcal{AS} (L^+ -L^- -h L^0)
= (-A)^{w(L^+)} \bar{L^+}-(-A)^{w(L^-)} \bar{L^-}
-(-A^2+A^{-2}) (-A)^{w(L^0)} \bar{L^0} \\
& =(-A)^{w(L^0)} (-A \bar{L^+} -A^{-1} \bar{L^0} -
(-A^2+A^{-2})  \bar{L^0})=0. 
\end{align*}
Let $d_1$ and $d_0$ be tangle diagrams
which are identical except for a disk, where they are a positive
twist and a straight line.
Then we have

\begin{align*}
& \psi'_\mathcal{AS} (T(d_1)-\exp(\rho h)T(d_0))
= (-A)^{w(T(d_1))} \bar{T(d_1)}-(-A)^{w(T(d_0))}A^4 \bar{T(d_0)}\\
& =(-A)^{w(T(d_0))}((-A) \bar{T(d_1)}-A^4 \bar{T(d_0)}) =0. 
\end{align*}
Furthermore, we have
$\psi'_\mathcal{AS} ( (\mathrm{triv.knot})-\dfrac{2 \sinh \rho h}{h})
=(\mathrm{triv.knot})-\dfrac{A^4-A^{-4}}{ -A^2+A^{-2}}=0$.
This proves the proposition.
\end{proof}

We will prove $\psi'_\mathcal{AS} (F^n \tskein{\Sigma})
\subset F^n \skein{\Sigma}$
where the filtration $\filtn{F^n \skein{\Sigma}}$ 
of $\skein{\Sigma}$ is defined in \cite{Tsujipurebraid}.

\begin{thm}
\label{thm_filtration_AS}
We have 
$\psi'_\mathcal{AS} (F^n \tskein{\Sigma})
\subset F^n \skein{\Sigma}$.
\end{thm}

In order to prove the Theorem, we fix an embedding
$e :\coprod_{1 \leq j \leq n}
\Sigma_{0, b_j+1} \times I  \to \Sigma \times I
$ where $N= \sum_j b_j$.
For $\epsilon_1 , \cdots , \epsilon_N \in \shuugou{0,1}$,
let $T(\epsilon_1, \cdots, \epsilon_N)$
 be the oriented framed link 
\begin{equation*}
e( \zettaiti{\gamma_1^{\epsilon_1} \cdots \gamma_{b_1}^{\epsilon_{b_1}}}
\sqcup 
\zettaiti{\gamma_1^{\epsilon_{b_1+1}} \cdots \gamma_{b_2}^{\epsilon_{b_1+b_2}}}
\sqcup \cdots \sqcup
\zettaiti{\gamma_1^{\epsilon_{b_1+\cdots+ b_{j-1}+1}}
 \cdots \gamma_{b_j}^{\epsilon_{b_1+ \cdots +b_j}}})
\end{equation*}
whose framing is defined by surfaces $\Sigma_{0,b_i}$.
Using the first equation in the proof of Lemma 5.9
in \cite{Tsujipurebraid} repeatedly,
 we have the following.

\begin{lemm}
We have $\sum_{\epsilon_1, \cdots, \epsilon_N \in \shuugou{0,1}}
(-1)^{N-\sum_i \epsilon_i} a_1^{\epsilon_1} \cdots
a_N^{\epsilon_N}\bar{T}(\epsilon_1, \cdots, \epsilon_N)
\in F^N \skein{\Sigma}$
for $a_1, \cdots a_N \in 1+ ((A+1))$.
\end{lemm}

Let $R$ be the commutative algebra over $\Q$
 generated by
$\shuugou{a_{i,j}|i,j \in \Z_{\geq 1}, i<j}$
and $\shuugou{r_i| i \in \Z_{\geq 1}}$
and $F^nR$ the ideal of $R$ generated by
$(a_{i_1,j_1}-1) \cdots (a_{i_{N'},j_{N'}}-1)
(r_{k_1}-1) \cdots (r_{k_{M'}}-1)$
for $2N+M \geq n$.

\begin{lemm}
\label{lemm_filtration_AS}
We have
\begin{equation*}
\sum_{1 \leq i_1 <i_2 < \cdots <i_k \leq n}
(-1)^{n-j}
(\prod_{1 \leq k <l \leq k} a_{i_k,i_l})
(\prod_{1 \leq m \leq k} r_{i_m})
\in F^n R.
\end{equation*}
\end{lemm}

\begin{proof}
 We will prove the lemma by induction on $n$.
If $n=1$, the equation is trivial.
We have
\begin{align*}
&\sum_{1 \leq i_1 <i_2 < \cdots <i_k \leq n}
(-1)^{n-j}
(\prod_{1 \leq k <l \leq k} a_{i_k,i_l})
(\prod_{1 \leq m \leq k} r_{i_m}) \\
&=\sum_{1 \leq i_1 <i_2 < \cdots <i_k \leq n-1}
(-1)^{n-j}
(\prod_{1 \leq k <l \leq k} a_{i_k,i_l})
(\prod_{1 \leq m \leq k} a_{i_m,n} r_{i_m})r_{n} \\
&-\sum_{1 \leq i_1 <i_2 < \cdots <i_k \leq n-1}
(-1)^{n-j}
(\prod_{1 \leq k <l \leq k} a_{i_k,i_l})
(\prod_{1 \leq m \leq k} r_{i_m}) .
\end{align*}
By inductive assumption, we have
\begin{equation*}
\sum_{1 \leq i_1 <i_2 < \cdots <i_k \leq n-1}
(-1)^{n-j}
(\prod_{1 \leq k <l \leq k} a_{i_k,i_l})
(\prod_{1 \leq m \leq k} a_{i_m,n-1} r_{i_m})r_{n} 
\in F^{n-1} R.
\end{equation*}
Since $a_{i,j} r_k-r_k  \in F^2 R$,
we obtain
\begin{align*}
&\sum_{1 \leq i_1 <i_2 < \cdots <i_k \leq n-1}
(-1)^{n-j}
(\prod_{1 \leq k <l \leq k} a_{i_k,i_l})
(\prod_{1 \leq m \leq k} a_{i_m,n} r_{i_m})r_{n} \\
&-\sum_{1 \leq i_1 <i_2 < \cdots <i_k \leq n-1}
(-1)^{n-j}
(\prod_{1 \leq k <l \leq k} a_{i_k,i_l})
(\prod_{1 \leq m \leq k} r_{i_m})  \mod F^n R\\
&=\sum_{1 \leq i_1 <i_2 < \cdots <i_k \leq n-1}
(-1)^{n-j}
(\prod_{1 \leq k <l \leq k} a_{i_k,i_l})
(\prod_{1 \leq m \leq k}  r_{i_m})r_{n} \\
&-\sum_{1 \leq i_1 <i_2 < \cdots <i_k \leq n-1}
(-1)^{n-j}
(\prod_{1 \leq k <l \leq k} a_{i_k,i_l})
(\prod_{1 \leq m \leq k} r_{i_m})  \mod F^n R\\
&=\sum_{1 \leq i_1 <i_2 < \cdots <i_k \leq n-1}
(-1)^{n-j}
(\prod_{1 \leq k <l \leq k} a_{i_k,i_l})
(\prod_{1 \leq m \leq k} r_{i_m})(r_n-1) =0 \mod F^n R.
\end{align*}

This proves the lemma.

\end{proof}

\begin{proof}[Proof of Theorem \ref{thm_filtration_AS}]
It suffices to show
\begin{equation*}
\psi'_\mathcal{AS} (
\sum_{\epsilon_1, \cdots, \epsilon_N \in \shuugou{0,1}}
(-1)^{N-\sum_i \epsilon_i} {T}(\epsilon_1, \cdots, \epsilon_N))
\in F^N \skein{\Sigma}.
\end{equation*}
By definition, we have
\begin{align*}
&\psi'_\mathcal{AS} (
\sum_{\epsilon_1, \cdots, \epsilon_N \in \shuugou{0,1}}
(-1)^{N-\sum_i \epsilon_i} {T}(\epsilon_1, \cdots, \epsilon_N)) \\
&=
\sum_{1 \leq i_1 <i_2 < \cdots <i_k \leq N}
(-1)^{N-j}
(\prod_{1 \leq k <l \leq k} a_{i_k,i_l})
T(\prod_{1 \leq m \leq k} r_{i_m})
\end{align*}
where $a_{i_k,i_l} =(-A)^{(\mathrm{linking \ number \ of \ } T_{i_k} \mathrm{\  and \ } 
T_{i_l})}$ and $T (r_1^{\epsilon_1} \cdots r_N^{\epsilon_N})= (-A)^{\sum_j \epsilon_j w(T_j)}
T(\epsilon_1, \cdots , \epsilon_N)$.
Here, when $\epsilon_1,  \cdots, \epsilon_{l-1}, \epsilon_{l+1}, \cdots, 
\epsilon_{n}=0,\epsilon_l=1$,
we denote $T_l \defeq  T(\epsilon_1 , \cdots , \epsilon_N)$.
Using Lemma \ref{lemm_filtration_AS} corresponding
$T (r_1^{\epsilon_1} \cdots r_N^{\epsilon_N})$
to $r_1^{\epsilon_1} \cdots r_N^{\epsilon_N}$,
we obtain
\begin{align*}
\sum_{1 \leq i_1 <i_2 < \cdots <i_k \leq n}
(-1)^{n-j}
(\prod_{1 \leq k <l \leq k} a_{i_k,i_l})
T(\prod_{1 \leq m \leq k} r_{i_m})
\in F^N \skein{\Sigma}.
\end{align*}
This proves the theorem.
\end{proof}

Let $\tzeroskein{\Sigma}$ be the submodule
generated by links whose homology class is $0$
in $\Sigma$.
We remark that, for a Conway triple
$(L^+, L^-, L^0)$, if a homology class of one of 
$(L^+, L^-,L^-)$ is $0$ in $\Sigma$,
the homology classes of $(L^+, L^-,L^0)$
are $0$. 

\begin{prop}
Let $L=l_1 \sqcup l_2 \sqcup  \cdots \sqcup l_N$
 be an oriented framed link in $\Sigma \times I$ with
$\sharp \pi_0 (L) =N$.
The injective map 
$\iota : \Sigma \to \Sigma \times I \backslash L,
x \mapsto (x,0)$
induces $\iota_* :
H_1 (\Sigma, \Z ) \to H_1 (\Sigma \times I \backslash L,
\Z )$. Let $z_i$ be the homology class of 
the right handed meridian of $l_i$ 
and $y_i$ the homology class of 
the longitude of $l_i$ for $i=1, \cdots, N$.
We remark $H_1 (\Sigma \times I \backslash L, \Z)
=\iota_* (H_1 (\Sigma , \Z)) \oplus \Z z_1 \oplus \cdots  \oplus \Z z_n$.
When $\sum_i y_i =\iota_* (A_0) \oplus a^1 z_1 \oplus \cdots \oplus a^n z_n$,
we define $w' (L) \defeq \sum_{j}a^j$.
Then we have $w(L) =w'(L)$.
\end{prop}

\begin{proof}
When an oriented link $L$ has no crossing,
we have $w(L)=w'(L)=0$.
So it is enough to show the following.
\begin{itemize}
\item Let $d_1$ and $d_2$ be link diagrams in 
$\Sigma$ which are identical except for a disk.
When they differ as shown in Figure \ref{figure_proof_w}
in the disk, then we have $w'(T(d_1))=w' (T(d_2))$.
\item Let $d_1$ and $d_2$ be link diagrams in
$\Sigma$ which are identical except for a disk.
When they are a positive twist and a straight line 
in the disk, then we have $w'(T(d_1))=w'(T(d_2))+1$.
\end{itemize}
We leave the details to the reader.

\begin{figure}
\begin{picture}(100,50)
\put(0,0){\includegraphics[width=100pt]{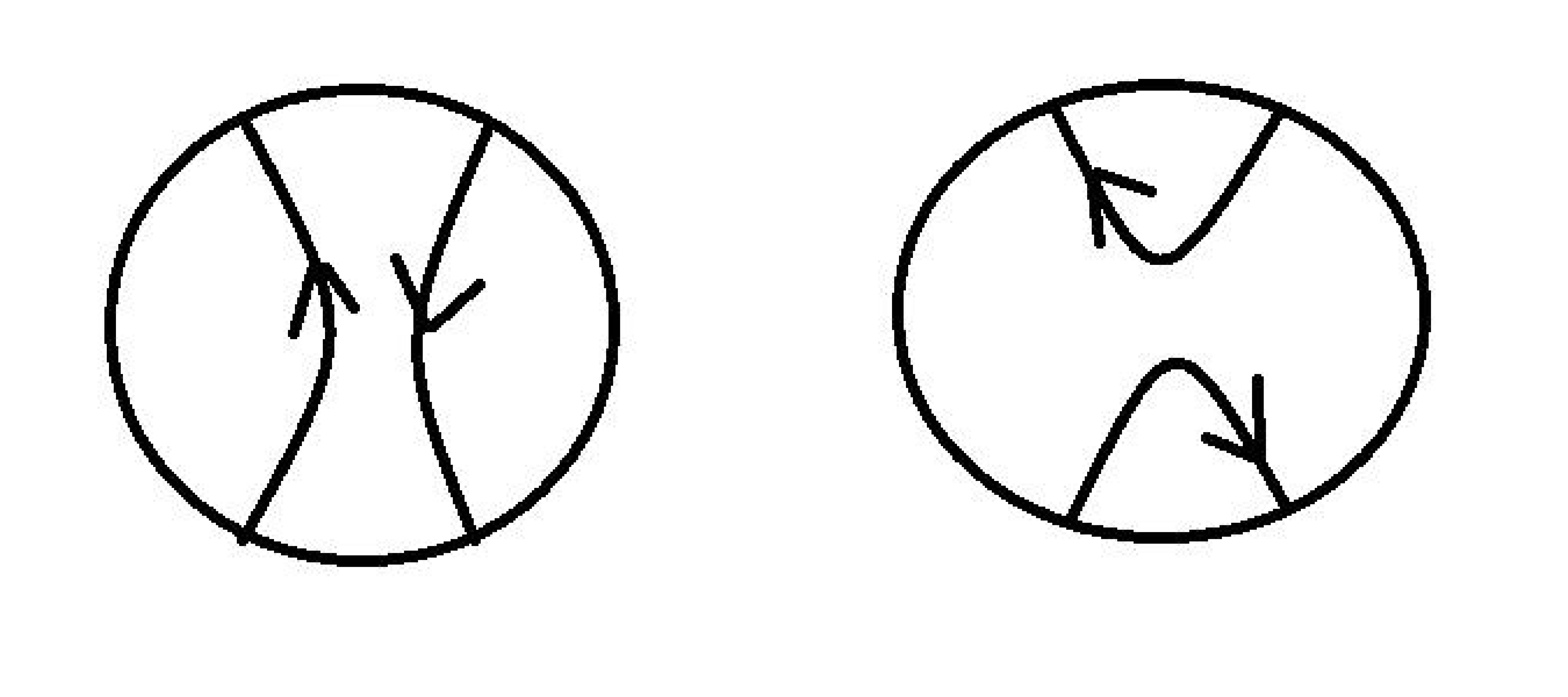}}
\end{picture}
\caption{$d_1,d_2$}
\label{figure_proof_w}
\end{figure}

\end{proof}

\begin{cor}
Let $e$ be an embedding from $\Sigma \times I \to \Sigma' \times I$.
Then, we have $w(L)=w(e(L))$ for any oriented framed link $L$
in $\Sigma \times I$
\end{cor}

\begin{cor}
\label{cor_embedding_AS}
Let $e$ be an embedding from $\Sigma \times I \to \Sigma' \times I$.
The embedding $e$ induces $e_*: \tskein{\Sigma} \to \tskein{\Sigma' }$,
$e_* :\tzeroskein{\Sigma} \to \tzeroskein{\Sigma' }$
and $e_*: \skein{\Sigma} \to \skein{\Sigma' }$.
Then we have $\psi'_\mathcal{AS} (e_* (x))=e_* (\psi'_\mathcal{AS} (x))$
for any $x \in \tzeroskein{\Sigma}$.

\end{cor}

Since $w(L_1 L_2) =w(L_1)+w(L_2)$ for null homologous oriented framed links
$L_1$ and $L_2$, $\psi'_\mathcal{AS}:\tzeroskein{\Sigma} \to \skein{\Sigma}$
 is an algebra homomorphism.
We define $\psi_\mathcal{AS}  \defeq (A+\gyaku{A}) \psi'_\mathcal{AS}$. 

\begin{prop}
The $\Q$-algebra homomorphism $\psi_\mathcal{AS}$
is a Lie algebra homomorphism.
\end{prop}

\begin{proof}
It suffices to show $\psi_\mathcal{AS} (x) \psi_\mathcal{AS} (y)-
\psi_\mathcal{AS} (y) \psi_\mathcal{AS} (x) -
(-A+A^{-1}) \psi_\mathcal{AS} ([x,y])=0$ for 
$x,y \in \tzeroskein{\Sigma}$.
We have
\begin{align*}
&(A+A^{-1})^2 (\psi_\mathcal{AS} (x) \psi_\mathcal{AS} (y)-
\psi_\mathcal{AS} (y) \psi_\mathcal{AS} (x) -
(-A+A^{-1}) \psi_\mathcal{AS} ([x,y])) \\
&= (\psi'_\mathcal{AS} (x) \psi'_\mathcal{AS} (y)-
\psi'_\mathcal{AS} (y) \psi_\mathcal{AS} (x)) -
(-A^2+A^{-2}) \psi'_\mathcal{AS} ([x,y]) \\
&=\psi'_\mathcal{AS} (xy-yx -h [x,y])=0
\end{align*}
Since $\skein{\Sigma}$ is a $\Q [A^{\pm 1}]$-free module,
we obtain 
$\psi_\mathcal{AS} (x) \psi_\mathcal{AS} (y)-
\psi_\mathcal{AS} (y) \psi_\mathcal{AS} (x) -
(-A+A^{-1}) \psi_\mathcal{AS} ([x,y])=0$.
This proves the proposition.
\end{proof}

Let $\Sigma_{g,1}$ be a compact connected oriented surface of genus $g$
with nonempty connected boundary.
There exist embeddings
\begin{align*}
&\zeta_\mathrm{Gol} : \mathcal{I} (\Sigma_{g,1}) 
\to (F^3 \widehat{\Q \hat{\pi} (\Sigma_{g,1})}, \bch), \\
&\zeta_\mathcal{S} :\mathcal{I} (\Sigma_{g,1})
\to (F^3 \widehat{ \skein{\Sigma_{g,1}}}, \bch).
\end{align*}
We remark we introduce the embedding $\zeta_\mathcal{S}$ in 
\cite{TsujiTorelli}.
We define some filtrations $\filtn{\mathcal{I}_g(n)}$,
$\filtn{\mathcal{M}_g (n)}$, $\filtn{\mathcal{M}^\mathcal{S}_g (n)}$
 and $\filtn{\mathcal{M}^\mathcal{A}_g (n)}$
by $\mathcal{I}_g (0) \defeq \mathcal{I} (\Sigma_{g,1})$,
$\mathcal{I}_g (n)=[\mathcal{I}_g(n-1), \mathcal{I} (\Sigma_{g,1})]$,
$\mathcal{M}_g(n)=\zeta_\mathrm{Gol}^{-1} (F^{n+3} \widehat{\Q \hat{\pi} (\Sigma_{g,1})})$, $\mathcal{M}^\mathcal{S}_g (n)
=\zeta_\mathcal{S}^{-1} (F^{n+3} \widehat{\skein{\Sigma_{g,1}}}$
and  $\mathcal{M}^\mathcal{A}_g (n)
=\zeta_\mathcal{A}^{-1} (F^{n+3} \widehat{\skein{\Sigma_{g,1}}}$
where $\zeta_\mathcal{A} =\zeta$.
We remark that $\filtn{\mathcal{M}_g (n)}$ is the Johnson filtration.
By definition, we have $\mathcal{I}_g (n) \subset \mathcal{M}_g(n)$,
$\mathcal{I}_g(n) \subset \mathcal{M}^\mathcal{S}_g (n)$
and $\mathcal{I}_g(n) \subset \mathcal{M}^\mathcal{A}_g (n)$.

Let $\psi_{\mathcal{AG}}:\tskein{\Sigma} \to \widehat{\Q \hat{\pi}} (\Sigma)$
be $\Q$-module homomorphism defined by
\begin{itemize}
\item $\psi_\mathcal{AG} (h)=0$, $\psi_\mathcal{AG} (\rho)=0$,
\item $\psi_\mathcal{AG} (L)$ is the conjugacy class of $L$ when $\sharp \pi_0 (L)=1$,
\item $\psi_\mathcal{AG} (L)=0$ when $\sharp \pi_0 (L) \geq 2$.
\end{itemize}
Then the $\Q$-module homomorphism $\psi_{\mathcal{AG}}$
 is a Lie algebra homomorphism 
and $\psi_{\mathcal{AG}} (F^n \tskein{\Sigma}) \subset
F^n \Q \hat{\pi} (\Sigma)$.
Using the Lie algebra homomorphism $\psi_\mathcal{AG}$,
we obtain $\mathcal{M}^\mathcal{A}_g (n) \subset
\mathcal{M}_g(n)$ for any $n \in \Z_{\geq 0}$.

\begin{prop}
\label{prop_filtration_Torelli}
We have 
$\mathcal{M}_g (1) = \mathcal{M}^\mathcal{S}_g(1)= \mathcal{M}^\mathcal{A}_g(1)$.
\end{prop}

\begin{proof}
Theorem 3.15 \cite{TsujiTorelli} implies that
$\mathcal{M}_g (1) = \mathcal{M}^\mathcal{S}_g(1)$.
It suffices to show that
$\mathcal{M}_g (1) \subset \mathcal{M}^\mathcal{A}_g (1)$.
By \cite{JohnsonKernel},
$\mathcal{M}_g (1)$ is generated by
Dehn twists along null homologous simple closed curves.
Since $\Lambda (c) \in F^4 \widehat{\tskein{\Sigma_{g,1}}}$
for null homologous simple closed curve $c$,
 we obtain 
$\mathcal{M}_g (1) \subset \mathcal{M}^\mathcal{A}_g (1)$.
This proves  the proposition.
\end{proof}

Since $\mathcal{M}_g (1) = \mathcal{M}^\mathcal{S}_g(1)= \mathcal{M}^\mathcal{A}_g(1)$
is generated by Dehn twists  along null homologous 
simple closed curves, we obtain
$\zeta_\mathcal{A} (\mathcal{M}^\mathcal{A}_g (1)) \subset
\tzeroskein{\Sigma_{g,1}}$.
So $\psi_\mathcal{AS}$ induces the group homomorphism
$\psi_\mathcal{AS}: (\zeta_\mathcal{A} (\Sigma_{g,1})\cap F^4 \widehat{\tskein{\Sigma_{g,1}}}, \bch)
\to (\zeta_\mathcal{S} (\Sigma_{g,1}) \cap F^4 \widehat{\skein{\Sigma_{g,1}}}, \bch)$.
By Theorem \ref{thm_filtration_AS}, we have the following.

\begin{thm}
We have $\mathcal{I}_g (n) \subset \mathcal{M}_g^\mathcal{A} (n)
\subset \mathcal{M}_g^\mathcal{S} (n) \cap \mathcal{M}_g (n)$.
\end{thm}

\section{Integral homology 3-spheres}
We choose the Heegard spliting $S^3 =H_g^+ \cup_\varphi H_g^-$
of $S^3$
where $\varphi : \Sigma_g =\partial H_g^+ \to \partial H_g^-$
is a diffeomorphism. We fix an closed disk $D$ in the boundary 
$\partial H_g^+=\partial H_g^-$
and denote $\Sigma_{g,1} \defeq \partial H_g^+ \backslash \mathrm{inn} D$.
The embedding $\Sigma_{g,1} \hookrightarrow S^3$ 
determines two natural subgroups of 
the mapping class group of $
\mathcal{M}(\Sigma_{g,1})
$
namely
\begin{equation*}
\mathcal{M}(H_{g,1}^\epsilon)
\defeq \mathrm{Diff}^+ (H_{g}^\epsilon, D)/
\mathrm{Diff}_0(H_{g}^\epsilon, D).
\end{equation*}
for $\epsilon \in \shuugou{+,-}$.
We denote $M(\xi) \defeq H_g^+ \cup_{\varphi \circ \xi} H_g^-$
for any mapping class $\xi \in \mathcal{M} (\Sigma_{g,1})$.
We remark that
there is an injective stabilization map $\mathcal{M}(\Sigma_{g,1}) \hookrightarrow
\mathcal{M}(\Sigma_{g+1,1})$, which is
compatible with  the above two subgroups
$\mathcal{M}(H_{g,1}^+)$ and $\mathcal{M} (H_{g,1}^-)$.

\begin{fact}[For example, see \cite{Morita1989} \cite{Pitsch2008}\cite{Pitsch2009}]
There exist a bijective map
\begin{equation*}
\underrightarrow{\lim}_{g \rightarrow \infty} (\mathcal{I}(\Sigma_{g,1})/\sim )
\to \mathcal{H}(3),\xi \mapsto M(\xi)
\end{equation*}
where $\mathcal{H}(3)$ is the set of integral homology
$3$-spheres.
Here, for $\xi_1$ and $\xi_2 \in \mathcal{I}(\Sigma_{g,1})$,
$\xi_1 \sim \xi_2$ if and only if there exist
$\eta^+ \in \mathcal{M}(H_{g,1}^+)$ and $\eta^- \in \mathcal{M}(H_{g,1}^-)$
satisfying $\xi_1 =\eta^- \xi_2 \eta^+$.
\end{fact}

Let $G 
$  be the subgroup 
of $\mathcal{M} (\Sigma_{g,1})$generated by
$
\shuugou{h_i|i \in \shuugou{1,2, \cdots, g}} \cup
\shuugou{s_{ij}|i \neq j}.
$
Here we denote by $h_i$ and $s_{ij}$
the half twist along $c_{h,i}$ as in Figure \ref{figure_handle_G_h_i}
and the element $t_{c_{i,j}}{t_{c_{a,i}}}^{-1}{t_{c_{b,j}}}^{-1}$
as in Figure \ref{figure_handle_G_S_j_i}
and Figure \ref{figure_handle_G_S_i_j}, respectively.

\begin{figure}
\begin{picture}(300,140)
\put(0,0){\includegraphics[width=300pt]{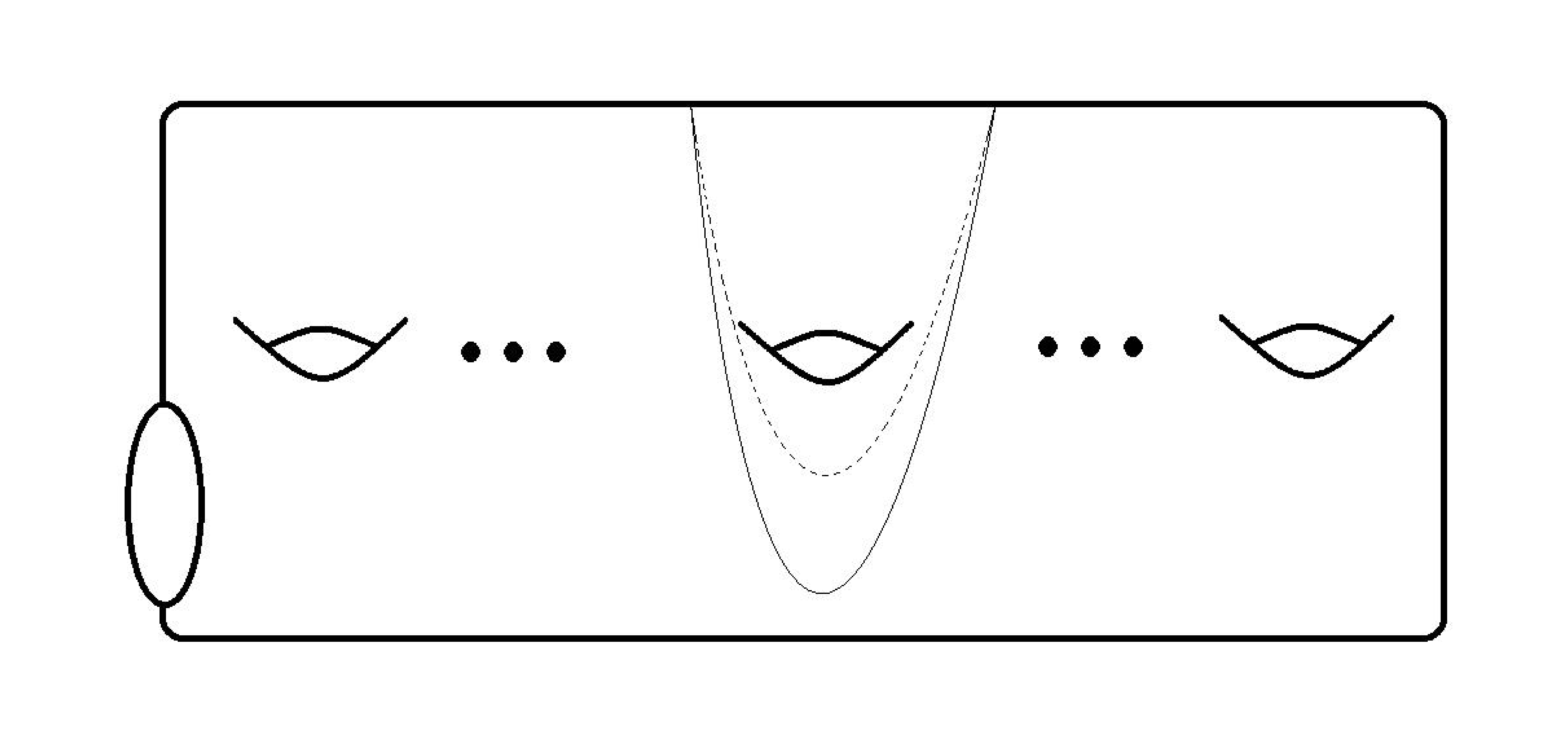}}
\put(37,90){$H_g^+$}
\put(57,90){$g$}
\put(250,90){$1$}
\put(160,90){$i$}
\put(180,63){$c_{h,i}$}
\put(14,90){$H_g^-$}
\end{picture}
\caption{$c_{h,i}$}
\label{figure_handle_G_h_i}
\end{figure}

\begin{figure}
\begin{picture}(300,140)
\put(0,0){\includegraphics[width=300pt]{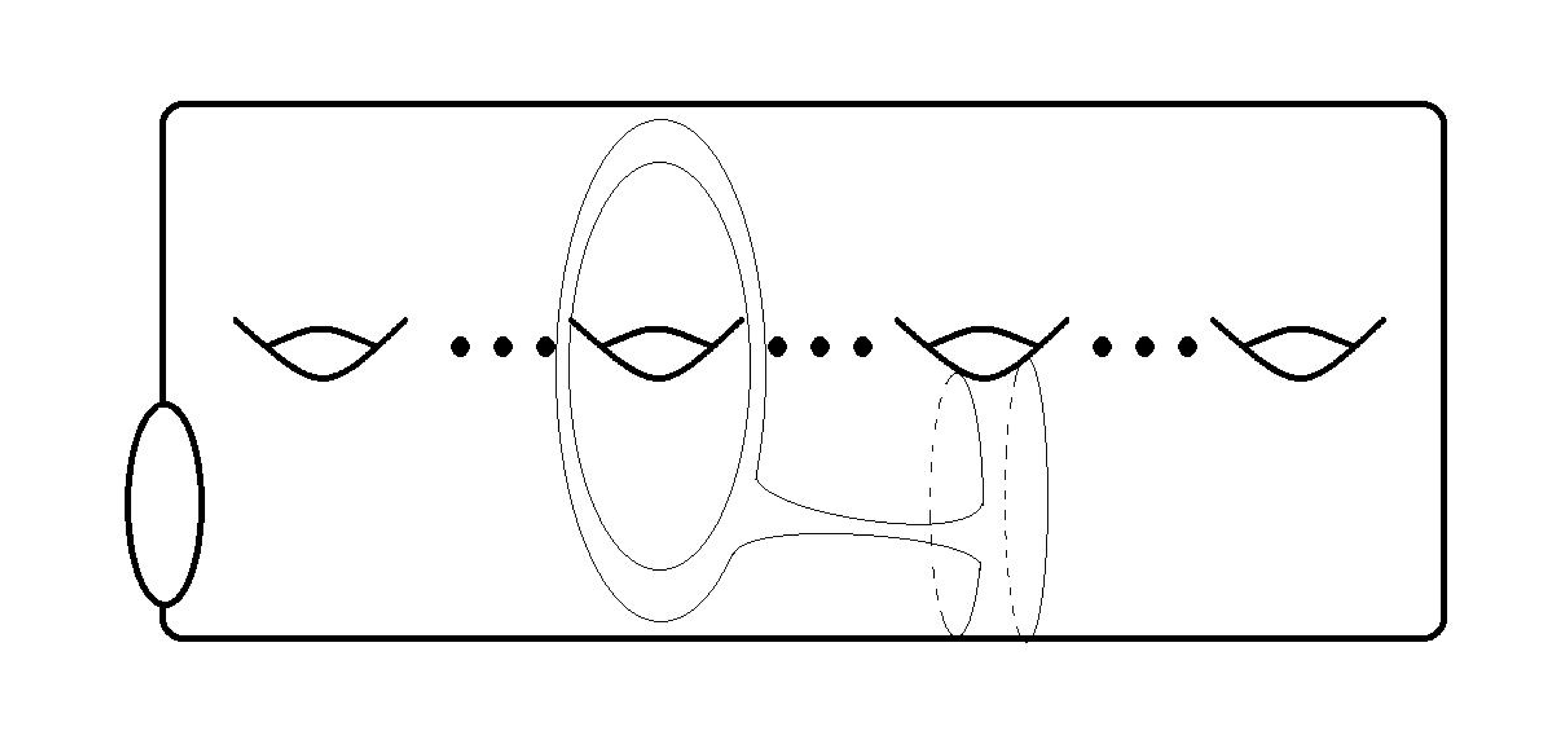}}
\put(37,90){$H_g^+$}
\put(57,90){$g$}
\put(250,90){$1$}
\put(204,43){$c_{a,i}$}
\put(160,27){$c_{i,j}$}
\put(120,45){$c_{b,j}$}
\put(187,90){$i$}
\put(125,90){$j$}
\put(14,90){$H_g^-$}
\end{picture}
\caption{$c_{a,i}$, $c_{b,j}$ and $c_{i,j}$ for $j>i$}
\label{figure_handle_G_S_j_i}
\end{figure}

\begin{figure}
\begin{picture}(300,140)
\put(0,0){\includegraphics[width=300pt]{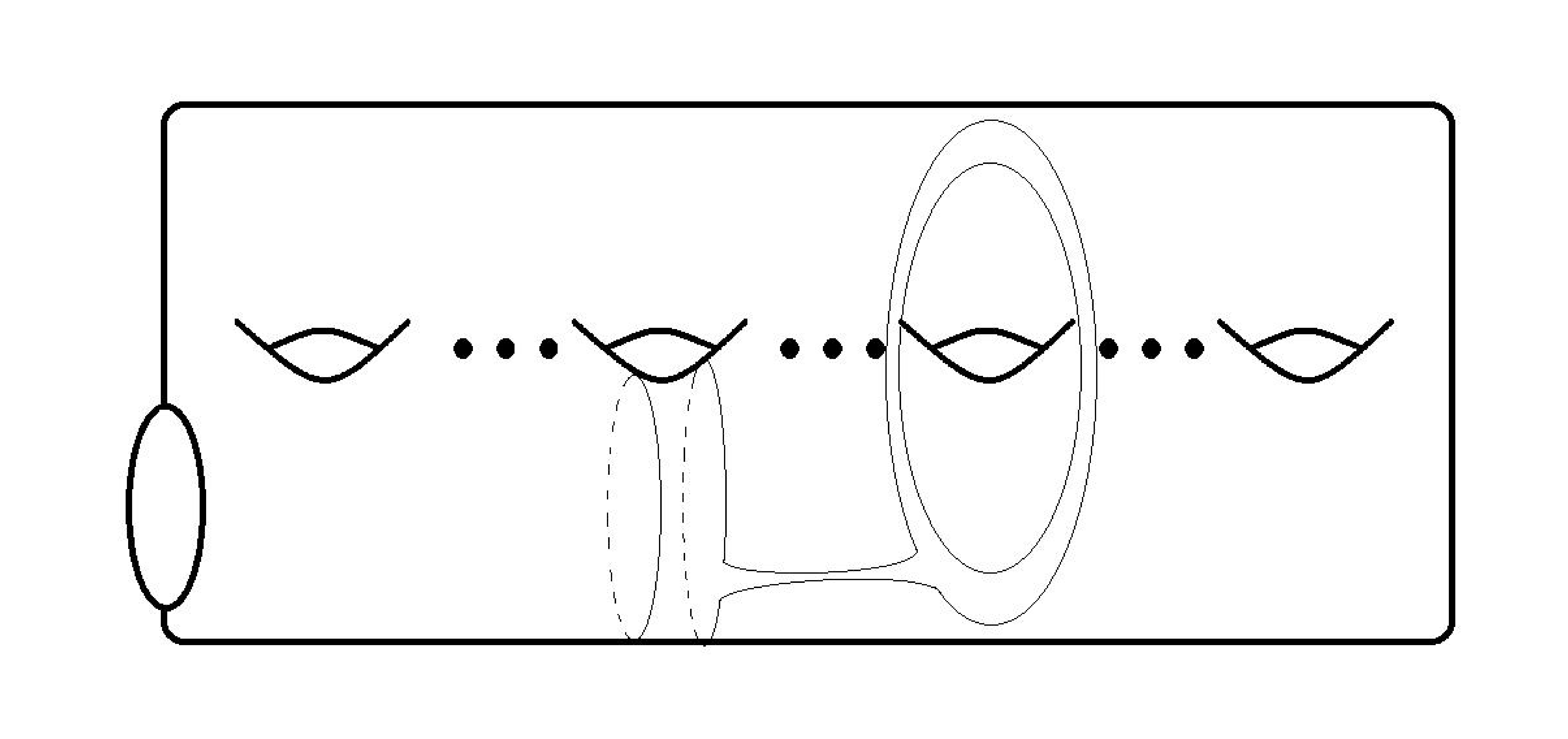}}
\put(37,90){$H_g^+$}
\put(57,90){$g$}
\put(250,90){$1$}
\put(213,63){$c_{i,j}$}
\put(180,53){$c_{b,j}$}
\put(100,43){$c_{a,i}$}
\put(187,90){$j$}
\put(125,90){$i$}
\put(14,90){$H_g^-$}
\end{picture}
\caption{$c_{a,i}$, $c_{b,j}$ and $c_{i,j}$ for $i>j$}
\label{figure_handle_G_S_i_j}
\end{figure}

\begin{lemm}[\cite{Tsujihom3} Lemma 2.5]
\label{lemm_torelli_check}
The equivalence relation $\sim$ in $\mathcal{I}(\Sigma_{g,1})$ is generated by
the relations
$\xi \sim \eta_{G} \xi {\eta_{G}}^{-1}$ for $\eta_G \in \shuugou{h_i, s_{i,j}}$,
$\xi \sim \xi \eta^+$ for $\eta^+ \in 
\shuugou{t_{\xi(c_a) \xi( c'_a)}|\xi \in \mathcal{M} (H_{g,1}^+)}
\cup \shuugou{t_{\xi(c_b) \xi(c'_b)}|\xi \in \mathcal{M} (H_{g,1}^+)}$
and 
$\xi \sim \eta^- \xi$ for $\eta^- \in 
\shuugou{t_{\xi(c_a) \xi( c'_a)}|\xi \in \mathcal{M} (H_{g,1}^-)}
\cup \shuugou{t_{\xi(c_b) \xi(c'_b)}|\xi \in \mathcal{M} (H_{g,1}^-)}$
\end{lemm}

\section{An invariant for integral homology 3-spheres}
Let $e$ be an orientation preserving embedding $\Sigma_{g,1} \times [0,1]$ satisfying
\begin{align*}
&e_{|\Sigma_{g,1} \times \shuugou{\frac{1}{2}}}:\Sigma \times \shuugou{\frac{1}{2}}
\to \Sigma, (x,\frac{1}{2}) \mapsto x, \\
\end{align*}
We denote by $e_*$ the $\Q[\rho][[h]]$-module homomorphism
$\widehat{\tskein{\Sigma_{g,1}}} \to \Q[\rho][[h]]$ induced by $e$.

The aim of this section is to prove the following.

\begin{thm}
\label{thm_main}
The map $Z: \mathcal{I}(\Sigma_{g,1}) \to \Q[\rho][[h]]$
defined by 
\begin{equation*}
Z(\xi) \defeq \sum_{i=0}^\infty \frac{1}{(h)^i i!}e_*
((\zeta (\xi))^i)
\end{equation*}
induces
\begin{equation*}
z:\mathcal{H} (3) \to \Q[\rho][[h]], M(\xi) \to Z(\xi).
\end{equation*}

\end{thm}

By Proposition \ref{prop_filt_disk}, the map
$Z: \mathcal{I}(\Sigma_{g,1}) \to \Q[\rho][[h]]$ is well-defined.

For $\epsilon \in \shuugou{+,-}$,
let $\tskein{H_g^\epsilon}$ be the quotient of
$\Q [\rho][[h]] \mathcal{T} (H_g^\epsilon)$
modulo the skein relation, the trivial knot relation and
the framing relation ,
where $\mathcal{T}(H_g^\epsilon)$
is the set of oriented framed links
in $H_g^\epsilon$. We can consider  its completion
$
\widehat{\tskein{H_g^\epsilon}}$.
We denote the embedding $\iota'^+ :
\Sigma_{g,1}  \times [0,\frac{1}{2}] \to{H_g^+}$
and the embedding $\iota'^- :
\Sigma_{g,1}  \times [\frac{1}{2},1] \to {H_g^-}$.
The embeddings
\begin{align*}
&\iota^+:\Sigma_{g,1} \times I  \to H_g^+, (x,t) \mapsto \iota'^+(x,\frac{t}{2}), \\
&\iota^-:\Sigma_{g,1} \times I  \to H_g^-, (x,t) \mapsto \iota'^+(x,\frac{t+1}{2}) \\
\end{align*}
induces
\begin{align*}
&\iota^+:\widehat{\tskein{\Sigma_{g,1}}} \to \widehat{\tskein{H_g^+}}, \ \
\iota^-:\widehat{\tskein{\Sigma_{g,1}}} \to \widehat{\tskein{H_g^-}}. \\
\end{align*}

By definition, we have the followings.

\begin{prop}
\label{prop_ideal}
\begin{enumerate}
\item The kernel of $\iota^+$ is a right ideal
of $\widehat{\tskein{\Sigma_{g,1}}}$.
\item The kernel of $\iota^-$ is a left ideal
of $\widehat{\tskein{\Sigma_{g,1}}}$.
\item We have $e_* (\ker \iota^\epsilon) =\shuugou{0}$
for $\epsilon \in \shuugou{+,-}$.
\end{enumerate}
\end{prop}

\begin{prop}
\label{prop_bch_Z}
We have $Z(\xi_1 \xi_2) =\sum_{i,j \geq 0}\frac{1}{h^{i+j} i! j!}e_* ((\zeta(\xi_1))^i
(\zeta(\xi_2))^j)$.
\end{prop}

\begin{lemm}
\label{lemm_proof_key}
\begin{enumerate}
\item For $\epsilon \in \shuugou{+,-}$, 
we have 
\begin{align*}
\iota^\epsilon (\xi(\Lambda(c_a)-\Lambda(c'_a))=0,
\iota^\epsilon (\xi(\Lambda(c_b)-\Lambda(c'_b))=0
\end{align*}
for $\xi \in \mathcal{M}(H_{g,1}^\epsilon)$.
\item For $\epsilon \in \shuugou{+,-}$, 
we have 
\begin{align*}
\iota^\epsilon (\Lambda(c_{i,j})-\Lambda(c_{a,i})-\Lambda(c_{b,j}))=0
\end{align*}for $i \neq j$.
\end{enumerate}

\end{lemm}

By Lemma \ref{lemm_torelli_check},
in order to prove Theorem \ref{thm_main},
it suffices to check the following lemmas.

\begin{lemm}
For any $i$, we have $e_* \circ h_i =e_*$
\end{lemm}
\begin{proof}
The embeddings $e \circ h_i $ and $e$ are
isotopic in $S^3$.
This proves the lemma.
\end{proof}

\begin{lemm}
For any $i \neq j$, we have $e_* \circ s_{ij} =e_*$.
\end{lemm}

\begin{proof}
We fix an element $x$ of $\widehat{\tskein{\Sigma_{g,1}}}$.
We have $s_{ij} (x) =\exp (\sigma(\Lambda(c_{i,j})-\Lambda(c_{a,i})-\Lambda(c_{b,j})))(x)$.
Using Lemma \ref{lemm_proof_key}(2) and Proposition
\ref{prop_ideal} (1)(2)(3),
we have $e_*(\exp (\sigma(\Lambda(c_{i,j})-\Lambda(c_{a,i})
-\Lambda(c_{b,j})))(x))=e_*(x)$.
This  proves the lemma.
\end{proof}

\begin{lemm}
\begin{enumerate}
\item We have $Z(\xi \eta^+) =Z(\xi)$ for any 
$\xi \in \mathcal{I}(\Sigma_{g,1})$ and 
any $\eta^+ \in \shuugou{t_{\xi(c_a) \xi( c'_a)}|\xi \in \mathcal{M} (H_{g,1}^+)}
\cup \shuugou{t_{\xi(c_b) \xi(c'_b)}|\xi \in \mathcal{M} (H_{g,1}^+)}$.
\item We have $Z(\eta^- \xi ) =Z(\xi)$ for any 
$\xi \in \mathcal{I}(\Sigma_{g,1})$ and 
any $\eta^- \in \shuugou{t_{\xi(c_a) \xi( c'_a)}|\xi \in \mathcal{M} (H_{g,1}^-)}
\cup \shuugou{t_{\xi(c_b) \xi(c'_b)}|\xi \in \mathcal{M} (H_{g,1}^-)}$.
\end{enumerate}
\end{lemm}

\begin{proof}
We prove only (1) (because the proof of (2) is almost the same.)
By Proposition \ref{prop_bch_Z}, we have 
$Z(\xi \eta^+) =\sum_{i,j \geq 0}\frac{1}{h^{i+j} i! j!}e_* ((\zeta(\xi))^i
(\zeta(\eta^+))^j)$.
Using Lemma \ref{lemm_proof_key} and Proposition \ref{prop_ideal} (1)(3),
we obtain
\begin{equation*}
Z(\xi \eta^+) =\sum_{i,j \geq 0}\frac{1}{h^{i+j} i! j!}e_* ((\zeta(\xi))^i
(\zeta(\eta^+))^j)=\sum_{i \geq 0}\frac{1}{h^{i} i!}e_* ((\zeta(\xi))^i)=Z(\xi).
\end{equation*}
This proves the lemma.
\end{proof}

This invariant satisfies the following.

\begin{prop}
For $\xi_1 \in \zeta^{-1} (F^{n_1+2}\widehat{\tskein{\Sigma_{g,1}}}), 
\cdots, \xi_k \in \zeta^{-1}(F^{n_k+2} \widehat{\tskein{\Sigma_{g,1}}})$,
we have
\begin{equation*}
\sum_{\epsilon_i \in \shuugou{1,0}}(-1)^{\sum \epsilon_i}
z (M(\xi_1^{\epsilon_1} \cdots \xi_k^{\epsilon_k}))
\in h^{\gauss{n_1 + \cdots +n_k}}\Q [\rho][[h]].
\end{equation*}

\end{prop}

\begin{proof}
We have
\begin{align*}
&\sum_{\epsilon_i \in \shuugou{1,0}}(-1)^{\sum \epsilon_i}
z (M(\xi_1^{\epsilon_1} \cdots \xi_k^{\epsilon_k})) \\
&=e_*((1-\exp (\frac{\zeta(\xi_1)}{h}))\cdots(1-\exp (\frac{\zeta(\xi_k)}{h})).
\end{align*}
By Proposition \ref{prop_filt_disk}, we have
\begin{equation*}
\sum_{\epsilon_i \in \shuugou{1,0}}(-1)^{\sum \epsilon_i}
z (M(\xi_1^{\epsilon_1} \cdots \xi_k^{\epsilon_k}))
\in h^{\gauss{n_1 + \cdots +n_k}}\Q [\rho][[h]].
\end{equation*}
This proves the proposition.

\end{proof}

\begin{cor}
\label{cor_finite_type}
The invariant $z(M) \in \Q[\rho][[h]]/(h^n)$
is a finite type invariant for $M \in \mathcal{H} (3)$
order $n$.
\end{cor}

In \cite{Tsujihom3}, 
we define $z_\mathcal{S} :\mathcal{H} (3) \to \Q [[A+1]]$ by
\begin{equation*}
M (\xi) \mapsto  \sum_{i=0}^\infty \frac{1}{i!(-A+A^{-1})^i}e_* ( (\zeta (\xi))^i).
\end{equation*}
By Corollary \ref{cor_embedding_AS} and Proposition \ref{prop_filtration_Torelli},
we have the following.

\begin{thm}
We have $(z_\mathcal{A} (M))_{| \exp (\rho h)=A^4. h=-A^2+A^{-2}}=z_\mathcal{S} (M)$
for any integral homology $3$-sphere $M$,
where $z_\mathcal{A} =z$.
\end{thm}


\end{document}